\documentclass[a4paper,oneside,11pt]{article}
\usepackage{xcolor}
\usepackage[a4paper]{geometry}
\usepackage{aeguill}                
\usepackage{graphicx}
\usepackage{amsmath,amsfonts,amssymb,amsthm}
\usepackage{pstricks,comment} 
\usepackage{epstopdf,enumerate}
\usepackage{srcltx}
\usepackage[utf8]{inputenc} 
\usepackage[english]{babel} 
\usepackage{hyperref} 
\hypersetup{
    colorlinks=true,
    linkcolor=blue,
    filecolor=blue,      
    urlcolor=blue,
    citecolor=blue
  }

\title{Measure equivalence rigidity of $\Out(F_N)$}
\author{Vincent Guirardel and Camille Horbez}

\usepackage[all]{xy}
\usepackage{bbm}

\newcommand{\un}{\mathbbm{1}}

\newcommand{\ie}{i.~e.\ }
\newcommand{\calk}{\mathcal{K}}
\newcommand{\Prob}{\mathrm{Prob}}
\newcommand{\imp}{\Rightarrow}
\newcommand{\ra}{\rightarrow}
\newcommand{\m}{^{-1}}
\newcommand{\dunion}{\sqcup}
\newcommand{\Dunion}{\bigsqcup}
\newcommand{\eps}{\varepsilon}
\renewcommand{\epsilon}{\varepsilon}

\newcommand{\calf}{\mathcal{F}}

\newcommand{\calm}{\mathcal{M}}
\newcommand{\calt}{\mathcal{T}}

\newcommand{\caly}{\mathcal{Y}}
\newcommand{\calb}{\mathcal{B}}
\newcommand{\calq}{\mathcal{Q}}
\newcommand{\calh}{\mathcal{H}}
\newcommand{\calp}{\mathcal{P}}
\newcommand{\cala}{\mathcal{A}}
\newcommand{\calx}{\mathcal{X}}
\newcommand{\cald}{\mathcal{D}}
\newcommand{\calo}{\mathcal{O}}
\newcommand{\calc}{\mathcal{C}}
\newcommand{\cals}{\mathcal{S}}
\newcommand{\Z}{\mathcal{Z}}
\newcommand{\bbR}{\mathbb{R}}
\newcommand{\bbP}{\mathbb{P}}

\newcommand{\bbN}{\mathbb{N}}
\newcommand{\bbZ}{\mathbb{Z}}

\newcommand{\actson}{\curvearrowright}
\newcommand{\es}{\emptyset}
\newcommand{\grp}[1]{\langle #1 \rangle}
\newcommand{\ngrp}[1]{\langle\langle #1 \rangle\rangle}

\newcommand{\rank}{\mathrm{rank}}

\makeatletter
\edef\@tempa#1#2{\def#1{\mathaccent\string"\noexpand\accentclass@#2 }}
\@tempa\rond{017}
\makeatother
\newcommand{\cale}{\mathcal{E}}
\newcommand{\Stab}{\mathrm{Stab}}
\newcommand{\Mod}{\mathrm{Mod}}

\newcommand{\id}{\mathrm{id}}
\newcommand{\Out}{\mathrm{Out}}
\newcommand{\Aut}{\mathrm{Aut}}
\newcommand{\Inn}{\mathrm{Inn}}

\newcommand{\AT}{\mathcal{AT}}

\newcommand{\calz}{\mathcal{Z}}
\newcommand{\caln}{\mathcal{N}}

\usepackage{mathrsfs}
\usepackage{euscript}
\newcommand{\TT}{\EuScript{T}}

\newtheorem{de}{Definition} [section]
\newtheorem{theo}[de]{Theorem} 
\newtheorem{prop}[de]{Proposition}
\newtheorem{lemma}[de]{Lemma}

\newtheorem{cor}[de]{Corollary}

\newtheorem{claim}[de]{Claim}
\newtheorem*{claim*}{Claim}

\theoremstyle{remark}
\newtheorem{rk}[de]{Remark}
\newtheorem{ex}[de]{Example}

\newcommand{\Name}{\text{Name}}

\newcommand{\ad}{\mathrm{ad}}
\newcommand{\calg}{\mathcal{G}}

\newcommand{\Fix}{\mathrm{Fix}}

\newcommand{\zmax}{{\mathcal{Z}_{\mathrm{max}}}}
\newcommand{\Zmax}{{\mathcal{Z}_{\mathrm{max}}}}
\newcommand{\IA}{\mathrm{IA}_N(\bbZ/3\bbZ)}

\newcommand{\Tw}{\mathrm{Tw}}

\newcommand{\FF}{\mathrm{FF}}
\newcommand{\stab}{\mathrm{Stab}}

\newcommand{\ATa}{\AT^{\mathrm{a}}}

\newcommand{\baro}{\overline{\mathcal{O}}}
\newcommand{\PAT}{\mathbb{P}\AT}
\newcommand{\PATsim}{\PAT\!/{\scriptstyle\sim}}
\newcommand{\PATa}{\mathbb{P}\ATa}

\newcommand{\PaPATsim}{\calp_{<\infty}^{\mathrm{amen}}(\PAT\!/{\scriptstyle\sim})}
\newcommand{\PnaPATsim}{\calp_{<\infty}^{\mathrm{na}}(\PAT\!/{\scriptstyle\sim})}
\newcommand{\PaPAT}{\calp_{<\infty}^{\mathrm{amen}}(\PAT)}
\newcommand{\PaPATN}{\calp_{=N}^{\mathrm{amen}}(\PAT)}
\newcommand{\amen}{\mathrm{amen}}
\newcommand{\PnaPAT}{\calp_{<\infty}^{\mathrm{na}}(\PAT)}

\newcommand{\PATnasimd}{\calp_{\le 2}^{\mathrm{na}}(\PAT\!/{\scriptstyle \sim})}

\newcommand{\Da}{(\cald^N)_{\mathrm{amen}}}

\newcommand{\FS}{\mathrm{FS}}
\newcommand{\NS}{\mathrm{FS}^{\mathrm{ns}}}
\newcommand{\FSii}{\mathrm{FS}^{2,2^+}}

\newcommand{\ZmaxS}{\Zmax\mathrm{S}}

\newcommand{\Pbaro}{\mathbb{P}\baro}

\newcommand{\Inc}{\mathrm{Inc}}

\newcommand{\lk}{\mathrm{lk}}
\newcommand{\ia}{\mathrm{IA}_k(\mathbb{Z}/3\mathbb{Z})}

\newcommand{\Isom}{\mathrm{Isom}}
\newcommand{\Conv}{\mathrm{Conv}}

\newcommand{\ol}{\overline}

\newcommand{\FSG}{\mathbb{FS}}
\newcommand{\FFG}{\mathbb{FF}}
\newcommand{\NSG}{\FSG^{\mathrm{ns}}}
\newcommand{\Maps}{\mathrm{Inj}(\NSG\to\FSG)}
\newcommand{\Turn}{\mathrm{Turn}}
\newcommand{\Cay}{\mathrm{Cay}}
\newcommand{\afs}{almost free factor system}
\newcommand{\HF}{{\hat\calf}}
\newcommand{\AHF}{A_{\HF}}
\newcommand{\THF}{T_{\HF}}
\newcommand{\onto}{\twoheadrightarrow}
\newcommand{\semidirect}{\ltimes}
\newcommand{\UZ}{U^{z}}
\newcommand{\Uun}{U^{1}}
\newcommand{\Ud}{U^{d}}
\newcommand{\normal}{\triangleleft}
\newcommand{\Inj}{\mathrm{Inj}}

\renewcommand{\phi}{\varphi}

\renewcommand{\sqsupseteq}{\succeq}
\renewcommand{\sqsubseteq}{\preceq}
\newcommand{\sqsupsetneq}{\succneqq}
\newcommand{\sqsubsetneq}{\precneqq}

\renewcommand{\subset}{\subseteq}
\renewcommand{\supset}{\supseteq}

\newcommand{\PI}{\hyperlink{Pi}{\ensuremath{(P_1)}}}
\newcommand{\PImax}{\hyperlink{Pimax}{\ensuremath{(P_1^{\max})}}}
\newcommand{\PII}{\hyperlink{Pii}{\ensuremath{(P_2)}}}
\newcommand{\PIII}{\hyperlink{Piii}{\ensuremath{(P_3)}}}

\begin{document}
\maketitle

\begin{abstract}
  We prove that for every $N\ge 3$, the group $\Out(F_N)$ of outer automorphisms of a free group of rank $N$ is superrigid from the point of view of measure equivalence: any countable group that is measure equivalent to $\Out(F_N)$, is in fact virtually isomorphic to $\Out(F_N)$.

  We introduce three new constructions of canonical splittings associated to a subgroup of $\Out(F_N)$ of independent interest.
  They encode respectively the collection of invariant free splittings, invariant cyclic splittings, and maximal invariant free factor systems.
  Our proof also relies on the following improvement of an amenability result by Bestvina and the authors: given a free factor system $\calf$ of $F_N$, the action of $\Out(F_N,\calf)$ (the subgroup of $\Out(F_N)$ that preserves $\calf$) on the space of relatively arational trees with amenable stabilizer is a Borel amenable action.
\end{abstract}

\section{Introduction}

Measure equivalence theory is concerned with classifying countable groups up to the following equivalence relation.

\begin{de}[{Gromov \cite[Definition~0.5.$\mathrm{E}_1$]{Gro}}]
Two countable groups $\Gamma_1$ and $\Gamma_2$ are \emph{measure equivalent} if there exists a measure-preserving action of $\Gamma_1\times\Gamma_2$ by Borel automorphisms on a standard measure space $\Sigma$ (i.e.\ a Polish space equipped with its Borel $\sigma$-algebra and endowed with a nonzero $\sigma$-finite Borel measure) such that for every $i\in\{1,2\}$, the $\Gamma_i$-action on $\Sigma$ is free and has a finite measure fundamental domain, i.e.\ there exists a Borel subset $X_i\subseteq\Sigma$ of finite measure such that $$\Sigma=\coprod_{\gamma\in \Gamma_i}\gamma X_i.$$
\end{de}

A typical example is that two lattices in the same locally compact second countable group are always measure equivalent (see \cite[Example~0.5.$\mathrm{E}_2$]{Gro}).
This notion can be viewed as a measure-theoretic analogue to the notion of quasi-isometry between finitely generated groups.
Indeed, a theorem of Gromov \cite[0.2.$\mathrm{C}'_2$]{Gro} asserts that two finitely generated groups $\Gamma_1$ and $\Gamma_2$ are quasi-isometric if and only if they are \emph{topologically equivalent}, i.e.\ there exists a continuous action of $\Gamma_1\times\Gamma_2$ on some locally compact space $X$, such that for every $i\in\{1,2\}$, the $\Gamma_i$-action on $X$ is properly discontinuous and cocompact. We insist however that this remains purely at the level of analogy: there is no implication in either direction between measure equivalence and quasi-isometry.

An alternative point of view on measure equivalence comes from studying measure-preserving group actions on standard probability spaces,
where one forgets everything but the partition into orbits.
Indeed, by works of Furman \cite{Fur2} and Gaboriau \cite[Theorem~2.3]{Gab2}, two countable groups are measure equivalent if and only if they admit stably orbit equivalent free actions on standard probability spaces by measure-preserving Borel automorphisms -- see Definition~\ref{deintro:oe} below. 

An easy way in which two countable groups $\Gamma_1$ and $\Gamma_2$ can be measure equivalent is if they are \emph{virtually isomorphic}. This means that there exist finite-index subgroups $\Lambda_i\subseteq\Gamma_i$ and finite normal subgroups $F_i\unlhd \Lambda_i$ for all $i\in\{1,2\}$, such that $\Lambda_1/F_1$ is isomorphic to $\Lambda_2/F_2$. A  countable group $\Gamma$ is said to be \emph{ME-superrigid} if every countable group which is measure equivalent to $\Gamma$ is in fact virtually isomorphic to $\Gamma$. 

Given $N\in\mathbb{N}$, we let $F_N$ be a free group of rank $N$, and $\Out(F_N)$ be its outer automorphism group. The group $\Out(F_N)$ has been subject to intense study from the point of view of geometric group theory over the past decades, often in analogy to mapping class groups of surfaces. The main theorem of the present paper is the following.

\begin{theo}[see Theorem \ref{thm_superrigid}]\label{theo:intro-me}
  For every $N\ge 3$, the group $\Out(F_N)$ is ME-superrigid:
 every countable group which is measure equivalent to $\Out(F_N)$, is in fact virtually isomorphic to $\Out(F_N)$.
\end{theo}

Notice that this does not hold when $N=2$, as $\Out(F_2)$ is virtually free, and the class of groups that are measure equivalent to a free group is wide, see e.g.\ \cite{Gab,BTW}. To our knowledge, the topological counterpart to Theorem~\ref{theo:intro-me}, i.e.\ knowing whether $\Out(F_N)$ is quasi-isometrically rigid, is still wide open. Actually, the geometry of $\Out(F_N)$ equipped with the word metric is still very mysterious, in large part due to the lack of a satisfying analogue for $\Out(F_N)$ of the Masur--Minsky theory \cite{MM} yielding a distance formula for mapping class groups of surfaces. 

\subsection{Some history}

\subsubsection{Measure equivalence rigidity and flexibility} 

We would like to mention here some previously known results that are milestones in the (still growing) theory of measure equivalence rigidity. We refer the reader to \cite{Sha-survey,Gab-survey,Fur-survey} for general surveys on measured group theory.

A first striking flexibility result in the classification of groups up to measure equivalence (in an orthogonal direction to that of rigidity results) was the proof by Ornstein and Weiss \cite{OW}, following previous work of Dye \cite{Dye1,Dye2} (and further extended by Connes, Feldman and Weiss \cite{CFW}), that all  countably infinite amenable groups are measure equivalent.
 More precisely, any two free ergodic measure-preserving actions of countably infinite amenable groups on standard probability spaces are orbit equivalent. In contrast, work by Hjorth, Gaboriau, Lyons, Popa, Ioana and Epstein
shows that any non-amenable countable group has uncountably many  ergodic actions that are  pairwise not orbit equivalent \cite{Hjorth,GabPopa,GabLyons,Ioana,Epstein}.
 
Still on the flexibility side, Gaboriau showed that many groups are measure equivalent to free groups \cite{Gab} (see also \cite{BTW}, showing more generally that every finitely generated group with the same elementary theory as a free group is measure equivalent to a  finitely generated free group). 
 
A groundbreaking rigidity result was the proof by Furman \cite{Fur1}, building on earlier work of Zimmer and most notably on Zimmer's celebrated cocycle superrigidity theorem \cite{Zim3,Zim4}, that any countable group which is measure equivalent to a lattice in a higher rank simple connected Lie group $G$ with finite center, is in fact virtually isomorphic to a lattice in $G$.  

Another remarkable result by Gaboriau shows that  $\ell^2$-Betti numbers are (projective) invariants of measure equivalence \cite{Gab_L2}.
 By \cite{GN}, the $\ell^2$-Betti numbers of $\Out(F_N)$ do not vanish in degree equal to the virtual cohomological dimension $2N-3$, which gives a short proof that $\Out(F_N)$ is not measure equivalent to $\Out(F_{N'})$ if $N'\neq N$.   
Monod and Shalom \cite{MS} used bounded cohomology techniques to establish striking measure equivalence rigidity results for products of negatively curved groups.

Later, Kida established that all mapping class groups of non-exceptional finite-type orientable surfaces are ME-superrigid \cite{Kid2}. Kida's approach to measure equivalence rigidity of mapping class groups has had a large influence on our work: the general strategy of our proof is similar to his,
though many new difficulties arose for $\Out(F_N)$.

Kida's theorem has then inspired further ME-superrigidity results \cite{Kid4,CK,HH}. Also, there are situations of groups that fail to be ME-superrigid, but for which one can still get interesting measure equivalence classification results within a given class, like Baumslag--Solitar groups \cite{HR} or right-angled Artin groups \cite{HH2}.  

In another direction, superrigidity results have been obtained for certain particular actions with few restrictions on the acting group: this started with Popa's remarkable cocycle superrigidity theorem for Bernoulli actions of Property (T) groups \cite{Pop}, leading to many further developments. 

We finally mention that measure equivalence theory has strong connections with the theory of von Neumann algebras, see e.g.\ \cite{Sin}. In combination to the aforementioned work of Kida, the notion of \emph{proper proximality} of a countable group, recently introduced by Boutonnet, Ioana and Peterson \cite{BIP}, led to strong rigidity results for von Neumann algebras associated to weakly compact (e.g.\ profinite)  free probability measure-preserving ergodic actions of mapping class groups \cite{HHL}. A natural question that arises from our work is to which extent an ergodic action of $\Out(F_N)$ is determined by its associated von Neumann algebra.

\subsubsection{Rigidity in $\Out(F_N)$} 

Rigidity of $\Out(F_N)$ also has a long history, starting with works of Khramtsov \cite{Khr} and of Bridson and Vogtmann \cite{BV} establishing that all automorphisms of $\Out(F_N)$ with $N\ge 3$ are inner. This algebraic rigidity statement is in close relation to more geometric rigidity statements concerning the symmetries of various spaces equipped with natural $\Out(F_N)$-actions, starting with the spine of reduced Outer space \cite{BV2}. Later, Farb and Handel proved \cite{FH} that $\Out(F_N)$ is its own abstract commensurator for all $N\ge 4$, i.e.\ every isomorphism between finite-index subgroups of $\Out(F_N)$ is given by a conjugation. A recent new proof of this result by Wade and the second named author \cite{HW} extends it to the $N=3$ case and allows to compute the abstract commensurator of various interesting subgroups of $\Out(F_N)$  such as its Torelli subgroup. 
The form in which we prove ME-rigidity (namely Theorem \ref{theo:intro-two-cocycles})
implies commensurator rigidity (see Remark~\ref{rk_commensurator}).
Thus, our result actually recovers the Farb--Handel theorem.
In the same way that Kida's theorem is a broad generalization of Ivanov's commensurator rigidity theorem as explained in \cite[p.~29--33]{Kid-survey},
the general approach of the present paper can also be viewed in a sense as 
a 
generalization of the strategy used in \cite{HW} (see Section~\ref{sec:intro-commensurateur}). 

\subsection{Some consequences}
We now mention two applications of the techniques we develop in the present paper.

\subsubsection{Orbit equivalence rigidity of ergodic actions of $\Out(F_N)$}

Using standard techniques from measured group theory originating from work of Furman \cite{Fur2}, Theorem~\ref{theo:intro-me} can be used to deduce a strong rigidity result regarding the orbit structure of ergodic actions of $\Out(F_N)$ on standard probability spaces. Before we state it, we recall the following definition.

\begin{de}\label{deintro:oe}
Let $\Gamma_1$ and $\Gamma_2$ be two countable groups, and for every $i\in\{1,2\}$, let $\Gamma_i\actson X_i$ be an ergodic measure-preserving action by Borel automorphisms on a standard probability space $X_i$. The two actions $\Gamma_1\actson X_1$ and $\Gamma_2\actson X_2$ are \emph{orbit equivalent} if there exists a measure space isomorphism $f:X_1\to X_2$ such that for a.e.\ $x\in X_1$, one has $f(\Gamma_1\cdot x)=\Gamma_2\cdot f(x)$.

More generally, the actions $\Gamma_1\actson X_1$ and $\Gamma_2\actson X_2$ are \emph{stably orbit equivalent} if there exist Borel subsets $Y_i\subseteq X_i$ of positive measure and a measure-scaling\footnote{i.e.\ $f$ induces a measure space isomorphism after rescaling the measures to turn $Y_1$ and $Y_2$ into probability spaces}
isomorphism $f:Y_1\to Y_2$ such that for a.e.\ $y\in Y_1$, one has $f((\Gamma_1\cdot y)\cap Y_1)=(\Gamma_2\cdot f(y))\cap Y_2$. 
\end{de}

An easy example of orbit equivalent actions is when the actions $\Gamma_1\actson X_1$ and $\Gamma_2\actson X_2$ are \emph{conjugate}. This means that there exists a group isomorphism $\rho:\Gamma_1\to\Gamma_2$ and a measure space isomorphism $f:X_1\to X_2$ such that for every $\gamma\in\Gamma_1$ and a.e.\ $x\in X_1$, one has $f(\gamma\cdot x_1)=\rho(\gamma)\cdot f(x_1)$. 

\begin{theo}[OE-superrigidity, see Theorem \ref{theo:oe-out(fn)}]\label{theo:intro-oe}
Let $N\ge 3$. Let $\Gamma\subset\Out(F_N)$ be a finite index subgroup, and $\Lambda$ a countable group with no non-trivial finite normal subgroup.
Consider two standard probability spaces $X_1,X_2$ and two free measure-preserving actions $\Gamma\actson X_1$, $\Lambda\actson X_2$.
Assume that these actions are aperiodic, i.e.\ that every finite index subgroup acts ergodically.

If the actions $\Gamma\actson X_1$ and $\Lambda\actson X_2$ are stably orbit equivalent, then they are conjugate.
\end{theo}

\subsubsection{Lattice embeddings of $\Out(F_N)$ into locally compact groups}

As mentioned earlier, a typical example of measure equivalent groups comes from two lattices in a common locally compact second countable group. A consequence of our work is that there are no interesting lattice embeddings of $\Out(F_N)$ into any locally compact second countable group.

\begin{theo}[see proof page~\pageref{proof:lattice}]\label{theo:lattice-embeddings}
Let $N\ge 3$, and let $\Gamma$ be a finite-index subgroup of $\Out(F_N)$. Let $H$ be a locally compact second countable group (equipped with its left Haar measure), and let $\sigma:\Gamma\to H$ be an injective homomorphism, such that $\sigma(\Gamma)$ is a lattice in $H$ (i.e.\ $\sigma(\Gamma)$ is discrete in $H$ and has finite Haar covolume). 

Then there exists a continuous homomorphism $\Phi_0:H\to\Out(F_N)$ with compact kernel such that for all $\gamma\in\Gamma$, one has $\Phi_0(\sigma(\gamma))=\gamma$.
\end{theo}

This implies in particular that $H$ has infinitely many connected components and $\sigma(\Gamma)$ is cocompact in $H$. By viewing $\Out(F_N)$ as a lattice in the automorphism group of its Cayley graph, we reach the following corollary, which was suggested to us by Jingyin Huang.

\begin{cor}[see proof page~\pageref{proof:lattice}]\label{corintro:cayley}
 Let $N\ge 3$. For any finite generating set $S$ of $\Out(F_N)$, any automorphism of the Cayley graph $\Cay(\Out(F_N),S)$ is at bounded
 distance from an automorphism induced by the left multiplication by an element of $\Out(F_N)$.

 If $\Gamma$ is a torsion-free finite-index subgroup of $\Out(F_N)$, and $S$ is a finite generating set of $\Gamma$, 
 then the group of automorphisms of the Cayley graph $\Cay(\Gamma,S)$ is isomorphic to a subgroup of $\Out(F_N)$ containing $\Gamma$.
In particular, it is countable.
\end{cor}

In this statement, $\Cay(G,S)$ is the simple graph with vertices $G$, and such that there is an edge between two distinct elements $g$ and $h$ if and only if $gh\m \in S\cup S\m$. 
As was recently proved by Leemann and de la Salle, every finitely generated group has a Cayley graph (for some finite generating set) whose automorphism group is countable \cite[Corollary~1.3]{LdlS}. But knowing that this holds for every Cayley graph  is a much more restrictive condition, which fails for free groups, for instance.
On the other hand, if a finitely generated infinite group $G$ has a non-trivial finite order element,
then it has a Cayley graph whose automorphism group is uncountable \cite[Lemma~6.1]{dlST}.

\subsection{New constructions of canonical splittings for subgroups of $\Out(F_N)$}

A substantial part of the present paper (Part~\ref{part2}) is devoted to constructions of canonical splittings associated to a subgroup $H$ of $\Out(F_N)$ which we believe to be of independent interest. 
The canonicity of a splitting $U_H$  associated to $H$ means that it is not only $H$-invariant, 
but even invariant under the normalizer $N_{\Out(F_N)}(H)$ (for the action of $\Out(F_N)$ on the set of all splittings) -- and in fact the map $H\mapsto U_H$ is $\Out(F_N)$-equivariant. 

This is reminiscent of Ivanov's \emph{canonical reduction system} of a reducible subgroup of the mapping class group, which is a canonical decomposition of the surface along a multicurve. As a consequence of these constructions, we obtain the following result which has the same flavor.

\begin{theo}[see Theorem \ref{theo:reducibility}]
  Let $N\ge 2$, and let $H\subset\Out(F_N)$ be an infinite subgroup.
  
    If there exists a conjugacy class of proper free factors of $F_N$ whose $H$-orbit is finite, 
    then there also exists a conjugacy class of proper free factors of $F_N$ whose orbit is finite under
    the normalizer $N_{\Out(F_N)}(H)$.
\end{theo}

We will not use this result directly, we will instead rely on a similar statement for a particular class  of splittings (see Corollary \ref{corintro:crs}).
One has to be careful with the class of splittings considered. Indeed, 
there exists a subgroup $H\subset \Out(F_N)$ having a non-trivial periodic free splitting, but whose normalizer
has none (Example~\ref{ex:arcs} gives an instance).
We actually carry out three separate constructions of canonical splittings.
The first two encode the sets of $H$-invariant free splittings and $H$-invariant $\Zmax$-splittings, respectively, relying on JSJ techniques. The third has a different flavor and deals with $H$-invariant free factor systems.

Recall that a \emph{splitting} of $F_N$ is an action of $F_N$ on a simplicial tree with no proper nonempty invariant subtree.
It is a \emph{free splitting} if edge stabilizers are trivial. Thus, one-edge free splittings (i.e.\ those for which the quotient graph under the $F_N$-action has a single edge)
are dual to free product decompositions $F_N=A*B$ or HNN extensions over the trivial group $F_N=A*$.
A \emph{$\Zmax$-splitting} is a splitting of $F_N$ whose edge stabilizers are isomorphic to $\bbZ$ and maximal
for the inclusion.

We will actually work in the finite index subgroup $\IA\subset\Out(F_N)$ consisting of 
outer automorphisms acting trivially on homology mod 3.
The point of this subgroup is that it avoids many finite order phenomena:
this group is torsion-free, any free splitting which is periodic by some element or subgroup of $\IA$
is in fact invariant; and so is any periodic conjugacy class of free factor, or any periodic conjugacy class of element (\cite{HM2}, see Section \ref{sec_IA}).

\subsubsection{Canonical splittings from collections of free splittings} 

The following theorem is at the heart of Section~\ref{sec:collection-free} of the present paper (see Theorem~\ref{theo:tree-of-cyl} for a more complete statement). It associates to any subgroup $H\subseteq\IA$ a canonical splitting from which one can read all $H$-invariant free splittings. 

\begin{theo}[see Theorem~\ref{theo:tree-of-cyl}]\label{theo:intro-tree-of-cyl}
Let $N\ge 2$. To any subgroup $H\subset \IA$, one can canonically assign a splitting $\Uun_H$ of $F_N$ which is non-trivial as soon as $H$ is infinite and preserves some  non-trivial free splitting of $F_N$.  
The assignment $H\mapsto U^1_H$ is $\Out(F_N)$-equivariant.

The splitting $\Uun_H$ comes with a bipartition $V(U^1_H)=V^0\dunion V^1$ of its vertex set, and the set of $H$-invariant one-edge free splittings coincides with the set of all  one-edge free splittings that can be obtained by blowing up a vertex of $V^1$.
\end{theo}

Our proof of Theorem~\ref{theo:intro-tree-of-cyl} uses techniques coming from JSJ theory to construct $\Uun_H$ from the collection of all $H$-invariant free splittings. 
It might be surprising that the result of the construction is a canonical splitting 
as opposed to a deformation space, as is usually the case when working with free splittings.
What rigidifies the situation here is the existence of the infinite group $H$ that preserves all these free splittings,
and  $\Uun_H$ is constructed using the action of $H$.
More precisely, denoting by $\tilde{H}$ the preimage of $H$ in $\Aut(F_N)$, we view $H$-invariant free splittings of $F_N$ as splittings of $\tilde{H}$, and construct $\Uun_H$ as a tree of cylinders for a deformation space of $\tilde{H}$-trees made of maximal $H$-invariant free splittings. 

The above theorem yields in particular a description of the stabilizer of a collection of free splittings of $F_N$ (see Proposition~\ref{prop:stab-u1}), and leads to the following chain condition for stabilizers of collections of free splittings.

\begin{theo}[see Proposition~\ref{prop:chain-FS2}]
  There is a bound, only depending on the rank $N$ of the free group, on the length of a chain $\calc_1\subseteq\dots \subseteq \calc_k$ of collections of free splittings of $F_N$ such that for all $i\in\{1,\dots,k-1\}$, we have $$\bigcap_{S\in \calc_{i+1}}\Stab(S)\subsetneq\bigcap_{S\in \calc_{i}}\Stab(S),$$ where $\Stab(S)$ denotes the $\Out(F_N)$-stabilizer of $S$.
\end{theo}

\subsubsection{Canonical splittings from collections of $\Zmax$-splittings}

 In the case where $H$ does not preserve any non-trivial free splitting, the splitting $\Uun_H$ constructed above is just the trivial splitting. But in this case, one may use the collection of invariant $\Zmax$-splittings, if any, to construct a  more interesting canonical splitting.

\begin{theo}[see Theorem~\ref{theo:JSJ_zmax}]\label{theo:intro-canonical-cyclic}
Let $N\ge 2$. To any subgroup $H\subset \IA$ that preserves a non-trivial $\Zmax$-splitting but does not preserve any non-trivial free splitting,
one can canonically assign a non-trivial $\Zmax$-splitting $\UZ_H$.
The assignment $H\mapsto\UZ_H$ is $\Out(F_N)$-equivariant.
\end{theo}

The splitting $\UZ_H$ is obtained as a JSJ splitting of all $H$-invariant $\Zmax$-splittings of $F_N$: it is maximal for domination among all universally elliptic $H$-invariant $\Zmax$-splittings, and equal to its own tree of cylinders.
Although surface groups are examples of hyperbolic groups having non-trivial $\Zmax$-splittings but trivial JSJ decomposition,
this possibility of a trivial JSJ decomposition is ruled out here because the infinite group $H$ has to act trivially
on surfaces of the JSJ decomposition (see Theorem~\ref{thm_JSJ}).

A collection of $\Zmax$-splittings is \emph{$\FS$-averse} if its elementwise stabilizer in $\IA$ does not fix any non-trivial free splitting of $F_N$. As above, we obtain a chain condition as a consequence of Theorem~\ref{theo:intro-canonical-cyclic}.

\begin{theo}[see Proposition~\ref{prop:chain-zmax2}]
There is a bound, only depending on the rank $N$ of the free group, on the length of a chain $\calc_0\subseteq\dots \subseteq \calc_k$ of $\FS$-averse collections of $\Zmax$-splittings of $F_N$ such that for all $i\in\{0,\dots,k-1\}$, one has $$\bigcap_{S\in \calc_{i+1}}\Stab(S)\subsetneq\bigcap_{S\in \calc_{i}}\Stab(S).$$
\end{theo}

\subsubsection{The dynamical decomposition}\label{sec_intro_dd}

Our third construction of a canonical splitting, carried out in Section~\ref{sec:collection-factors}, is of a different, more dynamical, flavor. It encodes the interaction between the maximal invariant free factor systems of a subgroup $H\subseteq\IA$ that does not fix any  non-trivial free splitting.
It is remarkable that from free factor systems, one obtains a canonical splitting, which is a much
more rigid object than a free factor system.

 To every subgroup $H$ of the mapping class group of a closed, connected, orientable, hyperbolic surface $\Sigma$, Ivanov associated a canonical subsurface decomposition which singles out \emph{active} parts of $\Sigma$ (where $H$ acts with at least one pseudo-Anosov homeomorphisms) and \emph{inactive} parts (where, up to a finite-index subgroup, $H$ acts trivially). By analogy, our \emph{dynamical decomposition} for subgroups $H\subseteq\IA$ extracts some \emph{active parts} of the free group -- although we cannot expect $H$ to act trivially on the complement.
 
A set $\calf$ of 
conjugacy classes of non-trivial proper subgroups of $F_N$ is a \emph{free factor system} if there exist representatives $P_1,\dots,P_k$ of the conjugacy classes such that $\calf=\{[P_1],\dots,[P_k]\}$
and $F_N$ decomposes as a free product $F_N=P_1*\dots *P_k *F_r$ for some free subgroup $F_r$ of $F_N$ (with $0\leq r\leq N$). 
We order free factor systems by saying that $\calf\preceq\calf'$ if every group in $\calf$ is conjugate into a group of $\calf'$.
Since chains of free factor systems have bounded length,
any subgroup $H\subseteq\IA$ has at least one maximal $H$-invariant free factor system $\calf$ (we allow $\calf=\es$).

We now assume that $H$ does not preserve any non-trivial free splitting.
By \cite{HM2,GH}, for any maximal $H$-invariant free factor system $\calf$,
the group $H$ contains a fully irreducible outer automorphism relative to $\calf$. 
However, there may exist 
$H$-invariant conjugacy classes of free factors which are not $\calf$-peripheral\footnote{For experts, the maximality of $\calf$ only says that there is no proper free factor \emph{relative to $\calf$} whose conjugacy class is $H$-invariant.}
(i.e.\ not conjugate into any of the subgroups $P_1,\dots, P_k$); these potential extra invariant free factors are the basis of our construction.

\begin{figure}[ht!]
  \centering
  \includegraphics[width= .9\textwidth]{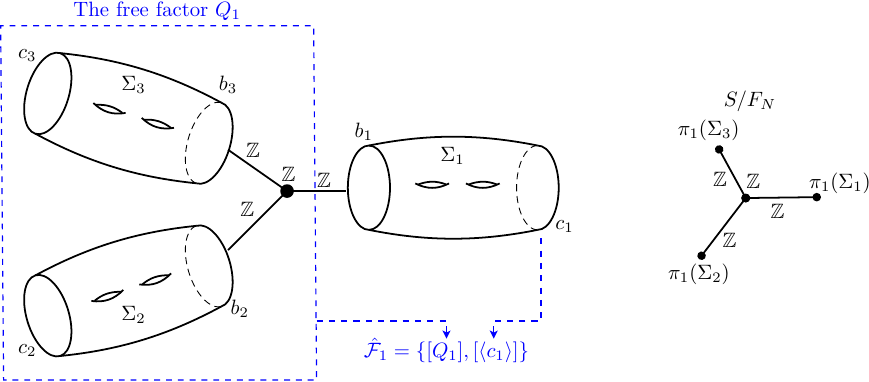}
  \caption{The dynamical decomposition of the group $H\subset\Out(F_N)$ induced by homeomorphisms of the three surfaces
     in  Example \ref{example_intro}.
The conjugacy class of the free factor $Q_1$ is $H$-invariant, and 
    $\HF_1=\{[Q_1],[\grp{c_1}]\}$ is a maximal $H$-invariant \afs.}\label{fig_dyn_intro}
\end{figure}

\begin{ex}[see Example \ref{ex:3-surfaces}]\label{example_intro}
The following example might be helpful (see Figure~\ref{fig_dyn_intro}).
We identify $F_N$ with the fundamental group of the space obtained by gluing 3 surfaces on a circle by homeomorphism along one of their boundary components, as shown in Figure~\ref{fig_dyn_intro}. 
We take for $H$ the group of outer automorphisms induced by homeomorphisms of the surfaces that are the identity on the boundary.
Let $Q_i$ be the fundamental group of the complement of $\Sigma_i$. Then $\calf_i=\{[Q_i]\}$ 
is a maximal $H$-invariant free factor system, and $Q_2$ is a free factor 
whose conjugacy class is $H$-invariant and which is not $\calf_1$-peripheral.
\end{ex}

For technical reasons, we need to work with $H$-invariant \emph{almost free factor systems}, usually denoted by $\HF$, which include extra situations involving unused boundary components of surfaces, see Definition~\ref{de:afs}. 
For example, in Example \ref{example_intro}, $\hat\calf_i=\{[Q_i],[ c_i]\}$ is an \afs{} (in fact a maximal $H$-invariant \afs).
In the case where $H$ has no invariant conjugacy class, every $H$-invariant almost free factor system is a free factor system, so one may ignore these subtleties at first.

If there is a unique maximal $H$-invariant almost free factor system, we say that $H$ is \emph{pure}. In this case, the dynamical decomposition of $H$ will be trivial. 
In Example~\ref{example_intro}, $H$ is not pure, but its restriction to each surface group is pure.

When there are several  maximal $H$-invariant \afs{}s, 
the dynamical decomposition produces a canonical splitting
having some vertex groups which are free factors in restriction to which $H$ is pure.
These vertices are called \emph{active}.
In the case of Example \ref{example_intro}, the dynamical decomposition is dual to the decomposition shown in
Figure \ref{fig_dyn_intro}, and the active vertices correspond to the three surfaces.

\begin{theo}[Dynamical decomposition, see Theorem~\ref{thm_dynamic_dec}]\label{theo:intro-dynamical-decomposition}
Let $N\ge 2$, and let $H\subseteq\IA$ be a subgroup which does not preserve any non-trivial free splitting. Then there exists a
canonical $H$-invariant splitting $\Ud_H$ coming with a bipartition of its vertex set $\Ud_H=V_p\dunion V_a$ such that 
\begin{enumerate}
\item for every $v\in V_p$, the stabilizer $G_v$ is peripheral in every maximal $H$-invariant \afs{} $\HF$, i.e.\ $G_v$ is conjugate into one of the factors in $\HF$;
\item for every $v\in V_a$, the stabilizer $G_v$ is a free factor of $F_N$ whose conjugacy class is $H$-invariant, and the restriction of $H$ to $G_v$ is pure.
\end{enumerate}
\end{theo}

Vertices in $V_a$ are called \emph{active}. More details about this decomposition are given in Section \ref{sec:collection-factors}.
 
Our proof involves considering, for each maximal $H$-invariant almost free factor system $\HF$, an element $\alpha\in H$ which is irreducible relative to $\HF$, and an $\alpha$-invariant $\bbR$-tree $T$.
If $A$ is any free factor whose conjugacy class is $H$-invariant, and which is not $\HF$-peripheral, 
then one gets a splitting of $F_N$ from a transverse covering of $T$ by the translates of the $A$-minimal subtree.
One then proves that there is a unique minimal such $H$-invariant conjugacy class of free factor $A_\HF$ which is not $\HF$-peripheral
and one can check that the associated splitting $T_\HF$ 
depends only on $\HF$ and $H$, not on the choice of $\alpha$ and $T$.
Then as $\HF$ varies among maximal $H$-invariant \afs{}s, the splittings $T_\HF$ are compatible
and yield the desired canonical dynamical decomposition.

In the end, the groups $A_{\HF}$ appear as the stabilizers of the  active vertices of $\Ud_H$. 
As a consequence, we obtain a bound on the number of maximal $H$-invariant free factor systems.

\begin{cor}[see Corollary \ref{coro:finitude-facteurs}]
There exists $K>0$, only depending on the rank of the free group, such that for every subgroup $H\subseteq\Out(F_N)$, assuming that no free splitting is invariant under a finite-index subgroup of $H$, then there are at most $K$ maximal free factor systems of $F_N$ that are invariant under a finite-index subgroup of $H$.
\end{cor}

The assumption that $H$ does not preserve any free splitting is crucial: one easily builds examples (including examples coming from punctured surfaces) of subgroups $H\subseteq\Out(F_N)$ with infinitely many $H$-invariant free splittings, giving infinitely many invariant conjugacy classes of maximal free factor systems (which are \emph{sporadic}, i.e.\ of the form $\calf=\{[P_1]\}$ with $F_N=P_1\ast\mathbb{Z}$ or $\calf=\{[P_1],[P_2]\}$ with $F_N=P_1\ast P_2$).

\subsubsection{Nice splittings}

Theorems \ref{theo:intro-tree-of-cyl} and \ref{theo:intro-canonical-cyclic} construct
from an $H$-invariant splitting, a canonical one, which is in particular invariant under the normalizer of $H$.
The output is no longer a cyclic splitting in general, and neither is the dynamical decomposition.
To iterate this kind of argument in a unified fashion, we use the following definition.

\begin{de}
A splitting $S$ of $F_N$ is \emph{nice} if either 
\begin{enumerate}
\item $S$ is a non-trivial free splitting, or 
\item $S$ is a non-trivial $\Zmax$-splitting, or
\item $S$ is \emph{bi-nonsporadic}, i.e.\ all edge stabilizers of $S$ are finitely generated and non-abelian, and $S$ contains two vertices in distinct $F_N$-orbits such that the Grushko decomposition of their stabilizer with respect to the collection of all incident edge stabilizers is non-sporadic.
\end{enumerate}
\end{de}

With this definition at hand one can deduce the following result from the three constructions above (strictly speaking, the canonical splittings described above are not always nice, but they are easily turned into nice splittings in a canonical way, see Corollary~\ref{cor:biflexible2nice}).

\begin{cor}[see Corollary \ref{cor_nice}]\label{corintro:crs}
 Let $N\ge 2$. If $H\subseteq\Out(F_N)$ is an infinite subgroup preserving a nice splitting, then its normalizer preserves a nice splitting.
\end{cor}

As a matter of analogy, when $\Sigma$ is a closed, connected, oriented, hyperbolic surface, and $H$ is an infinite subgroup of $\Mod(\Sigma)$ that preserves the isotopy class of a multicurve, then $H$ has a canonical reduction system, i.e.\ a multicurve whose isotopy class is also invariant by its normalizer. Thus, Corollary~\ref{corintro:crs} and the three canonical constructions on which it is based,
can be viewed as providing a (nice!) analogue of canonical reduction systems for free groups.

\subsection{On the proof of the measure equivalence rigidity theorem}

We now present the techniques we use in the proof of our main theorem. 

\subsubsection{Classical measure group theoretic reductions}

We first use a standard argument in measured group theory, developed in the successive works of Furman \cite{Fur1,Fur2}, Monod and Shalom \cite{MS} and Kida \cite{Kid2}, to reduce the proof of our main  theorem to a statement about measured groupoids associated to measure-preserving $\Out(F_N)$-actions on probability spaces.

It is actually enough (see e.g.\ \cite[Theorem~6.1]{Kid2}) to prove a rigidity statement for self measure equivalence couplings of a finite index subgroup of $\Out(F_N)$. As above, we will work in the finite index subgroup $\IA$ made of outer automorphisms acting trivially on homology mod $3$. 
A \emph{self coupling} of a countable group $\Gamma$ is a measure-preserving action $\Gamma\times\Gamma\actson \Sigma$
by Borel automorphisms on a standard measure space $\Sigma$ such that the action of each factor
is free and has a finite measure fundamental domain.
The restriction to self couplings should not surprise the reader familiar with quasi-isometry rigidity results:
self couplings may be thought as
analogues of self quasi-isometries of $\Gamma$, and quasi-isometric rigidity of groups
is often deduced from rigidity properties of these self quasi-isometries.

There is a \emph{standard} self coupling  $\IA\times\IA\actson \Sigma_0=\Out(F_N)$, where the action is by left and right multiplication. Our goal is to show that every
self coupling $\Sigma$ of $\IA$ factors through this standard coupling, i.e.\ that there always exists a Borel map $\Sigma\to\Out(F_N)$ which is almost everywhere equivariant with respect to the action of $\IA\times\IA$. This is called \emph{coupling rigidity} in \cite{Kid4} and \emph{tautness} in \cite{BFS}.

As in \cite{Fur2}, this is then rephrased in the language of stable orbit equivalence. The space $\Sigma$ is equipped with an action of $\Gamma_1\times\Gamma_2=\IA\times\IA$. Letting $X_1$ be a finite measure fundamental domain for the $\Gamma_1$-action on $\Sigma$, there is a natural action of $\Gamma_2$ on $X_1$ 
(identified with $\Sigma/\Gamma_1$).
Likewise, letting $X_2$ be a finite measure fundamental domain for the $\Gamma_2$-action on $\Sigma$, there is a natural action of $\Gamma_1$ on $X_2$. We choose the fundamental domains $X_1$ and $X_2$ such that, letting $Y=X_1\cap X_2$, one has $(\Gamma_1\times \Gamma_2)Y=\Sigma$ up to a null set. The orbit partitions 
coming from the $\Gamma_1$-action on $X_2$ and from the $\Gamma_2$-action on $X_1$ restrict to the same equivalence relation on $Y$. 

Informally, one can think of $Y$ as being equipped with arrows: there is an arrow from $x$ to $y$ whenever $x$ and $y$ are in the same $\Gamma_1$-orbit, or equivalently in the same $\Gamma_2$-orbit. The arrows are naturally equipped with two $\IA$-valued labels, indicating the element of $\Gamma_i$ that sends $x$ to $y$. 

This is made formal in the language of measured groupoids (see Section~\ref{sec:background} for precise definitions): we have a measured groupoid $\calg$ (the set of all arrows) over the base space $Y$, with two cocycles to $\IA$ (the two labelings). In our terminology, these cocycles coming from actions of $\IA$ are important instances of \emph{action-type} cocycles (see Definition~\ref{de:action-type}). By careful bookkeeping of the above construction, the proof of measure equivalence rigidity of $\Out(F_N)$ reduces to the proof of the following statement, as explained in Section~\ref{sec:me}.

\begin{theo}[See Theorem \ref{theo:two-cocycles}]\label{theo:intro-two-cocycles}
Let $N\ge 3$. Let $\calg$ be a measured groupoid over a standard measure space $Y$. Then any two action-type cocycles $\rho,\rho':\calg\to\IA$ are $\Out(F_N)$-cohomologous, i.e.\ there exist a Borel map $\phi:Y\to\Out(F_N)$ and a conull Borel subset $Y^*\subseteq Y$ such that for every arrow $g\in\calg$ joining two points $x,y\in Y^*$, one has $\rho'(g)=\phi(y)\rho(g)\phi(x)^{-1}$.
\end{theo}

\begin{rk}\label{rk_commensurator}
Restricting to the case where $Y$ is reduced to a point, one recovers commensurator rigidity as follows:
an isomorphism $j$ between two finite index subgroups $\Gamma_1,\Gamma_2$ of $\IA$ yields two cocycles
on the groupoid $\Gamma_1$ given by $j$ and the inclusion, and Theorem \ref{theo:intro-two-cocycles}
shows that $j$ is the restriction of an inner automorphism of $\Out(F_N)$.
\end{rk}

\subsubsection{The overall strategy}\label{sec:intro-commensurateur}
Let us first describe a general strategy to prove
 the commensurator rigidity of a group $\Gamma$. This is the approach used by Ivanov
 for the commensurator rigidity of the mapping
class group of a surface \cite{Iva}.
One uses a graph $X$ with an action of $\Gamma$ and 
whose group of automorphisms is precisely $\Gamma$.
One characterises vertex stabilizers of $X$ 
in terms of algebraic conditions that are robust under taking finite index subgroups,
and similarly for pairs of adjacent vertices.
Then any isomorphism $\phi$ between finite index subgroups of $\Gamma$
will induce an automorphism of $X$, and one concludes that $\phi$
is the restriction of an inner automorphism of $\Gamma$.

Our strategy of proof of Theorem~\ref{theo:intro-two-cocycles}, following Kida's, is a generalization of the above, where subgroups of $\Gamma$ are replaced by measured subgroupoids coming with cocycles towards $\Gamma$. 
In this paragraph, we explain a simplified version of the strategy we use; we will make this more accurate below.

We use a graph $X$ with an action of $\Out(F_N)$ 
whose group of automorphisms is precisely $\Out(F_N)$.
Given a vertex $v\in X$, a natural candidate to play the role of the stabilizer of $v$ is 
the subgroupoid $\calh_v:=\rho\m(\stab(v))$.
The goal is to give a characterization of such subgroupoids independently of $\rho$,
only in terms of groupoid-theoretic conditions.
This takes the following form:
there is  a countable Borel partition $Y^*=\dunion_{i\in \bbN} Y_i$ of a conull subset $Y^*\subset Y$,
and for each $i\in \bbN$ a unique vertex $v'_i\in X$ such that
$\calg_{|Y_i}=\rho'^{-1}(\Stab(v'_i))_{|Y_i}$
(where $\calg_{|Y_i}$ is the set of arrows with both endpoints in $Y_i$).
This defines for almost every $y\in Y$ 
a map $\phi(y):X\ra X$: if $y\in Y_i$, $\phi(y)$ maps the vertex $v$ to the vertex $v'_i$ above.
Using a groupoid-theoretic characterization of stabilizers of
pairs of adjacent vertices, one proves that for almost every $y\in Y$, 
$\phi(y)$ is an automorphism of $X$, i.e.~an element of $\Out(F_N)$:
this is the desired map $\phi:Y\ra \Out(F_N)$ that shows that $\rho$ and $\rho'$ are cohomologous.

The main groupoid-theoretic notions that we use in our characterizations
of subgroupoids
are amenability in Zimmer's sense, and the fact that a subgroupoid normalizes 
or contains another one.
In particular, one cannot use notions such as the rank of an abelian group 
(because of the Ornstein--Weiss theorem \cite{OW}).

On the other hand, the algebraic statements used in \cite{HW} to prove the commensurator rigidity of $\Out(F_N)$, phrased in terms of normal direct products of free groups, happen to be adaptable to the groupoid-theoretic framework, as we will explain in more detail in Section~\ref{sec:intro-adams}.

\subsubsection{Two rigid graphs} 
In fact, we need to use two rigid $\Out(F_N)$-graphs, namely
the \emph{free splitting graph} $\FSG$ and its sibling, the  \emph{non-separating free splitting graph} $\NSG$. 
The vertices of $\NSG$ are the decompositions of $F_N$ as an HNN extension $F_N=A\ast$, up to equivariant isomorphism of their Bass--Serre trees (in $\FSG$, one also allows for decompositions as a free product $F_N=A\ast B$). 
Two such decompositions $S,S'$ are joined by an edge when there is a two-edge graph of groups decomposition of $F_N$ collapsing to $S$ and $S'$. The group $\Out(F_N)$ naturally acts on $\NSG$ and $\FSG$. Building on work of Bridson and Vogtmann \cite{BV2}, Pandit proved \cite{Pan} that for $N\geq 3$, $\Out(F_N)$ is exactly the group of simplicial automorphisms of $\NSG$  (the same is true for $\FSG$, see \cite{AS}, but we will not use it).  
We will use additionally that every injective map $\NSG\ra \FSG$ preserving adjacency and non-adjacency
takes values in  the subgraph $\NSG$ (Proposition~\ref{prop:ns-fs}).

As explained above, to prove Theorem~\ref{theo:intro-two-cocycles}, 
we should ideally characterize in a purely groupoid-theoretic way 
(i.e.\ with no reference to the cocycles $\rho$ and $\rho'$) those subgroupoids of $\calg$ that preserve a vertex of $\NSG$ for either $\rho$ or $\rho'$ (and characterize adjacency in a similar manner).
In reality, we rather prove a partial characterization showing, roughly speaking, that the $\rho$-stabilizer of a vertex in $\NSG$
preserves a vertex in $\FSG$. More precisely, 
given a vertex $v\in\NSG$, one associates to the subgroupoid $\rho\m(\Stab(v))$
a countable Borel partition 
$Y^*=\Dunion_i Y_i$  of a conull subset $Y^*\subseteq Y$
such that $\rho'(\calg_{|Y_i})$ fixes a unique vertex $v_i\in \FSG$.
This defines for almost every $y\in Y$ an injective map $\phi(y)$ from the vertex set of $\NSG$ to $\FSG$,
and we prove that $\phi$ preserves compatibility of splittings.
Using the aforementioned facts, we show that in fact, for almost every $y\in Y$, the map $\phi(y)$ is an automorphism of $\NSG$, i.e.\ an element of $\Out(F_N)$.

\subsubsection{Towards a characterization of stabilizers of free splittings}\label{sec:intro-adams} 

The above discussion shows that the main theorem reduces to a groupoidal (partial) characterization 
of the stabilizers of non-separating free splittings of $F_N$ (Theorem~\ref{theo:non-characterization}), an analogue of the algebraic
version proved in \cite{HW}.

In this introduction, we will focus on one crucial statement towards this goal, first presented in a group-theoretic fashion, before explaining the additional arguments required to improve it to a groupoid-theoretic version.

\paragraph*{Finding invariant nice splittings (group version).} 
Recall that a one-edge non-separating free splitting is the same as a decomposition as an HNN extension of the form $F_N=A\ast$.
An important observation in their characterization, already exploited in \cite{HW}, is that for $N\geq 3$, the stabilizer of such a splitting 
contains a normal subgroup (the \emph{group of twists}) isomorphic to a direct product of two non-abelian free groups 
(made of all outer automorphisms having a representative in $\Aut(F_N)$ which acts as the identity on $A$, and multiply a stable letter $t$ by elements of $A$, on the left and on the right).
This somehow explains the choice of the graph $\mathbb{FS}$ as for instance, the stabilizer of a free factor (of corank at least 2) does not have this property.
The next proposition provides a partial converse.

\begin{prop}\label{prop:emblematic}
Let $N\ge 3$, and let $H\subseteq\IA$ be a subgroup. Assume that $H$ contains a subgroup which is isomorphic to a direct product $K_1\times K_2$ of two non-amenable subgroups $K_1$ and $K_2$, both normal in $H$.

Then $H$ fixes a nice splitting of $F_N$.
\end{prop}

The proof goes as follows. Since by Corollary \ref{corintro:crs},
the property of preserving a nice splitting remains true by going to the normalizer, we are done if $K_1$ or $K_2$ preserves a nice
splitting, hence also if some infinite cyclic  subgroup $\grp{a_1}\subset K_1$ preserves a nice splitting.
We can therefore assume that $\grp{a_1}$ is pure (as otherwise it preserves its dynamical decomposition), and to simplify the discussion, let us assume that
its unique maximal invariant almost free factor system is a free factor system.
The relative free factor graph $\FFG$ is a Gromov-hyperbolic graph on which $a_1$ acts loxodromically \cite{BF-FF,HM},
and its invariant endpoints in the Gromov boundary correspond to
equivalence classes of $\bbR$-trees
$T_-,T_+$.
It follows that $K_2$ preserves the  pair of equivalence classes of $T_{\pm}$. Using arguments from \cite{GL-on}, the fact that $T_{\pm}$ has non-amenable stabilizer
is witnessed by
a canonical nice splitting (see Corollary \ref{cor:PATna}) that will thus be  $K_2$-invariant.

\paragraph*{Finding invariant nice splittings (groupoid version).}

We need a groupoid-theoretic version of Proposition~\ref{prop:emblematic}, given in Lemma~\ref{lemma:1-implies-splitting}. Its statement is similar, using a notion of amenable subgroupoid that originates in the work of Zimmer \cite{Zim2}, a notion of normal subgroupoid coined by Kida in \cite{Kid1} as a generalization of earlier work of Feldman, Sutherland and Zimmer \cite{FSZ}, and a notion of pseudo-product of subgroupoids coined to replace the direct product in the statement (see Definition~\ref{de:pseudo-product}).

Our groupoid-theoretic proof stays very close to its group-theoretic version as long as it only involves objects taking values in a countable set, like nice splittings or free factor systems. One reason is that a map from the base space $Y$ to a countable set $\cald$ becomes constant after restriction to each part of a countable partition of $Y$, and all our groupoid-theoretic statements will allow to perform such partitions on the base space. In fact, the chain conditions given in the previous section also play a role here, to ensure essentially that here one never needs to perform infinitely many countable partitions at once. 

For instance, Corollary~\ref{corintro:crs} translates directly to the groupoid setting, as follows (the assumption that $\rho_{|\calh}$ is nowhere trivial is the natural replacement of the assumption that $H$ is infinite in Corollary~\ref{corintro:crs}, see Definition~\ref{dfn:nowhere_trivial}). 

\begin{prop}[see Proposition \ref{prop:nice-preserved}]\label{prop_normalisateur_groupoide}
Let $\calh'$ be a measured groupoid coming with a cocycle $\rho$ towards $\IA$. Let $\calh$ be a normal subgroupoid of $\calh'$, and assume that 
\begin{itemize}
\item $\rho_{|\calh}$ is nowhere trivial; 
\item $\rho(\calh)$ is contained in the stabilizer of a nice splitting.
\end{itemize}
Then, up to passing to a conull subset and taking a countable partition of the base space, $\rho(\calh')$ is also contained in the stabilizer of a nice splitting.  
\end{prop}

Our groupoidal analogue of Proposition \ref{prop:emblematic} follows from Proposition \ref{prop_normalisateur_groupoide} and the following crucial statement. 
All relevant definitions are given in Section~\ref{sec:background}.


\begin{theo}[see Theorem \ref{theo:normalisateur_moyennable}]\label{theointro:adams}
  Let $N\ge 3$. Let $\calg$ be a measured groupoid over a standard measure
  space $Y$, and let $\rho:\calg\to\IA$ be an action-type cocycle. Let $\cala_1$ be an amenable subgroupoid of $\calg$
  of infinite type,
  and let $\calk_2$ be a subgroupoid of $\calh$ that normalizes $\cala_1$.

  Then  either $\calk_2$ is amenable, or there exists a  Borel subset $U\subset Y$ of positive measure such that $(\calk_{2|U},\rho)$ preserves a nice splitting.
\end{theo}

Here $\cala_1$ and $\calk_2$ are the groupoid analogues of the groups $\grp{a_1}$, $K_2$ from the 
sketch of proof of Proposition \ref{prop:emblematic}.

Our proof of Theorem~\ref{theointro:adams} is the place where the argument significantly differs in the group and groupoid settings, as it involves dealing with objects taking values in non-discrete spaces, like spaces of actions on $\bbR$-trees. In particular, an additional ingredient is the amenability of the action of $\Out(F_N,\calf)$ on a space of arational trees, given by the following statement, a variation on earlier work by Bestvina and the authors \cite{BGH}.
See Definition~\ref{de:borel-amenable} for the definition of Borel amenability of a group action. 

\begin{theo}[see Proposition \ref{prop:action-amenable-lift}]\label{thm_amen_intro}
 Let $N\ge 2$, and let $\calf$ be a non-sporadic free factor system of $F_N$. Let $\PATa$ be the space of all projective arational $(F_N,\calf)$-trees whose stabilizer in $\Out(F_N,\calf)$ is amenable. 

The action of $\Out(F_N,\calf)$ on $\PATa$ is Borel amenable.

Similarly, the action of $\Out(F_N,\calf)$ on the set of pairs of projective arational $(F_N,\calf)$-trees $\{T_+,T_-\}$
whose stabilizer is amenable, is a Borel amenable action.
\end{theo}

The strategy of  our proof of Theorem~\ref{theointro:adams} (carried out in Section~\ref{sec:adams})
relies on work of Adams \cite{Ada} that was already exploited by Kida in the context of mapping class groups.
As above, we can assume that $\cala_1$ is pure, and let us assume 
that its unique maximal invariant almost free factor system is a free factor system $\calf$.
Amenability of $\cala_1$ provides an equivariant map $\mu$ that assigns to every point $y\in Y$ a probability measure $\mu_y$ on the compactification of Outer space (relative to $\calf$). We then take advantage of Reynolds' partition of this compactification into \emph{arational} and \emph{non-arational} trees \cite{Rey}, and reduce to the case where $\mu_y$ is supported almost everywhere on arational trees (otherwise we find invariant relative free factors, contradicting the maximality of $\calf$). Arational trees are in fact precisely those that describe points of $\partial_\infty\FFG$ by \cite{BR,Ham2,GH}. We then use a \emph{barycenter map} constructed by Lécureux and the authors in \cite{GHL}, which associates a finite set of conjugacy classes of free factors to every triple of pairwise inequivalent arational trees. This is used to ensure that $\mu_y$ is almost everywhere supported on at most two equivalence classes of arational trees $T_{\pm}(y)$. Among the possible maps $\mu$, there is an essentially unique one
 $\mu_{\max}$ with maximal support,
so  $\mu_{\max}$ is also equivariant under the groupoid $\calk_2$. If in restriction to a subset $U\subset Y$ of positive measure,
the pair $\{T_-(y),T_+(y)\}$ has non-amenable stabilizer, 
this is witnessed by a canonical nice splitting which is invariant under $\cala_1$ and $\calk_2$ as in the group-theoretic case. Otherwise, the stabilizer of $\{T_-(y),T_+(y)\}$ is amenable for almost every $y\in Y$, and we reach a contradiction to the non-amenability of $\calk_2$ using Theorem \ref{thm_amen_intro}.


\paragraph*{Additional arguments.} Once the groupoid-theoretic version of Proposition~\ref{prop:emblematic} has been established, more work is still required to distinguish free splittings among nice splittings. 
Briefly,  $\Zmax$-splittings are distinguished by the fact that their stabilizer contains a normal non-trivial abelian subgroup (its \emph{group of twists}). The groupoid-theoretic statement used to distinguish bi-nonsporadic splittings is quite technical; roughly, the two vertices from the definition of such a splitting enable us to find two commuting normal subgroups in its stabilizer that both contain direct products of free groups, and this does not happen in the stabilizer of a non-separating free splitting.

More work is also needed to establish a characterization of adjacency in $\NSG$. Again, the details are quite involved (see Section~\ref{sec:compatibility}); as in \cite{HW}, one underlying idea is that two one-edge non-separating free splittings of $F_N$ are adjacent in $\NSG$ if and only if their common stabilizer does not fix any third one-edge free splitting.

\subsection{Structure of the paper}

The paper has three parts. 

The first part gives preliminary background, both on $\Out(F_N)$ (Section~\ref{sec:background-out}) and on measured groupoids (Section~\ref{sec:background}) and their use in the study of measure equivalence (Section~\ref{sec:me}). In particular, Section~\ref{sec:me} explains how to reduce the main theorem of the present paper and its consequences to the rigidity statement for groupoids and their action-type cocycles (Theorem~\ref{theo:intro-two-cocycles} in this introduction).

The second part does not involve any groupoids, it is devoted to our constructions of canonical splittings and our descriptions of stabilizers of collections of splittings. Section~\ref{sec:additional-background} collects further background regarding invariant splittings and trees of cylinders. The construction of a canonical splitting encoding all invariant free splittings is carried out in Section~\ref{sec:collection-free}, the variant for $\Zmax$-splittings is the contents of Section~\ref{sec:collection-zmax}, and the dynamical decomposition is constructed in Section~\ref{sec:collection-factors}.  Section~\ref{sec_nice} contains the consequence to the fact that having an invariant nice splitting passes to the normalizer.  
Finally, in Section~\ref{sec:witness}, we build a witness map that associates a canonical splitting to any arational tree with non-amenable stabilizer. 

The third part is the heart of our proof of measure equivalence rigidity of $\Out(F_N)$. In Section~\ref{sec_canonical}, we translate
to the groupoid setting the
fact that having invariant nice splittings passes to the normalizer.
We also set up the notions of pure, nice-averse and stably nice subgroupoids that provide the groupoid-theoretic framework in which the later sections are phrased. As observed in this introduction, direct products of free groups play an important role in the present paper, and Section~\ref{sec:direct-product} is concerned with measured groupoids with cocycles towards such direct products. 
Section~\ref{sec:amenable_actions} contains the proof of Theorem~\ref{thm_amen_intro} showing the amenability of the action of $\Out(F_N,\calf)$ on the space of arational trees with amenable stabilizer. 
In Section~\ref{sec:adams}, we run our version of Adams' argument and prove Theorem~\ref{theointro:adams}. 
Section~\ref{sec:big} culminates in Theorem~\ref{theo:non-characterization} which gives the necessary conditions for stabilizing a non-separating free splitting, and sufficient conditions for stabilizing a free splitting. 
Section~\ref{sec:compatibility} deals with the characterization of adjacency in $\NSG$. In Section~\ref{sec:maps}, we prove that every injective map from $\NSG$ to $\FSG$ which preserves adjacency and non-adjacency must take its values in $\NSG$
(as explained above, this compensates the slight asymmetry in the statement of Theorem~\ref{theo:non-characterization}). The proofs of our main theorem and its consequences are completed in Section~\ref{sec:conclusion}.

\paragraph*{Acknowledgments.}   
We would like to warmly thank Damien Gaboriau for suggesting this problem, organizing a session at the IHP in 2011 with this problem in view within the framework of the trimester
\emph{Von Neumann algebras and ergodic theory of group actions}, and for helpful discussions and encouragements.
We thank Ric Wade and Martin Bridson for pointing out a missing assumption in an earlier version of Corollary \ref{corintro:cayley}. Finally, we warmly thank the referees 
for their exceptionally thorough reading and numerous suggestions that helped us improve this paper.

Both authors acknowledge support from the french ANR grant ANR-22-CE40-0004 GOFR.
The first author was partially supported by  European Research Council (ERC GOAT 101053021),
and benefited from the support of the French government ”Investissements d’Avenir” program integrated to France 2030 (ANR-11-LABX-0020-01).
The second named author acknowledges support from the Agence Nationale de la Recherche under Grant ANR-16-CE40-0006, and from the European Union (ERC) under Grant Artin-Out-ME-OA, 101040507. Views and opinions expressed are however those of the authors only and do not necessarily reflect those of the European Union or the European Research Council. Neither the European Union nor the granting authority can be held responsible for them. 

\newpage
{\small
\setcounter{tocdepth}{1}
\tableofcontents
}
\newpage

\part{Preliminaries}

\section{Preliminaries on $\Out(F_N)$}\label{sec:background-out}

\subsection{Splittings and deformation spaces}\label{sec_splittings_def}

\paragraph*{General definitions.}

Let $G$ be a group, and let $T$ be either a simplicial $G$-tree, or an $\mathbb{R}$-tree (i.e.\ a $0$-hyperbolic geodesic metric space) equipped with an isometric $G$-action. An element $g\in G$ acts \emph{elliptically} on $T$ if $g$ fixes a point in $T$, and \emph{hyperbolically} otherwise. The $G$-action on $T$ is \emph{minimal} if $T$ does not contain any proper nonempty $G$-invariant subtree. Assuming that some element of $G$ acts hyperbolically on $T$, the $G$-action on $T$ always contains a unique nonempty minimal $G$-invariant subtree \cite[Proposition~3.1]{CM}. 

A \emph{splitting} of $G$ is a minimal, simplicial $G$-action on a simplicial tree $T$ with no edge inversion. 

A subgroup $A\subseteq G$ is \emph{elliptic} in $S$ if it fixes a point in $S$. 
Given a collection $\calf$ of conjugacy classes of subgroups of $G$, we say that a splitting $S$ is \emph{relative to $\calf$} (or is a $(G,\calf)$-splitting) if every subgroup in $\calf$ is elliptic in $S$.

When $G$ is finitely generated, any splitting of $G$ has finitely many orbits of vertices and edges. This also holds for a $(G,\calf)$-splitting
if $G$ is finitely generated relative to a finite collection of subgroups $\calf$, i.e.\ $G$ is
generated by $\calf$ together with a finite set of elements of $G$.
By Bass-Serre theory, a splitting of $G$ corresponds to writing $G$ as the fundamental group of a graph of groups \cite{Ser},
and we use the word \emph{decomposition} as a synonym of splitting.
Splittings are always considered up to equivariant graph isomorphism.

There is a natural action of $\Aut(G)$ by precomposition on the set of ($G$-equivariant isomorphism classes of) splittings of $G$. Given a splitting $S$, precomposing the action $G\actson S$ by an inner automorphism $\ad_g$ yields a new action $G\actson S$, and the isomorphism of $S$ defined by $x\mapsto gx$ is equivariant with respect to these two actions. 
Therefore, the action of $\Aut(G)$ factors through an action of $\Out(G)$ on the set of splittings.

Given a splitting $S$ of $G$ and a vertex $v\in S$, we denote by $G_v$ the $G$-stabilizer of $v$. Likewise, given an edge $e\subseteq S$, we denote by $G_e$ the $G$-stabilizer of $e$.
We will also always denote by $\Inc_v$ the set of all $G_v$-conjugacy classes of stabilizers of edges incident on $v$.
When $G$ is finitely generated, 
the group $G_v$ is finitely generated relative to $\Inc_v$ (see \cite[Lemma~1.11]{Gui_actions}).

Let $S$ and $S'$ be two splittings of $G$. The splitting $S$ \emph{dominates} $S'$ if there exists a $G$-equivariant map $S\to S'$. This is equivalent to saying that every subgroup of $G$ which fixes a point in $S$, also fixes a point in $S'$. Two splittings are \emph{in the same deformation space} if they dominate each other (see \cite{For_deformation,GL_deformation}).

The splitting $S$ \emph{refines} $S'$ if $S$ collapses onto $S'$, i.e.\ $S'$ is obtained from $S$ by collapsing every edge in a $G$-invariant collection of edges to a point (in other words, there exists a $G$-equivariant alignment-preserving map $S\to S'$). 
Two splittings $S$ and $S'$ are \emph{compatible} if they admit a common refinement. 
Assume further that $G$ is finitely generated and that the splittings $S,S'$ are irreducible (this means
that $G$ contains a non-abelian free subgroup acting freely).
If $S$ and $S'$ are compatible, then they have a smallest common refinement $S\vee S'$, onto which every other common refinement collapses \cite[Proposition A.26]{GL-jsj}. 
This smallest common refinement has the following properties: every subgroup of $G$ which is elliptic in both $S$ and $S'$ is elliptic in $S\vee S'$ \cite[Proposition~A.26(3)]{GL-jsj};
and every edge stabilizer in $S\vee S'$ is an edge stabilizer in either $S$ or $S'$ \cite[Lemma~A.24]{GL-jsj}. 
More generally, if $S_1,\dots,S_n$ are pairwise compatible irreducible splittings of a finitely generated group, then they have a smallest common refinement $S_1\vee\dots\vee S_n$, and it satisfies the above properties.

A \emph{one-edge} splitting of $G$ is a splitting of $G$ that contains a single orbit of edges.

\paragraph*{Free splittings and $\Zmax$-splittings.} 
A \emph{free splitting} of $G$ is a splitting of $G$ whose edge stabilizers are all trivial. 

Given a collection $\calf$ of conjugacy classes of subgroups of $G$, a \emph{Grushko splitting} of $G$ relative to $\calf$ (or a \emph{Grushko $(G,\calf)$-splitting}) 
is a free splitting $S$ of $G$ relative to $\calf$ which is maximal for domination
\cite{Grushko}, \cite[\S 3.2]{GL-jsj}.
If $G$ is finitely generated, or even finitely generated relative to a finite collection of subgroups $\calf$,
then a Grushko splitting of $G$ relative to $\calf$ exists by Grushko theorem.
In particular, if $S$ is a splitting of a finitely generated group $G$,
for any vertex $v\in S$, since $G_v$ is finitely generated relative to $\Inc_v$,
there exists a $(G_v,\Inc_v)$-Grushko splitting.
Any two such splittings are in the same deformation space. 
A \emph{Grushko factor} of a group $G$ relative to a collection of subgroups $\calf$
is a vertex stabilizer in any Grushko $(G,\calf)$-splitting.

A \emph{$\Zmax$-splitting} of $F_N$ is a splitting of $F_N$ in which all edge stabilizers are infinite cyclic and root-closed. The following lemma follows for instance from \cite{BF91}. 

\begin{lemma}\label{lem_bound_zmax}
There is a bound, only depending on $N$, on the number of orbits of edges of a $\Zmax$-splitting of $F_N$ with no vertex of valence $2$.
\end{lemma}

We mention that given a splitting $S$ of $F_N$ with non-trivial cyclic edge stabilizers, one can construct a canonical $\Zmax$-splitting $S_{\Zmax}$ as follows.

\begin{de}[see {\cite[Lemma~9.27]{GL-jsj}}]\label{dfn_zmaxise} 
Let $S$ be a splitting of $F_N$ whose edge stabilizers are all isomorphic to $\mathbb{Z}$. 
Given an edge $e$ of $S$, define $\hat G_e$ as the maximal cyclic subgroup containing $G_e$.
Let $\sim$ be the equivalence relation on $S$ generated by
$x\sim x'$ if there exists an edge $e$ containing $x$ such that $x'\in \hat G_e.x$.

One defines $S_{\Zmax}$ as the minimal subtree of the quotient $S/\sim$.
\end{de}

By \cite[Lemma~9.27]{GL-jsj}, $S_{\Zmax}$ is a $\Zmax$ splitting of $F_N$, and 
every  $\Zmax$-splitting of $F_N$ dominated by $S$ is dominated by $S_{\Zmax}$ (although $G$ is assumed to be a one-ended hyperbolic group in \cite[\S 9.5]{GL-jsj}, this is not used in the proof of \cite[Lemma 9.27]{GL-jsj}).
It will be useful to describe $S$ assuming that $S_\Zmax$ is trivial.

\begin{lemma}[See Figure \ref{fig_chaussettes}]\label{lem:zmaxise_trivial}
  Let $S$ be a splitting of $F_N$ with infinite cyclic edge stabilizers.
Assume that $S_\Zmax$ is trivial (i.e.\ reduced to a point).

Then $S/F_N$ is a tree of groups, with a unique vertex $v$ whose stabilizer is not abelian.

In particular, assuming moreover that every edge of $S$ is adjacent to at least one vertex with non-abelian stabilizer,
it follows that $S/F_N$ is a star, all vertices $u_1,\dots,u_n$ distinct from $v$ are joined by a unique edge $e_i$ to $v$,
$G_{u_i}$ is cyclic and $[G_{u_i}:G_{e_i}]\geq 2$.
\end{lemma}

\begin{figure}[ht!]
  \centering 
  \includegraphics{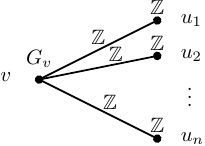}
  \caption{Splittings such that $S_{\Zmax}$ is trivial.}
  \label{fig_chaussettes}
\end{figure}

\begin{proof} 
  Denote by $\tilde S_{\Zmax}$ the quotient of $S$ by the equivalence relation of Definition~\ref{dfn_zmaxise},
  and note that $S/F_N$ and $\tilde S_{\Zmax}/F_N$ are isomorphic as graphs.
  Our assumption says that $F_N$ fixes a point in $\tilde S_\Zmax$. In particular $\tilde S_\Zmax/F_N$ is a tree, hence so is $S/F_N$.

  If $v,v'\in S$ are two vertices with non-abelian stabilizer and in distinct orbits, then $v$ and $v'$ have distinct images in $\tilde S_{\Zmax}/F_N$,
and since edge stabilizers of $S_{\Zmax}$ are cyclic, this prevents $F_N$ from having a global fixed point in $\tilde S_{\Zmax}$.
\end{proof}

\subsection{Free factor systems}\label{sec:factor-systems}

A \emph{free factor} of $F_N$ is a subgroup of $F_N$ that arises as a point stabilizer in a (maybe trivial) free splitting of $F_N$ (in particular, $\{1\}$ and $F_N$ are free factors). Given a collection $\calf$ of conjugacy classes of subgroups of $F_N$, a \emph{free factor} of $F_N$ relative to $\calf$ is a subgroup of $F_N$ that arises as a point stabilizer in some free splitting of $F_N$ relative to $\calf$. 
A free factor $A$ of $F_N$ relative to $\calf$ is \emph{proper} if $A\neq F_N$, $A\neq \{1\}$ and $A$ is not conjugate into a group appearing in $\calf$.

A \emph{free factor system} of $F_N$ is a collection of conjugacy classes of
 non-trivial free factors that arise as the non-trivial point stabilizers in some non-trivial (i.e.\ not reduced to a point) free splitting of $F_N$, see \cite[\S~2.6]{BFH}. 
For instance $\es$ is a free factor system, but $\{[F_N]\}$ is not (on this latter point our convention differs from \cite{BFH}). 

A free factor system $\calf$ of $F_N$ is \emph{sporadic} if either $F_N$ splits as an HNN extension $F_N=A\ast$ and $\calf=\{[A]\}$, or $F_N$ splits as a free product $F_N=A\ast B$ and $\calf=\{[A],[B]\}$. Otherwise $\calf$ is \emph{non-sporadic}. 
The main difference between sporadic and non-sporadic free factor systems is that
a sporadic free factor system $\calf$ satisfies the following property:
there exists a unique free splitting $S$ of $F_N$ such that $S/F_N$ has only one edge, and $\calf$ is the collection of all conjugacy classes of vertex stabilizers of $S$. In particular when $\calf$ is sporadic, this splitting $S$ is invariant under all outer automorphisms of $F_N$ preserving $\calf$. 

In train-track theory (developed in \cite{BH,BFH,FeH}), sporadic free factor systems correspond to \emph{one-edge extensions} \cite{HM2}, and non-sporadic free factor systems to \emph{multi-edge extensions}; they are related to polynomially-growing and exponentially-growing strata in relative train tracks.

If $\calf=\{[P_1],\dots,[P_k]\}$ is a collection of conjugacy classes of subgroups of $F_N$, we say that a subgroup $A\subset F_N$ is \emph{$\calf$-peripheral} if $A$ is trivial or conjugate into some $P_i$.
If $\calf,\calf'$ are two collections of conjugacy class of subgroups of $F_N$, we write
$\calf\sqsubseteq\calf'$ if every group in $\calf$ is $\calf'$-peripheral.
When $\calf,\calf'$ are free factor systems, this happens if and only if any Grushko $(F_N,\calf)$-splitting dominates any Grushko $(F_N,\calf')$-splitting.
This defines a partial ordering on free factor systems of $F_N$.

If $\calf$ if a free factor system of $F_N$, and $A$ is a proper free factor of $F_N$ relative to $\calf$,
then there is a smallest free factor system $\calf'$ with $\{[A]\},\calf\sqsubseteq\calf'$,
it consists of $[A]$ together with the conjugacy classes of groups in $\calf$ that are not conjugate into $A$.
We denote it by $\calf\vee \{[A]\}$.

If $\calf=\{[P_1],\dots,[P_n]\}$ is a collection of conjugacy classes of subgroups of $F_N$, and $A$ is a subgroup of $F_N$,
we let $\calf_{|A}$ be the set of all $A$-conjugacy classes of non-trivial subgroups of $A$ of the form $P_i^g\cap A$ for $g\in F_N$.
When $\calf$ is a free factor system,
$\calf_{|A}$ also coincides with the set of all conjugacy classes (in $A$) of non-trivial vertex stabilizers in the $A$-minimal subtree of any Grushko $(F_N,\calf)$-splitting. It is a free factor system of $\cala$ unless $A$ is $\calf$-peripheral.

\subsection{Relative outer automorphism groups}

Given a collection $\calf$ of conjugacy classes of subgroups of $F_N$, we denote by $\Out(F_N,\calf)$ the subgroup of $\Out(F_N)$ made of all outer automorphisms $\alpha$ such that for every conjugacy class $[A]\in\calf$, one has $\alpha([A])=[A]$. We denote by $\Out(F_N,\calf^{(t)})$ the subgroup made of all outer automorphisms $\alpha$ such that for every representative $\tilde\alpha$ of $\alpha$ in $\Aut(F_N)$, and every subgroup $A\subseteq F_N$ whose conjugacy class belongs to $\calf$, there exists $g_{\tilde\alpha,A}\in F_N$ such that $\tilde\alpha_{|A}$ is given by the conjugation by $g_{\tilde\alpha,A}$.

An outer automorphism $\alpha\in\Out(F_N,\calf)$ is \emph{atoroidal} relative to $\calf$ if any element whose conjugacy class is periodic under $\alpha$ is $\calf$-peripheral. We say that $\alpha$ is \emph{fully irreducible} relative to $\calf$ if
no power $\alpha^k$ with $k\neq 0$ preserves the conjugacy class of a proper free factor of $F_N$ relative to $\calf$.

\subsection{$\bbR$-trees, transverse families and transverse coverings}
A subtree of an $\bbR$-tree $T$ is \emph{non-degenerate} if it is non-empty and not reduced to one point. 
Non-degenerate segments are called \emph{arcs}.
Given a point $x\in T$, the \emph{directions} at $x$ are the connected
components of $T\setminus\{x\}$.
A \emph{branch point} $x\in T$ is a point such that there are at least 3 directions at $x$.

Let $T$ be an $\bbR$-tree equipped with a minimal isometric $F_N$-action. 
A \emph{transverse family} in $T$ is an $F_N$-invariant collection $\caly$ of nondegenerate subtrees of $T$  (i.e.\ nonempty and not reduced to a single point) such that any two distinct trees in $\caly$ intersect in at most a point. A \emph{transverse covering} of $T$ is a transverse family $\caly$ made of closed subtrees such that every segment in $T$ is covered by finitely many subtrees in $\caly$.

The \emph{skeleton} of a transverse covering $\caly$ is the bipartite simplicial tree $S$ having one vertex $v_Y$ for every subtree in $\caly$, one vertex $v_x$ for every point $x\in T$ that belongs to at least two trees in $\caly$, with an edge joining $v_x$ to $v_Y$ whenever $x\in Y$, see \cite[Definition~4.8]{Gui04}.
The skeleton $S$ is compatible with $T$ in the following sense, which generalizes the notion of compatibility of splittings (see \cite[\S A.3]{GL-jsj}): there exists an $\bbR$-tree $
\hat T$ with two equivariant alignment preserving projections $\hat T\ra T$ and $\hat T\ra S$.
Indeed, by \cite[Lemma~4.7]{Gui04}, $T$ splits as a graph of actions with underlying graph of groups $S/G$ (in the sense of \cite[Definition~4.3]{Gui04}), and this clearly implies compatibility of $S$ and $T$.

An $F_N$-tree $T$ is \emph{mixing} (in the sense of Morgan \cite{Mor}) if given any two segments $I,J\subseteq T$, there exists a finite collection $\{g_1,\dots,g_k\}$ of elements of $F_N$ such that $J\subseteq g_1I\cup\dots\cup g_kI$. In a mixing tree, every transverse family made of closed subtrees is a transverse covering.

\subsection{Outer space and its closure}

Let $\calf$ be a free factor system of $F_N$. A \emph{Grushko $(F_N,\calf)$-tree} is a minimal, simplicial metric $(F_N,\calf)$-tree whose underlying simplicial tree is a Grushko splitting of $F_N$ relative to $\calf$. The \emph{unprojectivized relative Outer space} $\calo(F_N,\calf)$ is the space of all $F_N$-equivariant isometry classes of Grushko $(F_N,\calf)$-trees (see \cite{GL-os}). When $\calf=\emptyset$, this is nothing but the unprojectivized version of Culler and Vogtmann's Outer space \cite{CV}.
When $\calf$ is clear from the context, we will simply write $\calo$ instead of $\calo(F_N,\calf)$. We denote by $\mathbb{P}\calo$ the projectivized outer space, where trees are considered up to homothety instead of isometry. The group $\Out(F_N,\calf)$ naturally acts on $\calo$ and $\mathbb{P}\calo$ by precomposition of the action.

The projectivized Outer space $\mathbb{P}\calo$ can be compactified by looking at its closure $\mathbb{P}\overline{\calo}$ in the space of all projective classes of $(F_N,\calf)$-trees  (i.e.\ $\mathbb{R}$-trees equipped with a minimal isometric action of $F_N$ for which every subgroup in $\calf$ is elliptic), equipped with the equivariant Gromov--Hausdorff topology introduced by Paulin in \cite{Pau} (or equivalently,  see \cite{Pau3}, with the length function topology studied by Culler and Morgan \cite{CM}). The closure is identified with the space of all projective classes of \emph{very small} $(F_N,\calf)$-trees, i.e.\ $(F_N,\calf)$-trees in which all arc stabilizers are either trivial or nonperipheral, cyclic and root-closed, and all stabilizers of tripods are trivial (\cite{CL,BF} for the case where $\calf=\emptyset$, and \cite{Hor1} in the relative setting). We also denote by $\overline{\calo}$ the unprojectivized version of $\mathbb{P}\overline{\calo}$, where trees are considered up to isometry instead of homothety. The $\Out(F_N,\calf)$-actions on $\calo$ and $\mathbb{P}\calo$ extend to actions on $\overline{\calo}$ and $\mathbb{P}\overline{\calo}$. We let $\partial\calo=\overline{\calo}\setminus\calo$.


By \cite{GL,Hor1}, if $T\in \ol\calo$, there are finitely many orbits of branch points in $T$,
and finitely many orbits of directions at branch points.
If every orbit is dense, then
the stabilizer of any arc is trivial (see e.g.\ \cite[Proposition~5.17]{Hor1}).


\subsection{Arational trees} 

Arational trees were first introduced by Reynolds \cite{Rey} as a free group analogue of arational measured foliations on surfaces; 
see \cite{Hor} for the version relative to a free factor system needed here.
A tree $T\in\partial\calo(F_N,\calf)$ is \emph{arational} if 
for every proper $(F_N,\calf)$-free factor $A$, $A$ is not elliptic in $T$ and the $A$-action on its minimal subtree $T_A$ is a Grushko $(A,\calf_{|A})$-tree. 

In any arational $(F_N,\calf)$-tree, $F_N$ orbits are dense (see \cite[Lemma~4.9]{Hor}). Therefore, in an arational $(F_N,\calf)$-tree, the stabilizer of any non-degenerate segment is trivial (see e.g.\ \cite[Proposition~5.17]{Hor1}).
In fact, 
arational $(F_N,\calf)$-trees are always mixing \cite[Lemma~4.9]{Hor}.

In the following lemma, we record a basic fact on arational trees mentioned in the paragraph following  \cite[Definition 12.2]{GH1}, relying on \cite[Lemma~4.6]{Hor}. 

\begin{lemma}\label{lemma:arational-stabilizers}
Let $T\in \partial\calo(F_N,\calf)$ be an arational tree. Then either 
\begin{enumerate}
\item  $T$ is relatively free, i.e.\ every point stabilizer is $\calf$-peripheral, 
\item or up to conjugation, there is a unique point stabilizer which is not $\calf$-peripheral, and this point stabilizer is cyclic.
\end{enumerate}
\end{lemma}

We mention that the second option in the lemma occurs when $T$ comes from an arational measured lamination on a surface in the precise sense of \cite[Definition~4.3]{Hor}.

For every $T\in\partial\calo(F_N,\calf)$ and every $\alpha\in\Out(F_N,\calf)$, it follows from the definition of the action of $\Out(F_N,\calf)$ on $\mathbb{P}\baro(F_N,\calf)$ that $\alpha$ fixes the homothety class of $T$ if and only if for every $\tilde\alpha\in\Aut(F_N)$ representing $\alpha$, there exists a homothety $H_{\tilde\alpha}$ of $T$ (i.e.\ a map that multiplies all distances by the same factor $\lambda$) which is $\tilde\alpha$-equivariant in the sense that $$H_{\tilde\alpha}(gx)=\tilde\alpha(g)H_{\tilde\alpha}(x)$$ for every $g\in F_N$ and every $x\in T$. Once the representative $\tilde\alpha$ is chosen, the homothety $H_{\tilde\alpha}$ is unique. The factor $\lambda$ does not depend on the chosen representative and is called the \emph{dilatation} of $\alpha$ with respect to $T$.

\begin{prop}\label{prop:invariant-arational-tree}
Let $\calf$ be a non-sporadic  free factor system of $F_N$, and let $\alpha\in\Out(F_N,\calf)$ be an outer automorphism which is fully irreducible relative to $\calf$.

Then there exists an arational $(F_N,\calf)$-tree $T$ whose homothety class is $\alpha$-invariant, and such that the dilatation of $\alpha$ with respect to $T$ is not equal to $1$. 
\end{prop}

\begin{proof}
It follows from work of Francaviglia and Martino \cite[Theorem~8.23]{FM} that $\alpha$ has an invariant axis in $\calo(F_N,\calf)$. Let $T$ be a limiting tree of such an axis. Then $T$ is arational, as otherwise $\alpha$ would have a non-trivial power that preserves one of the finitely many canonical reducing factors of $T$ (see \cite{Rey} or \cite[Section~4.4]{Hor}). That the dilatation of $\alpha$ acting on $T$ is not equal to $1$ follows from \cite[Remark~6.4]{GH}.  
\end{proof}

\subsection{QH vertex groups}

The notions introduced under this heading will only be used in Sections~\ref{sec:collection-zmax} and~\ref{sec:collection-factors}. 

QH vertex groups are a key notion in the theory of JSJ decompositions of groups, they are used in the description of certain canonical splittings \cite{Sel}. 
Informally, this is a surface attached to the rest of the group by its boundary.
 QH stands for \emph{quadratically hanging}, and refers to the fact that the attachment can be described algebraically by quadratic words.

 Let $\Sigma$ be a compact connected surface with boundary with negative Euler characteristic. A \emph{boundary subgroup} is a subgroup of $\pi_1(\Sigma)$ which is conjugate to a subgroup of the fundamental
group of a boundary component of $\Sigma$. Maximal boundary subgroups are the conjugates
of the fundamental groups of the boundary components of $\Sigma$. 
\begin{de}\label{dfn_cleanQH}
  Let $G$ be a group, and let $S$ be a splitting of $G$. A vertex $v$ of $S$ is a \emph{clean QH vertex} if it comes with an identification $G_v\simeq\pi_1(\Sigma)$ where $\Sigma$ is a (maybe
  non-orientable)  compact connected hyperbolic surface with geodesic boundary, and each incident
  edge group is a maximal boundary subgroup of $\pi_1(\Sigma)$, and
  this induces an injection from the set of incident edges to the set
  of boundary components of $\Sigma$.  

  Boundary components not in the image of this injection are called \emph{unused}.
\end{de}

This definition will be sufficient in most situations, but some arguments
require a more general definition.

\begin{de}[{\cite[Definition~5.13]{GL-jsj}}]\label{dfn_QH}
Let $A$ be a group and $\calq$ a collection of conjugacy classes of subgroups of $A$. The pair $(A,\calq)$ is \emph{QH with trivial fiber} if there exists a connected compact hyperbolic $2$-orbifold $\Sigma$ with boundary, 
and an identification $A\simeq \pi_1(\Sigma)$ such that, under this identification, each group in $\calq$ is either finite or a boundary subgroup of $\Sigma$.

A pair $(A,\calq)$ is \emph{QH with fiber $F$} if $F\normal A$, and $(A/F,\bar\calq)$ is QH with trivial fiber, where $\bar\calq$
is the image of $\calq$ in $A/F$.

Given a collection $\calp$ of conjugacy classes of subgroups of $G$, a vertex $v$ of a splitting of $(G,\calp)$ is a \emph{QH vertex relative to $\calp$} if 
$(G_v,\Inc_v\cup \calp_{|G_v})$ is QH (maybe with fiber). 
\end{de}

\subsection{Twists of splittings}\label{sec_twists}
Given a group $G$ and an element $g\in G$, we denote by $\ad_g$ the inner automorphism $x\mapsto gxg\m$.

Let  $S$ be a splitting of a group $G$, and $e=uv$ an edge of $S$.
Following Levitt \cite{Lev}, given an element $z\in G_u$ that centralizes $G_e$,
one defines as follows an outer automorphism $\tau\in\Out(G)$ called the \emph{twist} by $z$ around $e$ near $u$.

First assume that the image of $e$ in $S/G$ is separating, and consider the corresponding amalgam $G=A*_{G_e} B$, where $A$ contains $G_u$ and $B$ contains $G_v$.
Then $\tau$ is the outer automorphism of $G$ defined as the identity on $A$ and as $\ad_z$ in restriction to $B$.
Similarly, if $e$ is non-separating, and $t$ maps the origin of $e$ to its terminus after collapsing all edges  that are not in the orbit of $e$,
then the corresponding HNN extension gives a presentation of the form
$G=\grp{A,t| t\m g t= \phi(g)}$. Then  $\tau$ is the outer automorphism of $G$ defined as
the identity on $A$ and sending $t$ to $zt$. Notice that the outer automorphism $\tau$ defined in this way does not depend on the choice of the stable letter $t$, as any other choice is of the form $t'=ta$ for some $a\in G_u$.

In both cases (separating and non-separating), changing $e,u$ by a translate, and conjugating $z$ accordingly does not change $\tau$ as an outer automorphism,
so it may be described from the data in the graph of groups.
Moreover if $\tau$ is a twist around $e$ near $u$ and $\tau'$ is a twist around $e'$ near $u'$ and if the pair $(e,u)$ is not in the orbit of $(e',u')$,
then $\tau$ and $\tau'$  commute.
The \emph{group of twists} of $S$ is the subgroup $\Tw$ of $\Out(G)$ generated by all twists.

An important property is that $\tau$ preserves $S$: for any representative $\tilde \tau \in\Aut(G)$ of $\tau$,
there is a $\tilde \tau$-equivariant isomorphism $f:S\ra S$ (i.e.\ such that $f(gx)=\tilde{\tau}(g)f(x)$ for all $x\in S$ and all $g\in F_N$).
\\

We now give three examples that will play an important role.

If $S$ is a splitting of $F_N$ such that all edge stabilizers are non-abelian, then the group of twists is trivial because edge stabilizers
have trivial centralizer.

If $S$ is a $\Zmax$-splitting of $F_N$, and $e=uv$ is an edge in $S$, then $G_e$ is its own centralizer.
It follows that  twists around $e$ near $u$ agree with twists around $e$ near $v$: we just say that $\tau$ is a twist around $e$. Moreover, in this case, the group of twists is abelian.

If $S$ is a splitting of $F_N$ with trivial edge stabilizers, the group of twists can be much larger.
For instance, if $S$ is dual to the HNN extension $A*$ (with $A$ a corank one free factor),
then assuming that $A$ is non-abelian, the group of twists of $S$ is isomorphic to $A\times A$,
generated by twists around $e$  near both sides of $e$.
Explicitly, if $t$ is a stable letter of the HNN extension, then
$\Tw$ is the set of outer automorphisms that  have a representative in $\Aut(F_N)$ acting as the identity on $A$ and sending the stable letter $t$ to
$atb\m$ for some $(a,b)\in A\times A$.
If $A$ is cyclic, then there is an extra relation among the twists
and $\Tw\simeq A$ (in fact, if $A$ is an arbitrary group, the group of twists of the HNN extension $A*$ is isomorphic to $(A\times A)/Z(A)$
where $Z(A)$ is the center of $A$, diagonally embedded in $A\times A$).

If $S$ is dual to a splitting $F_N=A*B$ with $A,B$ non-abelian, then $\Tw\simeq A\times B$.
If $A$ is abelian, then twists near $A$ are trivial as outer automorphisms.
\\

We note that a twist around an edge of a splitting with trivial stabilizer can be seen
as a twist of a cyclic splitting.
For example, if $S$ is dual to a free product decomposition $F_N=A*B$,
and  $z\in A$, then the twist $\tau$ by $z$ around $e$ near $u$ can be seen as a twist of the cyclic splitting $S'$
dual to the decomposition $F_N=A*_{\grp{z}} \grp{B,z}$.
More generally, if $e=uv$ is an edge of $S$ with trivial stabilizer,
and $z\in G_u$, then the twist $\tau$ by $z$ around $e$ near $u$ can be seen as a twist of the cyclic splitting $S'$
obtained by folding $e$ with  its translate $ze$.
More precisely, if $e'=u'v'$ is the image of $e$ in $S'$,
then $\tau$ is also the twist by $z$ around $e'$ near $u'$.

\subsection{The exact sequence for automorphisms of splittings}\label{sec_exact}

Given a splitting  $S$ of $F_N$, we denote by $\Stab(S)$ the stabilizer of $S$ in $\Out(F_N)$.
Recall that $\alpha\in\Stab(S)$ if and only if for  some (equivalently every) $\tilde\alpha\in\Aut(F_N)$ representing $\alpha$
there exists a (unique if $S$ is not a line) $\tilde \alpha$-equivariant isometry $f_{\tilde\alpha}:S\ra S$.
The subgroup $\Stab^0(S)$ consisting of outer automorphisms $\alpha$ such that $f_{\tilde\alpha}$ acts trivially on $S/F_N$ (this does not depend on the choice of $\tilde\alpha$)
has finite index.
For each $\alpha\in \Stab^0(S)$, the restriction of $\alpha$ to $G_v$ gives a well defined outer automorphism of $G_v$.
One thus gets a restriction map $\rho:\Stab^0(S)\ra \prod_{v\in S/F_N} \Out(G_v)$.

\begin{prop}[Levitt \cite{Lev}]\label{prop_suite_exacte}
  Let $S$ be a splitting of $F_N$ with finitely generated edge stabilizers.
  Then one has an exact sequence
  $$1\ra \Tw \ra \Stab^0(S) \xrightarrow{\rho} \prod_{v\in S/F_N} \Out(G_v,\Inc_v)$$
  and the image of $\rho$ contains $\prod_{v\in S/F_N} \Out(G_v,\Inc_v^{(t)})$.
\end{prop}

\begin{proof}
 The assertion on the image of $\rho$ is \cite[Proposition 2.1]{Lev}.
 By \cite[Proposition 2.2]{Lev}, $\ker\rho$ is generated by \emph{bitwists} (see Section 2.4 in \cite{Lev}).
  Thus, we just have to check that bitwists are twists.
  
  Consider an edge $e=uv$ of $S$, and assume that $e$ is separating in $S/F_N$ and induces a splitting of the form $F_N=A*_{G_e} B$.
  Given $z\in G_u$ and $z'\in G_v$ normalizing $G_e$ and such that $zz'^{-1}$ centralizes $G_e$,
  one defines a bitwist by $\ad_z$ in restriction to $A$ and  $\ad_{z'}$ in restriction to $B$.
  Because $G_e$ is finitely generated, it has finite index in its normalizer $N(G_e)$, so $N(G_e)$ fixes a point in $S$.
  It follows that $z$ and $z'$ fix this point so one of them fixes $e$. This implies that $zz'^{-1}$ fixes one endpoint
  of $e$, and the bitwist is actually a twist by $zz'^{-1}$ at $e$ near one of its endpoints.
  The case of an HNN extension is similar and left to the reader.
\end{proof}

\begin{rk}\label{rk_normalizer}
  In many situations, one knows that a subgroup $H\subset \Stab^0(S)$ acts trivially on edge groups in the following sense:
  for every $\tilde\alpha\in \Aut(F_N)$ representing an element of $H$, and every edge group $G_e$, there exists $g\in F_N$
  such that $\tilde\alpha_{|G_e}=(\ad_g)_{|G_e}$.
  If each edge stabilizer is its own  normalizer, this implies that $\rho(H)$ is contained in $\prod_{v\in S/F_N} \Out(G_v,\Inc_v^{(t)})$.
\end{rk}

The following lemma is an immediate consequence.
\begin{lemma}\label{lem_produit}
  Let $S$ be a splitting of $F_N$ whose edge stabilizers are non-abelian and finitely generated.
  
  Then $\Stab^0(S)$ contains a subgroup that maps isomorphically to $\prod_{v\in S/F_N} \Out(G_v,\Inc_v^{(t)})$  under the restriction map.
\end{lemma}

\begin{proof}
  Since edge stabilizers are non-abelian, they have trivial centralizer in $F_N$ so $\Tw$ is trivial and
$\rho$ is one-to-one.
Since  $\prod_{v\in S/F_N} \Out(G_v,\Inc_v^{(t)})$ is contained in the image of $\rho$, the lemma follows.
\end{proof}

\subsection{Some properties of $\IA$}\label{sec_IA}

The group $\IA$ is the finite-index subgroup of $\Out(F_N)$ defined as the kernel of the natural homomorphism $\Out(F_N)\to \mathrm{GL}_N(\mathbb{Z}/3\mathbb{Z})$. In the present paper, it will often be convenient to work in $\IA$ to avoid finite-order phenomena. 
For instance, $\IA$ is torsion-free, see e.g.\ \cite[\S II.1, p.114]{HM2}.
We now collect the various properties of $\IA$ we will use.

\begin{theo}[{Handel--Mosher \cite[Theorem~II.3.1]{HM2}}]\label{theo:ia}
  Let $H\subseteq\IA$ be a subgroup, and assume that there exists a free factor $A$ of $F_N$ whose conjugacy class is
   invariant under a finite index subgroup of $H$.

Then the conjugacy class of $A$ is $H$-invariant.
\end{theo}

\begin{theo}[{Handel--Mosher \cite[Theorem~II.4.1]{HM2}}]\label{theo:ia-element}
  Let $H\subseteq\IA$ be a subgroup, and assume that there exists an element $g$ of $F_N$ whose conjugacy class is
  invariant under a finite index subgroup of $H$.

Then the conjugacy class of $g$ is $H$-invariant.
\end{theo}

The following lemma was established for example in \cite{HW} as an easy consequence of the above hard theorem of Handel and Mosher.

\begin{lemma}[{see e.g.\ \cite[Lemma~2.6]{HW}}]\label{lemma:ia-free-splitting}
Let $H\subseteq\IA$ be a subgroup, and let $S$ be a free splitting of $F_N$ which is invariant under a finite index subgroup of $H$. Then $S$ is $H$-invariant, and $H$ acts as the identity on $S/F_N$.
\end{lemma}


\begin{lemma}\label{lemma:ia-nice}
  Let $S$ be a splitting of $F_N$ and $v\in S$. 
  Assume that there is an edge incident on $v$ with trivial stabilizer, or that $G_v$ has a non-trivial free splitting relative to all incident edge stabilizers.
  
Let $\alpha\in\Stab_{\IA}(S)$, let $\tilde\alpha\in\Aut(F_N)$ be a representative of $\alpha$, and let $I_{\tilde\alpha}$ be an $\tilde\alpha$-equivariant isomorphism of $S$.

Then $I_{\tilde\alpha}(v)$ belongs to the same $F_N$-orbit as $v$.
\end{lemma}

\begin{proof}
  First assume that there is an edge $e$ with trivial stabilizer with origin at $v$.
  Let $\bar S$ be the free splitting obtained from $S$ by collapsing all edges with non-trivial stabilizer.
  Then $\bar S$ is $\alpha$-invariant.  By Lemma~\ref{lemma:ia-free-splitting}, $I_{\tilde\alpha}$ induces the identity on $\bar S/F_N$, so $I_{\tilde\alpha}$ maps $e$ to $ge$ for some $g\in F_N$. It follows that $I_{\tilde\alpha}(v)=gv$, and we are done.

  We denote by $W$ the set of such vertices $v\in S$ such that $G_v$ has a non-trivial free splitting relative to all incident edge stabilizers, and consider $v\in W$.
  Let $\hat{S}$ be a refinement of $S$ obtained in the following way. We first blow up
   $v$ into a Grushko decomposition $S_v$ of $G_v$ relative to the incident edge groups (this exists because $G_v$ is finitely generated relative to the incident edge groups). We then attach back every incident edge $e$ to the unique point of $S_v$ fixed by $G_e$. We denote by $\calf_v$ the free factor system of $F_N$ determined by collapsing all edges of $\hat{S}$ with non-trivial stabilizer to a point. Notice that this free factor system only depends on $S$ and $v$ and not on the choice of the refinement $\hat{S}$: indeed, $\calf_v$ is equal to the smallest free factor system of $F_N$ relative to the Grushko factors of $(G_v,\Inc_v)$ and to all other vertex groups of $S$. 

We now observe that if $v,v'\in W$ 
belong to different $F_N$-orbits, then $\calf_v\neq\calf_{v'}$: this is because $G_{v'}$ is conjugate into a subgroup of $\calf_v$, while $G_v$ is not.  

The element $\alpha$ permutes all free factor systems of the form $\calf_v$ where $v$ varies among the finitely many orbits of vertices in $W$.
As $\alpha\in\IA$, we deduce using Theorem~\ref{theo:ia} that $\alpha$ fixes each free factor system $\calf_v$. The conclusion follows.
\end{proof}

As a consequence of Lemma~\ref{lemma:ia-nice}, we get the following fact.

\begin{cor}\label{cor:ia-nice}
Let $S$ be a splitting of $F_N$ with all edge stabilizers non-trivial, and let $v\in S$ be a 
vertex such that $G_v$ has a non-trivial free splitting relative to all incident edge stabilizers.
Let $N_v$ be the subgroup of $\Stab_{\IA}(S)$ made of all outer automorphisms that act trivially on all vertex stabilizers but $G_v$.

Then $N_v$ is normal in $\Stab_{\IA}(S)$.
\qed
\end{cor}

\begin{prop}\label{prop:ia-zmax}
Let $H\subseteq\IA$, and let $S$ be a $\Zmax$-splitting which is $H$-invariant. Then $H$ acts trivially on $S/F_N$.
\end{prop}

\begin{proof}
Let $\Gamma=S/F_N$.
Let $K$ be the quotient of $H_1(F_N;\bbZ/3\bbZ)$ by the subgroup generated by the image of all edge groups of $\Gamma$.
Let $v$ be a terminal vertex of $\Gamma$, and $e$ the incident edge. 
As $S$ is a minimal $\Zmax$-splitting, the group $G_v$ is not cyclic.
Thus, $G_v$ is a free group of rank at least 2 and $H_1(G_v/\ngrp{G_e};\bbZ/3\bbZ)\neq \{0\}$.

We claim that for $v'\neq v$, the images of $G_v$ and $G_{v'}$ in $K$ are distinct.
Let $\tilde Q$ be the quotient of $F_N$ by the normal subgroups generated by all edge groups and all vertex groups of $\Gamma$ except $v$, 
and let $Q=H_1(\tilde Q;\bbZ/3\bbZ)$ be the corresponding quotient of $K$. 
Since the image of $G_{v'}$ in $Q$ is trivial for all $v'\neq v$, it suffices to show 
that the image of $G_v$ in $Q$ is non-trivial.
Now, $\tilde Q_v$ is the free product of  $G_v/\ngrp{G_e}$ by the fundamental group of the graph underlying $\Gamma$. It follows that the image of $G_v$ in $Q$ is non-trivial which proves our claim.

Since $H$ acts trivially on $K$, it follows that $H$ acts as the identity on the set of terminal vertices of $\Gamma$.
This concludes the proof if $\Gamma$ is a tree.

Otherwise, let $\Gamma_0\subseteq \Gamma$ be the core of $\Gamma$, i.e.\ the union of embedded loops.
If $\Gamma_0$ is not a circle, then $H$ acts as the identity on $\Gamma_0$ because it acts as the identity on
$H_1(\Gamma_0;\bbZ/3\bbZ)$. This implies that $H$ acts as the identity on $\Gamma$.
If $\Gamma_0$ is a circle, then $H$ acts on $\Gamma_0$ by orientation-preserving homeomorphisms because it acts trivially 
on $H_1(\Gamma_0;\bbZ/3\bbZ)$. So it suffices to prove that $H$ fixes a point in $\Gamma_0$. Since $H$ fixes all terminal points, this is clearly the case if $\Gamma\neq \Gamma_0$ so we can assume that $\Gamma$ itself is a circle.
By \cite[Proposition~6.6]{Hor1}, $\Gamma$ has at least one vertex whose vertex group admits a non-trivial free splitting relative to incident edge groups. Lemma \ref{lemma:ia-nice} shows that this vertex is fixed by $H$, concluding the proof.
\end{proof}

\section{Preliminaries on measured groupoids}\label{sec:background}

We do not assume the reader to be familiar with the notion of a measured groupoid and its use in proving measure equivalence rigidity results. In this section and the next one, we therefore thoroughly review standard facts from this theory. Most of the material we need can be found in \cite[Section~2.1]{AD} or \cite{Kid-survey} and the references therein.

\subsection{Measured groupoids: basic definitions}

\subsubsection{Groupoids, subgroupoids, restrictions}

A \emph{standard Borel space} is a  measurable space which is isomorphic to the 
measurable space of a separable complete metric space with its Borel $\sigma$-algebra.

\begin{de}[Discrete Borel groupoids]
A \emph{discrete Borel groupoid} 
over a standard Borel space $Y$
is a standard Borel space $\calg$ whose elements are called \emph{arrows}, equipped with the following data: 
\begin{enumerate}
\item two Borel maps $s,r:\calg\to Y$ called the \emph{source map} and the \emph{range map}, such that for every $y\in Y$, 
the set of arrows $g\in \calg$ such that $s(g)=y$ is at most countable.
\item a Borel map (called the \emph{composition law})
\begin{displaymath}
\begin{array}{cccc}
\calg^{(2)} & \to & \calg\\
(g_1,g_2) &\mapsto & g_1 g_2
\end{array}
\end{displaymath}
defined on the set $\calg^{(2)}:=\{(g_1,g_2)\in\calg\times\calg|s(g_1)=r(g_2)\}$ of \emph{composable arrows}, such that 
for all  $(g_1,g_2),(g_2,g_3)\in\calg^{(2)}$, the following hold:
\begin{itemize}
\item \textbf{Source/Range:} $s(g_1g_2)=s(g_2)$ and $r(g_1g_2)=r(g_1)$,
\item \textbf{Associativity:} $(g_1 g_2) g_3=g_1(g_2 g_3)$;
\end{itemize}
\item a Borel map (called the \emph{unit map}, whose image is called the \emph{space of units})
\begin{displaymath}
\begin{array}{cccc}
Y & \to & \calg\\
y &\mapsto & e_y
\end{array}
\end{displaymath}
such that for all $y\in Y$, the following hold:
\begin{itemize}
\item \textbf{Source/Range:} $s(e_y)=r(e_y)=y$ for all $y\in Y$, 
\item \textbf{Neutral element:} $g e_y=g$, $e_y g'=g'$ for all arrows $g,g'\in\calg$ with $s(g)=y$ and $r(g')=y$;
\end{itemize}
\item a Borel map (called the \emph{inverse map})
\begin{displaymath}
\begin{array}{cccc}
\iota: & \calg & \to & \calg\\
& g &\mapsto & g^{-1}
\end{array}
\end{displaymath}
such that for all $g\in\calg$, the following hold:
\begin{itemize}
\item \textbf{Source/Range:} $s(g^{-1})=r(g)$, $r(g^{-1})=s(g)$, 
\item \textbf{Inverse:} $gg^{-1}=e_{r(g)}$ and $g^{-1}g=e_{s(g)}$.
\end{itemize}
\end{enumerate} 
\end{de}

The word `discrete' in `discrete Borel groupoid' refers to the first condition from the above definition -- countability of the set of arrows having a given source. As all groupoids considered in the present paper are discrete, we will often drop the term `discrete' from the terminology and simply call them Borel groupoids. 

\begin{de}[Subgroupoids]
  Let $\calg$ be a Borel groupoid over a base space $Y$. A \emph{Borel subgroupoid} of $\calg$ is a Borel subset $\calh\subseteq\calg$ such that 
\begin{enumerate}
\item for every $(g_1,g_2)\in\calg^{(2)}\cap (\calh\times\calh)$, one has $g_1g_2\in\calh$,
\item for every $g\in\calh$, one has $g^{-1}\in\calh$, and
\item for every $y\in Y$, one has $e_y\in\calh$.
\end{enumerate}
Every Borel subgroupoid of $\calg$ has a natural structure of a Borel groupoid over the same base space $Y$.
\end{de}

\begin{de}
Given a Borel groupoid $\calg$ and two Borel subgroupoids $\calh,\calh'$ of $\calg$, the subgroupoid $\langle\calh,\calh'\rangle$ generated by $\calh$ and $\calh'$ is the smallest subgroupoid of $\calg$ that contains $\calh$ and $\calh'$. 

More explicitly, this is the set of all arrows of the form $h_1\dots h_k$, with each $h_i$ in either $\calh$ or $\calh'$.
\end{de}

The subgroupoid $\grp{\calh,\calh'}$ is a Borel subgroupoid of $\calg$ as can be seen by writing $\calh$ and $\calh'$
as a countable union of bisections (see below).

\begin{de}[Restrictions]\label{de:restriction}
Given a Borel groupoid over a base space $Y$, and a Borel subset $A\subseteq Y$, the set $$\calg_{|A}:=\{g\in\calg|s(g),r(g)\in A\}$$ has a natural structure of a Borel groupoid over the base space $A$, induced from the groupoid structure of $\calg$. 
The groupoid $\calg_{|A}$ is called the \emph{restriction} of $\calg$ to $A$.
\end{de}

\begin{rk}
  Note that $\calg_{|A}$ is not a subgroupoid of $\calg$ (it is not defined on the same base space).
\end{rk}

Every countable group has a natural structure of a Borel groupoid over a point. A more interesting -- and crucial -- example comes from restrictions of group actions, as explained right below (another interesting class of examples comes from equivalence relations, but we will not need those in the present paper). All countable groups will always be equipped with the discrete topology. 

\begin{ex}[Group actions]\label{ex-groupoid}
Let $Y$ be a standard Borel space, equipped with an action of a countable group $\Gamma$ by Borel automorphisms. 
Then the direct product $\Gamma\times Y$ has a natural structure of a discrete Borel groupoid over the base space $Y$: 
one defines $s(\gamma,y)=y$ and $r(\gamma,y)=\gamma\cdot y$, and the composition law is given by $$(\gamma_1,\gamma_2\cdot y)(\gamma_2,y)=(\gamma_1\gamma_2,y),$$ 
the unit map is given by $e_y=(1_\Gamma,y)$, and the inverse map by $(\gamma,y)^{-1}=(\gamma^{-1},\gamma\cdot y)$. This groupoid is denoted as $\Gamma\ltimes Y$. 

If $A\subseteq Y$ is any Borel subset (without any invariance hypothesis), then the groupoid $(\Gamma\ltimes Y)_{|A}$ is well defined;
its set of arrows is the set of pairs $(a,\gamma)\in A\times \Gamma$ such that $\gamma\cdot a\in A$.
\end{ex}

In the above example, the groupoid $(\Gamma\ltimes Y)_{|A}$  may be written as the countable union of subsets $B_\gamma=\{\gamma\}\times (A\cap \gamma\m A)$ for $\gamma\in\Gamma$. 
There is a partial isomorphism from $s(B_\gamma)=A\cap\gamma\m A$ to $r(B_\gamma)=\gamma A\cap A$ defined by $y\mapsto \gamma\cdot y$. 
More generally, a bisection is a collection of arrows that defines a partial isomorphism in the following sense.

\begin{de}[Bisections]\label{de:bisection}
Let $\calg$ be a Borel groupoid over a base space $Y$. A \emph{bisection} of $\calg$ is a Borel subset $B\subseteq \calg$ such that $s_{|B}$ and $r_{|B}$ are Borel isomorphisms to Borel subsets  $s(B),r(B)\subseteq Y$.  

The \emph{partial isomorphism} associated to a bisection $B$ is the Borel map $f_B:s(B)\ra r(B)$ defined by $f_B=r\circ s_{|B}\m$.
\end{de}

A theorem by Lusin and  Novikov (see \cite[Theorem~18.10]{Kec}) shows that any discrete Borel groupoid $\calg$ is always a countable union of bisections (this is obvious in the case of a restriction of a group action, but for the general case, it is crucial that Borel spaces are standard).

\begin{de}[Saturation]\label{de:saturation}
Let $\calg$ be a Borel groupoid over a base space $Y$, and let $y_0\in Y$. The \emph{$\calg$-orbit} of $y_0$ is the set of all $y\in Y$ such that there exists $g\in \calg$ with
$s(g)=y_0$ and $r(g)=y$. The \emph{saturation} $\calg A$ of a Borel subset $A\subseteq Y$ is the set of points that are in the orbit of some point in $A$, in other words $\calg A=r(s^{-1}(A))$.
\end{de}

If $A\subset Y$ is a Borel subset, then so is $\calg A$: indeed, writing $\calg$ as a countable union of bisections $B_n$, one has $\calg A=\cup_n f_{B_n}(A)$.

\subsubsection{Measured groupoids}

Let $\calg$ be a Borel groupoid over a base space $Y$. A measure $\mu$ on $Y$ is \emph{$\calg$-invariant} if for every bisection $B\subseteq \calg$, one has $\mu(s(B))=\mu(r(B))$. It is \emph{$\calg$-quasi-invariant} if for every bisection $B\subseteq \calg$, one has $\mu(s(B))=0$ if and only if $\mu(r(B))=0$. Note that in the case of an action of a countable group $\Gamma$ by Borel automorphisms on $Y$, the group $\Gamma$ preserves $\mu$ if and only if the groupoid $\Gamma\ltimes Y$ does.

A \emph{standard measure space} $(Y,\mu)$ is a standard Borel space $Y$ equipped with a $\sigma$-finite positive Borel measure $\mu$.

\begin{de}[Measured groupoid]
Let $(Y,\mu)$ be a standard measure space. A \emph{measured groupoid} over $(Y,\mu)$ is a discrete Borel groupoid $\calg$ over $Y$ such that $\mu$ is $\calg$-quasi-invariant.
\end{de}

\begin{rk}\label{rk:proba}
If we change the measure $\mu$ into another measure $\mu'$ in the same measure class, then 
$\calg$ is still a measured groupoid over $(Y,\mu')$, and this does not change $\calg$ up to isomorphism
(see Definition \ref{de:isomorphic} below). 
Thus, we can always assume without loss of generality that $\mu$ is a probability measure.
\end{rk}

When $\calg$ is a measured groupoid over a base space $(Y,\mu)$, and $\calh$ is a Borel subgroupoid of $\calg$, then $\mu$ is $\calh$-quasi-invariant. The Borel subgroupoid $\calh$ together with the measure $\mu$ is called a \emph{measured subgroupoid} of $\calg$.

From a measure $\mu$ on the base space $Y$, one can define a measure $\nu$ on the set of arrows $\calg$ by
$$\nu(A)=\int_Y\sum_{g\in s\m(y)}\un_A(g)d\mu(y)$$  
for any Borel subset $A\subseteq \calg$.
Then $\mu$ is preserved (resp.\ quasi-preserved)  by $\calg$ if and only if $\nu$ is preserved by the inversion map:
$\iota_{\ast}\nu=\nu$ (resp.\ $\nu$ and $\iota_{\ast}\nu$ have the same null sets).

\subsubsection{Cocycles}

\begin{de}[Groupoid homomorphisms and isomorphisms]\label{de:isomorphic}
Let $\calg_1,\calg_2$ be measured groupoids over standard measure spaces $(Y_1,\mu_1)$ and $(Y_2,\mu_2)$, respectively. A \emph{homomorphism of measured groupoids} from $\calg_1$ to $\calg_2$ is a Borel map $\phi:\calg_1\to \calg_2$ such that there exists a conull Borel subset $Y_1^*\subseteq Y_1$ on which the following hold:
\begin{enumerate}
\item for all $(g,h)\in ((\calg_1)_{|Y_1^*})^{(2)}$, one has $\phi(gh)=\phi(g)\phi(h)$, and
\item there exists a Borel map $\phi_Y:Y_1^*\ra Y_2$ such that
\begin{enumerate}
\item for every $y\in Y_1^*$, one has $e_{\phi_Y(y)}=\phi(e_y)$, and 
\item the measures $(\phi_{Y})_{\ast}((\mu_1)_{|Y_1^*})$ and  $\mu_2$ have the same null sets.
\end{enumerate}
\end{enumerate}  
It is an \emph{isomorphism} if in addition, there exist a groupoid homomorphism $\psi:\calg_2\to\calg_1$ and conull Borel subsets $Y_1^*\subseteq Y_1$ and $Y_2^*\subseteq Y_2$ such that $\psi\circ \phi$ and $\phi\circ \psi$ are the identity maps on $(\calg_1)_{|Y_1^*}$ and $(\calg_2)_{|Y_2^*}$, respectively.
\end{de}

\begin{de}[Cocycles]\label{dfn_cocycle} 
  Let $\calg$ be a measured groupoid and $\Gamma$ be a topological group.
  A \emph{cocycle} $\rho:\calg\ra\Gamma$ is a Borel map such that 
$\rho(gh)=\rho(g)\rho(h)$ for all $(g,h)\in\calg^{(2)}$.
\end{de}

\begin{rk}  Usually, in the definition of a cocycle one only asks that 
$\rho(gh)=\rho(g)\rho(h)$ for almost every (and not all) $(g,h)\in\calg^{(2)}$.
Cocycles satisfying Definition \ref{dfn_cocycle} are usually called \emph{strict}.
When $\Gamma$ is countable, for any non-strict cocycle, there exists a conull subset $Y^*\subseteq Y$ 
such that $\rho_{|Y^*}$ is strict. More generally, when $\Gamma$ is second countable, then
any non-strict cocycle coincides almost everywhere with a strict cocycle (\cite[Theorem~B.9]{Zim}).
All the cocycles we consider in this paper take values in second countable groups (and in fact mostly in countable discrete groups).
For these reasons, all the cocycles we consider are strict as in Definition~\ref{dfn_cocycle}.
\end{rk}

\begin{de}\label{dfn:trivial_kernel}
The \emph{kernel} of a cocycle $\rho$ is the set of all arrows $g\in \calg$ such that $\rho(g)=1_\Gamma$. We say that $\rho$ has \emph{trivial} kernel if its kernel is reduced to the set of unit elements of $\calg$.
\end{de}

Notice that the kernel of a cocycle $\calg\to\Gamma$ is always a Borel subgroupoid of $\calg$. More generally, if $\rho:\calg\to\Gamma$ is a cocycle and $H\subseteq \Gamma$ is a subgroup, then $\rho^{-1}(H)$ is a subgroupoid of $\calg$.

\begin{ex}[Natural cocycle associated to a group action]
Let $\calg:=\Gamma\ltimes Y$ be a groupoid associated to a quasi-measure-preserving action of a countable group $\Gamma$ by Borel automorphisms on a standard Borel space $Y$ equipped with a $\sigma$-finite Borel measure $\mu$. 
Let $A\subseteq Y$ be a Borel subset of positive measure.
The \emph{natural cocycle} associated to the restriction to $A$ of the $\Gamma$-action on $Y$ is the cocycle 
\begin{displaymath}
\begin{array}{cccc}
\rho:&(\Gamma\ltimes Y)_{|A} & \to & \Gamma\\
&(\gamma,y) & \mapsto & \gamma
\end{array}
\end{displaymath}  
Notice that $\rho$ is indeed a cocycle and has trivial kernel.
\end{ex}

The following lemma shows that conversely, every measured groupoid which admits a cocycle with trivial kernel towards a countable group $\Gamma$, is a restriction of a groupoid coming from a measure-preserving action of $\Gamma$.

\begin{lemma}\label{lemma:trivial-kernel}
Let $\calg$ be a measured groupoid over a base space $(Y,\mu)$, let $\Gamma$ be a countable group, and let $\rho:\calg\to\Gamma$ be a cocycle with trivial kernel.

Then there exist a standard Borel space $S$ with a $\sigma$-finite Borel measure $\nu$ and a quasi-measure-preserving $\Gamma$-action on $S$ by Borel automorphisms such that $\calg$ is isomorphic to a restriction of $\Gamma\ltimes S$, and $\rho$ is the restriction of the natural cocycle $\Gamma\ltimes S\to\Gamma$.
\end{lemma}

\begin{rk}  
If the measure $\mu$ is $\calg$-invariant, the constructed measure $\nu$ is $\Gamma\ltimes S$-invariant
but infinite in general even if $\mu$ is a probability measure.
\end{rk}

\begin{proof}
This is an adaptation of an argument given by Kida in his proof of \cite[Proposition~4.33]{Kid1}. Let $\eta$ be the counting measure on $\Gamma$, and let $\tilde{\mu}:=\mu\otimes\eta$ be the product measure on $Y\times \Gamma$. Let $\tilde\calg:=\calg\times \Gamma$. 
The space $\tilde\calg$ has a natural structure of a Borel groupoid over $Y\times\Gamma$ with $s(g,\gamma)=(s(g),\gamma)$ and $r(g,\gamma)=(r(g),\rho(g)\gamma)$ for all $(g,\gamma)\in\tilde{\calg}$, and one checks that $\tilde \calg$ quasi-preserves the measure $\tilde\mu$.

Let $\sim$ be the $\tilde\calg$-orbit equivalence relation on $Y\times\Gamma$, i.e.\ two points $(y_1,\gamma_1)$ and $(y_2,\gamma_2)$ are equivalent whenever there exists $h\in\tilde\calg$ such that $s(h)=(y_1,\gamma_1)$ and $r(h)=(y_2,\gamma_2)$. 
We claim that the equivalence relation $\sim$ has a fundamental domain $D$.
Indeed, triviality of the kernel of $\rho$ shows that for every $\gamma\in\Gamma$, the equivalence relation $\sim_{|Y\times \{\gamma\}}$ is trivial. 
Now choose an enumeration $\{\gamma_i\}_{i\in \bbN}$ of $\Gamma$, and define inductively $D_0=Y\times\{\gamma_0\}$, 
and $D_{i+1}=(Y\times\{\gamma_i\})\setminus \tilde\calg D_i$ where $\Tilde\calg D_i$ is the saturation of $D_i$.
Then $D=\cup_i D_i$ is a Borel fundamental domain for $\sim$.

Consider the space $S:=(Y\times\Gamma)/{\sim}$. Then $S$ is naturally isomorphic to $D$ (in particular $S$ is a standard Borel space),
and we define $\nu$ on $S$ as $\Tilde\mu_{|D}$ through this identification. Given $(y,\gamma)\in Y\times\Gamma$, we denote by $[y,\gamma]$ the class of $(y,\gamma)$ in $S$. The group $\Gamma$ acts on $S$ via $\gamma\cdot [y',\gamma']:=[y',\gamma'\gamma^{-1}]$ in a quasi-measure-preserving way. Then $\calg$ is isomorphic to $(\Gamma\ltimes S)_{|Y\times\{e\}}$ via $g\mapsto (\rho(g),[s(g),e])$, see \cite[Lemmas~4.34 and~4.35]{Kid1}.  
\end{proof}

The following definition will be used at various places in the paper.

\begin{de}[Nowhere trivial, stably trivial cocycle]\label{dfn:nowhere_trivial}
Let $\calg$ be a measured groupoid over a base space $Y$ and let $\Gamma$ be a countable group. A cocycle $\rho:\calg\to \Gamma$ is 
\begin{itemize}
\item \emph{nowhere trivial} if there does not exist any Borel subset $U\subseteq Y$ of positive measure such that $\rho(\calg_{|U})=\{1\}$;
\item  \emph{stably trivial} if there exists a Borel partition $Y^*=\Dunion_{i\in I}Y_i$ into at most countably many Borel subsets such that for every $i\in I$, one has $\rho(\calg_{|Y_i})=\{1\}$.
\end{itemize}
\end{de}

We make the following observation.

\begin{lemma}\label{lem_trivial_partition}
Let $\calg$ be a measured groupoid over a base space $Y$, $\Gamma$ be a countable group and $\rho:\calg\to \Gamma$ a cocycle.

Then there exists a Borel partition $Y=Y_1\dunion Y_2$
such that $\rho$ is nowhere trivial on $\calg_{|Y_1}$
and is stably trivial on $\calg_{|Y_2}$.
\end{lemma}

\begin{proof}
Without loss of generality, we assume that the measure on $Y$ is a probability measure (see Remark \ref{rk:proba}).
We choose for $Y_1$ a Borel subset of maximal measure such that $\rho$ is stably trivial on $\calg_{|Y_1}$: this exists because if $(Y_{1,n})_{n\in\mathbb{N}}$ is a measure-maximizing sequence of such sets, then letting $Y_{1,\infty}=\bigcup Y_{1,n}$, we have again that $\rho$ is stably trivial on $\calg_{|Y_{1,\infty}}$. The proof is then completed by letting $Y_2=Y\setminus Y_1$: the maximality of $Y_1$ ensures that $\rho$ is nowhere trivial on $\calg_{|Y_2}$.
\end{proof}

\subsubsection{Groupoids of infinite type and action-type cocycles}

We adopt the following definition of a measured groupoid of infinite type, which is equivalent to Kida's definition  \cite[Definition~3.1]{Kid2} in the case where $\mu$ is a finite $\calg$-invariant measure, see the discussion above \cite[Definition~2.25]{Kid-survey}. It is a notion of recurrence for a measured groupoid.

\begin{de}[Groupoids of infinite type]
A measured groupoid $\calg$ over a base space $(Y,\mu)$ is \emph{of infinite type} if for every Borel subset $U\subseteq Y$ of positive measure, and $\mu$-a.e.\ $y\in U$, the set $$\{g\in\calg_{|U}\, |\,s(g)=y\}= \{g\in \calg\, |\, s(g)=y,\ t(g)\in U\}$$ 
is infinite.
\end{de}

Notice that if $\calg$ is of infinite type, then so is $\calg_{|U}$ for every Borel subset $U\subseteq Y$ of positive measure.
A crucial example comes from Poincaré recurrence:
if an infinite group acts by preserving a probability measure on $Y$, then the associated groupoid is of infinite type
(here the fact that the measure is finite and invariant is essential).
The following definition reformulates this property in terms of the associated natural cocycle.
It will be of central importance in the present work.

\begin{de}[Action-type cocycles]\label{de:action-type}
Let $\calg$ be a measured groupoid, let $\Gamma$ be a countable group. 
A cocycle $\rho:\calg\to\Gamma$ is  \emph{action-type} if it has trivial kernel, 
and for every infinite subgroup $H\subseteq\Gamma$, the subgroupoid $\rho^{-1}(H)$ is of infinite type.
\end{de} 

\begin{lemma}[See for instance {\cite[Proposition~2.26]{Kid-survey}}]\label{lemma:infinite_type}
Let $\Gamma$ be a countable group, let $Y$ be a standard Borel space equipped with a $\Gamma$-action by Borel automorphisms which preserves a Borel probability measure $\mu$. Let $\calg=\Gamma\ltimes (Y,\mu)$, and let $\rho:\calg\to\Gamma$ be the natural cocycle.

Then $\rho$ is action-type.
\end{lemma}

We will systematically use the following obvious facts without mention.

\begin{lemma}\label{lemma:action-type}
 Let $\calg$ be a measured groupoid over a standard measure space $Y$, let $\Gamma$ be a countable group, and let $\rho:\calg\to\Gamma$ be an action-type  cocycle.
\begin{enumerate}\setcounter{enumi}{0}
\item For every Borel subset $U\subseteq Y$ of positive measure, the restricted cocycle $\rho:\calg_{|U}\to\Gamma$ is action-type.
\item For every subgroup $\Gamma'\subseteq \Gamma$, the induced cocycle $\rho\m(\Gamma')\to\Gamma'$ is action-type. \qed 
\end{enumerate}
\end{lemma}

\begin{rk}\label{rk:action-type}
We will often use the following consequence of the action-type property.
Let $\rho:\calg\ra\Gamma$ be action-type, $H\subseteq\Gamma$ be a subgroup, and $\calh=\rho\m(H)$.
Then for every element $h\in H$ of infinite order
and every Borel subset $U\subseteq Y$ of positive measure, $\rho(\calh_{|U})$ contains a power of $h$.
\end{rk}

\subsubsection{Stable equality of measured groupoids}

In the present paper, we will often consider measured groupoids up to stable  equality, defined as follows.

\begin{de}[Stable containment, stable  equality]
   Let $\calg$ be a measured groupoid over a base space $Y$, and 
let $\calh$ and $\calh'$ be two measured subgroupoids.

The groupoid $\calh$ is \emph{stably equal} to $\calh'$ (resp.\ \emph{stably contained} in $\calh'$)
if there exist a conull subset $Y^*\subseteq Y$ and a countable Borel partition $Y^*=\Dunion_{i\in I}Y_i$ such that for every $i\in I$, one has $\calh_{|Y_i}=\calh'_{|Y_i}$ (resp.\ $\calh_{|Y_i}\subseteq\calh'_{|Y_i}$).
\end{de}

Note that a finite partition is considered as countable.

\begin{rk}\label{rk:restriction-destabilisation}
In the present paper, 
we will often look for properties $(P)$ of measured groupoids that are preserved by restriction and stabilization: 
\begin{enumerate}
\item $(P)$ is preserved by restriction if, for every  measured groupoid $\calg$ over a base space $(Y,\mu)$, if $\calg$ satisfies $(P)$, 
then so does its restriction $\calg_{|A}$ for every Borel subset $A\subseteq Y$ of positive measure.
\item $(P)$ is preserved by stabilization if, 
for every measured groupoid $\calg$ over a base space $(Y,\mu)$, and for every countable Borel partition $Y^*=\Dunion_{i\in I}Y_i$ of a conull Borel subset $Y^*\subset Y$,
if the groupoid $\calg_{|Y_i}$ satisfies $(P)$ for every $i\in I$, then so does $\calg$.
Equivalently, if $\calg,\calg'$ are stably equal and $\calg$ satisfies $(P)$, then so does $\calg'$.
\end{enumerate}
\end{rk}

\subsubsection{Invariant objects, equivariant maps}

The following definition is a generalization to the groupoid setting of the notion of fixed points of a group action. 

\begin{de}[(Stably) equivariant maps] \label{dfn_invariant}
Let $\calg$ be a measured groupoid over a base space $(Y,\mu)$, let $\Gamma$ be a countable group, and let $\rho:\calg\to\Gamma$ be a cocycle. Let $\Delta$ be a measurable space equipped with a  measurable $\Gamma$-action, and let $\phi:Y\to\Delta$ be a  measurable map.

An element $x\in\Delta$ is \emph{$(\calg,\rho)$-invariant}  if 
there exists a conull Borel subset $Y^*\subseteq Y$ such that 
for all $g\in\calg_{|Y^*}$, one has $\rho(g)x=x$.

We say that $\phi$ is \emph{$(\calg,\rho)$-equivariant} if there exists a conull Borel subset $Y^*\subseteq Y$ such that for all $g\in\calg_{|Y^*}$, one has $\phi(r(g))=\rho(g)\phi(s(g))$.   

We say that $\phi$ is \emph{stably $(\calg,\rho)$-equivariant} if there exist a conull Borel subset $Y^*\subseteq Y$ and a countable Borel partition $Y^*=\Dunion_{i\in I} Y_i$ such that for every $i\in I$, the map $\phi_{|Y_i}$ is $(\calg_{|Y_i},\rho)$-equivariant.
\end{de}

\begin{rk}
The set $\rho(\calg_{|Y^*})$ is a subset of $\Gamma$ with no algebraic structure.
Yet, the element $x\in\Delta$ is $(\calg,\rho)$-invariant if and only if  $\rho(\calg_{|Y^*})$
is contained in the stabilizer $\Gamma_x$ of $x$ for some conull Borel subset $Y^*\subseteq Y$. 
\end{rk}

\begin{ex}
Let $\calg:=\Gamma\ltimes Y$ be a groupoid associated to a measure-preserving action of a countable group $\Gamma$ on some standard measure space $Y$, let $\rho:\calg\to\Gamma$ be the natural cocycle.
Let $\Delta$ be a measurable space equipped with a $\Gamma$-action by measurable automorphisms. Then a measurable map $Y\to\Delta$ is
$(\calg,\rho)$-equivariant if and only if it is almost everywhere $\Gamma$-equivariant.
\end{ex}

\begin{rk}\label{rk_eqv_vs_inv}
The case where the space $\Delta$ is a countable discrete set will be of particular importance. 
In this case, any Borel map $\varphi:Y\ra \Delta$ is \emph{stably constant}:
there exists a countable Borel partition $Y=\dunion_{i\in\bbN} Y_i$
such that $\varphi$ is constant in restriction to each $Y_i$.
For this reason, the two following statements are equivalent when $\Delta$ is countable:
\begin{itemize}
    \item 
there exists a stably $(\calg,\rho)$-equivariant map $\varphi:Y\to\Delta$ 
\item
there exists a countable Borel partition $Y^*=\dunion_{i\in\mathbb{N}}Y_i$ of a conull Borel subset $Y^*\subseteq Y$, such that for every $i\in I$, there exists an element $x_i\in\Delta$ which is $(\calg_{|Y_i},\rho)$-invariant (i.e.\ $\rho(\calg_{|Y_i})\subseteq\Stab_\Gamma(x_i)$). 
\end{itemize}


In Part~\ref{part_ME} of the present work, there will be two very distinct situations: we will have to deal both with stably equivariant maps with values in a countable set $\Delta$ (e.g.\ a set of splittings or free factors), and others with values in uncountable Polish spaces (e.g.\ spaces of actions on $\mathbb{R}$-trees). These two situations are dealt with very differently. The above suggests that, when $\Delta$ is countable, a crucial part of the work consists in having a fine understanding of (group-theoretic) stabilizers of collections of splittings or factors. On the other hand, the case where $\Delta$ is uncountable will require arguments with a dynamical flavour, involving amenability of group actions (see Section~\ref{sec_amenability} below).
\end{rk}





\begin{rk}\label{rk:stably-invariant}
  Although stably equivariant maps are not always equivariant, the following is true: 
  if there exists a stably $(\calg,\rho)$-equivariant measurable map $Y\to\Delta$, then there exists another map $Y\to\Delta$ which
  is $(\calg,\rho)$-equivariant.
The key point is that any map whose restriction to a Borel subset $A\subseteq Y$ of positive measure is
$(\calg_{|A},\rho)$-equivariant extends uniquely to an equivariant map on the saturation $\calg A$ (see \cite[Lemma~2.7]{Ada} or \cite[Lemma~4.16]{Kid1}).
\end{rk}

\begin{de}[Stabilizers]
  Let $\calg$ be a measured groupoid over a base space $Y$. Let $\Gamma$ be a countable group, and let $\rho:\calg\to\Gamma$ be a cocycle. Let $\Delta$ be a measurable space equipped with a  measurable $\Gamma$-action, and let $\phi:Y\to \Delta$ be a measurable map.

The \emph{$(\calg,\rho)$-stabilizer} of $\phi$ is the subgroupoid of $\calg$ made of all $g\in \calg$ such that $\phi(r(g))=\rho(g)\phi(s(g))$. 
\end{de}

\subsection{Normal subgroupoids}

A notion of normal subgroupoid of a measured groupoid was coined by Kida in \cite[Section~4.6.1]{Kid1}, building on earlier work by Feldman, Sutherland and Zimmer \cite{FSZ}. One can check that the definition of normality given below is equivalent to Kida's, but we will not use this fact.

Let $\calg$ be a measured groupoid over a base space $Y$, let $\calh$ be a measured subgroupoid of $\calg$, and let $B\subseteq\calg$ be a Borel subset (typically, $B$ may be a bisection, see Definition~\ref{de:bisection}, but it need not be). 
We say that $\calh$ is \emph{$B$-invariant} if there exists a conull Borel subset $Y^*\subseteq Y$ such that for all $g_1,g_2\in B\cap\calg_{|Y^*}$ and all $h\in\calg_{|Y^*}$ such that $g_2,h,g_1\m$ are composable, one has $h\in\calh$ if and only if $g_2hg_1^{-1}\in\calh$. 

\begin{de}[(Stable) normalization]\label{de:normal}
Let $\calg$ be a measured groupoid over a base space $Y$, and let $\calh$ and $\calh'$ be measured subgroupoids of $\calg$.

The groupoid $\calh'$ \emph{normalizes} $\calh$ if there exists a countable collection of Borel subsets $B_n\subseteq\calg$ such that $\calh'=\bigcup_{n\in\mathbb{N}}B_n$, and for every $n\in\mathbb{N}$, the subgroupoid $\calh$ is $B_n$-invariant.

The groupoid $\calh'$ \emph{stably normalizes} $\calh$ if there exists a countable Borel partition $Y^*=\Dunion_{i\in I}Y_i$ of a conull Borel subset $Y^*\subset Y$
such that for every $i\in I$, the restriction $\calh'_{|Y_i}$ normalizes $\calh_{|Y_i}$.
\end{de} 

\begin{rk}\label{rk:bisection}  
Note that $B_n$ is not assumed to be a subgroupoid.
  Since $\calg$ is a countable union of bisections, one may assume without loss of generality that every $B_n$ is a bisection of $\calg$.
\end{rk}

When $\calh\subseteq\calh'$ and $\calh'$ normalizes $\calh$, we say that $\calh$ is \emph{normal} in $\calh'$, denoted by $\calh\unlhd\calh'$. 
If $\calg=\Gamma\ltimes Y$ is the groupoid associated to the action of a countable group $\Gamma$ on a standard measure space $Y$, and if $H\unlhd \Gamma$ is a normal subgroup,
then the subgroupoid $H\ltimes Y$ is normal in $\calg$. Indeed, this is shown by letting $B_\gamma:=\{(\gamma,y)|y\in Y\}$: then the Borel subsets $B_\gamma$ are bisections that cover $\calg$, and normality of $H$ in $\Gamma$ implies that $H\ltimes Y$ is $B_\gamma$-invariant for every $\gamma\in\Gamma$.

If $Y^*\subset Y$ is a conull Borel subset, then $\calh'$ normalizes $\calh$ if and only if $\calh'_{|Y^*}$ normalizes $\calh_{|Y^*}$.
The notion of stable normalization is clearly preserved under restriction and stabilization.
Notice also that the preimage of a normal subgroupoid by a groupoid homomorphism is again a normal groupoid. In particular, we record the following statement for future use.

\begin{lemma}\label{lemma:normal-subgroup-subgroupoid}
  Let $\calg$ be a measured groupoid over a base space $(Y,\mu)$, let $\Gamma$ be a countable group, and let $\rho:\calg\to\Gamma$ be a cocycle. Let $H_1,H_2$ be two subgroups of $\Gamma$ with $H_1\unlhd H_2$.

Then $\rho^{-1}(H_1)\unlhd \rho^{-1}(H_2)$.
\qed
\end{lemma}

\subsection{Amenability}\label{sec_amenability}

\subsubsection{Amenability of a measured groupoid}

We review the definition of the amenability of a measured groupoid, which is a generalization of Zimmer's definition \cite[Definition~1.4]{Zim2} of an amenable group action, see e.g.\ \cite{ADR}. 
Let $\calg$ be a measured groupoid over a base space $Y$. 
Let $B$ be a separable real Banach space. We denote by $\Isom(B)$ the group of all linear isometries of $B$, equipped with the strong operator topology (this is a Polish group see e.g.~\cite[I.9.B.(9)]{Kec}). 
We denote by $B_1^\ast$ the unit ball of the dual Banach space $B^\ast$, equipped with the weak-$\ast$ topology. Given a linear isometry $T$ of $B$, we denote by $T^\ast$ the restriction to $B_1^\ast$ of the adjoint operator of $T$. 
The \emph{adjoint cocycle} $\rho^\ast$ of a cocycle $\rho:\calg\to\Isom(B)$ is defined as $\rho^{\ast}(g):=(\rho(g)^{-1})^\ast$. 
Then $\rho^{\ast}$ takes its values in the homeomorphism group of the compact metrizable space $X_1^\ast$, which is a Polish group (see e.g.\ \cite[I.9.B.(8)]{Kec}). We denote by $\Conv$ the set of all nonempty convex compact subsets of $B_1^\ast$. 
A map $K:Y\to\Conv$ (which we will refer to as a \emph{convex field over $Y$}) is \emph{measurable} if $$\{(y,k)\in Y\times B_1^\ast| k\in K(y)\}$$ is a Borel subset of $Y\times B_1^\ast$. A \emph{section} of a measurable convex field $K$ over $Y$ is a Borel map $s:Y\to B_1^\ast$ such that for a.e.\ $y\in Y$, one has $s(y)\in K(y)$. 

\begin{de}[Amenable groupoids]\label{dfn_amenable_groupoid}
A measured groupoid $\calg$ over a base space $(Y,\mu)$ is \emph{Zimmer amenable} (or simply \emph{amenable}) if for every separable real Banach space $B$ and every  cocycle $\rho:\calg\to\Isom(B)$, every stably $(\calg,\rho^\ast)$-equivariant measurable convex field over $Y$ has a stably $(\calg,\rho^\ast)$-invariant section. 
\end{de}

When $\calg=\Gamma\ltimes Y$ is the groupoid associated to a measure-preserving
ergodic action of a countable group $\Gamma$ on a standard probability space $Y$,
the definition above exactly coincides with Zimmer's Definition 1.4 in \cite{Zim2}.

\begin{rk}
The definition of an amenable groupoid is usually given in terms of invariant convex fields and invariant sections instead of stably equivariant ones. However, the above definition is equivalent to the usual one in view of Remark~\ref{rk:stably-invariant}. 
\end{rk}

Every restriction and every subgroupoid of an amenable measured groupoid is amenable, see e.g.\ \cite[Theorem ~4.16]{Kid-survey} whose proof relies on work of Connes, Feldman and Weiss \cite{CFW}. 
In addition, it is clear from the definition that amenability is preserved by stabilization, i.e.\ if
$Y^*=\Dunion_{i\in I} Y_i$ for some conull Borel subset $Y^*\subset Y$, and if for every $i\in I$, the groupoid $\calg_{|Y_i}$ is amenable, then $\calg$ is amenable.

\begin{de}\label{dfn:everywhere_nonamenable}
  A measured groupoid $\calg$ over a base space $(Y,\mu)$ is \emph{everywhere non-amenable}
  if for every Borel subset $U\subseteq Y$ of positive measure, $\calg_{|U}$ is not amenable.
\end{de}

\subsubsection{Amenable groupoids from Borel amenable actions}


We will need the notion of Borel amenability of a Borel group action $\Gamma \actson \Delta$, where $\Delta$ is a standard Borel space (in Part~\ref{part_ME} of this work, $\Delta$ will be a space of $\mathbb{R}$-trees). This notion is phrased without reference to any quasi-invariant measure on $\Delta$, but it implies the Zimmer amenability of the action with respect to any quasi-invariant measure on $\Delta$, see e.g.\ \cite[Proposition~2.5]{GHL}.

We denote by $\Prob(\Gamma)$ the space of all probability measures on the countable group $\Gamma$, equipped with the topology of pointwise convergence. 

\begin{de}[Borel amenability of a group action \cite{Ren}]\label{de:borel-amenable}
Let $\Gamma$ be a countable group, and let $\Delta$ be a standard Borel space equipped with a $\Gamma$-action by Borel automorphisms.
The $\Gamma$-action on $\Delta$ is \emph{Borel amenable} if there exists a sequence of Borel maps $\nu_n:\Delta\to\Prob(\Gamma)$ such that for every $\delta\in\Delta$, the norm $||\nu_n(\gamma\cdot\delta)-\gamma_*\nu_n(\delta)||_1$ converges to $0$ as $n$ goes to $+\infty$.
\end{de}

\begin{rk} In \cite[Definition 2.1]{Ren}, Renault gives the definition in the more general context of Borel amenable groupoids.
For equivalence relations, this corresponds to $1$-amenability in \cite[Definition~2.12]{JKL}. We will only use these notions in the context of a group action as in Definition \ref{de:borel-amenable}.
\end{rk}

The following proposition gives a sufficient criterion to ensure the Zimmer amenability of a measured groupoid. 

\begin{prop}[{Kida \cite[Proposition~4.33]{Kid1}}]\label{prop:mackey}
Let $\calg$ be a measured groupoid over a base space $(Y,\mu)$, let $\Gamma$ be a countable group, and let $\rho:\calg\to\Gamma$ be a cocycle with trivial kernel. 
Let $\Delta$ be a  standard Borel space equipped with a Borel amenable $\Gamma$-action. Assume that there exists a stably $(\calg,\rho)$-equivariant Borel map $Y\to\Delta$.

Then $\calg$ is amenable.
\end{prop}

\begin{proof}
The proof is due to Kida \cite[Proposition~4.33]{Kid1}, but as he phrased it in the specific context of mapping class groups, we include an argument for completeness of the exposition.

Since amenability is preserved by stabilization, we may assume without loss of generality that there exists a $(\calg,\rho)$-invariant Borel map $f:Y\to \Delta$. By Lemma~\ref{lemma:trivial-kernel}, there exist a standard measure space $(\hat Y,\hat \mu)$ containing $(Y,\mu)$ and a measure-class preserving $\Gamma$-action on $\hat Y$, such that $\calg$ identifies with the restriction $(\Gamma\ltimes \hat Y)_{|Y}$ in such a way that $\rho$ is the restriction of the natural cocycle $\hat \rho:(\Gamma\ltimes \hat Y)\ra \Gamma$.

We can assume that $\hat Y$ is the saturation of $Y$ for the groupoid $\Gamma\ltimes \hat Y$.
Let $\nu$ be a probability measure on $\hat Y$ in the same measure class as $\hat \mu$.
Using Remark~\ref{rk:stably-invariant}, the map $f:Y\ra \Delta$ can be extended to a $(\Gamma\ltimes \hat Y,\hat \rho)$-equivariant
Borel map $\hat Y\to\Delta$.
In other words, there exists a $\Gamma$-invariant conull subset $\hat Y^*\subseteq \hat Y$
and $\Gamma$-equivariant Borel map $\hat f:\hat Y^*\to\Delta$. 
Since the $\Gamma$-action on $\Delta$ is Borel amenable, 
the $\Gamma$-action on $(\Delta,\hat f_*\nu)$ is Zimmer amenable \cite[Proposition~2.5]{GHL}.
It then follows from \cite[Corollary~C]{AEG} that the $\Gamma$-action on $(\hat Y,\hat\mu)$ is Zimmer amenable. In other words the groupoid $\Gamma\ltimes \hat Y$ is amenable, and so is $\calg$. 
\end{proof}

As a special case of Proposition~\ref{prop:mackey}, we mention the following fact.

\begin{cor}\label{cor:amenable-subgroup-subgroupoid}
Let $\calg$ be a measured groupoid, let $\Gamma$ be an amenable countable group. Assume that there exists a cocycle $\calg\to\Gamma$ with trivial kernel.

Then $\calg$ is amenable.\qed
\end{cor}

On the other hand, when the cocycle is action-type, we have the following. 

\begin{lemma}[{Kida \cite[Lemma~3.20]{Kid2}}]\label{lemma:not-amen}
Let $\calg$ be a measured groupoid, let $\Gamma$ be a countable group which contains a non-abelian free group. Assume that there exists an action-type cocycle $\calg\to\Gamma$.

Then $\calg$ is everywhere non-amenable (see Definition~\ref{dfn:everywhere_nonamenable}).
\end{lemma}

\subsubsection{Amenability and invariant probability measures}

One way in which amenability is often used is through the following proposition which is essentially due to Zimmer \cite[Proposition~4.3.9]{Zim}, see e.g.\ \cite[Proposition~4.14]{Kid-survey} for a proof in the language of measured groupoids. Given a compact, metrizable space $K$, we denote by $\Prob(K)$ the space of all probability measures on $K$: this is a convex closed subset (in the weak-$\ast$ topology) of the unit ball of the dual of $C(K)$, the space of continuous $\mathbb{C}$-valued functions on $K$ equipped with the sup norm.

\begin{prop}(Zimmer \cite[Proposition~4.3.9]{Zim})\label{prop:zimmer}
  Let $\calg$ be a measured groupoid over a base space $Y$. Let $\Gamma$ be a countable group, and let $K$ be a compact
    metrizable space on which $\Gamma$ acts continuously. Let $\rho:\calg\to\Gamma$ be a cocycle. Assume that $\calg$ is amenable.

Then there exists a $(\calg,\rho)$-equivariant Borel map $Y\to\Prob(K)$. 
\end{prop}

\section{Groupoids, measure equivalence, and applications}\label{sec:me}

In this section, we review how proving that a countable discrete group $\Gamma$ is rigid for measure equivalence can be reduced to a statement about measured groupoids with cocycles towards $\Gamma$. The criterion we present follows from the work of Kida \cite{Kid2,Kid3}, based on earlier works of Furman \cite{Fur1,Fur2} and Monod and Shalom \cite{MS}, see also \cite{BFS}.

Given a measured groupoid $\calg$ over a  base space $Y$ and a countable group $\Gamma$, two cocycles $\rho,\rho':\calg\to\Gamma$ are \emph{$\Gamma$-cohomologous} if there exist a conull Borel subset $Y^*\subseteq Y$ and a Borel map $\phi:Y\to\Gamma$ such that for every $g\in\calg_{|Y^*}$, one has $$\rho'(g)=\phi(r(g))\rho(g)\phi(s(g))^{-1}.$$ 

We consider the following rigidity property of a countable group $\Gamma$.
\begin{de} \label{de:rigid} A countable group $\Gamma$ is \emph{rigid with respect to action-type cocycles} if
there exists a finite index subgroup $\Gamma^0\subset \Gamma$ such that 
for every measured groupoid $\calg$,
any two action-type cocycles $\rho,\rho':\calg\to\Gamma^0$ are $\Gamma$-cohomologous.
\end{de}

As noted in Remark \ref{rk_commensurator} in the introduction, if $\Gamma$ is rigid with respect to action-type cocycles, applying this property for groupoids over a point yields commensurator rigidity in the following sense: the natural map from $\Gamma$ to its abstract commensurator is surjective. It is injective if and only if $\Gamma$ is \emph{ICC} (standing for \emph{infinite conjugacy classes}), i.e.\ every nontrivial conjugacy class of $\Gamma$ is infinite.

We will now explain why checking rigidity with respect to action-type cocycles for $\Out(F_N)$ is enough to deduce all the results announced in the introduction.

\subsection{Every self coupling of $\Gamma$ factors through the standard one}

Given a free action of a group $G$ by Borel automorphisms on a standard Borel space $\Sigma$, a \emph{Borel fundamental domain} for the $G$-action of $\Sigma$ is a Borel subset $X\subseteq\Sigma$ such that $\Sigma=\bigcup_{g\in G}gX$ and $gX\cap hX=\emptyset$ for any two distinct $g,h\in G$. Recall from the introduction that two countable groups are \emph{measure equivalent} if they admit a coupling in the following sense.

\begin{de}
Let $\Gamma$ and $\Lambda$ be two countable groups. A \emph{coupling} between $\Gamma$ and $\Lambda$ is a standard measure space $\Sigma$ equipped with an action of $\Gamma\times\Lambda$ by measure-preserving Borel automorphisms, such that both the $\Gamma$-action and the $\Lambda$-action on $\Sigma$ are free and have a Borel fundamental domain of finite measure.
\end{de}

A \emph{self coupling} of a countable group $\Gamma$ is a coupling between $\Gamma$ and $\Gamma$. If $\Gamma^1\subseteq \Gamma$ is a finite index subgroup, then there is a \emph{standard} self coupling of $\Gamma^1$ with respect to $\Gamma$, given by $\Sigma=\Gamma$, with $\Gamma^1\times \Gamma^1$ acting on $\Gamma$ by
$$(g,g')\cdot h= ghg'^{-1}.$$

\begin{prop}[{Furman \cite{Fur2}, Kida \cite{Kid2}}]\label{prop:oe-me}
Let $\Gamma$ be a countable ICC group which is rigid with respect to action-type cocycles.

Then for every finite-index subgroup $\Gamma^1\subseteq\Gamma$, every self coupling $\Sigma$ of $\Gamma^1$ factors through the standard self coupling of $\Gamma^1$ with respect to $\Gamma$, in the following sense: there exists a Borel map $\Phi:\Sigma\to\Gamma$ which is almost everywhere $(\Gamma^1\times\Gamma^1)$-equivariant.
\end{prop}

\begin{rk} For an ICC group, the conclusion of Proposition~\ref{prop:oe-me} coincides with the notion of coupling rigidity in \cite{Kid4} (for the pair $(\Gamma,\pi)$ where $\pi:\Gamma_1\ra \Gamma$ is the inclusion). This also agrees with the notion of tautness in \cite{BFS}.
\end{rk}

\begin{proof}
By definition of rigidity with respect to action-type cocycles (Definition~\ref{de:rigid}), there exists a finite-index subgroup $\Gamma^0$ of $\Gamma$ such that for every measured groupoid $\calg$, any two action-type cocycles $\calg\to\Gamma^0$ are $\Gamma$-cohomologous. 
Let $\Gamma^2:=\Gamma^0\cap\Gamma^1$. Let $\Gamma^3$ be a finite-index subgroup of $\Gamma^2$ which is normal in $\Gamma^1$. 

The given self-coupling $\Sigma$ of $\Gamma^1$ is also a self-coupling of $\Gamma^3$.
The proof relies on a standard construction by Furman \cite{Fur2} establishing a dictionary between measure equivalence and stable orbit equivalence.

Let $X$ be a Borel fundamental domain for the action of  $\Gamma^3\times\{1\}$ on $\Sigma$.
We claim that one can choose a Borel fundamental domain $X'$ for the action of  $\{1\}\times \Gamma^3$ on $\Sigma$
such that $Y:=X\cap X'$ satisfies that $(\Gamma^3\times\Gamma^3) Y=\Sigma$ up to a null set. This is e.g.\ \cite[Lemma~2.27]{Kid-survey}, but let us sketch an argument for convenience.
Let $X''$ be an arbitrary fundamental domain for the action of $\{1\}\times \Gamma^3$.
The basic fact is that if we partition $X''$ into countably many pieces, and move these
pieces by elements of $\{1\}\times \Gamma^3$, we get a new fundamental domain for the action of $\{1\}\times \Gamma^3$.
Choose a numbering $\gamma_i$ of the elements of $\Gamma^3$.
Using the fact that $X''$ is a fundamental domain for the action of $\{1\}\times \Gamma^3$,
write $X=\dunion_{i\in \bbN} (1,\gamma_i).Y_i$ for some measurable subset $Y_i\subset X''$.
For every $i\in\mathbb{N}$, let $Y'_i=Y_i\setminus (Y_0\cup \dots\cup Y_{i-1})$. And let $Y_\infty:=X''\setminus \bigcup_{i\in\mathbb{N}}Y_i$. Then $X':=\dunion_i (1,\gamma_i)Y'_i \cup Y_\infty$ is the desired fundamental domain.

Consider the groupoid $\calg=\left((\Gamma^3\times\Gamma^3)\ltimes \Sigma\right)_{|Y}$, and recall that $\calg$ is the set of pairs $((\gamma,\gamma'),y)$
such that both $y$ and $(\gamma,\gamma')y$ lie in $Y$.
Let $\rho:\calg\ra \Gamma^3$ and $\rho':\calg\to\Gamma^3$ be the cocycles defined by $\rho((\gamma,\gamma'),y)=\gamma$
and $\rho'((\gamma,\gamma'),y)=\gamma'$.
We claim that the cocyles $\rho$ and $\rho'$ are action-type. Indeed, 
$X$ comes equipped with a natural action of $\{1\}\times\Gamma^3$, by letting $(1,\gamma')\cdot x$ be the unique element of $X$ in the same $(\Gamma^3\times\{1\})$-orbit as $(1,\gamma')x\in\Sigma$.
There is a natural isomorphism between  $\calg$ and  $\left((\{1\}\times\Gamma^3)\ltimes X\right)_{|Y}$ given
by $((\gamma,\gamma'),y)\mapsto ((1,\gamma'),y)$, and under this identification, $\rho'$ becomes the natural cocycle of
$\left((\{1\}\times\Gamma^3)\ltimes X\right)_{|Y}$. By  Lemmas~\ref{lemma:infinite_type} and~\ref{lemma:action-type}, this shows that $\rho'$ is action-type. By a symmetric argument, so is $\rho$.

By cocycle rigidity, our choice of $\Gamma^0$ ensures that there exist a conull Borel subset $Y^*\subseteq Y$ and a Borel map 
$\phi:Y^*\to\Gamma$ such that for all $g\in\calg_{|Y^*}$, one has 
\begin{equation}\label{eq:coho}
\rho'(g)=\phi(r(g))\rho(g)\phi(s(g))^{-1}.
\end{equation}
Consider the conull Borel subset $\Sigma^*=(\Gamma^3\times\Gamma^3)Y^*$ in $\Sigma$.
We then claim that the assignment  $$(\gamma,\gamma')y\mapsto \gamma\phi(y)\m{\gamma'}\m$$
for all $(\gamma,\gamma')\in\Gamma^3\times\Gamma^3$ and all $y\in Y^*$ determines a $(\Gamma^3\times\Gamma^3)$-equivariant Borel map $\Phi:\Sigma^*\to\Gamma$.
To check that this is well defined, it suffices to check that if $(\gamma,\gamma').y=y'$ with $y,y'\in Y^*$, then
$\gamma\phi(y)\m{\gamma'}\m=\phi(y')\m$.
The arrow $g=((\gamma,\gamma'),y)\in\calg$ has source $s(g)=y$ and range $r(g)=y'$,
so Equation~\eqref{eq:coho} above says that $\gamma'=\phi(y')\gamma\phi(y)\m$ and concludes that $\Phi$ is well defined.

The map $\Phi$ is clearly $(\Gamma^3\times \Gamma^3)$-equivariant.
That $\Phi$ is actually almost everywhere $(\Gamma^1\times\Gamma^1)$-equivariant then follows from \cite[Lemma~5.8]{Kid2}, using the fact that the group $\Gamma$ is ICC.
\end{proof}

\subsection{Measure equivalence rigidity}

We recall from the introduction that two countable groups $\Gamma_1$ and $\Gamma_2$ are \emph{virtually isomorphic} if there exist finite-index subgroups $\Gamma_i^0\subseteq\Gamma_i$ and finite normal subgroups $F_i\unlhd\Gamma_i^0$ such that $\Gamma_1^0/F_1$ and $\Gamma_2^0/F_2$ are isomorphic. A countable group $\Gamma$ is \emph{ME-superrigid} if for every countable group $\Lambda$, if $\Lambda$ is measure equivalent to $\Gamma$, then $\Gamma$ and $\Lambda$ are virtually isomorphic. 

\begin{theo}[{Kida \cite[Theorem~6.1]{Kid2}}]\label{theo:me}
Let $\Gamma$ be a countable  ICC group which is rigid with respect to action-type cocycles.

Then $\Gamma$ is ME-superrigid.
\end{theo}

\begin{proof}
This is a consequence of Proposition~\ref{prop:oe-me} and \cite[Theorem~6.1]{Kid2} (see also \cite[Theorem~2.6]{BFS}).
\end{proof}

\subsection{Orbit equivalence rigidity}\label{sec:oe-rigidity}

Let $\Gamma_1$ and $\Gamma_2$ be two countable groups. Two free measure-preserving actions $\Gamma_1\actson X_1$ and $\Gamma_2\actson X_2$ on standard probability spaces are \emph{stably orbit equivalent} if there exist Borel subsets $Y_i\subseteq X_i$ of positive measure and a measure-scaling isomorphism $f:Y_1\to Y_2$ such that for a.e.\ $y\in Y_1$, one has $$f((\Gamma_1\cdot y)\cap Y_1)=(\Gamma_2\cdot f(y))\cap Y_2.$$ 

A particular case of stably orbit equivalent actions is when the two actions are conjugate, or virtually conjugate, in the following sense.

They are \emph{conjugate} if there exists an isomorphism $\phi:\Gamma_1\ra \Gamma_2$ and 
and a measure preserving isomorphism $f:X_1\ra X_2$ such that $f(\gamma x)=\phi(\gamma)f(x)$ for all $\gamma\in \Gamma_1$ and almost
every $x\in X_1$.

More generally, they are \emph{virtually conjugate} if there exist short exact sequences $1\to F_i\to\Gamma_i\to\overline{\Gamma}_i\to 1$ with $F_i$ finite for every $i\in\{1,2\}$,  finite-index subgroups $\overline{\Gamma}_i^0\subseteq\overline{\Gamma}_i$, and conjugate actions $\overline{\Gamma}_1^0\actson X'_1$ and $\overline{\Gamma}_2^0\actson X'_2$ such that for every $i\in\{1,2\}$, the action $\overline{\Gamma}_i\actson X_i/F_i$ is induced from the action $\overline{\Gamma}_i^0\actson X'_i$ as in \cite[Definition~2.1]{Kid3}. Notice that this implies in particular that the groups $\Gamma_1$ and $\Gamma_2$ are virtually isomorphic.

\begin{rk}\label{rk_aperiodic}
In the particular case where the groups $\Gamma_i$ have no non-trivial finite normal subgroup 
and the actions $\Gamma_i\actson X_i$ are aperiodic,
i.e.\ every finite index subgroup of $\Gamma_i$ acts ergodically on $X_i$,
the actions are conjugate if and only if they are virtually conjugate.
Indeed, an action induced from a proper subgroup of finite index is never aperiodic. 
\end{rk}

Note for instance that a Bernoulli action of any infinite group is aperiodic.

An ergodic measure-preserving free action $\alpha$ of a countable group $\Gamma$ on a standard probability space is \emph{OE-superrigid} if for every ergodic measure-preserving free action $\beta$ of a  countable group $\Lambda$ on a standard probability space, if $\alpha$ and $\beta$ are stably orbit equivalent, then $\alpha$ and $\beta$ are virtually conjugate.

\begin{theo}[{Furman \cite[Theorem~A]{Fur2}, \cite[Theorem~1.1]{Kid3}}]\label{theo:oe}
Let $\Gamma$ be a countable ICC group which is rigid with respect to action-type cocycles. 

Then for every finite-index subgroup $\Gamma^1\subseteq\Gamma$, every ergodic measure-preserving free action of $\Gamma^1$ on a standard probability space is OE-superrigid.
\end{theo}

\begin{proof}
The proof of \cite[Theorem~1.1]{Kid3} derives Theorem~\ref{theo:oe} from Theorem~\ref{theo:me}. The idea is to use again the dictionary between measure equivalence and stable orbit equivalence developed by Furman in \cite{Fur2}. See also \cite[Lemma~4.18]{Fur-survey}. 
\end{proof}

See \cite[Theorem~1.2]{Kid3} for variants of the theorem.

\subsection{Lattice embeddings}

Let $H$ be a locally compact second countable group with a bi-invariant Haar measure $\lambda_H$. Given a discrete group $\Gamma$, a  homomorphism $\phi:\Gamma\to H$ is a \emph{lattice embbeding} if $\phi$ is one-to-one and $\phi(\Gamma)$ is a lattice in $H$, i.e.\ $\phi(\Gamma)$ is discrete in $H$ and has finite $\lambda_H$-covolume.

\begin{theo}[{Furman \cite{Fur3}, Kida \cite[Theorem~8.1]{Kid2}}]\label{theo:lattice}
Let $\Gamma$ be a countable ICC group which is rigid with respect to action-type cocycles.

Then for every finite-index subgroup $\Gamma^1 \subseteq\Gamma$, every locally compact second countable group $H$, and every lattice embedding $\sigma:\Gamma^1\to H$, there exists a continuous homomorphism $\Phi_0:H\to\Gamma$ with compact kernel such that for all $\gamma\in\Gamma^1$, one has $\Phi_0(\sigma(\gamma))=\gamma$.
\end{theo}

\begin{proof}
We give the argument for completeness, following \cite[Theorem~8.1]{Kid2}.
 The group $\Gamma^1\times\Gamma^1$ acts on $H$ via $(\gamma,\gamma')\cdot h:=\sigma(\gamma)h\sigma(\gamma')^{-1}$. As $\sigma$ is a lattice embedding, this turns $H$ (equipped with the bi-invariant Haar measure $\lambda_H$) into a self coupling of $\Gamma^1$. Since $\Gamma$ is ICC and rigid with respect to action-type cocycles, Proposition~\ref{prop:oe-me} yields the existence of an almost $(\Gamma^1\times\Gamma^1)$-equivariant Borel map $\Phi:H\to\Gamma$.

One first proves that for a.e.\ $(h_1,h_2)\in H$, one has $\Phi(h_1h_2)=\Phi(h_1)\Phi(h_2)$. To see this,
consider the map
$$\begin{array}{cccl}
  F:&H\times H&\ra &\Gamma\\
  &(h_1,h_2)&\mapsto &\Phi(h_1^{-1})^{-1}\Phi(h_1^{-1}h_2)\Phi(h_2)^{-1}.
\end{array}$$
An easy computation shows that $F(h_1\gamma,h_2\gamma')=F(h_1,h_2)$ 
and that $F(\gamma h_1,\gamma h_2)=\gamma F(h_1,h_2)\gamma\m$ for almost every $h_1,h_2\in H$ and all $\gamma,\gamma'\in \Gamma^1$.
Denote by $X=H/\Gamma^1$ the right quotient, endowed with the left $H$-invariant probability measure $\mu$.
Then $F$ 
induces a map $\overline{F}:X\times X\to\Gamma$.
The image of $\mu\otimes\mu$ under $\ol F$ is a probability measure on $\Gamma$ which is invariant under conjugation by $\Gamma^1$.
Since $\Gamma$ is ICC and $\Gamma^1$ has finite index in $\Gamma$, this measure must be the Dirac mass at the identity. This precisely means that for a.e.\ $(h_1,h_2)\in H$, one has $\Phi(h_1h_2)=\Phi(h_1)\Phi(h_2)$.

By \cite[Theorems~B.2 and~B.3]{Zim}, $\Phi$ agrees almost everywhere with a continuous
homomorphism $\Phi_0:H\to\Gamma$.
Since $\Phi\m(\{1_\Gamma\})$ is a Borel fundamental domain for the left action of $\Gamma^1$ on $H$, $\Phi_0\m(\{1_\Gamma\})$ is a closed subgroup of finite positive Haar measure, hence compact. Finally, using the equivariance of $\Phi$ and the fact that $\Phi$ and $\Phi_0$ agree almost everywhere, we see that for all $\gamma\in\Gamma^1$ and almost every $h\in H$, one has $$\Phi_0(\sigma(\gamma))\Phi_0(h)=\Phi_0(\sigma(\gamma)h)=\Phi(\sigma(\gamma)h)=\gamma\Phi(h)=\gamma\Phi_0(h),$$ and therefore $\Phi_0(\sigma(\gamma))=\gamma$.
\end{proof}

Given a finitely generated group $\Gamma$, and a finite generating set $S$ of $\Gamma$, define the \emph{Cayley graph} $\Cay(\Gamma,S)$ to be the simple graph with vertices $\Gamma$, and such that there is an edge between two distinct elements $g$ and $h$ if and only if $gh\m \in S\cup S\m$. The following consequence of Theorem~\ref{theo:lattice} was suggested to us by Jingyin Huang.

\begin{cor}\label{cor:cayley}
Let $\Gamma$ be a finitely generated ICC group which is rigid with respect to action-type cocycles.
\begin{enumerate}
\item For every finite generating set $S$ of $\Gamma$, every automorphism of $\mathrm{Cay}(\Gamma,S)$ is at bounded distance from an automorphism induced by the left multiplication by an element of $\Gamma$.
\item For every torsion-free finite-index subgroup $\Gamma^1\subset \Gamma$, and every finite generating set $S$ of $\Gamma^1$, the automorphism group of $\Cay(\Gamma^1,S)$ is isomorphic to a subgroup of $\Gamma$ containing $\Gamma^1$. In particular, it is countable.
\end{enumerate}
\end{cor}

Notice that the first conclusion does not hold as such if $\Gamma$ is replaced by an arbitrary finite-index subgroup. Indeed, if $\Gamma^1$ is a normal finite index subgroup of $\Gamma$ and $S$ is a finite generating set of $\Gamma^1$ which is
invariant under conjugation by a finite order element $\gamma\in\Gamma\setminus \Gamma^1$,
then the conjugation by $\gamma$ is an automorphism of $\Cay(\Gamma^1,S)$, and using the ICC condition, one can check that it is not
at bounded distance from the left translation by any element of $\Gamma^1$.

Also, the second conclusion does not hold if $\Gamma^1$ has torsion by \cite[Lemma~6.1]{dlST}.

\begin{proof}
Let $\Gamma^1$ be any finite index subgroup of $\Gamma$ (we will take $\Gamma^1=\Gamma$ or assume $\Gamma^1$ torsion-free below).
Denote by $G$ the group of all graph isomorphisms of $\Cay(\Gamma^1,S)$, equipped with the compact-open topology. As $\Cay(\Gamma^1,S)$ is a locally finite graph, this is a locally compact second countable group and point stabilizers are compact subgroups of $G$. 
As $\Gamma^1$ acts transitively on the vertex set of $\Cay(\Gamma^1,S)$, it follows that $\Gamma^1$ is a (cocompact) lattice inside $G$. Theorem~\ref{theo:lattice} thus yields the existence of a group homomorphism $\Phi_0:G\to\Gamma$ with compact kernel $K$, such that for every $\gamma\in\Gamma^1$, denoting by $L_\gamma\in G$ the automorphism of $\Cay(\Gamma^1,S)$ induced by the left multiplication by $\gamma$, one has $\Phi_0(L_\gamma)=\gamma$. 

As $K$ is compact, 
the $K$-orbit of $1_{\Gamma^1}$ is finite, and since $K$ is normalized by $\Gamma^1$, all $K$-orbits have the same diameter
so $K$ is at bounded distance of the identity.

If $\Gamma^1=\Gamma$, every isometry of $\Cay(\Gamma,S)$ can be written as a product of an element of $K$ and of the left multiplication by an element of $\Gamma$, so the first conclusion follows.

We now assume that $\Gamma^1$ is torsion-free and derive the second conclusion. Since $K$ is normal, $\Gamma^1$ permutes the $K$-orbits
in $\Cay(\Gamma^1,S)$: $\gamma.K(\gamma')=K(\gamma\gamma')$. It follows that $K(1_{\Gamma^1})$ is a finite subgroup of $\Gamma^1$,
so $K(1_{\Gamma^1})=\{1_{\Gamma^1}\}$. It follows that $K(\gamma)=\{\gamma\}$ for all $\gamma\in\Cay(\Gamma^1,S)$, hence $K=\{\id\}$
and $\Phi_0$ gives the desired isomorphism.
\end{proof}

\newpage

\part{Constructions of canonical splittings}\label{part2}

The goal of the second part of our work is to provide three constructions of canonical splittings associated to a subgroup of $\Out(F_N)$  (with $N\ge 2$). The first construction relies on the existence of invariant non-trivial free splittings (see Section~\ref{sec:collection-free}).
The second relies on the existence of invariant $\Zmax$-splittings assuming that there is no invariant non-trivial free splitting (see Section~\ref{sec:collection-zmax}).
The third one produces a splitting from the existence of several maximal invariant almost free factor systems, still
assuming that there is no invariant non-trivial free splitting
(see Section~\ref{sec:collection-factors}).
A convenient unifying statement is given in Section~\ref{sec_nice}.

Given an action of a group $\Gamma$ on a set $X$,
and a collection $\calc\subset X$,
we denote respectively by $\Gamma_\calc$ and $\Gamma_{\{\calc\}}$ the elementwise and setwise stabilizer of $\calc$, namely
$$\Gamma_\calc=\{\gamma\in \Gamma\mid\forall x\in \calc,\ \gamma x=x\}$$
and $$\Gamma_{\{\calc\}}=\{\gamma\in \Gamma\mid \gamma\calc=\calc\}.$$
Below, we take $\Gamma$ as a shorthand for $\IA$,
and when $\IA$ acts on a set $X$ (such as a set of splittings or the set of free factor systems) and given a collection $\calc\subset X$, we still denote by $\Gamma_{\calc}$
and $\Gamma_{\{\calc\}}$ the elementwise and pointwise stabilizer of $\calc$
for the action of $\IA$.

\section{$H$-invariant splittings and trees of cylinders}\label{sec:additional-background}

This section records preliminary material that will be used in the next three sections.

\subsection{$H$-invariant splittings}\label{sec:tildeH}
Let $H\subseteq\Out(F_N)$ be a subgroup
and let $\tilde{H}$ be the full preimage of $H$ in $\Aut(F_N)$. We recall from Section~\ref{sec_splittings_def} that $\Out(F_N)$ acts on the set of all splittings of $F_N$ (considered up to equivariant isometry).
In this subsection,
we recall the relation between $H$-invariant splittings of $F_N$ and actions of $\tilde H$ on trees.

Given a splitting $S$ of $F_N$, recall that $S$ being $H$-invariant means that for every
 $\tilde\alpha\in\tilde{H}$,
there exists an isomorphism $I_{\tilde\alpha}:S\ra S$ which is $\tilde\alpha$-equivariant in the following sense: 
for every $g\in F_N$ and every $x\in S$, one has $$I_{\tilde\alpha}(gx)=\tilde\alpha(g)I_{\tilde\alpha}(x).$$
Endowing $S$ with its combinatorial metric, we view automorphisms of $S$ as isometries.
As soon as $S$ is not a line (which will always be the case in our setting), $I_{\tilde\alpha}$ is unique,
and $\tilde\alpha\in \tilde H\mapsto I_{\tilde\alpha} \in\Aut(S)$ defines an action of $\tilde H$ on $S$
for which the action of an inner automorphism $\ad_g$ is the same as the original action of $g$.
We will thus consistently identify $F_N$ with $\Inn(F_N)$, thus viewing $F_N$ as a normal subgroup $F_N\normal \tilde H$.
This way, the restriction to $F_N$ of the action of $\tilde H$ on $S$ is the same as the original $F_N$-action.
Conversely, given an action of $\tilde H$ on a tree $S$, restricting the action to $F_N$ yields an $H$-invariant $F_N$-tree.

We will denote by $\tilde H_e$, $\tilde H_v$ the $\tilde H$-stabilizers of an edge $e$ or a vertex $v$ of $S$,
and by $G_e$, $G_v$ their $F_N$-stabilizers.
Since $G_e=\Tilde H_e\cap F_N$,
$S$ is a free splitting (resp.\ a cyclic splitting) of $F_N$ if and only if for each edge $e\subseteq S$, $\tilde H_e\cap F_N$ is trivial (resp.\ cyclic).

\subsection{Trees of cylinders}\label{sec:cyl}
Trees of cylinders were introduced in \cite{GL-cyl} as a way to produce a canonical splitting of a group $G$ out of a deformation space. 
The construction is based on a notion of cylinders in a tree, associated to an equivalence relation on edge stabilizers.
This notion will be used in the next three sections, with $G$ being either the free group $F_N$ or a subgroup $\tilde H\subset \Aut(F_N)$ as above.

Let $G$ be a group, and $\cale$ be a collection of subgroups of $G$, stable under conjugation (this is the collection of allowed edge groups). Following \cite[Definition~3.1]{GL-cyl}, we say that an equivalence relation $\sim$ on $\cale$ is \emph{admissible}
if for every $A,B\in\cale$,
and every $G$-tree $T$ with edge groups in $\cale$,
the following three conditions hold: 
\begin{enumerate}
\item for every $g\in G$, if $A\sim B$, then $A^g\sim B^g$, and
\item if $A\subseteq B$, then $A\sim B$, and
\item for every pair of points $a,b\in T$ fixed by $A,B$ respectively, if $A\sim B$, then for every edge $e\subseteq [a,b]$, one has $A\sim G_e\sim B$.
\end{enumerate}

More generally, given a collection $\calp$ of subgroups of $G$, one says that an equivalence relation $\sim$ on $\cale$ is \emph{admissible relative to $\calp$} if the third condition is only required to hold for $G$-trees $T$ relative to $\calp$ with edge groups in $\cale$. 

One then defines the \emph{cylinder} of an edge $e$ in $T$ as the union of all edges $e'$ with $G_e\sim G_{e'}$.
Admissibility of $\sim$ implies that this is a connected subset of $T$.

A simple example of admissible relation is the commutation relation (equivalently, commensurability) on the class $\cale$ of non-trivial cyclic subgroups of $F_N$.
This will be used in Section~\ref{sec:collection-zmax}. 

One then defines the \emph{tree of cylinders} $T_c$ of $T$: 
this is the bipartite simplicial tree having one vertex $v_Y$ for each cylinder $Y\subseteq S$,
one vertex $v_x$ for each point $x\in S$ that does not belong to exactly one\footnote{%
except when $S$ is reduced to a point, this means that $x$ belongs to at least two cylinders.} cylinder, and an edge between $v_x$ and $v_Y$ whenever $x\in Y$.
We denote by $V^1(T_c)$ be the set of vertices corresponding to cylinders, and let $V^0(T_c)$ be the set vertices corresponding to points in at least two cylinders of $S$.
In the special case where $T$ is reduced to a point, $T_c=V^0(T_c)$ is a single vertex. 
As a mnemonic, $V^0$ and $V^1$ correspond to subsets of $S$ of dimension 0 and 1 respectively.
The tree $T_c$ is minimal \cite[Lemma~4.9]{Gui04}. 
It may happen that $T_c$ is reduced to a point even if $T$ is non-trivial.
In general, it may also happen that edge stabilizers of $T_c$ are not in $\cale$.

Recall that two $G$-trees \emph{belong to the same deformation space} 
if they have the same elliptic subgroups.
By \cite[Theorem 1]{GL-cyl}, any two trees with edge stabilizers in $\cale$ that lie in the same deformation space
have the same tree of cylinders.

\section{Stabilizers of collections of free splittings and associated canonical splittings}\label{sec:collection-free}

 Throughout the section, we fix an integer $N\ge 2$. Given a subgroup $H\subseteq\IA$, let $\FS^H$ be the collection of all $H$-invariant non-trivial free splittings of $F_N$. The goal of this section is to construct a canonical splitting $\Uun_H$ of $F_N$  (not a free splitting in general) encoding all the $H$-invariant free splittings of $F_N$. The splitting $\Uun_H$ will be non-trivial as soon as $H$ is infinite and $\FS^H\neq\es$.
To make things a bit more concrete, let us start with an example.

\begin{figure}[ht!]
  \centering
  \includegraphics{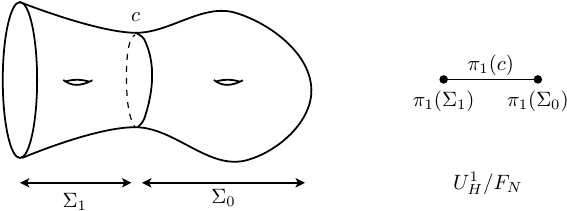}
  \caption{The canonical splitting $\Uun_H$ encoding all $H$-invariant free splittings.}
\end{figure}

\begin{ex}\label{ex:arcs}
Let $\Sigma$ be a connected surface with one boundary component, and let $c\subset \Sigma$ be 
an essential simple closed curve on $\Sigma$ which separates $\Sigma$ into two connected components $\Sigma_0$ and $\Sigma_1$, with $\Sigma_1$ containing the boundary curve. Choose an identification $F_N\simeq \pi_1(\Sigma)$.  
Let $H_1\subset \Out(F_N)$ be the subgroup induced by all
homeomorphisms of $\Sigma$ 
which restrict to
the identity on $\Sigma_1$, and let $H=H_1\cap \IA$.
Then any essential, properly embedded arc in $\Sigma_1$ defines a free splitting of $F_N$ which is $H$-invariant.
These arcs fill the subsurface $\Sigma_1$, and the canonical splitting $\Uun_H$ will be the splitting dual to the amalgamated product
$F_N=A_0*_{\grp{a}} A_1$ 
corresponding to the decomposition $\Sigma=\Sigma_0\cup \Sigma_1$
 (see the argument right after Remark~\ref{rk:cross-connected}).
Then any free splitting of $A_1$ relative to $\grp{a}$ induces a free splitting of $F_N$ which is $H$-invariant.
It will follow from our construction that conversely, any $H$-invariant free splitting is of this form
(but it might not correspond to an arc drawn on $\Sigma_0$). 
\end{ex}

Dually, given any collection $\calc$ of free splittings of $F_N$, recall that
the elementwise stabilizer $\Gamma_\calc$ of 
$\calc$ in $\IA$ is defined as the subgroup of $\IA$ made of all outer automorphisms that fix all the splittings in $\calc$.
One can then describe $\Gamma_\calc$ as a particular subgroup of the stabilizer of  $U^1_{\Gamma_\calc}$.
In the example above, starting with the collection $\calc$ of all free splittings dual to properly embedded arcs in $\Sigma_1$,
the elementwise stabilizer of $\Gamma_\calc$ in $\IA$  coincides with the subgroup of $\IA$ made of all outer automorphisms that preserve the decomposition $F_N=A_0\ast_{\langle a\rangle} A_1$ and act as the identity on $A_1$.
Additionally, one can describe the normalizer of $\Gamma_\calc$ as the stabilizer of  $U^1_{\Gamma_\calc}$.
See Section~\ref{sec:free-stabilizer}.

Our construction also allows us to give a bounded chain condition for stabilizers of collections of free splittings,  see Section~\ref{sec:free-chain}.

Our construction  of the splitting $U^1_H$ goes in two steps: starting from any group $H\subseteq \IA$ and the collection $\FS^H$ of non-trivial $H$-invariant free splittings,
we first construct  in Section~\ref{sec:free-maximum} a maximal $H$-invariant free splitting, which we also view as an $\tilde H$-tree where $\tilde H$ is the preimage
of $H$ in $\Aut(F_N)$.
Such an $\tilde H$-tree is not unique, but the collection
of all of them forms a deformation space  of $\tilde H$-trees.
 In Section~\ref{sec:free-cylinder}, we then constructs a tree of cylinders to get a canonical splitting from this deformation space.

\subsection{The deformation space of maximal $H$-invariant free splittings}\label{sec:free-maximum}

Let $H\subseteq\IA$ be a subgroup and $\tilde{H}$ be the full preimage of $H$ in $\Aut(F_N)$.
Identifying $F_N$ with $\Inn(F_N)$, we view $F_N$ as a normal subgroup of $\tilde H$.
Recall that $H$-invariant splittings of $F_N$ correspond to actions of $\tilde H$ that extend the action of $F_N$
(see Section \ref{sec:tildeH}).

Recall that given two $\tilde H$-trees $S,S'$, one says that  $S$ dominates $S'$ if the $\tilde H$-stabilizer of every vertex in $S$ is elliptic in $S'$. Two $\tilde{H}$-trees \emph{belong to the same deformation space of $\tilde{H}$-trees} if they dominate each other, i.e.\  if they have the same elliptic subgroups (in $\tilde{H}$).

If $S,S'$ are two $H$-invariant free splittings of $F_N$, and if $S$ dominates $S'$ as an $\tilde H$-tree,
then this also holds as $F_N$-tree. We will see in Remark~\ref{rk:domine} that the converse holds.

\begin{prop}\label{prop:deformation}
  Let $H\subseteq\IA$ be a subgroup.

There exists a (maybe trivial) maximum $H$-invariant free splitting, i.e.\
that dominates all $H$-invariant free splittings.

All maximum $H$-invariant free splittings are in the same
  deformation space of $\tilde H$-trees.
\end{prop}

\begin{figure}[ht]
  \centering
  \includegraphics{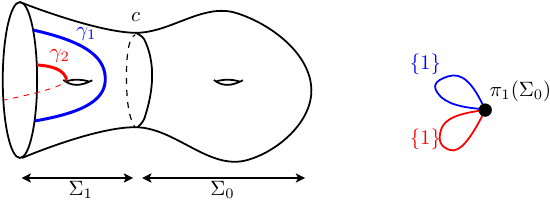}
  \caption{Example \ref{ex:arcs} continued. A maximum $H$-invariant free splitting, dual to the two arcs $\gamma_1,\gamma_2\subset \Sigma$.}
  \label{fig_ex2}
\end{figure}


\noindent\textit{Example \ref{ex:arcs} continued.} The splitting $S$ dual to the two arcs $\gamma_1,\gamma_2\subset\Sigma_1$
depicted on Figure \ref{fig_ex2}
is $H$-invariant. Since $\pi_1(\Sigma_0)$ is elliptic in any $H$-invariant free splitting,
$S$ dominates any $H$-invariant free splitting and therefore satisfies the conclusion of Proposition~\ref{prop:deformation}.

In the sequel, we will denote by $\cald_H$ the set of
maximum $H$-invariant (maybe trivial) free splittings viewed as 
$\tilde H$-trees. Proposition~\ref{prop:deformation} says that they are in the same deformation space of $\tilde H$-trees. 
Equivalently, if $\FS^H\neq\emptyset$, then $\cald_H$ is the set of maximum elements in $\FS_H$ viewed as $\tilde H$-trees; and if $\FS^H=\emptyset$, then $\cald_H$ is reduced to the trivial $\tilde{H}$-tree.

The main point in the proof of Proposition~\ref{prop:deformation} is the following lemma.

\begin{lemma}\label{lemma:edge-stabilizers}
Let  $H\subseteq\IA$ be a subgroup, and let $S,S'\in\FS^H$. Then the $\tilde{H}$-stabilizer of every edge of $S$ is elliptic in $S'$.
\end{lemma}

\begin{proof}
As $H\subseteq\IA$, every automorphism $\tilde\alpha\in\tilde{H}$ acts trivially on the quotient graph $S/F_N$ (Lemma~\ref{lemma:ia-free-splitting}), and therefore induces an isometry of every collapse of $S$. Therefore, without loss of generality, we can assume that $S$ has a single $F_N$-orbit of edges. Let $e\subseteq S$ be an edge, and let $\tilde\alpha\in\Stab_{\tilde{H}}(e)$. We denote by $I_{\tilde\alpha}$ the $\tilde\alpha$-equivariant isometry of $S$, and by $I'_{\tilde\alpha}$ the $\tilde\alpha$-equivariant isometry of $S'$. Our goal is to find a point in $S'$ that is fixed by  $I'_{\tilde\alpha}$, independently of $\tilde\alpha\in\Stab_{\tilde{H}}(e)$.

Let $A$ and $B$ be the $F_N$-stabilizers of the two extremities $v_A,v_B$ of $e$. Then $A$ and $B$ are proper free factors of $F_N$, and they are both non-trivial because $S$ has a single orbit of edges. Since $I_{\tilde\alpha}(e)=e$, we have in particular $I_{\tilde\alpha}(v_A)=v_A$. Therefore, for every $g\in A$, we have $\tilde\alpha(g)v_A=I_{\tilde\alpha}(gv_A)=I_{\tilde\alpha}(v_A)=v_A$, so $\tilde\alpha(g)\in A$. This shows that $\tilde\alpha(A)\subseteq A$, and since $\tilde\alpha$ is an automorphism, we deduce that $\tilde\alpha(A)=A$. Likewise $\tilde\alpha(B)=B$. 

If $A$ (or $B$) fixes a point in $S'$, then this point is unique (because $S'$ is a free splitting of $F_N$), and it is fixed by $I'_{\tilde\alpha}$. We can therefore assume that neither $A$ nor $B$ fixes a point in $S'$.

Let $S'_A$ (resp.\ $S'_B$) be the minimal $A$-invariant (resp.\ $B$-invariant) subtree of $S'$. The isometry $I'_{\tilde\alpha}$ sends the axis of every element $a\in A$ acting hyperbolically on $S'$ to the axis of $\tilde\alpha(a)$. As $S'_A$ is equal to the union of the axes of all elements of $A$ acting hyperbolically on $S'$, we deduce that $S'_A$ is $I'_{\tilde\alpha}$-invariant. Likewise $S'_B$ is $I'_{\tilde\alpha}$-invariant. 

As $A\cap B=\{1\}$, the intersection $S'_A\cap S'_B$ is compact (possibly empty). In addition this intersection is $I'_{\tilde\alpha}$-invariant. If $S'_A\cap S'_B\neq\emptyset$, then $I'_{\tilde\alpha}$ fixes the circumcenter of $S'_A\cap S'_B$. If $S'_A\cap S'_B=\emptyset$, then $I'_{\tilde\alpha}$ fixes the bridge between $S'_A$ and $S'_B$. In both cases, we have found a point fixed by all $I'_{\tilde\alpha}$, as required.     
\end{proof}

\begin{rk}\label{rk:domine} If $S,S'$ are $H$-invariant free splittings such that $S$ dominates $S'$ as an $F_N$-tree,
then $S$ dominates $S'$ as an $\tilde H$-tree. Indeed, if $v\in S$ is a vertex whose $F_N$-stabilizer $G_v$ is trivial,
then $\tilde H_v$ is commensurable to the stabilizer of an incident edge, so Lemma~\ref{lemma:edge-stabilizers} shows that $\tilde H_v$ fixes a point in $S'$.
If $G_v$ is non-trivial, then it fixes a unique point in $S'$ and this point is $\tilde H_v$ invariant since $G_v$ is normal in $\tilde H_v$.
\end{rk}

\begin{lemma}\label{lemma:maximal-refine}
  Let  $H\subseteq\IA$ be a subgroup, and let $S,S'\in\FS^H$. Assume that there does not exist any splitting in $\FS^H$ that properly refines $S$. Then
 $S$ dominates $S'$ as an $\tilde H$-tree.
\end{lemma}

\begin{proof}
Assume not. Then there exists a point $v\in S$ whose $\tilde{H}$-stabilizer $\tilde{H}_v$ is not elliptic in $S'$. Let $U$ be the minimal $\tilde{H}_v$-invariant subtree of $S'$. As $\tilde{H}$-edge stabilizers of $S$ are elliptic in $S'$ (Lemma~\ref{lemma:edge-stabilizers}), we can blow-up $S$ to an $\tilde{H}$-tree $\hat{S}$, using the $\tilde{H}_v$-action on $U$ (see e.g.\ \cite[Proposition~2.2]{GL-jsj}).

We claim that $\hat{S}\in\FS^H$, which will contradict our hypothesis. First, $\hat{S}$ is $H$-invariant (because it is an $\tilde{H}$-tree). Second, $F_N$-stabilizers of edges of $\hat{S}$ are trivial: indeed, every $\tilde{H}$-edge stabilizer of $\hat{S}$ stabilizes an edge in either $S$ or $S'$, and $\tilde{H}$-edge stabilizers of $S$ and $S'$ have trivial intersection with $F_N$ because $S$ and $S'$ are free splittings of $F_N$. This proves our claim, and concludes the proof of the lemma. 
\end{proof}

\begin{proof}[Proof of Proposition \ref{prop:deformation}]
  Consider an $H$-invariant free splitting $S$ of $F_N$ without redundant vertex, whose number of orbits of  edges is maximum
  (the trivial splitting if $\FS^H=\es$).
  By Lemma \ref{lemma:maximal-refine}, the tree $S$ is a maximum $H$-invariant free splitting.

  If $S,S'$ are two maximum $H$-invariant free splittings, then they dominate each other as $\tilde H$-trees by definition, so they are
  in the same deformation space of $\tilde H$-trees.
\end{proof}

\subsection{The tree of cylinders of maximal $H$-invariant free splittings}\label{sec:free-cylinder}

In this section we are going to construct a tree of cylinders of the deformation space $\cald_H$ following \cite{GL-cyl}  (see Section \ref{sec:cyl} for relevant definitions). 

Let $\cale$ be the collection of all $\tilde H$-stabilizers of edges of $H$-invariant free splittings of $F_N$.
To construct a tree of cylinders, we need an admissible equivalence relation on $\cale$.
We will take the trivial equivalence relation (i.e.\ equality),
and the following lemma implies that this is an admissible equivalence relation on $\cale$.

\begin{lemma}\label{lemma:stab-edges-containment}
 Let $H\subseteq\IA$ be a subgroup, and let $S,S'\in\FS^H$. Let $e\subseteq S$ and $e'\subseteq S'$ be edges. If the $\tilde{H}$-stabilizer $\tilde{H}_e$ of $e$ is contained in the $\tilde{H}$-stabilizer $\tilde{H}_{e'}$ of $e'$, then $\tilde{H}_e=\tilde{H}_{e'}$.
\end{lemma}

Lemma~\ref{lemma:stab-edges-containment} was proved in \cite[Lemma~6.9]{BGH}, we recall the proof for completeness.

\begin{proof}
Let $\tilde\alpha \in \tilde{H}_{e'}$; we aim to prove that $\tilde\alpha\in\tilde{H}_e$. As $H\subseteq \IA$ and $S$ is a free splitting of $F_N$, Lemma~\ref{lemma:ia-free-splitting} ensures that every automorphism in $\tilde{H}$ acts trivially on the quotient graph $S/F_N$, so there exists $g \in F_N$ such that $\tilde\alpha e=ge$. Then $\ad_g^{-1}\tilde\alpha\in\tilde{H}_e$, and hence $\ad_g^{-1}\tilde\alpha\in\tilde{H}_{e'}$ because $\tilde H_e\subseteq \tilde H_{e'}$. As $\ad_g^{-1}\tilde\alpha$ and $\tilde\alpha$ both fix $e'$, we deduce that $g$ fixes $e'$, and therefore $g=1$. Hence $\tilde\alpha\in\tilde{H}_e$, which concludes our proof. 
\end{proof}

For future reference, we record the following consequence of Lemma~\ref{lemma:stab-edges-containment}.


\begin{cor}\label{cor_Tc}
Equality is an admissible relation
 on $\cale$. 
 
 In particular, for every $\tilde H$-tree $S\in\cald_H$, the tree of cylinders of $S$ is well defined, and it does not depend on the choice of $S$.
\end{cor}

\begin{proof}
The 3 conditions for an admissible relation recalled in Section \ref{sec:cyl}
immediately follow from Lemma \ref{lemma:stab-edges-containment}.
This implies that the tree of cylinders of $S$ is well defined.
Since any two elements of $\cald_H$ are in the same deformation space (Proposition~\ref{prop:deformation}),
they have the same tree of cylinders by \cite[Theorem 1]{GL-cyl}.
\end{proof}

Recall that the cylinder of an edge $e\subseteq S$ is the union of all edges whose $\tilde H$-stabilizers are equal to $\tilde H_e$
(admissibility implies that this is a connected subset).

\begin{lemma}\label{lemma:trivial-on-cylinders}
Let $H\subseteq\IA$ be a subgroup, let $S\in\cald_H$, and let $Y\subseteq S$ be a cylinder. The subgroup $\tilde H_Y$ of elements of $\tilde H$ fixing $Y$ pointwise
is a lift of $H$ in $\Aut(F_N)$: the restriction of $\Aut(F_N)\onto \Out(F_N)$ to $\tilde H_Y$ is an isomorphism
$\tilde H_Y \xrightarrow{\sim } H$.

In particular, identifying $F_N$ with $\Inn(F_N)$, we have $\tilde H=\tilde H_Y\ltimes F_N\simeq H\ltimes F_N$.
\end{lemma}

\begin{rk}\label{rk_fix}
  Note that $\tilde H_Y$ coincides with $\tilde H_e$
for any edge $e\subseteq Y$. Indeed, if $\tilde\alpha\in \tilde H_e$, then for any edge $e'\subset Y$, $\tilde H_{e'}=\tilde H_e\ni\tilde\alpha$. 
\end{rk}

\begin{proof}
Since $F_N$-stabilizers of edges of $S$ are trivial, $\tilde H_Y\cap F_N=\{1\}$ which proves injectivity.

To prove surjectivity, consider $\alpha\in H$.
Let $e\subseteq Y$ be an edge. Since $H\subseteq\IA$, it acts trivially on the quotient graph $S/F_N$, so $\alpha$ has a representative $\tilde{\alpha}\in\tilde{H}$ that fixes the edge $e$.
As noted above, we conclude that $\tilde\alpha\in\tilde H_e=\tilde H_Y$.
\end{proof}

Given an automorphism $\tilde\alpha$ of $F_N$, we denote by $\Fix_{F_N}(\tilde\alpha)$ the subgroup of $F_N$ made of all elements $g\in F_N$ such that $\tilde\alpha(g)=g$. More generally, if $\calx$ is any collection of automorphisms of $F_N$, we let $$\Fix_{F_N}(\calx):=\bigcap_{\tilde\alpha\in\calx}\Fix_{F_N}(\tilde\alpha).$$

\begin{lemma}\label{lemma:stab-cylinder}
Let $H\subseteq\IA$ be a subgroup, let $S\in\cald_H$, and let $Y\subseteq S$ be a cylinder. 
Denote by $\tilde H_{\{Y\}}$ and $G_{\{Y\}}$ the setwise stabilizer of $Y$ in $\tilde H$ or in $F_N$ respectively,
and by $\tilde H_Y$ the pointwise stabilizer of $Y$ in $\tilde H$.

Then $G_{\{Y\}}=\Fix_{F_N}(\tilde{H}_Y)$ and $\tilde H_{\{Y\}}=\tilde H_Y \times G_{\{Y\}}$.
\end{lemma}

\begin{proof}
Fix an edge $e\subset Y$.
To prove that $G_{\{Y\}}\subset\Fix_{F_N}(\tilde{H}_Y)$,
consider $g\in G_{\{Y\}}$ and $\tilde\alpha\in\tilde{H}_Y$. 
As $e$ and $ge$ both belong to $Y$, we have $ge=I_{\tilde\alpha}(ge)=\tilde\alpha(g)I_{\tilde\alpha}(e)=\tilde\alpha(g)e$. As the $F_N$-stabilizer of $e$ is trivial, this implies that $\tilde\alpha(g)=g$. This shows that $G_{\{Y\}}\subseteq\Fix_{F_N}(\tilde{H}_Y)$.

Conversely, let $g\in F_N$ be such that for every  $\tilde\alpha\in\tilde{H}_Y$, one has $\tilde\alpha(g)=g$. 
To prove that $g\in G_{\{Y\}}$ it suffices to check that $ge\subseteq Y$.
For every $\tilde\alpha\in\tilde{H}_e$, one has $I_{\tilde\alpha}(ge)=\tilde\alpha(g)I_{\tilde\alpha}(e)=ge$, showing that $\tilde\alpha\in\tilde{H}_{ge}$. Therefore $\tilde{H}_{e}\subseteq\tilde{H}_{ge}$. By Lemma~\ref{lemma:stab-edges-containment}, this implies that $\tilde{H}_{e}=\tilde{H}_{ge}$. Thus $e$ and $ge$  belong to the same cylinder $Y$, 
and we conclude that $G_{\{Y\}}=\Fix_{F_N}(\tilde{H}_Y)$.

Now  the groups $\tilde H_Y$ and $G_{\{Y\}}=\tilde H_{\{Y\}}\cap F_N$ are normal in $\tilde H_{\{Y\}}$,
and they intersect trivially because $F_N$-stabilizers of edges of $S$ are trivial.
To conclude, it suffices to check that $\tilde H_{\{Y\}}=\tilde H_Y . G_{\{Y\}}$.
Given $\tilde\alpha\in \tilde H_{\{Y\}}$, there exists $g\in F_N$ such that $I_{\tilde\alpha}(e)=ge$ (as follows from Lemma \ref{lemma:ia-free-splitting} since $H\subseteq\IA$). 
Therefore $g\in G_{\{Y\}}$.
Since  $\ad_g\m\tilde\alpha$ fixes $e$, we have $\ad_g\m\tilde\alpha \in \tilde H_Y$ by Remark \ref{rk_fix}, which concludes the proof.
\end{proof}

Corollary \ref{cor_Tc} allows us to make the following definition (see Section \ref{sec:cyl}).
\begin{de}\label{dfn_UH1}
Given a subgroup $H\subseteq\IA$, we denote by $\Uun_H$ the tree of cylinders of any $\tilde H$-tree $S\in \cald_H$. 
\end{de}

Observe that the assignment $H\mapsto U_H^1$ is $\Out(F_N)$-equivariant: indeed, the set of $H$-invariant free splittings depends equivariantly on $H$, and therefore so do the deformation space $\cald_H$ and its tree of cylinders.

We denote by $V^1$ the set of vertices $v_Y$ of $\Uun_H$ corresponding to cylinders $Y$ of $S$,
and by $V^0$ the set of vertices $v_x$ of $\Uun_H$ corresponding to points $x\in S$ not lying in a unique cylinder. The tree $\Uun_H$ is minimal as an $\tilde{H}$-tree \cite[Lemma~4.9]{Gui04} (it may be reduced to a point). Since $F_N$ is normal in $\tilde H$, the tree $\Uun_H$ is also minimal as an $F_N$-tree.
In general, $F_N$-stabilizers of edges of $\Uun_H$ can be non-abelian. We will see below that it
is not trivial when $H$ is infinite and trees in $\cald_H$ are non-trivial (i.e.\ $\FS^H\neq \es$),
see Theorem~\ref{theo:tree-of-cyl}.

\begin{rk}\label{rk:cross-connected}
  Another more geometric interpretation of the tree of cylinders in terms of cross-connected components of $H$-invariant almost invariant sets is given in \cite{GL5}.
\end{rk}

\noindent\textit{Example \ref{ex:arcs} continued.} Recall that $H\subset \IA$ is the group of  outer automorphisms
induced by homeomorphisms of the surface $\Sigma$ that are the identity on $\Sigma_1$.
Let $S$ be the splitting dual to the arcs $\gamma_1,\gamma_2$ as in Figure \ref{fig_ex2}.
Let $\tilde \Sigma$ be the universal cover of $\Sigma$, and $\tilde \Sigma_1\subset\tilde \Sigma$
a connected component of the preimage of $\Sigma_1$.
Let $\tilde E$ be the collection of arcs of $\tilde \Sigma$ that are lifts of $\gamma_1$ or $\gamma_2$.
Thus, $\tilde E$ is in bijection with the set of edges of $S$.
Denote by $\tilde E_{\tilde\Sigma_1}$ the subcollection of $\tilde E$ consisting of all arcs contained in $\tilde \Sigma_1$,
and by $Y\subset S$ the corresponding collection of edges of $S$.
Any outer automorphism $h\in H$ has a lift acting as the identity on $\tilde \Sigma_1$, thus fixing
all the arcs in $\tilde E_{\tilde\Sigma_1}$.
The collection of these lifts defines a lift $\tilde H\subset \Aut(F_N)$ of $H$ which fixes
all the edges in $Y$, so $\tilde H=\tilde H_e$ for all $e\in Y$.
In particular, all edges in $Y$ are in the same cylinder.
Using the fact that $H$ contains a Dehn twist around $c$,
one can check that edges outside $Y$ are not fixed by all elements of $\tilde H$, so $Y$ is precisely
a cylinder of $S$, and the tree of cylinders $U^1_H$ of $S$ is the cyclic splitting dual to the curve $c\subset \Sigma$.

\begin{ex}
  The following three examples show that as an $F_N$-tree, $\Uun_H$ is not determined by the sole collection of $H$-invariant free splittings, but it really depends on $H$.

  Let $S$ be the free splitting  obtained as the Bass--Serre tree of the decomposition $F_N=A*B$, with $A,B$ non-abelian.
    Denote by $e\subseteq S$ the edge joining the vertices fixed by $A$ and $B$.
  Let $H\subset \IA$ be the stabilizer of $S$. The group $\Aut(A)\times\Aut(B)$ naturally embeds in $\Aut(F_N)$,
  and $\tilde H_e\simeq H$ has finite index in $\Aut(A)\times \Aut(B)$.
  Then $S$ is the only $H$-invariant  non-trivial free splitting, so $S\in \cald_{H}$,
  and we claim that every cylinder is reduced to a single edge.
  Indeed, for any $a\in A\setminus\{1\}$, the edges $e$ and $ae$ 
  cannot be in the same cylinder because $a\notin \Fix_{F_N}(\tilde H_e)$. Using the same argument in $B$, we see that
  the cylinder of $e$ contains no edge adjacent to $e$, which proves our claim.
  It follows that $\Uun_H$ is isomorphic to the barycentric subdivision of $S$, with $V^1$ corresponding to the midpoints of edges of $S$.

  From now on we further assume that $\mathrm{rk}(A)\ge 3$. Let $a\in A$ be an element which is part of a free basis of $A$, and let $H'\subset H$ be the stabilizer  (inside $H$) of the cyclic splitting $S'$ dual to the cyclic decomposition
  $F_N=A*_{\grp{a}} (\grp{a}*B)$.
  Now $\tilde H'_e\simeq H'$ has finite index in $\Aut(A,a)\times \Aut(B)$.
  Again, $S$ is the only $H'$-invariant free splitting. 
  But now, the cylinder of $e$ is the orbit of $e$ under $\grp{a}$ because $\Fix_{F_N}(\tilde H'_e)=\grp{a}$.
  As an $F_N$-tree, the tree $\Uun_{H'}$ is then dual to the two-edge splitting
  $$\xymatrix{ A \ar@{-}[r]_{\grp{a}} &\grp{a} \ar@{-}[r]_{\{1\}} & B}$$
  and as an $\tilde H'$-tree, it is isomorphic to
  $$\xymatrix@=2cm{(\tilde H'_e\semidirect A)\ar@{-}[r]_{\tilde H'_e\times \grp{a}}
    &\tilde H'_e\times \grp{a} \ar@{-}[r]_{\tilde H'_e} & (\tilde H'_e\semidirect B)}.$$
  The product structure on the central vertex illustrates Lemma \ref{lemma:stab-cylinder} and the second assertion of Theorem \ref{theo:tree-of-cyl} below.

  Finally, assume furthermore that $\mathrm{rk}(B)\geq 3$, and let $b\in B$ be an element in a free basis, and let $H''\subset H$ be the stabilizer of the cyclic splitting $T$ given by
  $$\xymatrix{ A \ar@{-}[r]_{\grp{a}} &\grp{a,b} \ar@{-}[r]_{\langle b\rangle} & B}.$$
  Now $\tilde H''_e\simeq H''$ has finite index in $\Aut(A,a)\times \Aut(B,b)$.
  Again, $S$ is the unique $H''$-invariant non-trivial free splitting.
  The cylinder containing $e$ is now the orbit of $e$ under $\grp{a,b}$ (it is unbounded) and $\Uun_{H''}=T$ as an $F_N$-tree
  (it is no more in the same deformation space as $S$). 
  As an $\tilde H''$-tree,
it is isomorphic to
  $$\xymatrix@=2cm{(\tilde H'_e\semidirect A)\ar@{-}[r]_{\tilde H'_e\times \grp{a}}
    &\tilde H'_e\times \grp{a,b} \ar@{-}[r]_{\tilde H'_e\times \grp{b}} & (\tilde H'_e\semidirect B}).$$
\end{ex}

\begin{theo}\label{theo:tree-of-cyl}\renewcommand{\theenumi}{(\arabic{enumi})}\renewcommand{\labelenumi}{\theenumi}
  Consider a subgroup $H\subseteq \IA$. Recall that the vertex set of the tree $\Uun_H$ is bipartitioned into $V^0\dunion V^1$. Then the following properties hold.
  \begin{enumerate}
  \item the $F_N$-stabilizers of vertices in $V^0$ are the maximal subgroups of $F_N$ that are elliptic in all $H$-invariant free splittings; in particular, they are free factors of $F_N$.
  \item\label{it_V1} if $v\in V^1$, and if $Y\subseteq S$ is the corresponding cylinder, then the $\tilde{H}$-stabilizer of $v$ splits as 
$\tilde H_v=G_v\times \tilde H_Y$
where $$G_{v}=G_{\{Y\}}=\Fix_{F_N}(\tilde H_Y)$$ 
and $\tilde H_Y$ maps isomorphically onto $H$
under $\Aut(F_N)\onto\Out(F_N)$;
\item \label{it_carac}
Consider the free splittings of $F_N$ obtained
from $\Uun_H$ by blowing up some vertices in $V^1$ (using for each $v\in V^1$ a maybe non-minimal free splitting of $G_v$ relative to $\Inc_v$) and collapsing all edges coming from $\Uun_H$.
Then any such free splitting is $H$-invariant, and conversely
every $H$-invariant free splitting may be obtained in this way.
\item \label{it_factor_system}
 for every vertex $v\in V^1$ such that every edge incident on $v$ has non-trivial $F_N$-stabilizer,
 the collection of  all conjugacy classes of $F_N$-stabilizers of incident edges forms a proper free factor system of $G_v$.
  \end{enumerate}
  Moreover, $\Uun_H$ is non-trivial if $H$ is infinite and $\FS^H\neq \es$. 
\end{theo}

\begin{rk}
  In Assertion \ref{it_carac}, one allows to blow up $v\in V^1$ using a non-minimal splitting of $G_v$ to get all $H$-invariant free splittings; on the other hand, this is only necessary when
 some edge incident on $v$ has trivial stabilizer.
\end{rk}

\begin{proof}
  Let $S$ be any $\tilde H$-tree in $\cald_H$. Then $\Uun_H$ is the tree of cylinders of $S$.
  For each vertex $v_x\in V^0$ corresponding to a vertex  $x\in S$, one has $\tilde H_{v_x}=\tilde H_x$ and $G_{v_x}=G_x$.
  Since $x$ lies in two distinct cylinders, $G_{x}\neq \{1\}$. By definition of  $\cald_H$, the subgroup $G_x$ is elliptic in all $H$-invariant splittings,
  and is maximal for this property because $F_N\actson S$ is an $H$-invariant splitting of $F_N$.
  This proves the first assertion.

  The second assertion follows from Lemmas \ref{lemma:trivial-on-cylinders} and \ref{lemma:stab-cylinder} noting that
  if $v\in V^1$ corresponds to a cylinder $Y\subset S$, then $\tilde H_v=\tilde H_{\{Y\}}$ is the setwise stabilizer of $Y$.

  To prove the third assertion, given a vertex $v\in V^1$, consider a free splitting $Z$ of $G_v$ relative to $\Inc_v$ (we do not assume that it is invariant under any group of automorphisms,
  and we allow it to  be non-minimal).
  Writing $\tilde H_v=G_v\times \tilde H_Y$, the action of $G_v$ on $Z$ naturally extends to an action of $\tilde H_v$ where $\tilde H_Y$ acts trivially.
  For each edge $e$ incident on $v$, there exists a point $p_e\in Z$ fixed by $\tilde H_e=\tilde H_Y\times G_e$ (not unique if $G_e$ is trivial).  Doing this in each orbit of vertices in $V^1$,
  this allows to blow up $\Uun_H$ into an $\tilde H$-tree
  $\hat{U}^1_H$, and collapsing all edges of $\hat{U}^1_H$ coming from $\Uun_H$ yields an $\tilde H$-tree $T$ (which is minimal if $Z$ is the convex hull
  of the orbits of the attaching points $p_e$). Edges of $T$ have trivial
  $F_N$-stabilizer, so $F_N\actson T$ is an $H$-invariant free splitting. 

Conversely, consider any $H$-invariant free splitting $T$ of $F_N$, and view it as an $\tilde H$-tree.
  By maximality, there is an $\tilde H$-equivariant map $S\ra T$. By \cite[Proposition~8.1]{GL-cyl}, $\Uun_H$ is compatible with $T$,
  and the proof actually shows that one can construct a common refinement $\hat U$ by blowing up vertices in $V^1$ (using the corresponding cylinder in $T$),
  and that one obtains $T$ from $\hat U$ by collapsing
  all edges coming from $\Uun_H$.
This concludes the proof of Assertion \ref{it_carac}.

 We now prove Assertion~\ref{it_factor_system}: assuming that all groups in $\Inc_v$ are non-trivial, we will prove that $\Inc_v$ is a free factor system of $G_v$.
   We use the fact that $G_v=G_{\{Y\}}$ is the setwise stabilizer of $Y$, and incident
  edge stabilizers are the intersections $G_{\{Y\}}\cap G_x$ where $x\in Y$ is a point lying in at least two cylinders.  
  Since $G_{\{Y\}}\actson Y$ has trivial edge stabilizers, it follows that either $\Inc_v=\{[G_{\{Y\}}]\}$ or $\Inc_v$ is a free factor system of $G_v=G_{\{Y\}}$.
  To prove Assertion~\ref{it_factor_system}, it thus suffices to prove that $G_{\{Y\}}$ does not fix a point in $Y$.
  So assume that $G_{\{Y\}}$ fixes a point $x_0\in Y$. Recall that each edge $\eps$ of $\Uun_H$ incident on $v$ is corresponds to a pair $(x,Y)$ with $x\in Y$
  lying in at least two cylinders. 
  Since by assumption, all groups in $\Inc_v$ are non-trivial, $G_\eps$ fixes no other point apart from $x_0$ so $x=x_0$. 
  It follows that there is a single edge in $\Uun_H$ incident on $v$, contradicting the minimality of $\Uun_H$.

  Finally, assume that $\Uun_H$ is reduced to a point $v$. If $v\in V^0$, then $G_v=F_N$ is elliptic in every $H$-invariant free splitting,
  so $\FS^H=\es$.
  If $v\in V^1$, then by Assertion (2), $\tilde H_v=G_v\times \tilde H_Y=F_N\times \tilde H_Y$ and $F_N=\Fix_{F_N}(\tilde H_Y)$  i.e.\ $\tilde H_Y=\{\id\}$.
  Since $\tilde H_Y\simeq H$, $H$ is trivial.
\end{proof}

\subsection{The canonical splitting $\Uun_H$ is biflexible}\label{sec:UH1_biflexible}

\begin{de}\label{de:biflexible}
Let $A$ be a group, and $\calp$ a finite collection of conjugacy classes of subgroups of $A$.
  We say that $(A,\calp)$ is \emph{flexible} if $\Out(A,\calp^{(t)})$ is infinite.

A splitting $S$ of $F_N$ is \emph{biflexible} if all its edge stabilizers are finitely generated non-abelian, and $S$ has
  two vertices $v_1,v_2$ in distinct $F_N$-orbits such that for every $i\in\{1,2\}$, the pair $(G_{v_i},\Inc_{v_i})$ is flexible.
\end{de}

\begin{cor}\label{cor:uh1-biflexible}
 Let $H\subseteq\IA$ be a subgroup. If $H$ is infinite and $\FS^H\neq\es$  
 then, as an $F_N$-tree, either $\Uun_H$ has an edge with trivial or maximal-cyclic stabilizer,
or $\Uun_H$ is biflexible.
\end{cor}

\begin{proof}
  We know that $\Uun_H$ is a non-trivial splitting. We view it as an $F_N$-tree.
  Note that edge stabilizers of $\Uun_H$ are root-closed because vertex stabilizers are by Assertions~(1) and~(2) of Theorem~\ref{theo:tree-of-cyl}.  
  Additionally, $F_N$-stabilizers of vertices in $V^0$ are finitely generated because they are free factors (Assertion~(1) of Theorem~\ref{theo:tree-of-cyl}) and so are $F_N$-stabilizers of vertices in $V^1$ because they arise as the fixed subgroups of a set of automorphisms of $F_N$ (Assertion~(2) of Theorem~\ref{theo:tree-of-cyl}), and these are finitely generated by \cite[Corollary 4.5.8]{DV}. It follows that edge stabilizers are finitely generated by the Howson property.
  
  Denote by $\bar V^0,\bar V^1$ the images of $V^0,V^1$ in $\Uun_H/F_N$.
  Assume that all $F_N$-stabilizers of edges of $\Uun_H$ are non-cyclic  (in particular non-trivial).
  Consider a vertex $v\in \bar V^1$.  By Theorem~\ref{theo:tree-of-cyl}\ref{it_factor_system}, the collection $\Inc_v$ forms a 
proper free factor system of $G_v$, so
  $\Out(G_v,\Inc_v^{(t)})$ is infinite;
  indeed the group of twists of any Grushko decomposition
  of $G_v$ relative to $\Inc_v$ is infinite (and non-abelian) because $\Inc_v$ contains a non-abelian free group
  (see Section \ref{sec_twists} or \cite[Proposition 3.1]{Lev}).
  Thus, we are done if $\#\bar V^1\geq 2$.

  To conclude, we claim that there is at least one vertex $v\in \bar V^0$ such that $\Out(G_v,\Inc_v^{(t)})$ is infinite.
Let $H^0\subset H$ be the finite index subgroup of $H$ acting trivially on the quotient graph $\Uun_H/F_N$.
Since the group of twists of $\Uun_H$ is trivial,
the natural homomorphism
$$\rho:H^0\to\prod_{v\in \bar V}\Out(G_v,\Inc_v)$$
is injective (Proposition~\ref{prop_suite_exacte}).
By Theorem~\ref{theo:tree-of-cyl}\ref{it_V1}, $H$ acts trivially on each vertex group $G_v$ for $v\in\bar V^1$,
so $\rho$ takes its values in $\prod_{v\in \bar V^0}\Out(G_v,\Inc_v)$.
This also implies that $H$ acts trivially on each edge group,
and that $\rho$ takes values in $\prod_{v\in \bar V^0}\Out(G_v,\Inc_v^{(t)})$ (see Remark \ref{rk_normalizer}; each edge stabilizer is its own normalizer
as an intersection of a fixed subgroup and a free factor).  Since $H$ is infinite, there is at least one vertex $v\in \bar V^0$ such that $\Out(G_v,\Inc_v^{(t)})$ is infinite.
\end{proof}

\subsection{A bounded chain condition}\label{sec:free-chain}

\begin{prop}\label{prop:chain-FS}
 There is a bound, depending only on the rank $N$ of the free group,
on the length of a chain of subgroups  $H^1\subset\dots \subset H^r$ of $\IA$
  such that $$\FS^{H^1}\supsetneqq \dots \supsetneqq \FS^{H^r}.$$
\end{prop}

\begin{proof}
  Choose $S_i\in\cald_{H^i}$ a maximal $H^i$-invariant free splitting of $F_N$.
  Since for $i\leq j$,  $S_j$ is $H^{i}$-invariant, $S_i$ dominates $S_j$.
  Since there is a bound on domination chains of deformation spaces of free splittings,
  we may assume that the trees $S_1,\dots,S_r$ all belong to the same deformation space of $F_N$-trees.

  Denote $S=S_r$.
  For all $i\leq r$, $S$ is a maximal $H^i$-invariant free splitting.
  Viewing $S$ as an $\tilde H^i$-tree,  let $\Uun_{H^i}$ be the tree of cylinders of $S$. 
  We fix an edge $e\subseteq S$,  and for every $i\le r$, we let $Y_i\subseteq S$ be the cylinder of $e$ when $S$ is viewed as an $\tilde H^i$-tree. 
  In other words $Y_i$ is the union of all edges $e'\subset S$ such that $\tilde H^i_{e'}=\tilde H^i_{e}$.
  Since for $i\leq j$, $\tilde H^i_{e'}=\tilde H^j_{e'}\cap \tilde H^i$, one has $Y_1\supseteq Y_2\supseteq \dots \supseteq Y_r$.
  
  We claim that $Y_i$ takes boundedly many values.
  Indeed, by Lemma~\ref{lemma:stab-cylinder}, we have $G_{\{Y_i\}}=\Fix_{F_N}(\tilde H^i_e)$, so by the bounded chain condition of auto-fixed groups \cite{MV}, the group
  $G_{\{Y_i\}}$ takes boundedly
  many values. Since $Y_i\cap (F_N\cdot e)=G_{\{Y_i\}}\cdot e$, and  since the number of edges of $S/F_N$ is finite (in fact bounded, with a bound only depending on $N$), it follows that $Y_i$ takes boundedly many values (with a bound only depending on $N$). 

  Using again the bound on the number of  $F_N$-orbits of edges of $S$, 
  we conclude that  $\Uun_{H^i}$ takes boundedly many values, as an $F_N$-tree.
Since $\FS^{H^i}$ can be recovered from $\Uun_{H^i}$ by Theorem~\ref{theo:tree-of-cyl}\ref{it_carac}, $\FS^{H^i}$ takes boundedly many values.
\end{proof}

\subsection{The stabilizer of a collection of free splittings}\label{sec:free-stabilizer}

We conclude this section by applying  the previous results to study the elementwise stabilizer $\Gamma_\calc$ in $\IA$ of any collection $\calc$ of non-trivial free splittings.
We denote by $\hat \calc=\FS^{\Gamma_\calc}$ the collection of all $\Gamma_\calc$-invariant non-trivial free splittings. Notice that $\calc\subseteq \hat\calc$,
and $\Gamma_\calc=\Gamma_{\hat\calc}$.

Proposition \ref{prop:chain-FS} can be reformulated as follows.

\begin{prop}\label{prop:chain-FS2}
  There is a bounded chain condition on the set of elementwise stabilizers (in either $\IA$ or $\Out(F_N)$) of collections of free splittings.
\end{prop}

\begin{proof}
We first work in $\IA$. If $\Gamma_{\calc_1}\subsetneqq\dots \subsetneqq \Gamma_{\calc_r}$,
  then $\FS^{\Gamma_{\calc_1}}\supsetneqq\dots \supsetneqq \FS^{\Gamma_{\calc_r}}$,
  so Proposition \ref{prop:chain-FS} concludes.
  
 The chain condition for elementwise stabilizers in $\Out(F_N)$ follows from the chain condition for elementwise stabilizers in $\IA$, since the length of a chain of subgroups of $\Out(F_N)$ having the same given intersection with $\IA$ is bounded by the index $[\Out(F_N):\IA]$.   
\end{proof}

Now let $\Gamma_{\!\{\hat\calc\}}$ be the setwise stabilizer of $\hat\calc$ in $\IA$. 
Our next result allows to recover $\Gamma_\calc$ and $\Gamma_{\!\{\hat\calc\}}$ from $\Uun_{\Gamma_\calc}$. Before we state it, we start with a definition.

\begin{de}\label{dfn_V0-rigid}
Let $H\subseteq\IA$ be a subgroup. The \emph{$V^1$-rigid stabilizer of $\Uun_H$}
is the group of all outer automorphisms $\alpha\in\IA$ such that for every $v\in V^1(\Uun_H)$, there exists a lift $\tilde\alpha$ of $\alpha$ to $\Aut(F_N)$ 
and an $\tilde\alpha$-equivariant isometry $I_{\tilde\alpha}:\Uun_H\ra \Uun_H$ acting as the identity on the star of $v$. 
\end{de}

\begin{prop}\label{prop:stab-u1}
Let $\calc$ be any collection of non-trivial free splittings, and view $\Uun_{\Gamma_\calc}$ as an $F_N$-tree. 

The normalizer in $\IA$ of $\Gamma_\calc$ is $\Gamma_{\!\{\hat\calc\}}$. Moreover,
  \begin{itemize}
  \item $\Gamma_{\!\{\hat\calc\}}$ is the stabilizer of $\Uun_{\Gamma_\calc}$ in $\IA$;
  \item $\Gamma_\calc$ is the $V^1$-rigid stabilizer of $\Uun_{\Gamma_\calc}$.
  \end{itemize}
\end{prop}

\begin{proof}
Any subgroup of $\IA$ normalizing $\Gamma_{\calc}$ preserves $\hat\calc$.
Since $\Gamma_{\!\{\hat\calc\}}$ normalizes $\Gamma_{\hat\calc}=\Gamma_\calc$, 
this shows that
 $\Gamma_{\!\{\hat\calc\}}$ is precisely the normalizer of $\Gamma_\calc$ in $\IA$.

  
  The set $\cald_{\Gamma_\calc}$ consists of all (maybe trivial) maximum $\Gamma_\calc$-invariant free splittings,
    viewed as $\tilde \Gamma_\calc$-trees.
   Since $\Gamma_\calc$ is normalized by  $\Gamma_{\!\{\hat\calc\}}$, $\cald_{\Gamma_\calc}$ is invariant under $\Gamma_{\!\{\hat\calc\}}$. Since all trees in $\cald_{\Gamma_\calc}$ have the same tree of cylinders (by Corollary \ref{cor_Tc}), it follows that
  $\Uun_{\Gamma_\calc}$ is $\Gamma_{\!\{\hat\calc\}}$-invariant as a $\tilde \Gamma_\calc$-tree, hence also as an $F_N$-tree.

 Conversely, let $H\subset \IA$ be the stabilizer of $\Uun_{\Gamma_\calc}$.
  By Lemma \ref{lemma:ia-nice}, $H$ preserves  the $F_N$-orbit of every vertex in $V^1$.
  Since the splittings in $\hat\calc$ are precisely those that can be read from $V^1$ as in Theorem \ref{theo:tree-of-cyl}\ref{it_carac},
  $H$ preserves $\hat\calc$, so $H\subseteq \Gamma_{\!\{\hat\calc\}}$ which proves the first assertion.

  To prove the second assertion, choose a  $\tilde \Gamma_\calc$-tree $S\in \cald_{\Gamma_\calc}$,
   then  $\Uun_{\Gamma_\calc}$ is the tree of cylinders of $S$. Fix a vertex $v\in V^1(\Uun_{\Gamma_\calc})$,
  and let $Y\subset S$ be the cylinder corresponding to $v$.
  By Lemma~\ref{lemma:trivial-on-cylinders}, the group $\Gamma_\calc$ has a lift $\tilde{\Gamma}_\calc^Y$ to $\Aut(F_N)$ which acts as the identity on $Y$.
  Now, every edge $\eps$ of $\Uun_{\Gamma_\calc}$ incident on $v$    
  corresponds to a pair $(Y,x)$ with $x\in Y$.
  As $\tilde{\Gamma}_\calc^Y$ acts trivially on $Y$, it fixes $x$, and therefore $\tilde{\Gamma}_\calc^Y$ fixes the edge $\eps$ in  $\Uun_{\Gamma_\calc}$. It follows that $\tilde{\Gamma}_\calc^Y$ acts trivially on the star of $v$ in  $\Uun_{\Gamma_\calc}$, so $\Gamma_\calc$ is contained in the $V^1$-rigid stabilizer of $\Uun_\calc$.
  
  Conversely, let $\alpha\in\IA$ be contained in the $V^1$-rigid stabilizer of $\Uun_{\Gamma_\calc}$. Let $\tilde\alpha$ be a lift of $\alpha$ to $\Aut(F_N)$. We denote by $I_{\tilde\alpha}$ the $\tilde\alpha$-equivariant isometry of $\Uun_{\Gamma_\calc}$.
  Let $S\in {\calc}$. Theorem \ref{theo:tree-of-cyl}\ref{it_carac} can be reformulated as follows:
  there is an $F_N$-tree $\hat U$ which is a common refinement of
  $S$ and $\Uun_{\Gamma_\calc}$, with a collapse map $p:\hat U\ra \Uun_{\Gamma_\calc}$ such that $p\m(v)$ is reduced to a point for all $v\in V^0(\Uun_{\Gamma_\calc})$, and
  such that $S$ is obtained from $\hat U$ by collapsing every edge $e\subseteq \hat U$ which is not collapsed by $p$.
 
  For every vertex $v\in V^1$, we claim that there exists a unique element $g(\tilde\alpha,v)\in F_N$ such that $I_{\tilde\alpha}$ coincides with the action of $g(\tilde\alpha,v)$ on the star of $v$.
  Indeed, existence follows 
from the assumption on $\alpha$.
Uniqueness comes from the fact that for any two distinct vertices $w,w'\in V^0$, the stabilizers $G_w$ and $G_{w'}$ are $F_N$-stabilizers of distinct vertices of the same free splitting of $F_N$ (namely $S$), so $G_w\cap G_{w'}=\{1\}$.
In particular, if $e$ and $e'$ are two distinct edges in $\Uun_\calc$ incident on $v$, then $G_e\cap G_{e'}=\{1\}$.
This is enough to deduce the uniqueness of $g(\tilde\alpha,v)$.

For every point $w$ in the star of $v$ and every $h\in F_N$, one has $g(\tilde\alpha,hv)hw=I_{\tilde\alpha}(hw)=\tilde\alpha(h)I_{\tilde\alpha}(w)=\tilde\alpha(h)g(\tilde\alpha,v)w$.
By uniqueness, this shows that $$\tilde\alpha(h)g(\tilde\alpha,v)=g(\tilde\alpha,hv)h.$$
Let $\hat{I}_{\tilde\alpha}:\hat{U}\to \hat{U}$ be the map defined in the following way:
\begin{itemize}
\item for every vertex $v\in V^1$ and every point $x\in p^{-1}(v) $, we let $\hat{I}_{\tilde\alpha}(x):=g(\tilde\alpha,v)x$, and 
\item for every point $y\in\hat{U}$ whose $p$-image does not belong to $V^1$, we let $\hat{I}_{\tilde\alpha}(y)$ be the unique preimage of $I_{\tilde\alpha}(y)$ in $\hat{U}$. 
\end{itemize}
Then $\hat{I}_{\tilde\alpha}$ is an isometry, and the above relation shows that $\hat{I}_{\tilde\alpha}$ is $\tilde\alpha$-equivariant. This shows that $\hat{U}$, whence $S$, is $\alpha$-invariant. Since this holds for all $S\in\calc$, we conclude that $\alpha\in \Gamma_{\calc}$
which finishes the proof.
\end{proof}

\subsection{When all splittings in $\calc$ are fixed by the stabilizer of $\Uun_{\Gamma_\calc}$}
The following technical results will be used to characterize compatibility of free splittings in terms of their stabilizers in
Section~\ref{sec:compatibility}.

\begin{lemma}\label{lem_V1-trivial}
  Let $\calc$ be a collection of 
  non-trivial free splittings with infinite elementwise stabilizer $\Gamma_\calc$.
  Denote by $V^0\dunion V^1$ the bipartition of the vertex set of $\Uun_{\Gamma_\calc}$.
  Assume that $\Gamma_\calc$ almost coincides with the stabilizer of $\Uun_{\Gamma_\calc}$ in the following sense:
  for every $\alpha\in\IA$ preserving ${\Uun_{\Gamma_\calc}}$, there exists $k\geq 1$ such that $\alpha^k\in\Gamma_\calc$.

  Then for every $v\in V^1$,
  \begin{enumerate}
  \item if $G_v$ is non-abelian, then the Grushko decomposition of $G_v$ relative to $\Inc_v$ is sporadic with free rank 0 (\ie dual to an amalgam $G_v=A*B$); 
  \item if $G_v$ is non-abelian, then there is no edge incident on $v$ with trivial stabilizer; 
  \item if $G_v$ is infinite cyclic, then there is exactly one  $G_v$-orbit of edges with trivial stabilizer incident on $v$.
  \end{enumerate}
\end{lemma}

\begin{proof}
To prove the first assertion, assume that $G_v$ is non-abelian, and that the Grushko decomposition of $G_v$ relative to $\Inc_v$
  is non-sporadic, and argue towards a contradiction. Denote by $\calf_v$ the peripheral factors of this decomposition.
  Then consider an outer automorphism  $\alpha_v\in \Out(G_v,\calf_v^{(t)})$ 
  which is fully irreducible relative to $\calf_v$.
  In particular, no power of $\alpha_v$ preserves a non-trivial  free splitting of $G_v$ relative to $\calf_v$.
  Since $\alpha_v$ acts trivially on incident edge groups,
  it extends to an outer automorphism $\alpha\in\Out(F_N)$ preserving  $\Uun_{\Gamma_\calc}$.
  By  Theorem \ref{theo:tree-of-cyl}\ref{it_carac} every free splitting of $F_N$ obtained by blowing up $v$ is $\Gamma_\calc$-invariant.
  This prevents that $\alpha^k\in\Gamma_\calc$ for some $k\geq 1$. 

  Assume now that $G_v$ is non-abelian and that  $G_v$ has a Grushko decomposition  relative to $\Inc_v$ of the form of an
  HNN extension $G_v= A*$ (in particular, all groups in $\Inc_v$ are conjugate into $A$ within $G_v$).
  We claim that exists a free splitting $S_v$ of $G_v$ relative to $A$ and an outer automorphism $\alpha_v\in \Out(G_v,A^{(t)})$
   such that no power of $\alpha_v$ preserves $S_v$. The claim then leads to a contradiction as in the first case, thus concluding the proof of the first assertion.
  To prove the claim, let $t$ be a stable letter of the HNN extension, so that $G_v=A*\grp{t}$, and let $S_v$ be the free splitting dual
  to this amalgamated product. Choose $a\in A\setminus\{1\}$ and consider $\alpha_v$ acting as the identity on $A$ and sending $t$ to $ta$. Then no power of $\alpha_v$ preserves $S_v$, which proves the claim.

  To prove the second assertion, assume that $G_v$ is non-abelian, but there is an edge $e$ with trivial stabilizer incident on $v$.

  Let $a\in G_v\setminus\{1\}$, and let $\alpha\in \Out(F_N)$ be the twist by $a$ around $e$ near $v$: this has a representative $\tilde\alpha\in\Aut(F_N)$
  such that  $\tilde\alpha_{|G_v}=\id$ and $I_{\tilde \alpha}(e)=ae$.

Then for all $g\in G_v$, one has $I_{\tilde\alpha}(ge)=\tilde \alpha(g)I_{\tilde\alpha}(e)=gae$ and since $G_v$ is not abelian,
$\alpha$ is not in the $V^1$-rigid stabilizer of $U^1_{\Gamma_\calc}$, and neither are its non-trivial powers, a contradiction
which concludes the proof of the second assertion.

To prove the third assertion, notice first that there has to be at least one $G_v$-orbit of edges with trivial stabilizer incident on $v$, as otherwise there would be no possible blowup of $v$ into a free splitting. The uniqueness of this orbit of edges is proved as follows: if $e,e'$ are two edges with trivial stabilizer incident on $v$ which are not in the same orbit under $G_v$, one can consider a twist $\tilde \alpha$ around $e$ near $v$
that acts as the identity on  $G_v$, and such that $I_{\tilde\alpha}$ fixes $e$ and maps $e'$ to $ae'$ for some non-trivial
$a\in G_v$. Then $\alpha$ and its non-trivial powers are not in the $V^1$-rigid stabilizer of $U^1_{\Gamma_\calc}$,
which proves the  third assertion.
\end{proof}

\begin{cor}\label{cor_V1_compatible}
Under the assumptions of Lemma~\ref{lem_V1-trivial}, then any one-edge free splitting obtained (as in Theorem \ref{theo:tree-of-cyl}\ref{it_carac}) from  $\Uun_{\Gamma_\calc}$ 
 by blowing up a vertex $v\in V^1$ with non-trivial stabilizer,  is compatible with all splittings in $\calc$.
\end{cor}

\begin{proof} 
  Let $S$ be a one-edge free splitting as in the statement.
  It suffices to prove that $S$ is compatible with any  $\Gamma_{\calc}$-invariant one-edge free splitting $S'$.
  By  Theorem \ref{theo:tree-of-cyl}\ref{it_carac}, $S'$ can be obtained by from $\Uun_{\Gamma_\calc}$
  by blowing up a vertex $v'\in V^1$.
  If $v',v$ are not in the same orbit, then $S$ and $S'$ are clearly compatible.
  If $v',v$ are in the same orbit, and $G_v$ is non-abelian, then by Assertions 1 and 2 of Lemma~\ref{lem_V1-trivial},
  there is a unique possible free splitting (obtained by blowing up $v$ using the decomposition $G_v=A*B$), so $S=S'$.
  If $G_v$ is infinite cyclic, then up to the action of $G_v$, there is exactly one edge $\eps$ with trivial stabilizer incident on $v$ (Assertion~3 of Lemma~\ref{lem_V1-trivial}). If there is an edge with non-trivial stabilizer incident on $v$,
  then there is a unique possible free splitting, namely the one dual to $\eps$, so again $S=S'$.

  In the remaining case, $v$ corresponds to a terminal vertex $\bar v$ of $\Uun_{\Gamma_\calc}/F_N$ with infinite cyclic vertex group, and the incident edge $\bar\eps$ carries the trivial group.
  Then there are exactly two possible blow ups of $\Uun_{\Gamma_\calc}$ at $v$, namely one dual to $\bar\eps$,
  and one obtained by replacing $\bar v$ by a circle. These two splittings are compatible, so in any case, $S$ and $S'$ are compatible.
\end{proof}

\section{Stabilizers of collections of $\Zmax$-splittings and associated canonical splittings}\label{sec:collection-zmax}

Throughout this section, we fix an integer $N\ge 2$. Let $H\subseteq\IA$ be a subgroup. When $H$ does not fix any non-trivial free splitting, the canonical splitting $\Uun$ constructed in the previous section is trivial. In this section, we assume that $\FS^H=\es$, and we use JSJ theory to construct a canonical $\Zmax$-splitting of $F_N$ encoding the collection of all $H$-invariant $\Zmax$-splittings, see Theorem~\ref{theo:JSJ_zmax} below.

\subsection{Basic facts on cyclic splittings and trees of cylinders}

We recall that a $\calz$-splitting of $F_N$ is a splitting of $F_N$ whose edge stabilizers are either trivial or isomorphic to $\mathbb{Z}$.
A \emph{$\zmax$-splitting} is a splitting whose edge stabilizers are maximally cyclic (in particular non-trivial).

\begin{lemma}\label{lemma:acylindrical}
Let $S$ be a $\calz$-splitting of $F_N$, and let $A\subseteq F_N$ be a subgroup isomorphic to $\mathbb{Z}$ which is elliptic in $S$. 

Then the subtree $Y_A\subseteq S$ made of all points fixed by $A$ is a finite subtree.
\end{lemma}

\begin{proof}
Consider an edge $e\subset Y_A$. Let $\hat A$ be the maximal cyclic group containing $A$.
We claim that $Y_A$ contains only finitely many edges in the $F_N$-orbit of $e$.
Indeed, if $e'=ge\subseteq Y_A$, then $G_{e'}=gG_e g\m$ and $G_e$ and $G_{e'}$ both contain $A$.
It follows that $g\in \hat A$, so there are at most $\#\hat A/A$ possibilities for $e'$.
This concludes the proof since there are only finitely many orbits of edges in $S$.
\end{proof}

In the case where $S$ is a $\Zmax$-splitting of $F_N$, we have the following more precise version.

\begin{lemma}\label{lemma:acylindrical-zmax}
Let $S$ be a $\Zmax$-splitting of $F_N$. Then two distinct edges of $S$ with the same stabilizer belong to different $F_N$-orbits.
\end{lemma}

\begin{proof}
Let $g\in F_N$, and let $e$ be an edge with stabilizer $G_e$. Then the stabilizer of $ge$ is equal to $gG_eg^{-1}$. Hence either $g$ belongs to the $\Zmax$ subgroup $G_e$ and $e=ge$, or else $e$ and $ge$ have distinct stabilizers. 
\end{proof}

For a $\Zmax$-splitting $S$ of $F_N$, 
equality among edge stabilizers is an admissible equivalence relation
(which coincides with commensurability and commutation).
The cylinder of an edge $e$ is thus the union of edges of $S$ having the same stabilizer.
See Section \ref{sec:cyl} for the definition of admissible equivalence relations and of the tree of cylinders.

\begin{lemma}\label{lemma:cyl_Zmax}
  Let $S$ be a $\Zmax$-splitting of $F_N$, and let $S_c$ be its tree of cylinders.

Then $S_c$ is a $\Zmax$-splitting in the same deformation space as $S$,
and every edge stabilizer of $S_c$ is an edge stabilizer of $S$.
\end{lemma}

\begin{proof}
Lemma \ref{lemma:acylindrical} shows that each cylinder $Y$ is finite, so its setwise stablizer $G_{\{Y\}}$
contains $G_e$ with finite index. Since $G_e$ is maximal cyclic, it follows that $G_{\{Y\}}=G_e$.
Since $G_{\{Y\}}$ is elliptic in $S$, this implies that $S$ and $S_c$ are in the same deformation space.

Additionally, if $\eps$ is an edge of $S_c$, it corresponds to a pair $(x,Y)$ with $x\in Y$,
and $G_\eps$ contains the stabilizer of
any edge $e$ of $Y$ incident on $x$, so $G_e\subset G_\eps\subset G_{\{Y\}}=G_e$. This shows that $S_c$ is a $\Zmax$-splitting
and concludes the proof of the lemma.
\end{proof}

One can check that a $\Zmax$-splitting $S$ is its own tree of cylinders if its vertex set is bipartite into $V_{cyc}\dunion V_{nab}$
where  $G_v$ is not abelian for $v\in V_{nab}$, and for all $v\in V_{cyc}$, $G_v$ is maximal cyclic and
the set of fixed points of $G_v$ is precisely the star of $v$ in $S$.

\subsection{Statement of the result}\label{sec:statement}

Let $H\subset \IA$ be a subgroup, and assume that there is no non-trivial  $H$-invariant free splitting of $F_N$.
Let $\ZmaxS^H$ be the collection of all $H$-invariant  non-trivial $\Zmax$-splittings of $F_N$.
Out of this collection of splittings we are going to construct a JSJ decomposition that will be a canonical splitting
associated to $H$.

If $S,S'$ are two splittings of a group $A$, we say that $S$ is \emph{elliptic} with respect to $S'$ if all its edge stabilizers
are elliptic in $S'$.
If $\calt$ is a collection  of splittings of a group $A$, 
we say that a subgroup $B\subseteq A$ is \emph{$\calt$-universally elliptic} if  $B$ is elliptic in all splittings in $\calt$.
We say that a splitting of $A$ is \emph{$\calt$-universally elliptic} if all its edge stabilizers are $\calt$-universally elliptic.

\begin{theo}[Canonical JSJ decomposition for $H$-invariant $\Zmax$-splittings]\label{theo:JSJ_zmax}
Let $H\subset \IA$ be a subgroup, and assume that there is no non-trivial $H$-invariant free splitting of $F_N$.

There exists a unique $H$-invariant $\Zmax$-splitting $\UZ_H$ with the following properties:
\begin{enumerate}\renewcommand{\theenumi}{(\arabic{enumi})}\renewcommand{\labelenumi}{\theenumi}
\item $\UZ_H$ is $\ZmaxS^H$-universally elliptic;
\item $\UZ_H$ dominates every $H$-invariant $\Zmax$-splitting which is $\ZmaxS^H$-universally elliptic; 
\item $\UZ_H$ is its own tree of cylinders.
\end{enumerate}
Moreover, $\UZ_H$ is non-trivial if $\ZmaxS^H\neq\es$.
\end{theo}

\begin{rk}
  The assumption that $H$ preserves no  non-trivial free splitting implies that $H$ is non-trivial, i.e.\ infinite.
\end{rk}

To prove the existence of this JSJ decomposition, one has to work consistently with $H$-invariant splittings,
which we interpret as $\tilde H$-splittings, where $\tilde H$ is the full preimage of $H$ in $\Aut(F_N)$
(see Section \ref{sec:tildeH}).
So in fact, we will construct $\UZ_H$ as a JSJ decomposition of $\tilde H$ over a suitable collection of edge groups.
Proving the existence and uniqueness of $U^z_H$ is actually easy, and done in Section~\ref{sec:existence-uniqueness} below; the hard part of the work is to prove the non-triviality of $\UZ_H$ in Section~\ref{sec:non-trivial}.

\subsection{Uniqueness and existence}\label{sec:existence-uniqueness}

Let us start with the proof of the uniqueness of $\UZ_H$ which is almost immediate.

\begin{proof}[Uniqueness of $\UZ_H$.]
  Let $U_1,U_2$ be two $H$-invariant $\Zmax$-splittings as in Theorem~\ref{theo:JSJ_zmax}.
  We view $U_1$ and $U_2$ as $F_N$-trees.
   Combining Assertions~(1) and~(2) shows that $U_1$ and $U_2$ are in the same deformation space. 
  Since they are their own trees of cylinders,
  this implies that $U_1=U_2$ (\cite[Theorem 1]{GL-cyl}).
\end{proof}

\begin{lemma}\label{lemma:acylindrical2}
  Let $S$ be an $H$-invariant $\Z$-splitting of $F_N$.
  Let $\tilde B$ be a subgroup of $\tilde H$ such that $\tilde B\cap F_N$ is non-trivial and fixes a point in $S$.

  Then $\tilde B$ fixes a point in $S$.  
\end{lemma}

\begin{proof}
Since $A:=\tilde B\cap F_N$ is normal in $\tilde B$, 
the subtree  $Y_A\subset S$ consisting of all points fixed by $A$ is $\tilde B$-invariant.
If $A$ is not cyclic, then $Y_A$ is reduced to a point because edge stabilizers are cyclic. Otherwise,
 $Y_A$ has finite diameter by Lemma~\ref{lemma:acylindrical}, and we deduce that $\tilde{B}$ fixes the circumcenter of $Y_A$.
\end{proof}

\begin{lemma}\label{lemma:edge-stabilizers-cyclic}
Let $S,S'\in\ZmaxS^H$. The tree $S$ is elliptic with respect to $S'$ as an $F_N$-tree  if and only if $S$ is elliptic with respect to $S'$ as an $\tilde{H}$-tree. 
\end{lemma}

\begin{proof}
Given an edge $e\subseteq S$, we denote by $G_e$ the $F_N$-stabilizer of $e$, and by $\tilde{H}_e$ its $\tilde{H}$-stabilizer. 
By definition of ellipticity (see Section \ref{sec:statement}), 
we have to prove that $\tilde H_e$ is elliptic in $S'$ if and only if $G_e$ is.

Since $G_e\subseteq \tilde H_e$, one implication is obvious.
Conversely, assume that 
$G_e$ is elliptic in $S'$.
Since $G_e=\tilde H_e\cap F_N$ is non-trivial and elliptic in $S'$,
Lemma~\ref{lemma:acylindrical2} shows that $\tilde H_e$ fixes a point in $S'$.
\end{proof}

\begin{cor}\label{coro:blowup_zmax}
Let $S,S'\in\ZmaxS^H$ viewed as $\tilde H$-trees. Assume that $S$ is elliptic with respect to $S'$.

Then there exists an $\Tilde H$-tree $\hat S$ which is a blowup of $S$ and dominates $S'$,
and such that  each edge stabilizer of $\hat S$ fixes an edge in $S$ or in $S'$.

Moreover, as an $F_N$-tree, $\hat S$ is a ($H$-invariant) $\Zmax$-splitting of $F_N$.
\end{cor}

\begin{proof}
Since $S$ is elliptic with respect to $S'$ as an $\tilde H$-tree by Lemma~\ref{lemma:edge-stabilizers-cyclic},
one can apply \cite[Proposition~2.2]{GL-jsj} to construct an $\tilde H$-tree $\hat S$
that refines $S$ and dominates $S'$ and whose edge stabilizers fix an edge in $S$ or $S'$ (as $\tilde H$-trees, therefore also as $F_N$-trees).
Thus, $\hat S$ can be viewed as an $H$-invariant $F_N$-tree, and $\hat S$ is still a $\Zmax$-splitting by \cite[\S~9.5]{GL-jsj}.  
\end{proof}

We can now prove the existence of $\UZ_H$.

\begin{proof}[Existence of $\UZ_H$.]
Note that the trivial splitting is $\ZmaxS^H$-universally elliptic.
By Lem\-ma~\ref{lem_bound_zmax}, there is a bound on the number of orbits of edges of a $\Zmax$-splitting of $F_N$ without vertex of valence 2.
Therefore, there exists an $H$-invariant $\ZmaxS^H$-universally elliptic $\Zmax$-splitting $U$ such that no proper refinement of $U$ satisfies this property. We claim that $U$ satisfies Assertion (2). Indeed, consider $U'\in \ZmaxS^H$ any $\ZmaxS^H$-universally elliptic splitting. 
Apply Corollary \ref{coro:blowup_zmax} to construct
an $\tilde H$-tree $\hat U$ that refines $U$ and dominates $U'$.
Since edge stabilizers of $\hat U$ fix an edge in $U$ or $U'$,
$\hat U$ is $\ZmaxS^H$-universally elliptic. By maximality of $U$, we thus have $\hat{U}=U$, showing that $U$ dominates $U'$. 
This proves that $U$ satisfies Assertions (1) and (2).

Now let $\UZ_H$ be the tree of cylinders of $U$, viewed as a $\Zmax$-splitting of $F_N$.
By Lemma~\ref{lemma:cyl_Zmax}, $\UZ_H$
 is a $\Zmax$-splitting in the same deformation space as $U$, and edge stabilizers of $\UZ_H$ are edge stabilizers of $U$, so $\UZ_H$ is  $\ZmaxS^H$-universally elliptic. It follows that $\UZ_H$ satisfies (1), (2) and (3).
\end{proof}

\subsection{Description and non-triviality of $\UZ_H$}\label{sec:non-trivial}

The goal of the present section is to prove the non-triviality of $\UZ_H$ when $H$ has no invariant  non-trivial free splittings and $\ZmaxS^H\neq\emptyset$ (Theorem~\ref{thm_JSJ} below). This relies on arguments coming from JSJ theory, the key tool being to understand \emph{flexible} vertices of $\UZ_H$, defined as follows, and show that they come from QH situations (all relevant definitions are given right below).

\begin{de}\label{de:flexible}
A vertex $v$ of  $\UZ_H$ is \emph{flexible} if its $F_N$-stabilizer $G_v$ is not  $\ZmaxS^H$-universally elliptic. 
\end{de}

\begin{figure}[ht]
  \centering
  \includegraphics{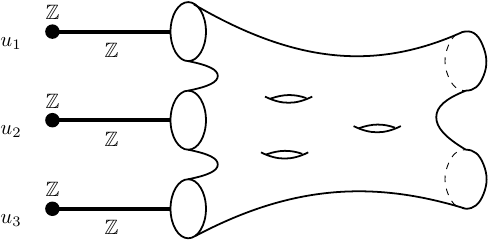}
  \caption{A socket graph of groups with 3 proper sockets and 2 improper ones.}
\end{figure}

A \emph{socket graph of groups} $\Lambda$ is a minimal tree of groups whose underlying graph is a star over a central vertex $v$,
where $v$ is a clean QH vertex (see Definition \ref{dfn_cleanQH}) and the other vertex groups are cyclic (see \cite{Sel}).
Equivalently, denoting by $u_1,\dots,u_r$ the vertices of $\Lambda\setminus\{v\}$, and by $e_i$ the edge joining $u_i$ to $v$, 
\begin{itemize}
\item the group $G_v=\pi_1(\Sigma)$ is identified to the fundamental group of a connected compact hyperbolic surface  $\Sigma$ with $n\geq r$ boundary components, each edge group $G_{e_i}$ being conjugate in $\pi_1(\Sigma)$  to $B_i$, where $B_i$ is the fundamental group of the $i$-th boundary component of $\Sigma$;
\item each vertex group $G_{u_i}$ is cyclic and $[G_{u_i}:G_{e_i}]\geq 2$.
\end{itemize}
We note that if $\pi_1(\Lambda)$ is a free group, then $n>r$.

For $i\leq n$, let $C_i$ be the maximal cyclic subgroup of $\pi_1(\Lambda)$ containing $B_i$.
Note that $C_i=G_{u_i}$ for $i\leq r$ and $C_i=B_i$ for $i>r$.
The conjugacy classes of the groups $C_i$ are called the \emph{sockets}.
A socket is \emph{improper} if $C_i=B_i$. 

\begin{rk}\label{rk:zmaxise_trivial}
  If $\Lambda$ is a socket graph of groups, then the associated $\Zmax$-splitting $\Lambda_\Zmax$ defined in Definition \ref{dfn_zmaxise}
  is trivial.
The following lemma gives a converse.  
\end{rk}

\begin{lemma}\label{lem_QH_Zmax} Let $S$ be a splitting of $F_N$ with infinite cyclic edge stabilizers. Assume that $S$ has at least one QH vertex with trivial fiber, and that every edge is adjacent to at least one vertex with non-abelian stabilizer.

If the associated $\Zmax$-splitting $S_\Zmax$ (Definition \ref{dfn_zmaxise}) is trivial,
then $S/F_N$ is a socket graph of groups. 
\end{lemma}

\begin{proof}
By Lemma~\ref{lem:zmaxise_trivial}, $S/F_N$ is a tree of groups which is a star with all vertex groups cyclic except the central one. The central vertex $v$ is QH with trivial fiber by assumption.
There remains to prove that $v$ is clean QH. First, the underlying orbifold $\Sigma$ is a surface since $F_N$ is torsion-free and the fiber is trivial. Second, no two edges of $S/F_N$ are incident on the same boundary component of $\Sigma$, as otherwise we get an element of $F_N$ that has two distinct roots in $F_N$, which is not possible. For the same reason, each incident edge group is a maximal boundary subgroup of $\pi_1(\Sigma)$.
\end{proof}

\begin{de}\label{dfn_sockets}
Let $A$ be a finitely generated group, and let $\calp$ be a finite collection of subgroups of $A$. The pair $(A,\calp)$ is \emph{weakly QH with sockets} if $A$ splits as a socket graph of groups $A\simeq \pi_1(\Lambda)$ 
and $\calp$ is a subset of the conjugacy classes of the socket groups. 

It is \emph{QH with sockets} if there exists a socket graph of groups as above where additionally $\calp$ contains all conjugacy classes of improper sockets. Such a decomposition of $A$ is called a \emph{socket decomposition} of $(A,\calp)$.

If $\Lambda$ is a socket decomposition of $(A,\calp)$, a \emph{Möbius socket} of $\Lambda$ is a socket $C_i$ which is not in $\calp$ and such that $[C_i:B_i]=2$.
\end{de}

\begin{rk}\label{rk:weak-qh} 
If $(A,\calp)$ is weakly 
QH with sockets but not QH with sockets, then there is a free splitting
of $A$ relative to $\calp$ dual to a properly embedded arc in the underlying surface, whose endpoints are on a boundary component corresponding to an improper socket not appearing in $\calp$.
Thus, if $(A,\calp)$ is weakly QH with sockets and if there is no free splitting
of $A$ relative to $\calp$, then $(A,\calp)$ is in fact QH with sockets.
\end{rk}

\begin{figure}[ht]
  \centering
  \includegraphics[scale=0.9]{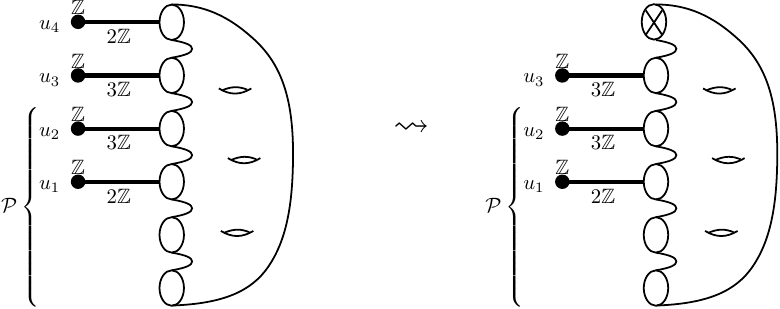}
  \caption{$u_4$ is the only Möbius socket, and one can replace it by a Möbius band. 
  The decomposition on the right is the cyclic JSJ decomposition relative to $\calp$.}
  \label{fig:socket}
\end{figure}

\begin{rk} 
If $(A,\calp)$ is QH with sockets, then one can choose a socket decomposition of $(A,\calp)$ without any Möbius socket (see Figure~\ref{fig:socket}).
Indeed, start with any socket decomposition $\Lambda$ of $(A,\calp)$ 
with underlying surface $\Sigma$.
If there are Möbius sockets (which are by definition not in $\calp$), one can remove the corresponding edges of $\Lambda$
and  glue back a Möbius bands on the corresponding boundary components of $\Sigma$.

If $\Lambda$ is a socket decomposition of $(A,\calp)$, then $\Lambda$ is its cyclic JSJ decomposition by \cite[Proposition~4.24]{DG_isomorphism}. In particular, it is unique.
\end{rk}

In the following statement, we view $\UZ_H$ as an $F_N$-tree.
Given a vertex $v\in \UZ_H$, we denote by $\Inc_v$ the set of all $F_N$-stabilizers of edges that are incident on $v$.

We denote by $H_{|G_v}\subset \Out(G_v)$ the subgroup obtained by restriction of outer automorphisms in $H$.
More precisely, since $\UZ_H$ is $H$-invariant and $H$ acts trivially on the quotient graph $\UZ_H/F_N$ (Proposition~\ref{prop:ia-zmax}),
the conjugacy class of $G_v$ is $H$-invariant. Thus for every $\tilde\alpha\in H$, there is a representative
$\tilde \alpha\in \Aut(F_N)$ such that $\tilde\alpha(G_v)=G_v$, and the image of $\tilde \alpha_{|G_v}$ in
$\Out(G_v)$ does not depend on the choice of $\tilde\alpha$ because $G_v$ is its own normalizer, we denote it by $\alpha_{|G_v}$.
The map $\alpha\mapsto \alpha_{|G_v}$ defines a morphism $H\to \Out(G_v)$ whose image we denote by $H_{|G_v}$.
An element $\alpha$ lies in its kernel if it has a lift $\tilde \alpha\in \Aut(F_N)$ such that $\tilde\alpha_{|G_v}=\id$; we then
say that $\alpha$ \emph{acts trivially} on $G_v$.

\begin{theo}\label{thm_JSJ}
Assume that $H\subset\IA$ (is infinite and) preserves no  non-trivial free splitting, and that $\ZmaxS^H\neq \es$. Then $\UZ_H$ is non-trivial. 

Moreover, if $v$ is a flexible vertex of $\UZ_H$, then 
$(G_v,\Inc_v)$ is QH with sockets, and $H_{|G_v}=\{1\}$.
\end{theo}

\begin{rk}\label{rk:2imp1}
  The second assertion implies the first. Indeed,
if $\UZ_H$ is reduced to a single point $v$, then $v$ is flexible because there exists a non-trivial $H$-invariant $\Zmax$-splitting
(we assume $\ZmaxS^H\neq \es$).
Since $G_v=F_N$, $H_{|G_v}=H$ so  $H$ is trivial by the second assertion, a contradiction.
\end{rk}

\begin{lemma}\label{lemma:restriction}
  Let $v$ be a flexible vertex of $\UZ_H$, and let $\calt$ be the set of all non-trivial
   $\Zmax$-splittings of $G_v$ relative to $\Inc_v$ which are $H_{|G_v}$-invariant. Then
  \begin{itemize}
  \item $G_v$ has no non-trivial $H_{|G_v}$-invariant free splitting relative to $\Inc_v$;
  \item $\calt\neq\es$ \ie $G_v$ has some non-trivial $H_{|G_v}$-invariant $\Zmax$-splitting relative to $\Inc_v$;
  \item no splitting in $\calt$ is $\calt$-universally elliptic.
  \end{itemize}
\end{lemma}

\begin{proof}
  We claim that any $H_{|G_v}$-invariant cyclic splitting $S$ of $G_v$ relative to $\Inc_v$ can be used to blow up
  $\UZ_H$ into an $H$-invariant cyclic splitting $\hat U$ of $F_N$. Indeed, view $\UZ_H$ as an $\tilde H$-tree,
  and let $\tilde H_v$ be the stabilizer of $v$ in $\tilde H$.
  Since $S$ is $H_{|G_v}$-invariant,
  it can be viewed as a $\tilde K$-tree where $\tilde K\subset \Aut(G_v)$ is the preimage of $H_{|G_v}$
  in $\Aut(G_v)$. We note that $\tilde K$ coincides with the image of 
  the natural restriction map $r:\tilde H_v \ra \Aut(G_v)$. 
  The map $r:\tilde H_v\ra \tilde K$ allows to view the $\tilde K$-tree $S$ as an $\tilde H_v$-tree.
  There remains to check that the $\tilde H$-stabilizer $\tilde H_e$ of each edge $e$ incident on $v$ fixes a point in $S$.
  Since $S$ is relative to $\Inc_v$, the $F_N$-stabilizer $G_e$ of $e$ fixes a point in $S$.
  By Lemma \ref{lemma:acylindrical2} (applied to $\tilde K\actson S$ and $\tilde B=r(\tilde H_e)$), 
$\tilde H_e$ is elliptic in $S$ which concludes the proof of the claim.

  The claim implies the first assertion: any non-trivial $H_{|G_v}$-invariant free splitting of $G_v$ relative to $\Inc_v$
  would induce an $H$-invariant splitting $\hat U$ of $F_N$ having an edge with trivial $F_N$-stabilizer, contradicting
  the assumption that $H$ preserves no non-trivial free splitting.
  
  Since $G_v$ is flexible, there exists an $H$-invariant $\Zmax$-splitting $S$ of $F_N$
in which $G_v$ is not elliptic.
Let $S_v\subset S$ be the minimal $G_v$-invariant subtree, with its action by $G_v$.
Then $S_v$ is an $H_{|G_v}$-invariant $\Zmax$-splitting  of $G_v$, and since $\UZ_H$ is universally elliptic, the groups in $\Inc_v$
fix a point in $S$ hence in $S_v$. This proves the second assertion.

 Finally, let $S\in \calt$; we aim to prove that $S$ is not $\calt$-universally elliptic. Let $\hat U$ be an $H$-invariant blowup of $\UZ_H$ given by the claim above, using $S$ to blow up $\UZ_H$ at  $v$. 
Note that $G_v$ is root closed in $H$ because $\UZ_H$ is a $\Zmax$-splitting.
Edge stabilizers of $\hat U$ are $\Zmax$ because $S$ is a $\Zmax$-splitting and $G_v$ is root closed.
Since  $\UZ_H$ does not dominate $\hat U$, $\hat U$ is not  $\ZmaxS^H$-universally elliptic
so one can find some edge $e\subset S$ whose stabilizer $G_e\subset G_v$ is not elliptic 
in some $H$-invariant $\Zmax$-splitting $T$ of  $F_N$. Let $T_v\subset T$ be the minimal $G_v$-invariant tree.
The action $G_v\actson T_v$ is an $H_{|G_v}$-invariant $\Zmax$-splitting of $G_v$ relative to $\Inc_v$.
Thus this is a splitting in $\calt$ in which $G_e$ is not elliptic.  In particular $S$ is not $\calt$-universally elliptic.
\end{proof}

Theorem \ref{thm_JSJ} is an easy consequence of the following result.

\begin{prop}\label{prop:QH}
  Let $G$ be a non-cyclic finitely generated free group, $\calp$ be a finite family of conjugacy classes of maximal cyclic subgroups of $G$, and $K\subset \Out(G,\calp)$.
  Let $\calt$ be the set of all non-trivial minimal $\Zmax$-splittings of $G$
  relative to $\calp$ which are $K$-invariant. Assume that 
  \begin{itemize}
  \item $G$ has no non-trivial $K$-invariant free splitting relative to $\calp$;
  \item $\calt\neq\es$;
  \item no splitting in $\calt$ is $\calt$-universally elliptic.
  \end{itemize}
  Then $(G,\calp)$ is weakly QH with sockets and  $K$ is finite.
\end{prop}

\begin{proof}[Proof of Theorem \ref{thm_JSJ} from Proposition \ref{prop:QH}.]
  As noted in Remark \ref{rk:2imp1}, it suffices to prove the second assertion.
  So consider a flexible vertex $v$ of $\UZ_H$.
  Notice that $G_v$ is a finitely generated free group (being a point stabilizer in a $\Zmax$-splitting of $F_N$),
  and is not cyclic (because being flexible, it has a non-trivial  splitting relative to incident edge groups).

  We apply Proposition~\ref{prop:QH} with  $G=G_v$, $\calp=\Inc_v$ and $K=H_{|G_v}$.
 Lemma \ref{lemma:restriction} shows that the hypotheses of Proposition~\ref{prop:QH} are satisfied,
 so  $(G_v,\Inc_v)$ is weakly QH with sockets and $H_{|G_v}$ is finite.

 We now deduce that $H_{|G_v}$ is trivial using that $H\subset \IA$.
 Fix any $\alpha\in H_{|G_v}$, and consider
 $\tilde \alpha\in\tilde H$ preserving $G_v$ and inducing $\alpha$ in $\Out(G_v)$.
 By \cite[Theorem~B(3)]{KSS}, generic elements $u\in G_v$ satisfy that their stabilizer in $\Aut(G_v)$ is generated by $\ad_u$
 (the inner automorphism associated to $u$).
 Choose such an element $u\in G_v$ not conjugate into an incident edge group.
 Since $H_{|G_v}$ is finite, the $H$-orbit of the $F_N$-conjugacy class of $u$ is finite.
 Since $H\subseteq \IA$, it follows from \cite[Theorem~4.1]{HM2} that the $F_N$-conjugacy class of $u$ is fixed by $H$.
 Then $\tilde\alpha(u)\in G_v$ is conjugate to $u$ in $F_N$,
 and since $u$ fixes no edge, $\tilde \alpha(u)$ is conjugate to $u$ in $G_v$.
 Up to composing $\tilde\alpha$ by the conjugation by an element in $G_v$,
 we may assume that $\tilde \alpha(u)=u$. By choice of $u$, $\tilde \alpha_{|G_v}$ is inner, i.e.\ $\alpha$ is trivial.
 We thus conclude that $H_{|G_v}=\{1\}$.
 
 Finally, Lemma~\ref{lemma:restriction} ensures that $G_v$ has no non-trivial free splitting relative to $\Inc_v$, so $(G_v,\Inc_v)$ is QH with sockets (and not only weakly QH with sockets; see Remark~\ref{rk:weak-qh}).   
\end{proof}

\begin{proof}[Proof of Proposition \ref{prop:QH}]
 Let $\tilde K$ be the full preimage of $K$ under the quotient map $\Aut(G)\onto\Out(G)$.
We will see $K$-invariant splittings of $G$ as splittings of $\tilde K$.

We are going to follow  \cite[\S 6]{GL-jsj} giving the description 
of flexible vertices of the JSJ decomposition of
the group $\tilde K$ relative to $\calp$ over a well chosen class of allowed edge groups.
We cannot readily apply the results because we want the extra information that $K$ is finite,
and we do not know a priori that $\tilde K$ is finitely presented.

\newcommand{\Amax}{\cala_{max}}
\newcommand{\AZ}{\cala_{\bbZ}}
\newcommand{\Acyc}{\cala_{cyc}}
\newcommand{\Asl}{\cala_{sl}}

If $\cala$ is a collection of subgroups of $\tilde{K}$, an $(\cala,\calp)$-tree $S$ is an action of $\tilde K$ on 
a tree with edge stabilizers in $\cala$ and in which groups in $\calp$ are elliptic. 
We will work with the following collections of subgroups.
We identify $G$ with $\Inn(G)$ so that $G\normal \tilde K$.
Let $\Amax$ (respectively $\AZ$, $\Acyc$) be the collection of subgroups $E\subset \tilde K$ such that
$E\cap G$ is maximal cyclic (respectively isomorphic to $\bbZ$, respectively cyclic (maybe trivial)).
For instance, non-trivial $(\Amax,\calp)$-trees are exactly the same thing as non-trivial $K$-invariant $\Zmax$-splittings of $G$ relative to $\calp$ 
(i.e.\ trees in $\calt$), viewed as $\tilde{K}$-trees.

If $E\subset \tilde K$ and $S$ is a splitting of $\tilde K$, 
we say that $E$ is \emph{slender in $S$} if it preserves a line or a point of $S$. 
We will say that $E$ is \emph{hyperbolic} in $S$ if it is not elliptic.

Let $\Asl$ the collection of subgroups $E\subset \tilde K$ such that $E\in\Acyc$ and $E$ is slender in any
$(\Acyc,\calp)$-tree.
Unlike the other two classes, $\Asl$ and $\Acyc$ are stable under taking subgroups.

We have the following inclusions
$$\Amax\subset \AZ\subset\Asl \subset\Acyc.$$
We explain the second inclusion, the two others being obvious.
Let $E\in \AZ$, and let $C=E\cap G\simeq \bbZ$, a normal subgroup of $E$.
If $T$ is an $(\Acyc,\calp)$-tree in which $C$ is not elliptic, then $C$ has an axis $A_C$ in $T$, and $A_C$ is $E$-invariant. If $C$ is elliptic in $T$, then its set of fixed points is bounded (Lemma~\ref{lemma:acylindrical}) so
$E$ is also elliptic in $T$. This shows that $E$ is slender in $(\Acyc,\calp)$-trees, so $\AZ\subset \Asl$.

Since there is no non-trivial $K$-invariant free splitting of $G$ relative to $\calp$,
any $(\Acyc,\calp)$-tree is in fact an $(\AZ,\calp)$-tree. Thus,
$(\AZ,\calp)$-trees, $(\Asl,\calp)$-trees, and $(\Acyc,\calp)$-trees are the same objects:
we denote by $\calt'\supset \calt$ this set of actions of $\tilde K$ on trees.

\begin{claim}\label{claim_slender}
If $E\subset \tilde K$ is slender in trees in $\calt'$,
and hyperbolic in some tree in $\calt'$,
then $E\in\Asl$.
\end{claim}

\begin{proof}
We have to check that $E\cap G$ is cyclic (maybe trivial).
Indeed, consider $T\in\calt'$ such that $E$ is hyperbolic in $T$, 
and let $l\subset T$ be the $E$-invariant line.
Since the $G$-stabilizer of any edge of $T$ is cyclic, 
the setwise $G$-stabilizer of $l$ is a small subgroup of the free group $G$, so it is cyclic.
Thus $E\cap G$ is cyclic and $E\in \Asl$.
\end{proof}

\begin{claim}\label{cl_faithful}
 The $\tilde{K}$-action on every non-trivial tree $S\in \calt'$ is faithful.

 More generally, let $S\in \calt'$, $A\subset G$
 not elliptic in $S$,
and $\tilde N\subset\tilde K$ acting as the identity on an $A$-invariant subtree $S'\subset S$.
Then for all $\tilde\alpha \in \tilde N$ and all $a\in A$, one has $\tilde\alpha(a)=a$.
\end{claim}

\begin{proof}
One deduces the first assertion from the second remembering that $\tilde K\subset \Aut(G)$ 
by taking $A=G$, and $\tilde N$ the kernel of the action $\tilde K\actson S$.

In general, 
for $\tilde\alpha\in\tilde K$,
 we denote by $I_{\tilde\alpha}$ the isometry of $S$ determined by the action of $\tilde\alpha$.
 Fix $\tilde\alpha\in\tilde{N}$ and $a\in A$. Then for all $x\in S'$, we have $a.x=I_{\tilde\alpha}(a.x)=\tilde\alpha(a).I_{\tilde\alpha}(x)=\tilde\alpha(a).x$, so $a\m\tilde\alpha(a)$ is an element of $G$ acting as the identity on $S'$. 
Since $S'$ is unbounded, Lemma \ref{lemma:acylindrical} implies that $\tilde\alpha(a)=a$.
\end{proof}

We claim that $\Asl$ satisfies the stability condition $(SC)$ of \cite[Definition~6.1]{GL-jsj} with $\calf=\Asl$ as a class of fibers. Recall that this means that for every short exact sequence of groups $$1\to F\to B\to Q\to 1$$ with $B\subseteq\tilde{K}$, the following two conditions hold:
\begin{enumerate}
\item if $B\in\Asl$ and $Q$ is isomorphic to either $\mathbb{Z}$ or the infinite dihedral group $D_\infty$, then $F\in\Asl$, and 
\item if $F\in\Asl$ and $Q$ is isomorphic to a quotient of $\bbZ$ or $D_\infty$, then $B\in\Asl$.
\end{enumerate}
Indeed, the first condition is obvious since $\Asl$ is stable under taking subgroups. 
To check the second condition, assume that $F\in\Asl$, and that $B$ is an extension of $F$ as above.
Then $B\cap G$ is an extension of the cyclic group $F\cap G$ by a cyclic or dihedral group, so $B\cap G$ is cyclic, i.e.\ $B\in \Acyc$.
To check the slenderness condition, let $S\in\calt'$. 
Since $F\in \Asl$, either $F$ preserves a unique line $l$ in $S$, or $F$ fixes a point.
In the first case, $l$ is $B$-invariant because $F$ is normal in $B$, so $B$ is slender in $S$.
In the second case, let $Y\subset S$ be the subtree consisting of all points fixed by $F$.
Then $Y$ is $B$-invariant and the action of $B$ factors through an action of $Q$.
Since $Q$ is cyclic or dihedral, it preserves a line or fixes a point in $Y$, so $B$ is slender in $S$. This concludes the proof
that $\Asl$ satisfies the stability condition $(SC)$.\\

We recall some more definitions from \cite[Definition~6.9]{GL-jsj}, which generalize a usual assumption asking that
$G$ does not split over a subgroup commensurable with a subgroup of infinite index in an allowed edge group.
Given two subgroups $A,B\in\Asl$, we write $A\ll B$ if $A\subseteq B$ and there exists a tree $S\in\calt'$ such that $A$ is elliptic in $S$ while $B$ is not.
A subgroup $B\in\Asl$ is \emph{minuscule} if whenever $A\ll B$ and $A$ fixes an edge $e$ in some tree in $\calt'$,
then $A$ has infinite index in the $\tilde{K}$-stabilizer $\tilde{K}_e$ of $e$. A tree $S\in\calt'$ is \emph{minuscule} if all edge stabilizers of $S$ are minuscule.

\begin{claim}\label{claim:minuscule}
Every tree $S\in \calt'$ is minuscule.
\end{claim}

\begin{proof} 
  Let $e\subseteq S$ be an edge. We assume that there exists $A\ll\tilde{K}_e$, with $A$ of finite index in the stabilizer $\tilde K_{e'}$ of an edge $e'$ of another tree $S'\in\calt'$, and aim for a contradiction.
As noted above, $S'$ is an $\AZ$-tree so $A\cap G\simeq \bbZ$. Let $C$ be the maximal cyclic subgroup of $G$ that contains $A\cap G$.
The group $G_e=\tilde{K}_e\cap G$ is cyclic, of finite index in $C$, and therefore $\tilde{K}_e$ normalizes $C$.
As $A\ll \tilde{K}_e$, there exists a tree $S''\in\calt'$ such that $A$ is elliptic in $S''$ while $\tilde{K}_e$ is not.
Then $A\cap G$, whence $C$, is elliptic in $S''$. By Lemma~\ref{lemma:acylindrical2}, $\tilde{K}_e$ fixes a point in $S''$, a contradiction.
\end{proof}

We are now going to construct the underlying surface following the approach by Fujiwara and Papsoglu \cite{FP} as
implemented in \cite[\S6]{GL-jsj}. We will use the general definition of $QH$ vertices 
of \cite[Definition~5.13]{GL-jsj} recalled in Definition \ref{dfn_QH}:
if $S$ is a $\tilde K$-tree, $v\in S$ is a QH vertex with fiber $F$ relative to $\calp$
if $F\normal G_v$ and there is an identification $G_v/F=\pi_1(\Sigma)$ with the fundamental group 
of a hyperbolic $2$-orbifold with boundary so that incident edge stabilizers and groups in $\calp$
intersect $G_v$ into subgroups that map in finite or boundary subgroups of $\pi_1(\Sigma)$.
The results of \cite{GL-jsj} we use are stated for slender edge groups (i.e.~groups all of whose subgroups are finitely generated),
but as explained in \cite[Section~6.8]{GL-jsj}, they apply with edge groups in $\Asl$.

 Let now $U$ be a tree in $\calt$ whose number of orbits of edges is maximal.
 We claim that there exists a tree $V\in \calt'$ such that $U$ and $V$ are fully hyperbolic with respect to each other
(i.e.\ every edge stabilizer in one of these trees is hyperbolic in the other tree).

We would like to apply \cite[Proposition~6.28]{GL-jsj} but it does not formally apply: 
we know that trees in $\calt'$ are minuscule, but
we do not know that $\tilde K$ is totally flexible (in the sense of \cite[Definition~6.7]{GL-jsj}) 
with respect to trees in $\calt'$.
However, its proof applies because the assumption of total flexibility is only used
to ensure that each edge stabilizer of $U$ is hyperbolic in some tree in $\calt'$;
this is true because edge groups of $U$ are not $\calt$-universally elliptic by Assumption (3) of 
\ref{prop:QH}. This proves our claim.

We mention that $U$ and $V$ are also fully hyperbolic with respect to each other as $G$-trees
(this follows from Lemma \ref{lemma:acylindrical2} as in the proof of Lemma~\ref{lemma:edge-stabilizers-cyclic}).

\newcommand{\Vcut}{V_{\mathrm{cut}}}\newcommand{\VS}{V_{\cals}}\newcommand{\Cut}{\mathrm{Cut}}

We will now consider the asymmetric core $\calc_{\tilde{K}}(U,V)$ of $U$ and $V$, as defined by Fujiwara and Papasoglu in \cite{FP} (see \cite[Definition~6.16]{GL-jsj}).  
By definition, $\calc_{\tilde{K}}(U,V)$ is the subcomplex of $U\times V$
 whose fiber over each edge or vertex of $U$ is the minimal subtree of $V$ with respect
to the stabilizer of that edge or vertex.
The fiber over each edge $e$ is a line (by slenderness of $\tilde K_e$),
on which $\tilde{K}_e$ acts cocompactly;
in particular, there are only finitely many orbits of squares in $\calc_{\tilde{K}}(U,V)$.
Reversing the roles of $U$ and $V$, we get another subset $\check\calc_{\tilde K}(U,V)\subset U\times V$.
Notice that the core $\calc_{\tilde{K}}(U,V)$ of $U$ and $V$ viewed as $\tilde{K}$-trees is the same as the core $\calc_{G}(U,V)$ as $G$-trees; we simply denote it by $\calc(U,V)$ (likewise we will write $\check\calc(U,V)$). 
Indeed, given an edge $e\subseteq U$ with $G$-stabilizer $G_e$, the axis of $G_e$ in $V$ is $\tilde{K}_e$-invariant, so $\tilde{K}_e$ and $G_e$ have the same axis, and similarly, the minimal subtrees of $G_v$
and $\tilde K_v$ in $V$ agree for every vertex $v\in U$.

As $U$ and $V$ are minuscule and fully hyperbolic with respect to each other, the core $\calc(U,V)$ is symmetric, i.e.\ $\calc(U,V)=\check\calc(U,V)$ \cite[Proposition~6.19]{GL-jsj}, 
and $\calc(U,V)$ actually coincides with the core defined by the first author in \cite{Gui_core}, see \cite[Proposition~12.1]{Gui_core}. 
We denote by $V(\calc)$ the set of vertices of $\calc(U,V)$ and by $\Cut(\calc)\subset V(\calc)$ the set of cut vertices.
By \cite[Lemma~6.18]{GL-jsj}, each connected component $Z$ of $\calc(U,V)\setminus V(\calc)$ is a connected surface, 
tesselated by squares with missing vertices.

Let $R$ be the regular neighborhood of $U$ and $V$ as defined in \cite[Definition~6.23]{GL-jsj} following \cite{FP}. This is the bipartite $\tilde K$-tree defined as follows. 
The vertex set of $R$ is $\Vcut\dunion \VS$, having one vertex $v_x\in \Vcut$ for every cut vertex $x\in \Cut(\calc)$, 
one vertex  $v_Z\in\VS$ for every connected component $Z$ of 
$\calc(U,V)\setminus \Cut(\calc)$, with an edge joining $v_x$ to $v_Z$ whenever $x\in\overline{Z}$. 
We apply Proposition~6.25 of \cite{GL-jsj} with $\cala=\Asl$ which satisfies condition (SC) as proved above. 
First, we get that $R$ is an $(\Asl,\calp)$-tree.
If we view $U,V$ as $G$-trees, then their regular neighborhood as $G$-trees
is nothing but the tree $R$ where one restricts the action to $G\subseteq\tilde K$. 

We apply \cite[Proposition~6.25~(RN1)]{GL-jsj} to the vertex $v_Z$ of $R$ both viewed as a $\tilde K$-tree
and as a $G$-tree (the class of cyclic groups of $G$ satisfies condition (SC) 
with trivial fiber).
We claim that $G_{v_Z}$ is not slender and that $\tilde K_{v_Z}$ is not slender in trees in $\calt'$.
Indeed, if $G_{v_Z}$ is slender, then it is cyclic (as a subgroup of the free group $G$).
In particular $G_{v_Z}$ does not virtually surject onto $\bbZ^2$, so by \cite[Proposition~6.25~(RN1)]{GL-jsj}, 
$G_{v_Z}$ is QH, a contradiction.
If $\tilde K_{v_Z}$ is slender in trees in $\calt'$,
then since it is hyperbolic in $U$, Claim \ref{claim_slender} 
ensures that $G_{v_Z}$ is cyclic, a contradiction.

By \cite[Proposition~6.25~(RN1)]{GL-jsj}, we conclude that
$v_Z$ is a QH vertex relative to $\calp$ of $R$, both as a $\tilde K$-tree and as a $G$-tree.
Its fiber $F$, as a $\tilde{K}$-tree, belongs to $\Asl$. Its fiber as a $G$-tree is $F'=F\cap G$, so it is cyclic.
 In fact, the fiber $F'$ is trivial because 
 it is normal in $G_{v_Z}$. 
 Additionally, since $G$ is torsion-free and the fiber is trivial,
the underlying orbifold is a surface (we do not claim that this is the case when viewing $R$ as a $\tilde K$-tree). 

We first assume that $R$ is trivial (as a $\tilde K$-tree or equivalently as a $G$-tree), 
i.e.\ that $\Cut(\calc)=\es$ and
$Z=\calc(U,V)\setminus V(\calc)$ is connected, and $R$ is reduced to a single vertex $v_Z$. 
Then by \cite[Proposition~6.25~(RN1)]{GL-jsj}, $(\tilde K,\calp)$ is QH with fiber $F\in\Asl$, and $(G,\calp)$ is QH with trivial fiber.
Since $\calp$ consists of maximal cyclic subgroups,
$(G,\calp)$ is QH with sockets (all sockets being improper).
We claim that $K$ is finite, i.e.\ that  $G$ has finite index in $\tilde{K}$.
Indeed, the index of $G$ in $\tilde{K}$ is bounded by the number of $G$-orbits of squares in $\calc(U,V)$: if there exists a square $C\in\calc(U,V)$ such that $kC=gC$ with $k\in \Tilde K$ and $g\in G$, then $k\m g$ fixes $C$, and since 
$Z=\calc(U,V)\setminus V(\calc)$ is a connected tesselated surface, this implies that $k\m g$ acts as the identity on the core.
Since the natural projection from $\calc(U,V)$ to $U$ is onto,  $k\m g$ acts as the identity on $U$
so  $k\m g=\id$ by Claim~\ref{cl_faithful}.

Now we do not assume that $R$ is trivial any more.
We are going to prove that the $\Zmax$-splitting $R_{\Zmax}$ of $G$ associated to $R$
is trivial (see Definition \ref{dfn_zmaxise} for a definition of $R_{\Zmax}$). 
We thus view $R$ as a $K$-invariant cyclic splitting of $G$ relative to $\calp$. Then $R_\Zmax$ is a $K$-invariant $\Zmax$-splitting of $G$, i.e.\ $R_\Zmax\in\calt$. By \cite[Proposition~6.25~(RN3)]{GL-jsj}, the tree $U$ is dual to a family of curves in QH-vertices of $R$. Hence $R_{\Zmax}$ is compatible with $U$.
Let $S=U\vee R_{\Zmax}$ be the smallest common refinement of $U$ and $R_{\Zmax}$. Then $S$ belongs to $\calt$. Since $U$ has the maximal possible number of orbits of edges among trees in $\calt$, it follows that $U=S$.
Assuming by contradiction that $R_{\Zmax}$ is not reduced to a point, let $e$ be an edge of $U=S$ not collapsed in $R_{\Zmax}$.
Since $U$ is fully hyperbolic with respect to $V$, the stabilizer of $e$ is hyperbolic in $V$.
On the other hand, edge stabilizers of $R$ (hence of $R_{\Zmax}$) are elliptic in $U$ and $V$
by \cite[Proposition~6.25 (RN2)]{GL-jsj}, a contradiction.

This shows that $R_{\Zmax}$ is trivial. Viewing it as a $G$-tree, 
this implies by Lemma \ref{lem_QH_Zmax} 
that $R/G$ is a socket graph of groups.

We now prove that $K$ is finite.
Let $v\in R$ be a QH vertex and $Z$ the corresponding component of $\calc(U,V)\setminus \Cut(\calc)$. 
By Proposition \ref{prop_suite_exacte},
up to replacing $K$ with a finite index subgroup, one has an extension
$$1\ra K\cap\mathrm{Tw} \ra K \xrightarrow{\rho} K_{|G_v} \ra 1$$
where  $\mathrm{Tw}$ is 
the group of twists,
and $K_{|G_v}$ is the image of $K$ in $\prod_{u\in R/F_N}\Out(G_u,\Inc_u)=\Out(G_v,\Inc_v)$ under the restriction map $\rho$. Since  $R/G$ is a socket graph of groups, its group of twists is trivial so it suffices to prove that $K_{|G_v}$ is finite.
The argument is similar to the one above.
Since $Z\setminus V(\calc)$ is a connected tesselated surface, the $\tilde K$-stabilizer $\tilde N$ of a square $C$ in $Z$ acts as the identity on $Z$.
By Claim \ref{cl_faithful}, for every element $\tilde \alpha\in \tilde N$, one has $\tilde \alpha_{|G_v}=\id_{G_v}$.
Since $\tilde K_v$ permutes the $G_v$-orbits of squares, for $\tilde\alpha$ in finite index subgroup $\tilde K_0\subset \tilde K_v$,
there exists $g\in G$ such that $\tilde\alpha C=gC$, so 
$\ad_g\m\tilde\alpha \in \tilde N$.
This shows that all elements in a finite index subgroup of $K$ act trivially on $G_v$ 
and concludes the proof of Proposition \ref{prop:QH}.
\end{proof}

\subsection{Characterization of invariant $\Zmax$-splittings}

Recall  that an outer automorphism $\alpha\in \Out(F_N)$ 
is said to act trivially on a subgroup $A\subseteq F_N$
if it has a representative $\tilde \alpha$ in $\Aut(F_N)$
such that $\tilde \alpha_{|A}=\id_A$.

\begin{prop}\label{prop:carac_zmax}
 Let $H\subseteq\IA$ be a subgroup with no invariant  non-trivial free splitting. Then $\UZ_H$ is compatible with every $H$-invariant $\Zmax$-splitting  of $F_N$.
  More precisely, 
  any $H$-invariant $\Zmax$-splitting of $F_N$ may be obtained from $\UZ_H$
  by blowing up each flexible 
  vertex $v$ into a (minimal, maybe trivial)  $\Zmax$-splitting $Y_v$ relative to $\Inc_v$,
  and by blowing up each vertex $v$  with cyclic stabilizer into a finite tree
  on which $G_v$ acts as the identity, and then by collapsing all the edges coming from $\UZ_H$.

  Conversely, let $S$ be any $\Zmax$-splitting of $F_N$ obtained by blowing up $\UZ_H$ in this way.
Then $S$ is $H$-invariant.
  More generally, $S$ is invariant under any outer automorphism $\alpha\in \IA$ that preserves $\UZ_H$
  and that acts trivially on each flexible vertex group of $\UZ_H$.
\end{prop}

\begin{rk}
  The fact that any $\alpha\in H$ acts trivially on stabilizers of flexible vertices of $\UZ_H$ is a consequence of Theorem \ref{thm_JSJ}.  See Proposition~\ref{prop:stab-zmax} below for the precise relationship between the group of all such outer automorphisms and the stabilizer of $\UZ_H$.
\end{rk}

 In order to prove Proposition~\ref{prop:carac_zmax}, we will use the following fact.

\begin{lemma}\label{lem_sockets}
Let $(A,\calp)$ be QH with sockets. Let $\Lambda$ be a socket decomposition of $(A,\calp)$ with no M\"obius socket.
Let $Y_A$ be a $\Zmax$-splitting of $(A,\calp)$.

Then $Y_A$ is dual to a family of pairwise disjoint non-boundary parallel $2$-sided simple closed curves on the underlying surface  of $\Lambda$, and every group in $\calp$ fixes a unique point in $Y_A$.
\end{lemma}

\begin{proof}
  Let $v$ be the QH vertex of $\Lambda$ and
  denote by $\tilde\Lambda$ the Bass--Serre tree of $\Lambda$.

Let $Y_v\subset Y_A$ be the minimal $G_v$-invariant subtree of $Y_A$.
The groups in $\calp$ and the edge groups of $\Lambda$ are elliptic in $Y_A$ 
because $\Lambda$ is the cyclic JSJ decomposition of $(A,\calp)$ by \cite[Proposition 4.24]{DG_isomorphism}. Let $\Sigma$ be the underlying surface of the socket decomposition $\Lambda$. Since all improper sockets appear in $\calp$ by Definition~\ref{dfn_sockets}, 
the fundamental group of each boundary component of $\Sigma$ is elliptic in $Y_A$,
so $Y_v$ is dual to a collection of disjoint simple closed curves on $\Sigma$ not parallel to the boundary.
 Observe in particular that for each edge $e$ of  $\tilde \Lambda$ incident on $v$,
 $G_e$ and any finite index subgroup of $G_e$ fixes a unique point $p_e$ in $Y_v$.
 We also observe for future use that if $gv\neq v$ and  $e_1,e_2$ are two edges in $Y_v$ and $gY_v$
 respectively, then $G_{e_1}$ and $G_{e_2}$ do not commute: if they did, then
 $G_{e_1}\cap G_{e_2}$ would fix the segment $[v,gv]$ in $\tilde \Lambda$, hence some
 edge incident on $v$, contradicting the previous observation.

 Let $S$ be the blow up of  $\tilde \Lambda$ obtained by replacing $v$ by $Y_v$
 and gluing back each incident edge $e$ to the point $p_e\in Y_v$.

 The inclusion $Y_v\hookrightarrow Y_A$ extends to an equivariant map
 $f:S\ra Y_A$ which can be chosen linear on edges.
 Let $e=uv$ be an edge of  $\tilde \Lambda$ incident on $v$, and let $\tilde e=up_e$
 be the corresponding edge in $S$. Note that $G_u$ is a conjugate of some group in $\calp$. 
 Since edge stabilizers of $Y_A$ are $\Zmax$, $f$ maps the translates of $\tilde e$ by $G_u$
 to the same segment $[f(u),p_e]\subset Y_A$,
 and one may redefine $f$ to map $\tilde e$ to the point $p_e$.
 Thus $f$ factors through the tree $S'$ obtained from $S$ by collapsing all edges coming from  $\tilde\Lambda$. Since $f$ is isometric in restriction to each $Y_v$, and $f$ cannot fold
 two edges in distinct translates of $Y_v$ by the observation above, $S'\simeq Y_A$.
 Since no edge in $Y_v$ has a stabilizer conjugate to  some group in $\calp$, this concludes the proof.
 \end{proof}

\begin{proof}[Proof of Proposition \ref{prop:carac_zmax}]
  Let $S$ be any $H$-invariant $\Zmax$-splitting.
    
  To each vertex $v\in \UZ_H$ we associate in an equivariant way a subtree $Y_v\subset S$ as follows.
If $G_v$ is not cyclic, we define $Y_v\subset S$ as the minimal $G_v$-invariant subtree of $S$ (if $G_v$ fixes a point, this point is unique because $G_v$ is not cyclic).
 Since flexible vertices  of $\UZ_H$ are QH with sockets  (Theorem~\ref{thm_JSJ}), Lemma \ref{lem_sockets} applies to $G_v\actson Y_v$ for $v$ flexible  (with $\calp=\Inc_v$),
and in particular, the stabilizer $G_v\cap G_e$ of every edge  $e\subseteq Y_v$ is non-trivial and fixes no edge in $\UZ_H$.

If $G_v$ is cyclic, then we define $Y_v\subset S$ as the set of points fixed by $G_v$;
this is non-empty because edges of $\UZ_H$ incident on $v$ have stabilizer $G_v$ and are $\ZmaxS^H$-universally elliptic.

 We note that in all cases, if $e$ is an edge in $Y_v$, then the stabilizer of $e$ for the action $G_v\actson Y_v$ is equal to its stabilizer $G_e$ for the $G$-action on $S$, in other words $G_e\cap G_v=G_e$: this is because $G_v$ is root-closed and $G_v\cap G_e$ is infinite cyclic.

We claim that if $u,v\in \UZ_H$ are joined by an edge, then $Y_u\cap Y_v\neq \es$.
Since $\UZ_H$ is a tree of cylinders, we may assume that $G_u$
is cyclic and $G_v$ is not, and since edge stabilizers of $\UZ_H$ are maximal cyclic,
$G_u\subset G_v$. 
Since $Y_v$ is a $G_u$-invariant tree and $G_u$ is elliptic in $S$, $G_u$ fixes a point in $Y_v$ which proves the claim.

We next claim that if $v_1\neq v_2$ are distinct vertices of $\UZ_H$, then $Y_{v_1}\cap Y_{v_2}$ contains at most one point. Indeed, assuming by contradiction that there is an edge $e$ in $Y_{v_1}\cap Y_{v_2}$,
then 
$G_e\subset G_{v_1}\cap G_{v_2}$ so 
$G_e$ fixes all the edges in the segment joining $v_1$ to $v_2$ in $\UZ_H$.
As noticed above, this prevents $v_1$ and $v_2$ from being flexible vertices.
Since $Y_{v_1},Y_{v_2}$ are not reduced to a point, this implies that $v_1$ and $v_2$
are vertices with cyclic stabilizer. Since $G_{v_1}\cap G_{v_2}\neq \{1\}$,
and since  $\UZ_H$ is a tree of cylinders, this implies $v_1= v_2$, a contradiction.

Now we construct a blowup $\hat U$ of $\UZ_H$ as follows:
we start from the disjoint union of the trees $Y_v$ (where $v$ ranges over the set of vertices of $\UZ_H$) 
and the disjoint union of the edges of $\UZ_H$, and for each edge $e=uv$ of $\UZ_H$, we let $p$ be the unique point in $Y_u\cap Y_v$
and we glue the endpoints of $e$ to the copy of $p$ in $Y_u$ and $Y_v$ respectively.
We claim that collapsing all edges of $\hat U$ coming from $\UZ_H$ gives an $F_N$-tree $\overline U$ isomorphic to $S$.
Indeed, the inclusions $Y_v\subset S$ extend to an equivariant map $\hat U\ra S$ which maps each edge coming from $\UZ_H$ to a point, hence factors through $\overline U$.
Since $Y_u\cap Y_v$ contains at most one point for $u\neq v$, the induced map $\overline U\ra S$ is injective,
which proves the first assertion of Proposition \ref{prop:carac_zmax}.

We now prove the converse.
Let $\hat U$ be obtained by blowing up $\UZ_H$ by replacing in a $F_N$-equivariant way
each flexible vertex $v\in \UZ_H$ by a minimal action $G_v\actson Y_v$ relative to $\Inc_v$,
each vertex $u$ with cyclic stabilizer by a finite tree $Y_u$ on which $G_u$ acts as the identity,
and choosing for each  edge $\eps=uv$ attaching points $p_\eps\in Y_u$, $q_\eps\in Y_v$ fixed by $G_\eps$.
It suffices to check that $\hat U$ is $H$-invariant, and more generally, invariant under any  outer automorphism $\alpha\in \IA$
preserving $\UZ_H$ and acting trivially on flexible vertex groups of $\UZ_H$.

Let $\tilde\alpha$ be a preimage of $\alpha$ in $\Aut(G)$.
We have an $\tilde\alpha$-equivariant isometry $I_{\tilde\alpha}:\UZ_H\ra \UZ_H$, and we
now define a lift $J_{\tilde\alpha}:\hat U\ra \hat U$. 
For each edge $\eps\subset \hat U$ coming from $\UZ_H$, we define $J_{\tilde\alpha}(\eps)=I_{\tilde\alpha}(\eps)$.
We now define $J_{\tilde\alpha}$ on each tree $Y_v$.
If $v$ is rigid with non-cyclic stabilizer, $Y_v$ is a point and we define $J_{\tilde\alpha}(Y_v)=Y_{J_{\tilde\alpha}(v)}$.
If $v$ is flexible, we know that there exists $g_v\in F_N$ such that $\ad_{g_v\m}\circ \tilde\alpha$ restricts to the identity on $G_v$.
Since $\UZ_H$ is a tree of cylinders, 
there is no pair of edges incident on $v$ with the same stabilizer.
It follows that $g_v\m I_{\tilde\alpha}$ is the identity on the star of $v$.
We then define $J_{\tilde\alpha}(x)=g_v.x$ for all $x\in Y_v$, and this is compatible with the attaching points.
Now consider $u\in\UZ_H$ with cyclic stabilizer.
Recall that as $\alpha\in\IA$, Proposition~\ref{prop:ia-zmax} shows that $\tilde\alpha$ acts as the identity on $\UZ_H/F_N$. Therefore, there exists a (non-unique) $g_u\in F_N$ such that $\tilde\alpha$ coincides with $\ad_{g_u}$ in restriction to $G_u$,
and we define $J_{\tilde\alpha}(x)=g_u.x$ for all $x\in Y_u$. If $g'_u$ is another choice for $g_u$, then $g_u\m g'_u$ centralizes $G_u$ so lies in $G_u$,
so  $g_u.x=g'_u.x$ for all $x\in Y_u$ since $G_u$ acts as the identity on $Y_u$.
To check that this is compatible with the attaching data, consider $\eps=uv$ an edge of $\UZ_H$ with origin $u$. Then 
$I_{\tilde\alpha}(\eps)=g_\eps \eps$ for some $g_\eps\in F_N$ because $\tilde\alpha$ acts trivially on $\UZ_H/F_N$, so $g_\eps u=g_u u$ and
$g_\eps\m g_u\in G_u$ fixes the attaching point $p_\eps\in Y_u$.

This shows that $J_{\tilde\alpha}$ induces an $\tilde\alpha$-equivariant isometry $\hat U\ra \hat U$ and concludes the proof. 
\end{proof}

\subsection{Bounded chain condition}

\begin{prop}\label{prop:BCC_zmax}
There exists a bound, depending only on the rank of the free group $F_N$, on the length of
any chain 
  $$H_1\subseteq H_2\subseteq \dots\subseteq H_n$$
  of subgroups of $\IA$ such that $H_1$ (hence every $H_i$) fixes no  non-trivial free splitting of $F_N$, and such that
  $$\ZmaxS^{H_1}\supsetneqq \ZmaxS^{H_2}\supsetneqq \dots \supsetneqq \ZmaxS^{H_n}.$$
\end{prop}

\begin{proof}
  Denote $U_i=\UZ_{H_i}$. For $j\geq i$, $U_j$ is an $H_i$-invariant $\Zmax$-splitting,
so $U_i$ is compatible with $U_j$ by Proposition \ref{prop:carac_zmax}.
Let $\hat U=U_1\vee\dots\vee U_n$ be the smallest common refinement of the splittings $U_1,\dots,U_n$ given by \cite[Proposition~A.26]{GL-jsj}, it is again a $\Zmax$ splitting. 
Each $U_i$ can be obtained from $\hat U$ by collapsing some set of edges.
Since there is a bound on the number of edges of a $\Zmax$-splitting of $F_N$ (see for instance \cite{BF91}), this gives a bound on the number of edges of $\hat U$, and therefore a bound on the possible number of distinct trees $U_i$ (which are collapses of $\hat U$).
If $U_i=U_j$ and if the set of flexible vertices of $U_i$ agrees with the set of flexible vertices of $U_j$,
then $\ZmaxS^{H_i}=\ZmaxS^{H_j}$ by Proposition \ref{prop:carac_zmax}.
The proposition follows.
\end{proof}

\subsection{The stabilizer of a collection of $\Zmax$-splittings}

We conclude this section by applying the previous results to study the elementwise
stabilizer $\Gamma_\calc$ in $\IA$ of a collection $\calc$ of $\Zmax$-splittings, assuming that
this stabilizer does not preserve a non-trivial free splitting.
We denote by $\hat\calc=\ZmaxS^{\Gamma_\calc}$ the collection of all $\Gamma_\calc$-invariant
$\Zmax$-splittings of $F_N$. Notice that $\calc\subseteq\hat\calc$, and $\Gamma_\calc=\Gamma_{\hat\calc}$.

Proposition \ref{prop:BCC_zmax} may be reformulated as follows:
\begin{prop}\label{prop:chain-zmax2}
There exists a bound, depending only on the rank of the free group $F_N$, on the length of
any chain $$\calc_1\subseteq \dots\subseteq \calc_n$$ 
of collections of $\Zmax$-splittings of $F_N$ such that 
no $\Gamma_{\calc_i}$ fixes a non-trivial free splitting of $F_N$, and
$$\Gamma_{\calc_1}\supsetneq \dots \supsetneq \Gamma_{\calc_n}.$$
\end{prop}

\begin{proof}
  We apply Proposition~\ref{prop:BCC_zmax} with $H_i=\Gamma_{\calc_i}$. 
Denote by $\hat \calc_i=\ZmaxS^{H_i}$.
If $n$ is too large, then 
there exists $i_1<i_2$ such that $\hat\calc_{i_1}=\hat \calc_{i_2}$.
Since $\Gamma_{\calc_i}=\Gamma_{\hat\calc_i}$ for all $i\leq n$,
we get $\Gamma_{\calc_{i_1}}=\Gamma_{\calc_{i_2}}$, a contradiction.
\end{proof}

\begin{rk}
 As in Proposition~\ref{prop:chain-FS2}, the same conclusion holds if one replaces $\Gamma_{\calc_i}$ (i.e.\ the elementwise stabilizer of $\calc_i$ in $\IA$) by the elementwise stabilizer of $\calc_i$ in $\Out(F_N)$.
\end{rk}

As in Section~\ref{sec:collection-free}, we denote by $\Gamma_{\!\{\hat\calc\}}$ the setwise stabilizer of $\hat\calc$ in $\IA$,
and we note that $\Gamma_{\!\{\hat\calc\}}$ is the normalizer of $\Gamma_\calc$ in $\IA$.

The following result allows to recover $\Gamma_\calc$ and $\Gamma_{\!\{\hat\calc\}}$ from $\UZ_{\Gamma_\calc}$.

\begin{prop}\label{prop:stab-zmax}
Let $\calc$ be any  collection of $\Zmax$-splittings, and assume that $\Gamma_\calc$ fixes no  non-trivial free splitting.
Then
  \begin{itemize}
  \item $\Gamma_{\!\{\hat\calc\}}$ is the stabilizer of $\UZ_{\Gamma_\calc}$ in $\IA$;
  \item $\Gamma_\calc$ is the set of outer automorphisms $\alpha\in \IA$ that stabilize $\UZ_{\Gamma_\calc}$
and act trivially on each flexible vertex group.
  \end{itemize}
\end{prop}

\begin{rk}
  Note that $\alpha$ acts trivially on a flexible vertex stabilizer
  $G_v$ of $\UZ_{\Gamma_\calc}$ if and only if $\alpha$ has a lift
  acting as the identity on the star of $v$ (compare Definition
  \ref{dfn_V0-rigid}).  This is because two distinct edges incident on
  $v$ have distinct stabilizer.
\end{rk}

\begin{proof}
  Since $\Gamma_{\!\{\hat\calc\}}$ preserves the collection of all $\Gamma_\calc$-invariant $\Zmax$-splittings,
the uniqueness statement in Theorem \ref{theo:JSJ_zmax} ensures that
$\Gamma_{\!\{\hat\calc\}}$ preserves $\UZ_{\Gamma_\calc}$.
Conversely, if $\alpha\in\IA$ preserves $\UZ_{\Gamma_\calc}$, then $\alpha$ induces the identity on the quotient graph 
by Proposition~\ref{prop:ia-zmax}. In particular, it preserves the collection of flexible vertices,
so by Proposition~\ref{prop:carac_zmax}, it preserves the collection of all $\Gamma_\calc$-invariant $\Zmax$-splittings,
i.e.\ $\alpha\in\Gamma_{\!\{\hat\calc\}}$. This proves the first assertion.

Consider the group $H'\subset \IA$ consisting of outer automorphisms 
that preserve $\UZ_{\Gamma_\calc}$ and act trivially on each flexible vertex group. 
Theorem \ref{thm_JSJ} shows that $\Gamma_\calc\subseteq H'$.
Conversely, Proposition~\ref{prop:carac_zmax} says that every $\Gamma_\calc$-invariant splitting can be read from $\UZ_{\Gamma_\calc}$
and that every such splitting is $H'$-invariant. This implies that $H'\subseteq \Gamma_\calc$.
\end{proof}

\section{Pure subgroups of $\Out(F_N)$ and dynamical decomposition}\label{sec:collection-factors}

In this section, we consider subgroups $H\subseteq\IA$ that do not preserve any non-trivial free splittings, and for which there are several maximal $H$-invariant free factor systems. 
Our goal is to construct a non-trivial 
canonical $H$-invariant splitting encoding the interaction
between these free factor systems.
The construction does not work in the exceptional case where the free factor systems 
all come from a common QH-splitting. For this reason, the right setting for this section is to work with almost free factor systems, introduced in Definition~\ref{de:afs}.


 Throughout this section, we fix an integer $N\ge 2$.

\subsection{Almost free factor systems}
The goal of this section is to introduce almost free factor systems and basic related notions.
These objects also occur in the framework of weak attraction theory as
non-attracting subgroup systems, which are described in \cite[Theorem~F]{HM2}.
Many results concerning these objects appear in various forms in \cite{HM2}
but we give an independent treatment.

Recall from Definition \ref{dfn_cleanQH} that a vertex $v$ in a graph of groups is a clean QH vertex
if $G_v=\pi_1(\Sigma)$ and incident edge groups correspond bijectively to the fundamental groups of a subset of the boundary components of $\Sigma$. 

In order to introduce the notion of almost free factor systems, we start 
with the following definition which already appears in various forms in 
\cite{Sel}, \cite[Section~I.2.2]{HM2}, or \cite[Section 4.1]{Hor}.

\begin{de}[QH splittings and related notions]\label{dfn_QH_splitting}
A \emph{QH splitting} of $F_N$ is a cyclic splitting $S$ of $F_N$ 
such that the vertices of the quotient graph of groups $S/F_N$ are $v,u_1,\dots, u_n$ (maybe $n=0$) 
where $v$ is a clean QH vertex identified with the fundamental group of a surface $\Sigma$ with
at least one unused boundary component, and every edge joins $v$ to some $u_i$.

Its \emph{factor system} is $\hat\calf=\{[G_{u_1}],\dots,[G_{u_n}],[\grp{b_1}],\dots,[\grp{b_r}]\}$ where $[\grp{b_1}],\dots,[\grp{b_r}]$ are the conjugacy classes of the fundamental groups of unused boundary components.

The conjugates of the groups $\grp{b_1},\dots,\grp{b_r}$ are called the \emph{unused boundary subgroups}.

An \emph{extracted free factor system} is a collection of the form $\HF\setminus \{[\grp{b_i}]\}$
for some $i\leq r$.
\end{de}

\begin{figure}[ht]
  \centering
  \includegraphics{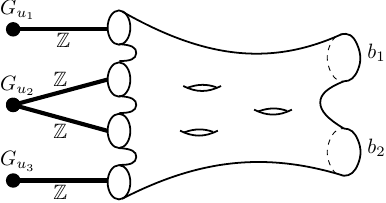}
  \caption{A QH splitting with $2$ unused boundary components. Its factor system is $\HF=\{[G_{u_1}],[G_{u_2}],[G_{u_3}],[\grp{b_1}],
    [\grp{b_2}]\}$. The two extracted free factor systems are
    $\{[G_{u_1}],[G_{u_2}],[G_{u_3}],[\grp{b_1}]\}$ and $\{[G_{u_1}],[G_{u_2}],[G_{u_3}],  [\grp{b_2}]\}$.
  }
\end{figure}

The quotient graph $S/F_N$ is not assumed to be a tree, there may be several edges joining $u_i$ to $v$.
Note that  $F_N$ is one-ended relative to $\hat \calf$ because the fundamental group of a surface
  is one-ended relative to its boundary components.

  \begin{rk}\label{rk:afs-fs}
Each  extracted free factor system is indeed a free factor system of $F_N$.   
This is shown by considering a maximal system of disjoint pairwise non-isotopic properly embedded arcs
with both endpoints on the boundary component of $\Sigma$ whose fundamental group is $\langle b_i\rangle$: such an arc system yields a free splitting of $F_N$ in which the elliptic subgroups are precisely given by  the collection $\HF\setminus\{[\langle b_i\rangle]\}$. 

In particular, if the number $r$ of unused boundary components is at least $2$, all the factors in $\HF$ are free factors of $F_N$
(although the factor system $\hat{\calf}$, as a collection, is never a free factor system of $F_N$, see Lemma~\ref{lem_HF_JSJ} and Remark~\ref{rk:hf} below).
When $r=1$, it may happen that the unused boundary subgroup is not a free factor of $F_N$.
\end{rk}

\begin{lemma}\label{lem_HF_JSJ}
 Let $S$ be a QH splitting of $F_N$, and $\HF$ be its factor system. Then $F_N$ is one-ended
  relative to $\HF$, and $S$ is a JSJ decomposition of $F_N$ relative to $\HF$.
\end{lemma}

\begin{proof}
    One-endedness follows from the fact that $\pi_1(\Sigma)$ is one-ended relative to its boundary subgroups.
    The splitting $S$ is universally
    elliptic because each edge group is conjugate into a group in $\HF$, hence
    elliptic in any cyclic splitting of $F_N$ relative to $\HF$.
    Maximality follows from the fact that any cyclic splitting of $\pi_1(\Sigma)$ relative to its boundary components
    is dual to a multicurve on $\Sigma$, and each non-peripheral curve on $\Sigma$ is intersected non-trivially by another one.
    This shows that $\pi_1(\Sigma)$ is elliptic in any JSJ decomposition (see also \cite[Corollary~5.30]{GL-jsj}) so $S$ itself is a JSJ decomposition.    
\end{proof}

\begin{rk}\label{rk:hf}
  One-endedness implies that $\HF$ cannot a free factor system.
  The lemma also implies that $\HF$ determines $S$ uniquely up to deformation,
  but one can easily deduce that it actually determines $S$ up to equivariant isomorphism:
  the tree of cylinders $S_c$ of $S$ is determined by $\HF$ up to equivariant isomorphism and
  $S$ is obtained from $S_c$ by collapsing the edges of $S$ not adjacent to the QH vertex.
\end{rk}

The following corollary implies that boundary subgroups and extracted free factor systems are permuted by
the stabilizer of $\HF$.

\begin{cor}\label{cor:hf}
  The notions of boundary subgroups, unused boundary subgroups, and extracted free factor systems
  depend only of $\HF$.
\end{cor}

\begin{proof}
  All JSJ decompositions are in the same deformation space so have the same non-cyclic vertex groups \cite[Corollary~4.4]{GL_deformation}.
  The group $\pi_1(\Sigma)$ is thus characterized as the unique vertex group in any JSJ decomposition that is not $\HF$-peripheral.
  Boundary subgroups of $\pi_1(\Sigma)$ are characterized as the maximal cyclic subgroups of $\pi_1(\Sigma)$ which
  are $\HF$-peripheral.
  Unused boundary subgroups are characterized as those that do not fix an edge  in $S$, hence in any reduced JSJ decomposition of $F_N$ relative to $\HF$
  \cite[Proposition~4.6]{GL_deformation}.
  Extracted free factor systems are obtained from $\HF$ by removing any unused boundary subgroup.
\end{proof}

\begin{de}[Almost free factor systems]\label{de:afs}
  An \emph{almost free factor system} $\HF$ of $F_N$ is
  a collection of conjugacy classes of subgroups of $F_N$ which is either a free factor system or the factor system
   of a QH splitting of $F_N$.

  A free factor system $\calf$ is \emph{extracted from $\HF$} if $\HF$ is a free factor system and $\calf=\HF$ or
  if $\HF$ the factor system of a QH splitting $S$ of $F_N$ and $\calf$ is an extracted free factor system.
\end{de}

\begin{rk}\label{rk_afs_malnormal}
  Any \afs\ $\HF$ is a \emph{malnormal family}: for each $[P],[Q]\in\HF$  and every $g\in F_N$,
  if $P^g\cap Q\neq\{1\}$ then $Q=P^g$ and $g\in P$.
\end{rk}

\begin{lemma}\label{lemma:af}
  Let $H\subseteq\IA$ be a subgroup
  that does not preserve any non-trivial free splitting.  

  \begin{enumerate}\renewcommand{\theenumi}{(\arabic{enumi})}

  \item\label{it:af1} For any maximal $H$-invariant free factor system $\calf$, there exists a unique maximal $H$-invariant \afs\ $\HF$
  such that $\calf\sqsubseteq \HF$. Moreover $\calf$ is extracted from $\HF$.

\item\label{it:af2} If $\HF$ is any maximal $H$-invariant \afs, then any free
    factor system $\calf$ extracted from $\HF$ is a maximal
    $H$-invariant free factor system.
  
\item\label{it:af3} If $\calf$ and $\HF$ are as in \ref{it:af1} or \ref{it:af2},
then every cyclic group whose conjugacy class is $H$-invariant is $\HF$-peripheral, and $H$ contains an element which
  is fully irreducible relative to $\calf$, and atoroidal relative to $\HF$.  
\end{enumerate}
\end{lemma}

\begin{rk}
  Notice that as $H$ does not preserve any non-trivial free splitting, every $H$-invariant free factor system is non-sporadic. 
\end{rk}

\begin{proof}[Proof of Lemma \ref{lemma:af}]
  We first prove \ref{it:af1} and  \ref{it:af3} assuming that every $H$-periodic
  (or equivalently, $H$-invariant) conjugacy class is $\calf$-peripheral.
We claim that every $H$-invariant \afs\ $\HF$ such that $\calf\sqsubseteq\HF$ is a free factor system.
This claim will imply that $\calf=\HF$, so \ref{it:af1} holds
and \ref{it:af3} follows from \cite[Theorem~2]{GH}.

To prove the claim, assume towards a contradiction that $\hat\calf$ is not a free factor system.
Consider a QH splitting defining $\HF$, and let $\Sigma$ be the underlying surface.
By Corollary \ref{cor:hf}, 
every boundary subgroup of $\pi_1(\Sigma)$
is $H$-periodic, hence $\calf$-peripheral by assumption.
Since $\pi_1(\Sigma)$ is freely indecomposable relative to its boundary components,
$\pi_1(\Sigma)$ is $\calf$-peripheral hence $\HF$-peripheral, a contradiction.

We now consider the case where there exists a 
 maximal cyclic subgroup $\grp{b}$ which is not $\calf$-peripheral and whose conjugacy class is $H$-periodic. 
 We claim that such a conjugacy class is unique. Indeed, by \cite[Theorem~1]{GH}, there exists an element $\alpha\in H$ which is fully irreducible relative to $\calf$.
Proposition~\ref{prop:invariant-arational-tree} ensures that there exists an
$\alpha$-invariant tree $T$ which is arational relative to $\calf$ and in which $b$ is elliptic.
It follows that $T$ is an arational surface tree \cite{Rey,Hor}, and, up to conjugacy, $\grp{b}$ is the unique point stabilizer of $T$ which is not $\calf$-peripheral. In particular, $\calf\cup\{[\grp{b}]\}$ is an almost free factor system, and 
we have proved that $\alpha$ is fully irreducible relative to $\calf$ and atoroidal relative to $\calf\cup\{[\grp{b}]\}$.

We now check that $\HF=\calf\cup\{[\grp{b}]\}$ is the unique maximal $H$-invariant almost free factor system 
relative to which $\calf$ is peripheral.
This will conclude the proof of Assertions~\ref{it:af1} and~\ref{it:af3}.
So consider $\HF'$ a maximal $H$-invariant \afs\ with  $\calf\sqsubseteq\hat\calf'$ and let us prove that $\HF=\HF'$.
Consider a QH splitting defining $\HF'$, and let $\Sigma'$ be the underlying surface.
By Corollary~\ref{cor:hf}, 
the boundary subgroups of $\pi_1(\Sigma')$ are $H$-periodic.
As above, at least one of these boundary subgroups 
is not $\calf$-peripheral,  so it is conjugate to a subgroup of $\grp{b}$,
so $\calf\cup\{[\grp{b}]\}$ is $\HF'$-peripheral.

If $\grp{b}$ is an unused boundary subgroup of $\HF'$, then write $\HF'=\calf'\cup\{[\grp{b}]\}$
where $\calf'$ is the corresponding extracted free factor system.
Since $\calf\sqsubseteq\HF'$  and $b$ is not $\calf$-peripheral, it follows that $\calf\sqsubseteq \calf'$.
By maximality of $\calf$, it follows that $\calf=\calf'$, so $\HF=\HF'$ and we are done.

If $\grp{b}$ is not an unused boundary subgroup of $\HF'$, we argue towards a contradiction.
The QH splitting $S'$ associated to $\HF'$
has an edge $e$ whose stabilizer is commensurable with $\grp{b}$. Let $U$ be the splitting of $F_N$ obtained from $S'$
by collapsing every edge not in the orbit of $e$.
Denote by $\Sigma$ the surface underlying the QH splitting $S$ associated to $\HF$, so that $\pi_1(\Sigma)$
is identified with a vertex stabilizer  of $S$.
Since $\HF\sqsubseteq\HF'$, the boundary subgroups of $\pi_1(\Sigma)$ are elliptic in $S'$ hence in $U$.
If $\pi_1(\Sigma)$ does not fix a point in $U$, then the action of $\pi_1(\Sigma)$ on
its minimal subtree in $U$ is dual to a collection of simple closed curves in $\Sigma$
whose stabilizer is conjugate to a power of $b$. But since $\grp{b}$ is a boundary subgroup of $\pi_1(\Sigma)$,
there is no such curve so $\pi_1(\Sigma)$ fixes a point in $U$.
This implies that there is an equivariant map $f:S\ra U$, but this map has to collapse every edge in $S$
because no edge stabilizer of $S$ is commensurable with $\grp{b}$.
It follows that $f$ is constant, a contradiction.

This concludes the proof of Assertions \ref{it:af1} and \ref{it:af3} in all cases.

 Assertion \ref{it:af2} is trivial if $\HF$ is a free factor system, so write
  $\HF=\calf\cup\{[\grp{b}]\}$ for some unused boundary subgroup $\grp{b}$ of $\HF$.
 The free factor system $\calf$ is $H$-invariant because $H\subseteq\IA$  (Theorem \ref{theo:ia}), and it suffices to show that $\calf$ is maximal.
 Assume that $\calf\sqsubseteq\calf'$ for some maximal $H$-invariant free factor system $\calf'$.
  By Assertion~\ref{it:af1}, there exists a unique maximal $H$-invariant \afs\ $\HF'$ such that $\calf'\sqsubseteq\HF'$.
Note that $\grp{b}$ cannot be $\calf'$-peripheral because $F_N$ is freely indecomposable relative to $\hat\calf$; but as $\grp{b}$ is $H$-invariant, Assertion~\ref{it:af3} implies that $\langle b\rangle$ is $\HF'$-peripheral. Therefore $\HF'=\calf'\cup\{[\grp{b}]\}$.
Hence $\hat\calf\sqsubseteq\HF'$,
and by maximality of $\hat\calf$ it follows that $\hat\calf=\HF'$, 
so $\calf=\calf'$.
\end{proof}

\begin{lemma}\label{lem_deborde}
Let $H\subseteq \IA$ be a subgroup that does not preserve any non-trivial free splitting,  and $\HF$ be a maximal $H$-invariant \afs. 

Then  there is a maximal $H$-invariant \afs\ distinct from $\HF$ if and only if
there exists a proper free factor $A$ which is not $\HF$-peripheral and
whose conjugacy class is $H$-invariant.
\end{lemma}

\begin{proof}
The `if' direction follows from the fact that every proper free factor whose conjugacy class is $H$-invariant is contained in a maximal $H$-invariant almost free factor system. For the `only if' direction, let $\HF'$ be a maximal $H$-invariant \afs{} distinct from $\HF$.
If $\HF'$ is a free factor system, then it contains a conjugacy class of free factor $A$ which is not $\HF$-peripheral, and the conjugacy class of $A$ is $H$-invariant (as $H\subseteq\IA$). So we assume that $\HF'$ is not a free factor system,
 and write $\HF'=\calf'\cup \{[\grp{b}]\}$
where $\calf'$ is a free factor system extracted from $\HF'$.
By Lemma~\ref{lemma:af}\ref{it:af2}, $\calf'$ is a maximal $H$-invariant free factor system.
If some free factor in $\calf'$ is not $\HF$-peripheral, then we are done.
Otherwise, we have $\calf'\sqsubseteq\HF$. But by 
Lemma~\ref{lemma:af}\ref{it:af3}, the conjugacy class of 
$\grp{b}$ is also $\HF$-peripheral, so $\HF'\sqsubseteq\HF$, and by maximality $\HF'=\HF$, a contradiction.
\end{proof}

\paragraph{Restricting an \afs\ to a subgroup.}

 Recall from Section~\ref{sec:factor-systems} that if $A$ is a subgroup of $F_N$, and $\calq$ a collection of conjugacy classes of subgroups of $F_N$,
we define $\calq_{|A}$ as the set of all $A$-conjugacy classes of non-trivial groups of the form  $Q\cap A$, where $Q\subseteq F_N$ is a subgroup whose conjugacy class belongs to $\calq$.

If $\calf$ is a free factor system of $F_N$, and $A$ is any subgroup, not $\calf$-peripheral, then
 the Kurosh lemma says that
$\calf_{|A}$ is a free factor system of $A$.  Our next lemma describes what happens for the restriction of an \afs.

Recall that a 
\emph{Grushko factor} of a group $G$ relative to a collection of subgroups $\calp$
is a vertex stabilizer in any Grushko $(G,\calp)$-splitting: with our usual conventions, if $G=G_1*\dots *G_p*F_r$ is a Grushko decomposition of $(G,\calp)$,
then its Grushko factors are the groups conjugate to $G_i$ for some $i\leq p$.
Note that $G_i$ is one-ended relative to $\calp_{|G_i}$.

\begin{lemma}\label{lem_restriction}
  Let $\HF$ be an \afs\ of $F_N$, and $A\subseteq F_N$ be a finitely generated malnormal subgroup
  which is not $\HF$-peripheral. Then one of the following holds: 
\begin{enumerate}
    \item[(1)] $\HF_{|A}$ is an \afs\ of $A$, 
\item[(2)] or the Grushko decomposition of $A$ relative to $\HF_{|A}$
  is non-trivial and there exists a Grushko factor $A_0$ of  $A$ relative to $\HF_{|A}$
  which is not $\HF_{|A}$-peripheral.
\end{enumerate}
\end{lemma}

\begin{proof}
If $\HF$ is a free factor system, then since $A$ is not $\HF$-peripheral,
$\HF_{|A}$ is a free factor system of $A$ by Kurosh theorem.
We may thus assume that $\HF$ is the factor system of a QH splitting $S$ of $F_N$ with QH vertex $v\in S$, with underlying surface $\Sigma$.

Assume first that the Grushko decomposition of $A$ relative to $\HF_{|A}$ is non-trivial.
If (2) does not hold, then all its factors are $\HF_{|A}$-peripheral,
so $\HF_{|A}$ is a free factor system of $A$, i.e.\ (1) holds.

We thus assume that $A$ is freely indecomposable relative to $\HF_{|A}$ and prove that $\HF_{|A}$ is an \afs. 

We first treat the case where $A$ is contained in the surface group $G_v=\pi_1(\Sigma)$.
If $A$ has finite index in $G_v$, then being malnormal, this implies that $A=G_v$ and $\HF_{|A}$ consists
in the conjugacy classes of fundamental groups of boundary components of $\Sigma$, so $\HF_{|A}$ is an almost
free factor system of $A$.  We now assume that $A$ has infinite index in $G_v$.
Since $A$ is not  $\HF$-peripheral, it is not conjugate into a boundary subgroup of $G_v$,
so 
 $A$ has a non-trivial free splitting relative to $\HF_{|A}$
(see for instance \cite[Lemma~3.11]{Perin} for this well known fact).

We may now assume that $A$ does not fix any point in $S$. Indeed, the case where $A$ fixes a point in the orbit of $v$ was
treated above, and if $A$ fixes another vertex, then $A$ is $\HF$-peripheral.

  Let $V_p\dunion V_s$ be the bipartition of vertices of $S$ where $V_s=F_N.\{v\}$.
  Let $S_A\subset S$ be the minimal $A$-invariant subtree. Since $A$ is freely indecomposable relative to  $\HF_{|A}$,
  for every edge $e$ of $S_A$,  the intersection $G_e\cap A$ is non-trivial
  (in fact maximally cyclic  in $F_N$ by malnormality of $A$).
  In particular, for every vertex $w\in V_s\cap S_A$, the intersection  $G_w\cap A$ is not cyclic since by minimality,
  there are at least two edges in $S_A$ incident on $w$.
  Therefore, if $G_w\cap A$ has infinite index in $G_w$, the argument above
  yields a non-trivial 
  free splitting of $G_w\cap A$ relative to the incident edge groups, hence a
  non-trivial free splitting of $A$ relative to $\HF_{|A}$, a contradiction.

  We can therefore assume that for every vertex $w\in V_s\cap S_A$, the intersection $G_w\cap A$ has finite index in $G_w$. As $A$ is malnormal, this implies that $A\cap G_w=G_w$.
  It also follows that all vertices in $V_s\cap S_A$ are in the same $A$-orbit:
    indeed, if $w,w'\in V_s\cap S_A$, then $w=gw'$ for some $g\in F_N$, so $G_w\subseteq A\cap A^{g\m}$
    so $g\in A$ by malnormality.

    We now check that $S_A$ is
    a QH splitting of $A$.
    Up to conjugating $A$, we may assume that $v\in S_A$, and
     we claim that
     $v$ is a clean QH vertex in $S_A$. Indeed, $G_v=\pi_1(\Sigma)$, and the stabilizers of incident
      edges in $S_A$ are subgroups of the boundary components of $\Sigma$. Moreover, since $A\cap G_v=G_v$, the $A$-stabilizer 
    of any edge in $S_A$ incident on $v$ agrees with its $F_N$-stabilizer. It follows that $v$ is clean QH as a vertex of $S_A$.
    Since $\Sigma$ has an unused boundary component in $S$, the same is true in $S_A$, so $S_A$ is a QH splitting of $A$.
    
    There remains to check that $\HF_{|A}$ coincides with the set of all conjugacy classes of $A$-stabilizers of vertices in $V_p\cap S_A$
    together with the conjugacy classes of fundamental groups of unused  (in $S_A$) boundary components of $\Sigma$.
    Each group $Q$ in $\HF$ is either of the form $Q=G_p$ for some vertex $p\in V_p$,
    or $Q$ is conjugate to the fundamental group $G_b$ of a boundary component $b$ of $\Sigma$ unused in $S$.
    If $Q=G_p$ with $p\in S_A$, then $Q\cap A$ is the stabilizer of the vertex $p$ for the action  $A\actson S$.
    If $p\notin S_A$, then either $G_p\cap A$ is trivial, or $p$ is adjacent to a vertex  $w\in V_s\cap S_A$ 
    (this is because any segment in $S$ with non-trivial stabilizer has length at most $2$ and contains at most one vertex in $V_p$).
    In this case, up to conjugation we may assume that $v=w$; then $G_p\cap A=G_p\cap G_v$ is conjugate to the fundamental 
    group of a boundary component of $\Sigma$ that is unused in $S_A$.
    If $Q$ is conjugate to the stabilizer $G_b$ of an unused  (in $S$) boundary component of $\Sigma$, then $Q$ fixes a unique vertex
    $w\in S$ (and $w\in V_s$). If $w\in S_A$, then $Q\cap A=Q$ because $G_w\subset A$, so $Q$ is conjugate in $A$
    to the stabilizer of an unused boundary component of $\Sigma$.
    Finally, if $w\notin S_A$, then $Q\cap A=\{1\}$. Otherwise,  $Q\cap A$
    has finite index in the cyclic group $Q$.
    But $Q\cap A$ fixes the arc joining $w$ to $S_A$, hence so does $Q$ since edge stabilizers are maximal cyclic,
    contradicting that $Q$ fixes no edge in $S$.
    This concludes the proof.
\end{proof}

\subsection{The dynamical decomposition: statement and examples}

Given a subgroup $H\subseteq\IA$,
the goal of this section is to construct a \emph{dynamical decomposition} of the free group $F_N$ for $H$. As was informally explained in Section~\ref{sec_intro_dd} of the introduction, this is an analogue of Ivanov's natural $H$-invariant decomposition of a surface into \emph{active} and \emph{inactive} subsurfaces, when $H$ is a subgroup of the mapping class group.

\begin{de}\label{dfn_pure}
We say that a subgroup $H\subseteq\IA$ is \emph{pure} if
  there is a unique maximal $H$-invariant \afs, but no $H$-invariant  non-trivial free splitting.  
\end{de}

We insist that in this definition, the unique maximal $H$-invariant \afs\ is allowed to be empty, 
which happens if $H$ contains an atoroidal fully irreducible outer automorphism.

\begin{rk}
  One can naturally extend this definition to any subgroup of $\Out(F_N)$ as follows:
  define $H\subset\Out(F_N)$ as \emph{pure} if $H\cap \IA$ is pure.
  Equivalently, $H$ is pure if  there is
 a unique maximal $H$-periodic \afs, and no $H$-periodic non-trivial free splitting.
 The equivalence 
 follows from the fact that any $H$-periodic \afs\ (or free splitting) is invariant under  $H\cap\IA$.
\end{rk}

\begin{de}\label{dfn_UP}
Let $H\subseteq\IA$ be a subgroup.  A subgroup $P\subseteq F_N$ is \emph{universally peripheral}  with respect to $H$ if it is $\HF$-peripheral for any maximal $H$-invariant \afs\ $\HF$.

  We denote by $\calp_H$ the set of all subgroups of $F_N$ that are universally peripheral with respect to $H$, and by $\calp_H^{\max}$ the subset of maximal
  universally peripheral subgroups.
\end{de}

We note that $\calp_H$ is stable under taking subgroups and contains the trivial group.
Since any intersection of free factors is a free factor,
every group $P\in\calp_H^{\max}$ is either a free factor or a maximal
cyclic subgroup of $F_N$.
If $H$ is not fully irreducible, then $\calp_H$ contains a non-trivial
proper free factor of $F_N$ (one can additionally show, using the chain condition on free factors, that $\calp_H^{\max}$ is a finite collection of conjugacy classes of free factors and maximal cyclic subgroups).
The case where $H$ 
is fully irreducible and atoroidal (i.e.\ does not preserve the conjugacy class of any proper free factor or of any cyclic group)
is a bit peculiar:
there is no 
non-empty $H$-invariant \afs\ (i.e.\ $\hat \calf=\es$ is the only one) so $\calp_H=\calp_H^{\max}$ 
 is reduced to the trivial subgroup.
But anyway the dynamical decomposition is clearly trivial in this case, so we may as well exclude this case from now on.

\begin{theo}[Dynamical decomposition]\label{thm_dynamic_dec}
Let $H\subseteq\IA$ be a subgroup that does not preserve any non-trivial free splitting of $F_N$.

Then there exists a unique $H$-invariant  splitting $U^d_H$ of $F_N$ relative to $\calp_H$  whose vertex set $V$ has a bipartition $V=V_p\dunion V_a$
(for peripheral vs.\ active vertices) with the following properties:
\renewcommand{\theenumi}{(\arabic{enumi})} 
\renewcommand{\labelenumi}{\theenumi}
\begin{enumerate}
\item \label{it_relP} for every $v\in V_p$, one has $G_v\in\calp_H^{\max}$, and $G_v\cap G_{v'}=\{1\}$ for distinct vertices $v,v'\in V_p$;
\item \label{it_Vapur} for every $v\in V_a$, the group $G_v$ is a free factor of $F_N$  whose conjugacy class is $H$-invariant,
  the restriction of $H$ to $G_v$ is pure,
  and  $(\calp_H^{\max})_{|G_v}$ is the unique maximal $H$-invariant \afs\ in $G_v$;
  moreover, the collection of incident edge groups is a non-sporadic free factor system of $G_v$;
\item \label{it_Vamin}  for every free factor $A$ whose conjugacy class is $H$-invariant and which is not universally peripheral, there exists $v\in V_a$ such that $G_v\subseteq A$;
\item \label{it_HF} for every vertex $v\in V_a/F_N$, there exists a unique
  maximal $H$-invariant \afs\ $\HF_v$ such that $G_v$ is not $\HF_v$-peripheral.
  Moreover the map $v\mapsto \HF_v$ is a bijection from  $V_a/F_N$ to the set of 
 maximal $H$-invariant \afs{}s.
\end{enumerate}
\end{theo}

As usual, the uniqueness of $\Ud_H$ is up to $F_N$-equivariant isomorphism.

The above theorem is already new in the case where $H=\grp{\Phi}$ is a cyclic subgroup of $\IA$, in which case it yields a new conjugacy invariant for $\Phi$. In Section~\ref{sec_rel_other_works} below, we will explain the connection between active vertices and attracting laminations of $\Phi$ and the relation to other concepts from the literature. 

\begin{rk}
  It is \emph{a priori} not obvious that there are only finitely many
  maximal $H$-invariant almost free factor systems (under the assumption that $H$ preserves no non-trivial free splitting),
  but this is a consequence of Assertion~\ref{it_HF}. In fact, Theorem~\ref{thm_dynamic_dec} can be used to derive a bound on their numbers, as will be done in Corollary~\ref{coro:finitude-facteurs} below. We insist that the assumption that $H$ preserves no free splitting is crucial here,
  see Example \ref{ex:arcs} for a situation where there are infinitely many invariant splittings  dual to
  arcs on a surface, hence infinitely many sporadic $H$-invariant free factor systems.
\end{rk}

\begin{rk}\label{rk_edge_fg}
  Since $H$ does not preserve any non-trivial free splitting, edge stabilizers of $\Ud_H$ are non-trivial.
  They are finitely generated because vertex stabilizers are either cyclic or free factors. 
  By Assertion~\ref{it_relP}, they are universally peripheral with respect to $H$.
  We will see that edge stabilizers are actually free factors of $F_N$ (see Proposition \ref{prop_additional}).
\end{rk}

\begin{rk}\label{rk_binonsporadic}
  If $H$ is pure, $\Ud_H$ is the trivial splitting, reduced to a vertex in $V_a$.
  If $H$ is not pure, then $\Ud_H$ is non-trivial, and there are at least two orbits of active vertices.
  One can refine each of these vertices into a non-sporadic free splitting.
\end{rk}

\begin{figure}
\centering
\includegraphics{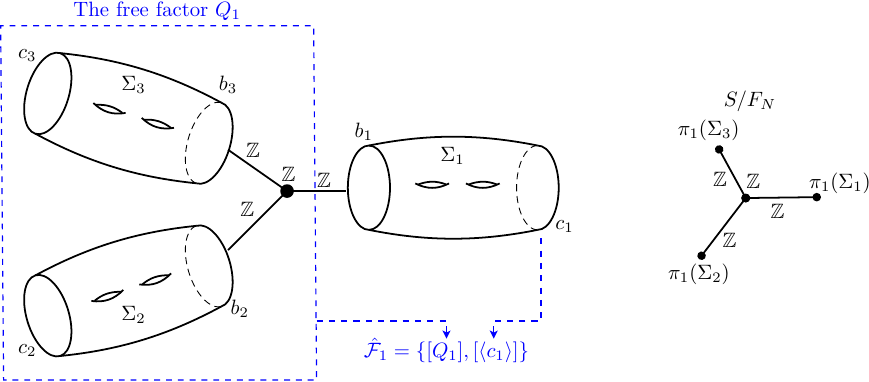}
\caption{The splitting in Example~\ref{ex:3-surfaces}.}
\label{fig:3-surfaces}
\end{figure}

\begin{ex}\label{ex:3-surfaces}
We give an example (see Figure~\ref{fig:3-surfaces}) illustrating the dynamical decomposition of Theorem \ref{thm_dynamic_dec}.

Let $\Sigma_1,\Sigma_2,\Sigma_3$ be three connected orientable surfaces of positive genus, each of which has exactly two boundary components. Let $F_N$ be the free group obtained by amalgamating the fundamental groups of these surfaces over  the fundamental group of one of their boundary components, and let $S$ be the splitting of $F_N$ represented in Figure~\ref{fig:3-surfaces}. Let $H$ be the subgroup of $\IA$ induced by the homeomorphisms of the three surfaces whose restriction to the boundary is the identity. 
 
Let us describe the $H$-invariant conjugacy classes of proper free factors.  Note that $\pi_1(\Sigma_i)$ is a free factor of $F_N$
as can be seen by
drawing arcs on $\Sigma_{i-1}$ and $\Sigma_{i+1}$ with endpoints in their unused boundary components. Similarly, so is 
$\langle c_i\rangle $, the fundamental group of the unused  boundary curve in $\Sigma_i$, and $\langle c \rangle$ the central vertex group.
Finally, the subgroup 
$Q_i=\pi_1(\Sigma_{i-1}\cup \Sigma_{i+1})$ is also a free factor of $F_N$. The conjugacy classes of all these free factors are clearly $H$-invariant.
Conversely, one may use the argument in Lemma \ref{lemma:minimal-debording-factor} below
and the fact that $H$ contains elements whose
restriction to $\pi_1(\Sigma_i)$ is pseudo-Anosov
to prove that these are these are the only invariant free factors.
It follows in particular that $H$ does not preserve any  non-trivial free splitting.

The almost free factor system  $\hat\calf_i=\{[Q_i],[\langle c_i\rangle]\}$
is $H$-invariant and maximal.
On the other hand, $\pi_1(\Sigma_i)$ is a free factor
which is not $\hat\calf_i$-peripheral  and whose conjugacy class is $H$-invariant, and it is minimal for this property.
Minimal non-peripheral free factors will play an important role in the proof
and appear as stabilizers of active vertices by Assertion \ref{it_Vamin}.

One can check that the dynamical decomposition $\Ud_H$ of Theorem \ref{thm_dynamic_dec} is the splitting dual to the decomposition shown in Figure~\ref{fig:3-surfaces}
where $V_p/F_N$ is the central vertex, and $V_a/F_N$
corresponds to the three surfaces.

In this example, active vertices correspond to surfaces, but this is not a general fact. For instance, one may replace any of the surfaces of the example by a free group
$F=\grp{x_1,\dots,x_p}$, and attach it to the central vertex along $x_1$, 
and in the above construction, replace the homeomorphisms of the surface by the automorphisms of $F$ fixing $x_1$.
\end{ex}

\subsection{A minimal non-peripheral invariant free factor}

The main step in the proof of the existence part of Theorem \ref{thm_dynamic_dec} is the following result.
We have been informed by a referee that this lemma can also be proved using results from \cite{HM2} in the framework of weak attraction theory.

\begin{lemma}\label{lemma:minimal-debording-factor}
  Let $H\subseteq\IA$ be a subgroup that does not preserve any non-trivial free splitting. Let $\hat\calf$ be a maximal $H$-invariant almost free factor system.
 
Among all conjugacy classes of non-trivial free factors of $F_N$ (including $F_N$ itself) which are $H$-invariant and not $\hat\calf$-peripheral, 
there is a unique minimal one, that we denote by $\AHF$.

In addition, $\AHF$ is non-abelian and
there is a  minimal 
$H$-invariant splitting $\THF$ 
of $F_N$ relative to $\hat\calf$ 
 coming with a bipartition of its vertex set as  $V^0\dunion V^1$ that satisfies the following properties.
\renewcommand{\theenumi}{(\arabic{enumi})} 
\renewcommand{\labelenumi}{\theenumi}
\begin{enumerate} 
\item \label{it0_relP} For every $v\in V^0$, one has $[G_v]\in\HF$  and $G_v$ is not an unused boundary subgroup of $\HF$.  
 Moreover $G_v\cap G_{v'}=\{1\}$ for distinct vertices $v,v'\in V^0$;
\item \label{it0_Va} $V^1$ consists of a single orbit of vertices, the stabilizer of every vertex in $V^1$ is conjugate to $\AHF$;
\item \label{it0_periph}  $\HF_{|\AHF}$ is a non-sporadic almost free factor system of $\AHF$,  and the collection of incident edge groups is a non-sporadic free factor system of $\AHF$.
\end{enumerate}
\end{lemma}

\begin{rk}
In this statement, we do not claim that $\AHF$ is a \emph{proper} free factor; in fact $\AHF=F_N$ if and only if
$H$ is pure (see Lemma~\ref{lem_deborde}). More generally, the defining property of $\AHF$ implies that the restriction of $H$ to $\AHF$ is pure, and its unique
maximal $H$-invariant maximal \afs\ is $\HF_{|\AHF}$.
\end{rk}

\begin{ex}
  Continuing with the notations from Example~\ref{ex:3-surfaces}, let $\hat\calf:=\hat\calf_1=\{[Q_1],[\langle c_1\rangle]\}$. Up to conjugation, among all $H$-invariant  free factors which are not $\hat\calf$-peripheral, the unique
  minimal one 
  is $A_{\HF_1}=\pi_1(\Sigma_1)$. The splitting $T_{\hat\calf_1}$ is the Bass--Serre tree of the decomposition $F_N=\pi_1(\Sigma_1)\ast_{\mathbb{Z}} Q_1$; the vertex with stabilizer $\pi_1(\Sigma_1)$ belongs to $V^1$, and the vertex with stabilizer $Q_1$ belongs to $V^0$.
Moreover,  $\HF_{|\pi_1(\Sigma_1)}=\{[\grp{b_1}],[\grp{c_1}]\}$. 
\end{ex}

In our proof of Lemma~\ref{lemma:minimal-debording-factor}, we will make use of the following fact.

\begin{lemma}\label{lemma:intersection-factors}
Let $H\subseteq\IA$. Let $A$ and $A'$ be two proper free factors of $F_N$ whose conjugacy classes are $H$-invariant.

Then for every $g\in F_N$,  $A\cap (A')^g$ is a free factor of $F_N$ whose conjugacy class is $H$-invariant. 
\end{lemma}

\begin{proof}
By the Kurosh subgroup theorem, 
there are only finitely many $A$-conjugacy classes (hence $F_N$-conjugacy classes) of 
subgroups of the form $A\cap (A')^g$ with $g\in F_N$, and these are free factors of $A$ 
(hence of $F_N$).
Since $[A]$ and $[A']$ are $H$-invariant, the group $H$ permutes the conjugacy classes $[A\cap (A')^g]$. As $H\subseteq\IA$, these are actually $H$-invariant.  
\end{proof}

\begin{proof}[Proof of Lemma~\ref{lemma:minimal-debording-factor}]
  Let $A,A'$ be two free factors of $F_N$ which are not $\HF$-peripheral, whose conjugacy classes are $H$-invariant 
 with $A$ of minimal rank. 
We aim to prove that $A$ is conjugate into $A'$.

Let $\calf$ be a free factor system extracted from  $\hat\calf$. By Lemma~\ref{lemma:af},
$\calf$ is a maximal $H$-invariant free factor system
(non-sporadic because $H$ does not fix any free splitting of $F_N$)
and $H$ contains an outer automorphism $\alpha$ which is fully irreducible relative to $\calf$ and atoroidal relative to $\hat{\calf}$.

Let $T$ be an $\alpha$-invariant arational tree (rel.\ $\calf$) on which $\alpha$ acts by homothety of dilatation factor $\lambda\neq 1$: this exists by Proposition~\ref{prop:invariant-arational-tree}. Notice that a subgroup is elliptic in $T$ if and only if it is $\hat{\calf}$-peripheral. 

Since $A$ is not elliptic in $T$, we can consider the minimal $A$-invariant subtree $T_A$ of $T$. Since $[A]$ is $\alpha$-invariant, we deduce that $T_A$ is $\tilde\alpha$-invariant for some preimage $\tilde\alpha\in\Aut(F_N)$ of $\alpha$. 
Since $\tilde \alpha$ acts on $T_A$ by homothety of dilatation $\lambda\neq 1$, $T_A$ has no simplicial edges,
so $A$ acts on $T_A$ with dense orbits.

Using an argument of Reynolds (as formulated in \cite[Corollary~11.9]{GH1}), we deduce that the collection $\{gT_A\}_{g\in F_N}$ is a transverse family of $T$, 
and more precisely, for every $g\in F_N\setminus A$, the intersection $gT_A\cap T_A$ contains at most one point (in particular $\Stab_{F_N}(T_A)=A$). As $T$ is arational, it is mixing \cite[Lemma~4.9]{Hor}, so $\{g T_A\}_{g\in F_N}$ is a transverse covering of $T$,  i.e.\  the subtree $T_A$ is closed and every segment of $T$ is covered by finitely many subtrees from this family. 
  
Likewise, the above argument also yields another transverse covering $\{gT_{A'}\}_{g\in F_N}$ of $T$. We can then form the refinement of the two transverse coverings, which consists of all nondegenerate subtrees of $T$ of the form $gT_A\cap g'T_{A'}$ with $g,g'\in F_N$. 
Since this forms a transverse covering of $T$, the stabilizer of each of these trees acts with dense orbits on it (see e.g.\ \cite[Lemma~5.2]{HW}) and in particular is not $\HF$-peripheral and not cyclic.

We observe that if $T_A\cap g T_{A'}$ is nondegenerate, then its stabilizer is equal to $A\cap (A')^g$. Indeed, it is clear that $A\cap (A')^g$ preserves both $T_A$ and $g T_{A'}$, and therefore preserves their intersection. Conversely, if $h\in F_N$ preserves $T_A\cap g T_{A'}$, then in particular $h T_A\cap T_A$ is nondegenerate, so $h\in A$, and similarly, $h\in  (A')^g$. 

Therefore, there exists $g\in F_N$ such that $A\cap (A')^g$ is not elliptic in $T$, hence not $\hat\calf$-peripheral.
By Lemma~\ref{lemma:intersection-factors}, its conjugacy class is $H$-invariant. By minimality of the rank of $A$, we have $A\subseteq (A')^g$, as desired. This concludes our proof of the existence and the uniqueness of $\AHF$, and of the fact that $\AHF$ is non-abelian.

We now construct the tree $\THF$. Our construction a priori depends on the choice of $\alpha\in H$  and of the $\mathbb{R}$-tree $T$, but 
we will check at the end of the proof that  in fact $\THF$ does not depend on these choices.
We keep the notations from above. Let $\THF$ be the skeleton (in the sense of \cite[Definition~4.8]{Gui04}) of the transverse covering $\{gT_A\}_{g\in F_N}$, i.e.\ $\THF$ is the bipartite simplicial tree having one vertex $v_Y$ for every subtree $Y$ of the form $gT_A$, one vertex $v_x$ for every point $x\in T$ that belongs to at least two subtrees from the family, with an edge between $v_x$ and $v_Y$ whenever $x\in Y$.  Minimality of $T_\HF$ follows from \cite[Lemma~4.9]{Gui04}. We denote by  $V^1$ the set of all vertices of the form $v_Y$, and by  $V^0$ the set of all vertices of the form $v_x$. Notice that $V^1$ consists of a single orbit of vertices, and the stabilizer of every vertex in $V^1$ is conjugate to $\AHF$, proving Assertion \ref{it0_Va}.
Edge stabilizers of $\THF$ are non-trivial because otherwise, this would produce an $(F_N,\calf)$-free factor acting non-discretely on $T$, contradicting arationality. 
In particular, the stabilizer of any vertex $v=v_x\in V^0$ is non-trivial, and since non-trivial point stabilizers of $T$ are precisely groups in $\HF$,  we deduce that $[G_v]\in\HF$. Since arc stabilizers of $T$ are trivial, it follows that the stabilizers of distinct vertices in
$V^0$ have trivial intersection.

To prove Assertion \ref{it0_relP}, there remains to prove that 
$G_v$ is not  an unused boundary subgroup $\langle b\rangle$ of $\HF$.
So assume otherwise that $G_v=\grp{b}$. Since there are at least two edges incident on $v$ in $\THF$, and since edge stabilizers are non-trivial,
there exists $k\geq 1$ such that $b^k$ is contained in two distinct conjugates of $A$, say  $b^k\in A^g\cap A^{g'}$.
By malnormality, $A^g=A^{g'}$, a contradiction.
This proves Assertion \ref{it0_relP}.

We now prove Assertion \ref{it0_periph}.
We first apply Lemma \ref{lem_restriction} to prove that $\HF_{|A}$ is an almost free factor system. If not, then there is
Grushko factor $A_0\subsetneq A$ of $A$ relative to $\HF_{|A}$ which is not $\HF_{|A}$-peripheral.
Since there are only finitely many conjugacy classes of such factors, $A_0$ is $H$-periodic hence $H$-invariant because $H\subseteq\IA$. This contradicts the choice of $A$. 
This proves that $\HF_{|A}$ is an \afs\ of $A$. If it was sporadic, then the action of $A$ on $T_A$ would be simplicial, a contradiction.

In order to prove the last part of Assertion~\ref{it0_periph},
 denote by $\calf_0\subseteq\HF$ the collection of all  conjugacy classes of stabilizers of $V^0$-vertices, and note that $(\calf_0)_{|A}$ coincides with the set of  $A$-conjugacy classes of incident edge groups.
 By Assertion \ref{it0_relP}, $\calf_0$ contains no unused boundary subgroup of $\HF$,
 so $\calf_0$ is a free factor system of $F_N$,  hence $(\calf_0)_{|A}$ is a free factor system of $A$. 
If $(\calf_0)_{|A}$ was sporadic, then as above, the action of $A$ on $T_A$ would be simplicial, a contradiction.

We  finally give an alternative description of the tree $T_{\HF}$ which shows that it does not depend on the choice of $\alpha$ and of the $\bbR$-tree $T$, and which also implies that it is $H$-invariant.
We note that since $T_{\hat\calf}$ is bipartite with non-trivial edge stabilizers, and since $G_v\cap G_{v'}=\{1\}$ for any two distinct vertices $v,v'\in V^0$, it follows that two vertices $v\in V^0$ and $v'\in V^1$ are adjacent if and only if $G_v\cap G_{v'}\neq\{1\}$.
So let $S$ be the  bipartite graph on the vertex set $V^0(S)\dunion V^1(S)$, where $V^0(S)$ is defined as the set
of subgroups of $F_N$ whose conjugacy class is in $\HF$, and $V^1(S)$  is the set of minimal non-$\HF$-peripheral free factors whose conjugacy class is $H$-invariant, and where two vertices  (one in $V^0(S)$ and one in $V^1(S)$) are adjacent if and only if the corresponding subgroups have non-trivial intersection. The action of $F_N$ by conjugation on its subgroups induces an action on $S$. 
The map  $f:v\in V(T_\HF)\mapsto G_v\in V(S)$ is injective, and preserves adjacency and non-adjacency.
We claim every vertex in $S\setminus f(T_\HF)$ is a terminal vertex of $S$ (in particular, $S$ is a tree).
Indeed, let $P$ be the subgroup  associated to a vertex  $v\in S\setminus f(T_\HF)$.
Then $[P]\in\HF$ (in other words $v\in V^0(S)$) and the  point $x\in T$ fixed by $P$ is not in $V^0(T_\HF)$.
This means that $x$ lies in a single translate $gT_A$ of $T_A$, so $P\subset A^g$.
Since arc stabilizers of $T$ are trivial, all other conjugates of $A$ intersect $P$ trivially.
This means that $v$ is adjacent to a unique vertex in $V^1(S)$, and the bipartite structure of $S$ implies that $v$ is a terminal vertex.
This shows that $T_\HF$ can be described as the complement of terminal vertices in $S$ and concludes the proof.
\end{proof}

The factors $\AHF$ have the following property.
 
  \begin{lemma}\label{lemma:A-periph}
    Let $H\subseteq\IA$ be a subgroup that does not preserve any non-trivial free splitting, and let $\HF$ and $\HF'$ be two distinct maximal $H$-invariant \afs{}s.
    
    Then $\AHF$ is $\hat\calf'$-peripheral, and it is not contained in an unused boundary subgroup of $\hat\calf'$.
  \end{lemma}

  \begin{proof}
    Since $\HF'\not\sqsubseteq \HF$, consider a factor $F'$ in $\HF'$ which is not $\HF$-peripheral. By 
    Lemma~\ref{lemma:af}\ref{it:af3},
    $F'$ is not cyclic, hence is not an unused boundary subgroup of $\HF'$. In particular, $F'$ is a free factor of $F_N$,
and its  conjugacy class is $H$-invariant. Therefore, $F'$ contains a conjugate of $\AHF$.
This concludes that $\AHF$ is $\calf'$-peripheral.
Since $\AHF$ is not abelian by Theorem \ref{lemma:minimal-debording-factor}, it cannot be contained in an unused boundary subgroup of  $\HF'$.
\end{proof}

We record the following consequences of Lemma~\ref{lemma:A-periph}.

  \begin{cor}\label{cor:UP}
 Let $H\subseteq\IA$ be a subgroup  that does not preserve any non-trivial free splitting, and let $\hat\calf$ be a maximal $H$-invariant almost free factor system. Then all edge stabilizers of $\THF$ are universally peripheral with respect to $H$.
  \end{cor}

\begin{proof}
Lemma~\ref{lemma:A-periph} implies that edge stabilizers of $\THF$ are $\HF'$-peripheral for $\HF'\neq \HF$ because $A_\HF$ is. They are also $\HF$-peripheral by Lemma \ref{lemma:minimal-debording-factor}\ref{it0_relP}.
\end{proof}

\begin{lemma}\label{lemma:filling}
    Let $H\subseteq\IA$ be a subgroup  that does not preserve any non-trivial free splitting, and let $\hat\calf_1,\HF_2$ be two distinct maximal $H$-invariant almost free factor systems. 
    
    Then $\HF_1\cup\HF_2$ is filling in the following sense: there is no \afs\ $\HF$ such that
    $\HF_1\sqsubseteq \HF$ and $\HF_2\sqsubseteq \HF$.
\end{lemma}

\begin{rk}
  The point in the lemma is that we  do not assume that $\HF$ is $H$-invariant.
\end{rk}

\begin{proof}
Assume on the contrary that there exists an \afs\ $\HF$ such that     $\HF_1\sqsubseteq \HF$ and $\HF_2\sqsubseteq \HF$.
If $\HF$ is a free factor system, 
then the Grushko decomposition of $F_N$ relative to $\HF_1\cup\HF_2$ is non-trivial, and provides 
a free factor system $\calg$ such that $\HF_1\sqsubsetneq\calg$ and $\HF_2\sqsubsetneq\calg$, and which is $H$-invariant.
This contradicts maximality of $\HF_1$ and $\HF_2$.

So assume that $\HF=\{[A_1],\dots,[A_n],[\grp{b_1}],\dots,[\grp{b_r}]\}$
is the factor system of a QH splitting $S$
with 
$r\geq 1$ unused boundary components.
Let $\Gamma'$ be the non-minimal graph of groups obtained from $S/F_N$ by adding $r$ terminal edges with stabilizers
$\grp{b_1},\dots,\grp{b_r}$,
and let $S'$ be the Bass-Serre tree of this splitting.
The splitting $S'$ has the property that for every $\HF$-peripheral subgroup $A\subset F_N$,
there is a unique vertex $v_A\in S'$ such that $A$ fixes $v_A$ and the stabilizer of $v_A$ is in $\HF$.

We claim that all vertex stabilizers of $T_{\HF_1}$ are $\HF$-peripheral.
Indeed, $A_{\HF_1}$ is $\HF$-peripheral because it is $\HF_2$-peripheral by Lemma~\ref{lemma:A-periph},
and the stabilizers of $V^0$-vertices of $T_{\HF_1}$ are $\HF_1$-peripheral
 (Lemma~\ref{lemma:minimal-debording-factor}\ref{it0_relP}).
It follows that there is an equivariant map $f:T_{\HF_1}\ra S'$
that sends each vertex $u\in T_{\HF_1}$ to the vertex $v_{G_u}$ defined above. 
Edge stabilizers of $T_{\HF_1}$ are non-trivial because $H$ preserves no non-trivial free splitting.
Since $\HF$ is a malnormal family, this forces $f$ to be constant, say $f(T_{\HF_1})=\{w\}$ with $[G_w]\in\HF$.
Thus $F_N$ is $\HF$-peripheral, a contradiction.
\end{proof}

\subsection{End of the proof via a tree of cylinders construction}

If $\HF_1,\dots,\HF_n$ are maximal $H$-invariant free factor systems,
Corollary \ref{cor:UP} implies that for all $i,j\leq n$,  the trees $T_{\HF_i}$, $T_{\HF_j}$ are elliptic with respect
to each other, i.e.\ each edge stabilizer fixes a point in the other tree. This allows to construct 
a refinement of $T_{\HF_1}$ dominating $T_{\HF_2}$, and repeating this procedure, a splitting that dominates all
the trees $T_{\HF_i}$.
To make things more canonical
it will be convenient to use a suitable notion of tree of cylinders (see Section~\ref{sec:cyl}).
We take as a class of allowed edge groups the class $\cale$ of non-trivial universally peripheral subgroups
(Definition~\ref{dfn_UP}), i.e.\ $\cale=\calp_H\setminus\{\{1\}\}$. 
We consider the following equivalence relation on $\cale$.

\begin{lemma}\label{lem_equiv}
  Let $H\subseteq\IA$ be a subgroup. If $P,P'\in\calp_H$
  and $P\cap P'\neq\{1\}$, then $\grp{P,P'}\in\calp_H$. 

  The relation $P\sim P'$ defined by $\grp{P,P'}\in\calp_H$ is an equivalence relation on
  the set $\cale$ of non-trivial universally peripheral subgroups.
\end{lemma}

\begin{proof}
  To prove the first assertion,
  let $\HF$ be a maximal $H$-invariant \afs. Let $F,F'$ be the factors in $\HF$ containing  $P$ and $P'$ respectively.
  Since $F\cap F'$ contains $P\cap P'$, malnormality of $\HF$ implies that $F=F'$. Thus $\grp{P,P'}$ is $\HF$-peripheral.
Since this holds for all $\HF$, the first assertion follows.

  We deduce the transitivity of the relation $\sim$. If $P\sim P'$ and $P'\sim P''$, then $\grp{P,P'}$ and $\grp{P',P''}$ are universally peripheral
  and have non-trivial intersection, so $\grp{P,P',P''}\in\calp_H$ by the first assertion.
\end{proof}

The lemma implies that for any $P\in\cale$, 
there exists a unique $\hat P\in\calp_H^{\max}$
such that $P\subset \hat P$, and $P\sim P'$ if and only if $\hat P=\hat P'$.

Notice that if $P,P'\in\cale$ with $P\sim P'$, then $\langle P,P'\rangle$ is elliptic in any $F_N$-tree relative to $\calp_H$. It thus follows from \cite[Lemma 3.2]{GL-cyl} that this equivalence relation is admissible relative to $\calp_H$, and can
therefore be used to construct the tree of cylinders of any tree $U$ relative to $\calp_H$ with edge groups in $\cale$
(see Section \ref{sec:cyl}):
cylinders are subtrees of $U$ defined by the partition of edges induced by the equivalence relation $G_e\sim G_{e'}$.
One then defines the tree of cylinders $U^c$ of $U$
as the tree having one vertex $v_Y$ for each cylinder $Y\subset U$, one vertex $v_x$ for each vertex $x\in U$
belonging to at least two cylinders, with an edge between $v_x$ and $v_Y$ if an only if $x\in Y$.
We denote by $V_p=\{v_Y\}$ and $V_a=\{v_x\}$ the two sets of vertices of $U^c$.

  \begin{lemma}\label{lemma:deformation}
    Let $H\subseteq\IA$ be a subgroup, let $U$ be a splitting of $F_N$ relative to $\calp_H$ 
    with edge stabilizers in $\cale$,
    and let $U^c$ be its tree of cylinders, coming with the bipartition $V=V_a\sqcup V_p$ of its vertex set.

    For every $v\in V_p$, one has $G_v\in \calp_H^{\max}$, and $G_v\cap G_{v'}=\{1\}$ if $v\neq v'\in V_p$.
    On the other hand, if $v\in V_a$, then $G_v$ is not universally peripheral.
    
    Moreover, $U$ and $U^c$ are in the same deformation space  (i.e.\ every vertex stabilizer of $U^c$ fixes a vertex in $U$).
  \end{lemma}
   
  \begin{proof}
    Consider the cylinder $Y_P$ corresponding to all edge groups contained in a given subgroup $P\in \calp_H^{\max}$.
    Then the stabilizer of $Y_P$ is the normalizer of $P$, which coincides with $P$.
    In particular, 
    $G_{v_{Y_P}}=P\in \calp_H^{\max}$. If $Y,Y'$ are two distinct cylinders corresponding to $P\neq P'\in \calp_H^{\max}$,
    then $G_{v_Y}\cap G_{v_{Y'}}=P\cap P'=\{1\}$.

    If the stabilizer $G_x$ of a vertex $x\in U$ is universally peripheral, then all incident edge groups are in the same equivalence class so $x$ lies in a unique cylinder and cannot correspond to a vertex in $V_a$.

    Since every subgroup $P\in\calp_H^{\max}$ is elliptic in $U$,
    all vertex stabilizers of $U^c$ fix a point in $U$, and it follows that $U$ and $U^c$ are in the same deformation space.    
  \end{proof}

\begin{proof}[Proof of Theorem \ref{thm_dynamic_dec}]  
  We first prove the uniqueness (see Proposition~\ref{prop_intrinseque} for an alternative proof).
  Let $U,U'$ be as in the theorem. We first prove that $U$ and $U'$ are in the same deformation space,
  i.e.\ that they have the same elliptic subgroups.
  For any $u\in V_p(U)$, $G_u$ is universally peripheral hence elliptic in $U'$. 
  For $v\in V_a(U)$, $G_v$ is not universally peripheral by Assertion \ref{it_HF}.
  Since $G_v$ is a free factor by \ref{it_Vapur},
  and since its conjugacy class is $H$-invariant, \ref{it_Vamin} implies that $G_v$ contains the stabilizer $G_{v'}$
  of a vertex $v'\in V_a(U')$. Symmetrically $G_{v'}$ contains $G_{v''}$ for some $v''\in V_a(U)$,
  which imposes $G_v=G_{v''}$ hence $G_v=G_{v'}$. This concludes that $U$ and $U'$ are in the same deformation space.
  
  To conclude that $U=U'$, we claim that $U$ and $U'$ are their own trees of cylinders.
 If $e_1,e_2$ are two edges of $U$ with $G_{e_1}\sim G_{e_2}$, 
  let $p_i$ denote the endpoint of $e_i$ in $V_p$. Then $G_{p_1}\sim G_{e_1}\sim G_{e_2}\sim G_{p_2}$ so $\grp{G_{p_1},G_{p_2}}$
  is universally peripheral. Since $G_{p_i}\in\calp_H^{\max}$, we get $G_{p_1}=\grp{G_{p_1},G_{p_2}}=G_{p_2}$ hence $p_1=p_2$
  by Assertion \ref{it_relP}.
  This shows that each cylinder is the star of a $V_p$-vertex
  so that $U$ is its own tree of cylinders, and so is $U'$ by the same argument.
  By \cite[Corollary 4.10]{GL-cyl}, it follows that $U$ and $U'$ are isomorphic.

  We now construct $\Ud_H$.  Choose a numbering $\HF_1,\HF_2,\dots$ of the set of maximal $H$-invariant \afs{}s (we do not know yet that it is finite),
  and let $T_i=T_{\HF_i}$ and $A_i=A_{\HF_i}$. By  Corollary~\ref{cor:UP}, edge stabilizers of $T_i$ are universally peripheral.
  We start with $U_1=T_1$ and construct inductively an $H$-invariant tree $U_i$
  with edge groups in $\calp_H$ and such that a subgroup $K\subseteq F_N$ is elliptic in $U_i$ if and only if it is elliptic in
  each of the trees $T_1,\dots,T_i$.
  Since $H$ preserves no free splitting, edge stabilizers are necessarily non-trivial.
  Since edge stabilizers of $U_{i-1}$ are elliptic in $T_{i}$ one can construct a blowup $U_i$ of $U_{i-1}$ that dominates $T_i$: replace
  each vertex $v$ such that $G_v$ is not elliptic in $T_i$ by the minimal $G_v$-invariant subtree of $T_i$, and attach every edge $e$ of $U_{i-1}$ incident on $v$
  to a point $p_e$ fixed by $G_e$. This can be done so that $U_{i}$ is $H$-invariant by choosing $p_e$ in a canonical way:
  if $G_e$ fixes a
  point in
  $V^0(T_i)$, then it is unique by Lemma \ref{lemma:minimal-debording-factor}\ref{it0_relP}, 
  and one can choose $p_e$ to be this vertex; otherwise, $G_e$ fixes no edge hence fixes a unique vertex in  $V^1(T_i)$.
  This concludes the construction of $U_i$.

Given any $i$ and $j$, the group $A_j$ fixes a unique vertex $v_{A_j}\in U_i$: the existence of $v_{A_j}$ follows from the fact that $A_j$ is elliptic in all trees $T_\ell$ (Lemma~\ref{lemma:A-periph}), and its uniqueness follows from the fact that $A_j$ is not universally peripheral, so cannot fix an edge.

  Let $U_i^c$ be the tree of cylinders of $U_i$.
  We  are going to prove that $T_j$ is a collapse of $U_i^c$ for all $j\leq i$. 
  We construct for all $j\leq i$ a map $f_{ij}:U^c_i\ra T_j$ as follows.
  For every vertex $v\in U^c_i$,  the group $G_v$ is elliptic in $U_i$ by Lemma~\ref{lemma:deformation}, hence also in $T_j$.
  If $G_v$ fixes a vertex $v'\in V^0(T_j)$, then this vertex is unique because the stabilizers of two distinct vertices in
   $V^0(T_j)$ have trivial intersection.  If $G_v$ fixes no vertex in $V^0(T_j)$, then it fixes a necessarily unique vertex $v'\in V^1(T_j)$.
  In both cases, we let $f_{ij}(v)=v'$.  
  We extend $f_{ij}$ linearly on edges.
  Notice that $f_{ij}$ is $F_N$-equivariant, so its image is an $F_N$-invariant subtree of $T_j$; the minimality of $T_j$ thus implies that $f_{ij}$ is surjective.

  (i) We first claim that if $v_1,v_2\in V_a(U^c_i)$ are such that $f_{ij}(v_1)=f_{ij}(v_2)=v'$ with $v'\in V^1(T_j)$, then $v_1=v_2$.
  Indeed, $G_{v'}$ is a conjugate of $A_j$ and therefore fixes a vertex $w\in U^c_i$.
  Now $G_{v_1}\subset G_{v'}\subset G_w$, and since $G_{v_1}$ fixes no edge of $U_i^c$ because $G_{v_1}\notin\calp_H$  (Lemma~\ref{lemma:deformation}),
  it follows that $v_1=w$. Symmetrically, one gets $v_2=w=v_1$, which proves the claim.

  (ii) We deduce that $f_{ij}(V_p(U^c_i))\subset V^0(T_j)$. Indeed, arguing by contradiction,
  consider $u\in V_p(U_i^c)$  such that $f_{ij}(u)\in V^1(T_j)$.
  By definition of $f_{ij}$, this means that $G_u$ fixes no vertex in $V^0(T_j)$. Let $e$ be an edge incident on $u$.
  If $G_e$ fixes a vertex  $w'\in V^0(T_j)$, then  $G_u\cap G_{w'}\neq \{1\}$, and since  $[G_{w'}]\in\HF_j$ and $G_u$ is universally peripheral  (Lemma~\ref{lemma:deformation}), if follows that $G_u\subset G_{w'}$, a contradiction. This shows that no edge incident on $u$ fixes a vertex in $V^0(T_j)$,
  so $f_{ij}$ collapses the star of $u$ to $f_{ij}(u)$, contradicting (i) and concluding the proof of (ii).
  
  (iii) We now claim that the image of an edge $e\subset U_i^c$ under $f_{ij}$ is either an edge or a vertex.
  Write $e=uv$ with $u\in V_p(U_i^c)$ and $v\in V_a(U_i^c)$. By (ii), $f_{ij}(u)\in V^0(T_j)$. Since stabilizers of
  distinct vertices in $V^0(T_j)$ intersect trivially and $G_e\neq\{1\}$,  it follows that $f_{ij}(e)$ is contained in the star of $f_{ij}(u)$ which proves (iii).

  To prove that $f_{ij}$ is a collapse map, consider $e_1=u_1v_1$, $e_2=u_2v_2$  two distinct edges of $U_i^c$ that are mapped to the same edge $e'=u'v'$ with $u'=f_{ij}(u_1)=f_{ij}(u_2)\in V^0(T_j)$ and $v'=f_{ij}(v_1)=f_{ij}(v_2)\in V^1(T_j)$.
  By (ii), $v_1,v_2\in V_a(U_i^c)$ hence $u_1,u_2\in V_p(U_i^c)$. By (i), it follows that $v_1=v_2$.
    Since $\grp{G_{e_1},G_{e_2}}$ fixes an edge in $T_j$, it is universally peripheral (Corollary~\ref{cor:UP}).
  It follows that $G_{u_1}\sim G_{e_1}\sim \grp{G_{e_1},G_{e_2}} \sim G_{e_2}\sim G_{u_2}$,
  and $U_i^c$ being a tree of cylinders, this implies $u_1=u_2$.
  This shows that no pair of noncollapsed edges gets identified by $f_{ij}$. This concludes that $f_{ij}$ is a collapse map.

  Additionally, for any vertex $v\in f_{ij}\m(V^1(T_j))$,  the map $f_{ij}$ induces a bijection between
  the star of $v$ and the star of $f_{ij}(v)$. Indeed,  as $f_{ij}$ is surjective, it suffices to show that it does not collapse any edge
  $uv$ incident on $v$, but if it did, then $f_{ij}(u)\in V^1(T_j)$ would contradict (ii).

By Lemma \ref{lemma:minimal-debording-factor}\ref{it0_periph}, this shows that the
  incident edge groups at  $v_{A_j}$
 form a non-sporadic free factor system of $A_j$.
This allows to refine $U_i^c$ simultaneously at each vertex $v_{A_j}$ into a free splitting of $A_j$; 
collapsing all edges with non-trivial stabilizer in this refinement yields a free splitting of $F_N$ with at least $i$ edges. This provides a bound on $i$, 
  and shows that there are only finitely many maximal $H$-invariant \afs{}s, so the process stops.
  
  We define $\Ud_H=U_i^c$ as the last constructed tree, and we denote by $f_j=f_{ij}:\Ud_H\ra T_j$.
  The tree $\Ud_H$ is relative to $\calp_H$ because every
  $T_j$ is.
  Assertion~\ref{it_relP} follows from the properties of the trees of cylinders (Lemma \ref{lemma:deformation}).

  To prove \ref{it_HF}, we check that each vertex  $v\in V_a(\Ud_H)$ is in the orbit of some $v_{A_j}$.
  Indeed, $G_v\notin \calp_H$ so $G_v$ is not $\HF_j$-peripheral for some $j$. 
  It follows that up to changing $v$ to a translate, the image of $v$ in $T_j$ under $f_j$ is the vertex fixed by $A_j$,
  so $G_v\subset A_j$.  
  Since $A_j$ fixes a vertex in  $U_H^d$ and $G_v$ fixes no edge, it follows that $G_v=A_j$. This proves \ref{it_HF}.

   We now prove Assertion \ref{it_Vamin}. 
  If $A$ is a free factor  whose conjugacy class is $H$-invariant and which is not universally peripheral, then $A$ is not $\HF_j$-peripheral for some $j$,
  so $A$ contains $A_j$ by definition of $A_j=A_{\HF_j}$. Since $A_j=G_{v_{A_j}}$, this proves Assertion~\ref{it_Vamin}. 
  
  We finally prove Assertion~\ref{it_Vapur}. 
  We first note that  $(\calp_{H}^{\max})_{|A_j}=(\HF_{j})_{|A_j}$ because $A_j$ is $\HF_{j'}$-peripheral for all $j'\neq j$
  by Lemma~\ref{lemma:A-periph}.
  By Lemma~\ref{lemma:minimal-debording-factor}\ref{it0_periph}, this is a non-sporadic \afs, and we denote it by $\hat\calg$.
  To prove purity, consider a proper free factor $A\subset A_j$  whose conjugacy class is $H$-invariant.
  The minimality of $A_{\HF_j}$  implies that $A$ is $\HF_j$-peripheral hence $\hat\calg$-peripheral.
  This implies that  $\hat\calg$ is the unique maximal $H$-invariant \afs\ of $A_j$.
 As already noted above, the incident edge groups form a non-sporadic free factor system of $A_j$ which concludes the proof
  of Assertion \ref{it_Vapur} and of the theorem.
\end{proof}

\subsection{Finiteness of maximal invariant almost free factor systems}

The previous proof shows that there are only finitely many maximal $H$-invariant \afs{}s. For future use, we record this fact in the next corollary, with an explicit bound.

\begin{cor}\label{coro:finitude-facteurs}
  Let $N\geq 2$. For every subgroup $H\subseteq \IA$
  that does not preserve any free splitting of $F_N$, there are at most  $\max\left\{1,\frac{N-1}{2}\right\}$ maximal $H$-invariant almost free factor systems of $F_N$. 
    
  Consequently, under the same hypothesis on $H$, there are at most $\frac{(N+1)(N-1)}{2}$ maximal $H$-invariant free factor systems of $F_N$. 
  \end{cor}

\begin{rk}\label{rk:finitude}
  As already mentioned, Example \ref{ex:arcs} shows that
  Corollary~\ref{coro:finitude-facteurs} does not hold true in general if we remove the assumption that $H$ does not preserve any free splitting of $F_N$.  
\end{rk}

\begin{proof}
  Let $\Ud_H$ be the splitting of $F_N$ associated to $H$ given by Theorem~\ref{thm_dynamic_dec}, coming with the bipartition $V_p\dunion V_a$ of its vertex set. Assertion~\ref{it_HF} from this theorem ensures that the number of maximal $H$-invariant almost free factor systems of $F_N$ is equal to the cardinality $K$ of $V_a/F_N$. We may assume that $K\geq 2$,  and aim to prove that $K\le\frac{N-1}{2}$. 

 Recall that stabilizers of edges of $\Ud_H$ are non-trivial because there is no $H$-invariant free splitting.
  By the last assertion of Theorem~\ref{thm_dynamic_dec}\ref{it_Vapur}, for each vertex $v\in V_a/F_N$
  there is a free splitting of $G_v$ relative to the incident edge groups
  with all vertex groups non-trivial and at least two  orbits of edges (indeed, the relative Grushko decomposition is non-sporadic and has at least one non-trivial vertex group because there is at least one edge incident on $v$, as $K\geq 2$).
  Blowing up every vertex in $V_a/F_N$ using this free splitting
  and  collapsing all edges coming from $\Ud_H/F_N$ gives a
  free splitting of $F_N$ with at least $2K$ edges and all vertex groups non-trivial. It follows that $2K\leq N-1$.

  For $N=2$, any non-empty free factor system is sporadic, so the empty set is the unique
  maximal $H$-invariant free factor system. Otherwise,
  since any maximal $H$-invariant free factor system can be extracted from a maximal $H$-invariant \afs\ by Lemma~\ref{lemma:af},
  and since an almost free factor system of $F_N$ contains of at most $(N+1)$
  conjugacy classes of cyclic groups, there are at most $(N+1)\frac{N-1}{2}$ maximal $H$-invariant
  free factor systems. 
\end{proof}

\subsection{Reducibility is inherited by the normalizer}

The following result will not be used in the rest of the paper, but we believe that it is of independent interest.
\begin{theo}\label{theo:reducibility}
    Let $H\subset\Out(F_N)$ be an infinite subgroup. If $H$ has a periodic conjugacy class of proper free factor, so does its normalizer in $\Out(F_N)$.
\end{theo}

\begin{proof}
    Up to replacing $H$ by $H\cap \IA$, we may assume that $H\subset \IA$.

    Assume first that there is no $H$-invariant non-trivial free splitting.
    Then by Corollary~\ref{coro:finitude-facteurs}, there are only finitely many maximal
    $H$-invariant free factor systems, and all of them are non-empty by assumption.
    The normalizer of $H$ permutes them and the theorem follows in this case.
    
    If there exists a non-trivial $H$-invariant free splitting, then by Proposition~\ref{prop:deformation}
    there exists a $H$-invariant free splitting $S$ which is maximum for domination among
    all $H$-invariant free splittings, and all such splittings are non-trivial and belong to the same deformation space.
    Let $\calf$ be the collection of conjugacy classes of non-trivial vertex stabilizers of $S$. 
    Then $\calf$ is a free factor system, and $\calf$ is non-empty because otherwise, the stabilizer of $S$ would be finite, contradicting that $H$ is infinite.
    Since $\calf$ is clearly invariant under the normalizer of $H$, this concludes the proof.
\end{proof}

\subsection{Side remarks about the dynamical decomposition}

We conclude this section by establishing a few additional properties of the splitting $\Ud_H$ and by relating it to other works involving the graph of free splittings and laminations.

This subsection will not be used in the sequel of the paper.

\subsubsection{Additional properties of the dynamical decomposition}

\begin{prop}\label{prop_additional} 
 Keeping the notations from Theorem~\ref{thm_dynamic_dec}, the following hold: 
  \renewcommand{\theenumi}{(\alph{enumi})} 
\renewcommand{\labelenumi}{\theenumi} 
  \begin{enumerate}
  \item \label{it_FF} All edge and vertex stabilizers of $\Ud_H$ are free factors of $F_N$.
  \item \label{it_P_Grushko} For every $v\in V_p$, the Grushko decomposition of $G_v$ relative to its incident edge groups is trivial.
  \end{enumerate}
\end{prop}

\begin{proof}
We first claim that that for every $v\in V_p$ and every maximal $H$-invariant \afs\ $\HF$,
$G_v$ is not contained in an  unused boundary subgroup $\grp{b}$ of $\HF$.
Otherwise, let $v_{\hat{\calf}}$ be a vertex in $\Ud_H$ whose orbit is sent to $\hat{\calf}$ under the bijection given in Theorem~\ref{thm_dynamic_dec}\ref{it_HF}.
Since collapsing all edges not adjacent to a vertex in the orbit of $v_\HF$ yields the tree $\THF$,
we deduce that $\grp{b}$ fixes a vertex in $V_0(\THF)$, contradicting
Lemma~\ref{lemma:minimal-debording-factor}\ref{it0_relP}.

To prove Assertion~\ref{it_FF}, it suffices to check that vertex stabilizers are free factors.
We already know this for vertices in $V_a$, so
consider $v\in V_p$. Since $G_v\in\calp_H^{\max}$, for each maximal $H$-invariant \afs\ $\HF$, there
is a factor of $\HF$ containing $G_v$, and the claim above says that this factor is  not an unused boundary subgroup.
The intersection of these free factors as $\HF$ varies is a free factor $G'_v$ containing $G_v$
and which is universally peripheral. Since $G_v\in\calp_H^{\max}$, we have $G_v=G'_v$, which proves  Assertion~\ref{it_FF}.

To prove Assertion~\ref{it_P_Grushko}, consider $v\in V_p$. If the Grushko decomposition of $G_v$ relative to its incident edge groups is non-trivial, then one can use it to blow up $\Ud_H$ at the vertex $v$,
and thus get an $H$-invariant free factor system $\calg$ of $F_N$.
Consider a  maximal $H$-invariant  free factor system $\calf$ such that $\calg\sqsubseteq\calf$.
By construction, for every vertex of $\Ud_H$ not in the orbit of $v$, its stabilizer is $\calg$-peripheral hence $\calf$-peripheral.
Let $\HF$ be the unique maximal $H$-invariant \afs\ containing $\calf$.
Being universally peripheral, $G_v$ is $\HF$-peripheral 
and the claim above shows that $G_v$ is in fact $\calf$-peripheral.
It follows that there is a domination map $\Ud_H\ra R_\calf$ where $R_\calf$ is a (non-trivial) free splitting whose elliptic subgroups are the elements of $\calf$. As all edge stabilizers of $\Ud_H$ are non-trivial, this map has to be constant, contradicting that $R_\calf$ is non-trivial. This contradiction concludes the proof.
\end{proof}

We  finally give a more intrinsic description of $\Ud_H$ by defining a new graph $S$ as follows.
Let $V'_a$ be the collection of minimal free factors of $F_N$ whose conjugacy class is $H$-invariant and which are not universally peripheral, i.e.\ the collection of subgroups that are conjugate to some $A_\HF$ where $\HF$ is a maximal $H$-invariant \afs.
Let $V'_p$ be the set of subgroups $P\in \calp_H^{\max}$ such that there exist at least two distinct subgroups
$A,A'\in V'_a$ that intersect $P$ non-trivially.
We define $S$ as the bipartite graph with vertex set $V'_a\cup V'_p$ and with an edge between
$A\in V'_a$ and $P\in V'_p$ if $A\cap P\neq \{1\}$.
There is a natural action of $F_N$ on $S$ by conjugation.

\begin{prop}\label{prop_intrinseque}
  The graph  $S$ is equivariantly isomorphic to $\Ud_H$.
\end{prop}

\begin{proof}
  For each $v\in V_a$,  we have $G_v\in V'_a$,
  and for each $v\in V_p$, we have $G_v\in V'_p$ because edge stabilizers of $\Ud_H$ are non-trivial.
  The assignment $v\mapsto G_v$ then defines a map $f$ from the vertices of $\Ud_H$ to the vertices of $S$.

  Since edge stabilizers of $\Ud_H$ are non-trivial, $f$ preserves adjacency.
  The map $f$ is injective on $V_p$ by Theorem \ref{thm_dynamic_dec}\ref{it_relP}.
  We claim that it is also injective on $V_a$.  Indeed, if $v,v'\in V_a$ are not in the same orbit, then $G_{v'}$ is $\HF_v$-peripheral but $G_{v}$ is not
  by Theorem~\ref{thm_dynamic_dec}\ref{it_HF}, so $G_v\neq G_{v'}$. If $v'=gv$ and $G_{v'}=G_v$, then $g\in G_v$ by malnormality of the free factor $G_v$, so $v'=v$. This  proves our claim that $f$ is injective on  $V_a$.
  Since groups in $V'_p$ are universally peripheral and groups in $V'_a$ are not, $V'_p\cap V'_a=\es$ so $f$ is injective.
  
  To check that $f$ preserves non-adjacency, consider $u\in V_p$ and $v\in V_a$ which are not adjacent in $\Ud_H$.
  Then they are at distance at least $3$ and the segment $[u,v]$ contains a vertex $u'\in V_p$ distinct from $u$.
  Since $G_u\cap G_{u'}=\{1\}$ by Theorem \ref{thm_dynamic_dec}\ref{it_relP}, it follows that $G_u\cap G_v=\{1\}$ so
  $f(u)$ is not adjacent to $f(v)$.

  We finally check that $f$ is onto. It is clear that $f(V_a)=V'_a$.
  So consider $P\in V'_p$. Since $\Ud_H$ is relative to $\calp_H$,  the group $P$ fixes a vertex $v\in \Ud_H$.
  If one can choose $v\in V_p$, then $P=G_v$ because $P\in\calp_H^{\max}$ so $P=f(v)$.
  If $P$ fixes no vertex in $V_p$, then $v\in V_a$ and for any $v'\in V_a\setminus \{v\}$, one has $P\cap G_{v'}=\{1\}$.
  This shows that $G_v$ is the unique group in $V'_a$ that $P$ intersects non-trivially, so $P$ does not fulfill the requirements
  to lie in $ V'_p$.
  This shows that $f$ is surjective and defines an isomorphism $\Ud_H\ra S$.
\end{proof}

\subsubsection{Relation with other works and concepts}\label{sec_rel_other_works}
In this subsection, we relate the dynamical decomposition of a subgroup $H\subset \IA$ with 
its dynamics on the free splitting graph $\mathbb{FS}$ (see Section \ref{sec:maps} for its definition). Recall that $\mathbb{FS}$ is a Gromov-hyperbolic space by \cite{HM3}.
Using \cite{HM2,HM4}, we also relate it to the set of attracting laminations 
in the case of a cyclic group $H$.\\


The construction of the dynamical decomposition of a subgroup $H\subset \IA$
assumes that there is no $H$-invariant non-trivial free splitting,
i.e.~that $H$ has no fixed point (equivalently no finite orbit) in $\mathbb{FS}$.

Theorem \ref{thm_dynamic_dec} implies that if a subgroup $H\subset \IA$ is not pure,
then $H$ is elliptic in $\mathbb{FS}$. 
Indeed, this is clear if there is an $H$-invariant non-trivial free splitting, so assume otherwise.
Since $H$ is impure, 
there exist two distinct active vertices $a_1,a_2$ in the dynamical decomposition
$\Ud/F_N$. 
Consider two free splittings $S_1,S_2$ obtained by refining $\Ud/F_N$ at $a_1$ and $a_2$ respectively.
Then for any $\alpha\in H$, the free splitting $\alpha.S_1$ is also obtained by refining 
$\Ud/F_N$ at $a_1$ (because $\Ud$ is $H$-invariant and $H$ fixes $a_1$ in $\Ud/F_N$ by Theorem \ref{thm_dynamic_dec}\ref{it_HF}) and is therefore compatible with $S_2$.
This shows that the $H$-orbit of $S_1$ is at bounded distance from $S_2$ in $\mathbb{FS}$, and that $H$ is elliptic.

Using the characterization of loxodromic elements of $\mathbb{FS}$ given in \cite{HM4},
one can deduce conversely that if $H=\grp{\alpha}\subset \IA$ is cyclic and pure, then 
$\alpha$ acts loxodromically on $\mathbb{FS}$.
Indeed, recall that to any maximal $\alpha$-invariant free factor system $\calf$ which is non-sporadic,
one can associate a pair of attracting and repulsive laminations $\Lambda_\calf^{\pm}$
(\cite{BFH}). 
Using \cite{HM4}, it suffices to prove that 
$\Lambda_\calf^+$ is not carried by a proper free factor to deduce 
that $\alpha$ acts loxodromically on $\mathbb{FS}$.
Now if $\Lambda_\calf^+$ was carried by a proper free factor, 
the minimal free factor carrying $\Lambda_\calf^+$
would be $\alpha$-invariant up to conjugacy,
hence $\hat\calf$-peripheral by purity, contradicting that 
$\Lambda_\calf^+$ is not carried by a free factor in $\calf$, nor by a cyclic group.

In the same way that several maximal $H$-invariant free factor systems $\calf$ may correspond
to the same $\hat\calf$,
several distinct $\calf$ may correspond to the same pair of 
laminations $\Lambda_\calf^\pm$. 
The canonical object associated to $\Lambda_\calf^+$ is 
the non-attracting subgroup system of $\Lambda_\calf^+$ defined in \cite[\S III.1]{HM2}
from a completely split train track representative of $\alpha$ in the sense of \cite{FeH};
this non-attracting subgroup system has a nice description given in \cite{HM2} in the remark before Remark III.1.3, page 197.
One can show that $\hat\calf$ coincides with this non-attracting subgroup system.

When $H=\grp{\alpha}$ is cyclic and does not preserve a non-trivial free splitting,
there is a natural bijection between active vertices in $\Ud/F_N$
and the set of pairs of laminations $\Lambda_\calf^{\pm}$ associated
to maximal $\alpha$-invariant free factor systems $\calf$.
This bijection can be defined using the bijections
$\Lambda_\calf^{\pm}\mapsto \hat \calf$ and $v\in V_a/F_N\mapsto {\hat \calf_v}$.

Notice that for every active vertex $v\in V_a/F_N$, $\alpha_{|G_v}$ is pure, so there is a top attracting lamination of $\alpha_{|G_v}$ which is not carried by any proper free factor of $G_v$. Viewing $\Lambda_v$ as a lamination of $F_N$, this is the unique attracting lamination of $\alpha$ which is not carried by $\hat\calf_v$. The map $v\mapsto\Lambda_v$ is the above bijection.

\section{Canonical nice splittings} \label{sec_nice}

In Sections \ref{sec:collection-free}, \ref{sec:collection-zmax}, and \ref{sec:collection-factors}
we constructed canonical splittings using invariant free splittings, invariant $\Zmax$-splittings,
and maximal invariant free factor systems.
The goal of this section is to unify them using the following notion of nice splittings.  We let $N\ge 2$.

\begin{de}[Nice and bi-nonsporadic splittings]\label{dfn_nice}
  A non-trivial splitting $S$ of $F_N$ is \emph{nice} if either
  \begin{itemize}
  \item $S$ is a free splitting
  \item $S$ is a $\Zmax$-splitting
  \item $S$ is \emph{bi-nonsporadic} in the following sense: its edge stabilizers are finitely generated non-abelian, and 
    there are at least two vertices $v_1,v_2\in S$ in different $F_N$-orbits such that for every $i\in\{1,2\}$, 
    the Grushko decomposition of $G_{v_i}$ relative to the incident edge groups is non-sporadic.
  \end{itemize}
\end{de}

 We say that a collection $\calc$ of 
almost free factor systems is \emph{filling}
there is no almost free factor system $\HF$ such that for every $\HF'\in \calc$,
one has $\HF'\sqsubseteq\HF$.

The main result we will prove is the following one.

\begin{theo}\label{thm_nice}
  There exists an $\Out(F_N)$-equivariant map assigning a (canonical) nice splitting $U_\calc$ to $\calc$
  when $\calc$ is either
  \begin{enumerate}
  \item a non-empty (possibly infinite) 
  collection of nice splittings whose elementwise stabilizer is infinite;
  \item a (possibly infinite) filling collection of \afs{}s whose elementwise stabilizer is infinite;
  \end{enumerate}
\end{theo}

The following corollary will not be directly used in this paper,
but it is good to have it in mind in preparation for the next section,
where we will establish groupoid-theoretic analogues that will be crucial in the proof of our main measure equivalence rigidity theorem. 

\begin{cor}\label{cor_nice}
  If $H\subseteq\Out(F_N)$ is an infinite group preserving a nice splitting, then its normalizer preserves a nice splitting.
\end{cor}

\begin{proof}
  Let $\calc$ be the collection of all nice splittings preserved by $H$, and let $U_\calc$
  be the associated nice splitting. Since $\calc$ is invariant under $N(H)$, so is $U_\calc$.
\end{proof}

We prove several intermediate results before proving the theorem.

In Section \ref{sec:UH1_biflexible}, we constructed from invariant free splittings
a canonical splitting which has an edge with trivial or $\Zmax$
stabilizer or is biflexible.
Recall from Definition~\ref{de:biflexible} that
that $(A,\calp)$ is \emph{flexible} if $\Out(A,\calp^{(t)})$ is infinite. A splitting $S$ of $F_N$ is \emph{biflexible} if all its edge stabilizers are
finitely generated non-abelian and there exist two vertices $v_1,v_2$ in different $F_N$-orbits such that for every $i\in\{1,2\}$, $(G_{v_i},\Inc_{v_i})$ is flexible.

\begin{lemma}\label{lem:flex} Let $A$ be a finitely generated free group and $\calp$ a non-empty collection of conjugacy classes of non-abelian subgroups.
Assume that $(A,\calp)$ is flexible. Then one of the following holds:
\begin{enumerate}
\item the Grushko decomposition of $A$ relative to $\calp$ is non-sporadic;
\item there is a canonical non-trivial free splitting of $A$ relative to $\calp$ ---in particular it is invariant under $\Out(A,\calp)$;
\item $A$ is one-ended relative to $\calp$ and the canonical $\Zmax$ JSJ decomposition of $A$ relative to $\calp$ is non-trivial.
\end{enumerate}
\end{lemma}

\begin{proof}
  Assume that (1) does not hold. If the Grushko decomposition of $A$ relative to $\calp$ is non-trivial (hence sporadic),
  then it contains a unique reduced free splitting (this is the splitting with a single orbit of edges),
  so Assertion (2) holds.
  Thus, we may assume that $A$ is one-ended relative to $\calp$. This allows to consider the canonical $\Zmax$ JSJ decomposition of $(A,\calp)$, and we will prove that it is non-trivial. Since $\Out(A,\calp^{(t)})$ is infinite,
  there exists a non-trivial $\Zmax$-splitting of $A$ relative to $\calp$ by Paulin's argument  and Rips theory (\cite{Pau2,BF95},   see for instance \cite[Theorem~7.14]{GL_automorphisms}
   for this precise statement).
   If the canonical $\Zmax$ JSJ decomposition of $A$ relative to $\calp$ is trivial, then the pair $(A,\calp)$ has to be QH with sockets
   (see \cite{Sel} or Proposition \ref{prop:QH})
  and in particular, every group in $\calp$ is cyclic. This is a contradiction.
\end{proof}

\begin{cor}\label{cor:biflexible2nice}
To any biflexible splitting of $F_N$, one can associate canonically (in an $\Out(F_N)$-equivariant way) a
nice splitting.
\end{cor}

\begin{proof}
  Let $S$ be a biflexible splitting, and let $v_1,\dots, v_k$ be the collection of flexible vertices, with $k\geq 2$.
  We construct the canonical nice splitting $S'$ as follows. 
  
  If there are at least two vertices $v_i$ such that the Grushko decomposition of $(G_{v_i},\Inc_{v_i})$ is non-sporadic, then $S$ is bi-nonsporadic and we take $S'=S$.
  Thus, we may assume that for some $i\leq k$, the Grushko decomposition of $(G_{v_i},\Inc_{v_i})$ is trivial or sporadic.

  If the Grushko decomposition of $(G_{v_i},\Inc_{v_i})$ is non-trivial but sporadic for some $i$,
  then consider $\hat S$ obtained from $S$ by blowing up each such vertex $v_i$ into the corresponding canonical free splitting of $G_{v_i}$ relative to $\Inc_{v_i}$
  (there is no choice in the attaching points of the edges of $S$ because the groups in $\Inc_{v_i}$ are non-trivial).
  Collapsing all edges of $\hat S$ coming from $S$ yields a non-trivial free splitting $S'$ that is invariant under the stabilizer of $S$.

  If the Grushko decomposition of $(G_{v_i},\Inc_{v_i})$ is trivial for some $i\leq k$, then
   the canonical $\Zmax$ JSJ decomposition of $G_{v_i}$ relative to $\Inc_{v_i}$ is non-trivial by Lemma \ref{lem:flex}.
So consider $\hat S$ obtained from $S$ by blowing up each such vertex $v_i$ using this JSJ decomposition
  (there is no choice in the attaching points of the edges of $S$ because the groups in $\Inc_{v_i}$ are non-abelian).
  Collapsing all edges of $\hat S$ coming from $S$ yields a non-trivial $\Zmax$-splitting $S'$ that is invariant under the stabilizer of $S$.
\end{proof}

\begin{lemma}\label{lemma:binonsporadic2F}
To any bi-nonsporadic splitting of $F_N$, one can canonically associate a filling finite set of non-empty free factor systems.

In particular, a subgroup $H\subseteq\IA$ preserving a bi-nonsporadic splitting of $F_N$ cannot be pure.
\end{lemma}

\begin{proof}
  Let $S$ be a bi-nonsporadic splitting, and $v_1,\dots,v_n$ with $n\geq 2$ be the collection of vertices of $S$ such that
  the Grushko decomposition of   $(G_{v_i},\Inc_{v_i})$ is non-sporadic.
  For each $i\leq n$, let $\hat S_i$ be obtained from $S$ by blowing up $v_i$ into a Grushko decomposition of $G_{v_i}$ relative to $\Inc_{v_i}$
  and let $S'_i$ be the free splitting obtained by collapsing all edges coming from $S$. Finally, let $\calf_i$ be the free factor system consisting
  of non-trivial vertex stabilizers of $S'_i$.
  Note that $\calf_i$ does not depend on the choice of the Grushko decomposition used to blow up $S$.
  
  We claim that $\{\calf_1,\dots,\calf_n\}$ is filling. Otherwise, consider a free factor system
  $\calf$ such that $\calf_i\sqsubseteq\calf$ for all $i\leq n$.
  Note that for every vertex $v\in S$ not in the orbit of $v_i$, the group $G_{v}$ is $\calf_i$-peripheral.
  Since $n\geq 2$, this implies that $G_{v}$ is $\calf$-peripheral for every vertex $v$ of $S$.
  Let $R$ be a non-trivial free splitting whose conjugacy classes of non-trivial elliptic subgroups are precisely the groups in $\calf$.
  Since vertex stabilizers of $S$ are elliptic in $R$, there is an equivariant map $f:S\ra R$.
  Since edge stabilizers of $R$ are trivial, $f$ maps every edge of $S$ to a point, so $f$ is constant.
  Since $R$ is non-trivial, this is a contradiction and concludes the proof.
\end{proof}

We now prove the main result of this section.

\begin{proof}[Proof of Theorem \ref{thm_nice}]
  Note that it suffices to assign to $\calc$ a splitting which is either bi-nonsporadic or which has
  an edge with trivial or $\Zmax$ stabilizer. Indeed, in the latter case, collapsing all edges with non-trivial stabilizer
  or all edges whose stabilizer is not $\Zmax$ accordingly yields a nice splitting.

  Assume first that $\Gamma_\calc$ preserves a non-trivial free splitting.
  Then Corollary \ref{cor:uh1-biflexible} assigns to $\Gamma_\calc$ a canonical splitting $U^1_{\Gamma_\calc}$
  which has an edge with trivial or $\Zmax$-stabilizer, or which is biflexible.
  As noted above, we may assume that $U^1_{\Gamma_\calc}$ is biflexible, and we conclude using Corollary \ref{cor:biflexible2nice}.
  
From now on, we will assume that $\Gamma_\calc$ preserves no non-trivial free splitting.
In particular, $\calc$ contains no free splitting.
  
If $\calc$ contains some non-trivial $\Zmax$ splitting, then 
Theorem \ref{theo:JSJ_zmax}, applied to $H=\Gamma_\calc$, provides a canonical non-trivial $\Zmax$ splitting $U_\calc:=\UZ_{\Gamma_\calc}$. 
  This proves the theorem in the case where $\calc$ contains is a collection of non-trivial $\Zmax$ or free splittings.


Before treating the case where $\calc$ consists of bi-nonsporadic splittings, 
we consider the case where $\calc$ is a filling collection of \afs{}s whose elementwise stabilizer $\Gamma_\calc$
  is infinite. 
  Let $\calc'$ be the collection of all maximal $\Gamma_\calc$-invariant proper \afs{}s. Note that for any $\HF$ in $\calc$,
  there is $\HF'\in\calc'$ with $\HF\sqsubseteq\HF'$. Since $\calc$ is filling, this implies that $\#\calc'\geq 2$, i.e.\
  $\Gamma_\calc$ is not pure (Definition \ref{dfn_pure}).
 Since $\Gamma_\calc$ does not preserve any non-trivial free splitting, we can apply Theorem~\ref{thm_dynamic_dec} to $H=\Gamma_\calc$ and get a (canonical) splitting $\Ud_{\Gamma_\calc}$. 
  By Remark \ref{rk_binonsporadic}, it is non-trivial and either it has an edge with $\Zmax$ stabilizer or
  else it is bi-nonsporadic (see also Remark~\ref{rk_edge_fg} for the fact that its edge stabilizers are finitely generated).
 
  Finally, we consider the case where $\calc$ is a collection of bi-nonsporadic splittings.
  Using Lemma \ref{lemma:binonsporadic2F}, consider for each $S\in \calc$ the filling finite collection
  $\calc_S$ of free factor systems associated to $S$.
  Each free factor system in $\calc_S$ is $\Gamma_\calc$-invariant because $\calc_S$ is finite and $\Gamma_\calc\subset \IA$. 
  Let $\calc'=\cup_{S\in \calc}\calc_S$.
  Then $\calc'$ is a filling collection of free factor systems canonically associated to $S$
  and whose elementwise stabilizer contains $\Gamma_\calc$. One can then define $U_\calc:=U_{\calc'}$.
\end{proof}


\section{Stabilizers of arational trees and associated canonical splittings}\label{sec:witness}

In this section, we associate to any projective arational tree with non-amenable stabilizer, 
a compatible splitting that witnesses that its stabilizer is infinite. 
More generally, to any finite set of projective arational trees whose common stabilizer is non-amenable, we associate a splitting that is compatible with all these trees (see Proposition \ref{prop:PATna}).
We proceed in two steps: in Subsection \ref{sec_witness_isometric} we deal 
with trees whose isometric stabilizer is infinite
and we deduce Proposition \ref{prop:PATna} in Subsection \ref{sec:witness2}.

In all this section, we fix a non-sporadic free factor system $\calf$
and work relative to $\calf$. In particular $\AT$ is the set of arational trees
relative to $\calf$, endowed with its action of $\Out(F_k;\calf)$.

\subsection{A witness map for trees with infinite isometric stabilizer}\label{sec_witness_isometric}


Given a tree $T\in \AT$, there are three naturally associated stabilizers:
the stabilizers of $T$, $[T]$ and $[[T]]$ for the action of
$\Out(F_k;\calf)$ on $\AT$, $\PAT$ and $\PAT/{\sim}$ respectively.
For disambiguation purposes, we sometimes say that the stabilizer of $T$ is its
\emph{isometric stabilizer}.
 Note that if $T,T'\in \AT$ represent the same projective tree $[T]=[T']\in\PAT$, then
 $T$ and $T'$ have the same isometric stabilizer, and the same compatible splittings.

In the following proposition, we denote by $\calp_{<\infty}(\PAT)$ the set of all non-empty finite sets of trees in $\PAT$. The definition of a biflexible splitting was given in Definition~\ref{de:biflexible}.

\begin{prop}\label{prop:S_T}
Let $\calf$ be a non-sporadic free factor system of $F_k$.

  There exists an $\Out(F_k,\calf)$-equivariant Borel map  
assigning to any finite set of trees $[\TT]=\{[T_1],\dots,[T_p]\}\in\calp_{<\infty}(\PAT)$
a splitting $S^0_{[\TT]}$ compatible with $T_1,\dots, T_p$, such that:
\begin{itemize}
\item if the common isometric stabilizer of $T_1,\dots,T_p$
is infinite, then $S^0_{[\TT]}$ is a  non-trivial splitting which is either biflexible or has an edge with $\Zmax$ stabilizer;
\item if the common isometric stabilizer of $T_1,\dots,T_p$
is finite, then $S^0_{[\TT]}$ is trivial.
\end{itemize}
\end{prop}

The proof follows the ideas introduced in \cite{GL-on}, but we give a self-contained proof.

Recall that we say that a subgroup $H$ of $\ia$ acts trivially on a subgroup $A$
if every $\alpha\in H$ has a lift in $\Aut(F_k)$ acting as the identity on $A$.

\begin{lemma}\label{lemma:maximal-fix}
Let $H\subseteq\ia$ be a subgroup, and assume that $H$ is contained in the isometric stabilizer of some tree in $\AT$. Then up to conjugacy, there exists a unique maximal subgroup $A_{H}\subseteq F_k$ which is  non-abelian and not $\calf$-peripheral
and such that $H$ acts trivially on  $A_{H}$.

Moreover, consider any $T\in\AT$ such that $H$ is contained in the isometric stabilizer of $T$, and let $Y\subset T$ be the minimal $A_H$-invariant subtree. Then the translates of $Y$ form a transverse covering of $T$, and $A_H$ is the setwise stabilizer of $Y$. 
\end{lemma}

\begin{proof}
Consider $T\in\AT$ such that $H$ is contained in the isometric stabilizer of $T$.
We are going construct a group $A_H$ using $T$.
Given $\alpha\in H$ and a representative $\tilde\alpha\in \Aut(F_k)$, we denote by $I_{\tilde \alpha}$
the $\tilde\alpha$-equivariant isometry of $T$.
Let $\eta$ be a germ of direction at a branch point of $T$.
Let $H^0$ be the subgroup of $H$ consisting of elements preserving the $F_k$-orbit of $\eta$.
As there are finitely many $F_k$-orbits of directions at branch points in $T$ (\cite{GL}), the subgroup $H^0$ has finite index in $H$.
We will see below that in fact $H^0=H$.
For each $\alpha\in H^0$, there exists a unique representative $\tilde \alpha\in \Aut(F_k)$
such that $I_{\tilde \alpha}$ fixes $\eta$.
The set of all such representatives $\tilde \alpha$ defines a lift $\tilde H^0$
of $H^0$ in $\Aut(F_k)$. 

Given $\tilde\alpha_1,\dots,\tilde\alpha_n\in \tilde H^0$,
let $Y_{\tilde\alpha_1,\dots,\tilde\alpha_n}$ be the set of points in $T$ fixed by
$I_{\tilde\alpha_i}$ for all $i\leq n$, and let 
$A_{\tilde \alpha_1,\dots,\tilde \alpha_n}=\Fix_{F_k}(\grp{\tilde \alpha_1,\dots,\tilde\alpha_n})$ be 
 the set of elements of $F_k$ fixed by $\tilde \alpha_i$ for all $i\leq n$. 
The subtree $Y_{\tilde\alpha_1,\dots,\tilde\alpha_n}$ contains at least a segment representing the direction $\eta$.

\begin{claim*}
  If $g\in F_k$ is such that $g.Y_{\tilde\alpha_1,\dots,\tilde\alpha_n}\cap Y_{\tilde\alpha_1,\dots,\tilde\alpha_n}$ contains a non-degenerate segment,
  then $g\in A_{\tilde \alpha_1,\dots,\tilde \alpha_n}$.
\end{claim*}

\begin{proof}[Proof of the claim]
The subtree $g.Y_{\tilde\alpha_1,\dots,\tilde\alpha_n}\cap Y_{\tilde\alpha_1,\dots,\tilde\alpha_n}$
is fixed by $I_{\tilde \alpha_i}$ and by  $gI_{\tilde\alpha_i} g\m$ hence by
$gI_{\tilde\alpha_i} g\m I_{\tilde \alpha_i\m}=g\tilde\alpha_i(g\m)$. 
Since $F_k$-stabilizers of arcs of $T$ are trivial, it follows that  $\tilde\alpha_i(g)=g$.
We conclude that $g\in  A_{\tilde \alpha_1,\dots,\tilde \alpha_n}$.
\end{proof}

Let now $J\subset Y_{\tilde\alpha_1,\dots,\tilde\alpha_n}$ be a non-degenerate segment. 
Since $T$ is mixing (\cite[Proposition~8.3]{Rey}, \cite[Lemma 4.9]{Hor}), 
$J$ contains infinitely many germs of segments representing a direction in the $F_k$-orbit of $\eta$.
The claim then implies that $A_{\tilde \alpha_1,\dots,\tilde \alpha_n}$ contains elements of arbitrary small positive translation length,
and is therefore is non-abelian.

By the bounded chain condition on fixed subgroups in $F_k$ \cite{MV}, there exists 
$\tilde\alpha_1,\dots,\tilde\alpha_n\in \tilde H^0$ such that $A_{\tilde\alpha_1,\dots,\tilde\alpha_n}=A_{\tilde\alpha_1,\dots,\tilde\alpha_n,\tilde\alpha}$
for any $\tilde\alpha\in\tilde H^0$.
We define $A_H=A_{\tilde\alpha_1,\dots,\tilde\alpha_n}$
and we note that $H^0$ acts trivially on $A_H$ (but we don't yet know that $A_H$ does not depend on $T$).

We now use the fact $H$ is contained in $\ia$ to deduce that $H$ itself acts  trivially on $A_H$.
Since $H^0$ preserves the conjugacy class of every element $g\in A_H$,
its orbit under every $\alpha\in H$ is finite. Since $\alpha\in\ia$, this implies that $\alpha$ preserves $[g]$ (Theorem~\ref{theo:ia-element} by Handel and Mosher). Since this holds for every $g\in A_H$, \cite[Corollary~1.6]{GL4} implies that $\alpha$ acts trivially on $A_H$, hence so does $H$.

In order to prove the first assertion of the lemma, there remains to prove that 
if $H$ acts trivially on some subgroup  $A'\subseteq F_k$ which is not abelian and not $\calf$-peripheral, then $A_H$ contains $A'$ up to conjugacy.
Since the centralizer of $A'$ in $F_k$ is trivial, 
every $\alpha\in H$ has a unique lift $\tilde \alpha\in \Aut(F_k)$
acting as the identity on $A'$.
This defines a lift
$\tilde H' $ of $H$ in $\Aut(F_k)$ with $A'\subseteq \Fix_{F_k} \tilde H'$.
Since $A'$ is non-abelian and not $\calf$-peripheral, $A'$ is not elliptic in $T$
and its minimal subtree $Y'\subset T$ is not a line.
Since $\tilde H'$ fixes $A'$, for every $\tilde\alpha\in \tilde H'$, 
$I_{\tilde\alpha}$ preserves the axis of every element of $A'$ and therefore acts as the identity on  $Y'$.
Since $T$ is mixing, up to conjugating $A'$ and $\tilde H'$ by some element of $F_k$,
we may assume that $Y'$ contains 
a representative of $\eta$, so $\tilde H'\subset \tilde H^0$.
Since $\tilde H'$ is a lift of $H$, it follows that $\tilde H'= \tilde H^0$ 
and $A'\subset A_H$ (and it also follows that $H^0=H$).
This proves the maximality and the uniqueness of $A_H$.

The fact that translates of the minimal $A_H$-invariant tree $Y$ are transverse and that $A_H$ is the setwise stabilizer of $Y$ follows from the claim above.
Since the action of $F_k$ on $T$ is mixing, 
this implies that the translates of $Y$ form a transverse covering.
\end{proof}

We denote by $S_{T,H}$ the skeleton of the transverse covering of $T$ by the translates of $Y$.
Its vertex set is $V_0\dunion V_1$ where $V_1$ is the set of translates of $Y$,
$V_0$ is the set of points  $x\in T$ lying in at least two distinct translates of $Y$, and there is an edge between $x$ and $gY$
if and only if $x\in gY$. 
Since $T$ and the conjugacy class of $A_H$ are $H$-invariant, $S_{T,H}$ is $H$-invariant.

Stabilizers of vertices in $V_0$ are point stabilizers in $T$ (as $T$ is arational, these are either groups in $\calf$ or maximal cyclic), and stabilizers of vertices in $V_1$ are conjugates of $A_H$ (which are fixed subgroups of $F_k$). In particular, vertex stabilizers of $S_{T,H}$ are finitely generated, hence so are edge stabilizers by Howson's property.
It follows that each edge stabilizer $G_e$ is its own normalizer: being a finitely generated subgroup of $F_k$,
$G_e$ has finite index in its normalizer, and since vertex stabilizers are stable under taking roots,
so is $G_e$.

\begin{rk}\label{rk_non-discrete}
The moreover part of Lemma~\ref{lemma:maximal-fix} implies that $A_H$ acts non-discretely on $T$.
Indeed, $T$ splits as a graph of actions over $S_{T,H}$ with vertex actions being
the action of a conjugate of $A_H$ on the corresponding translate of $Y$ 
(for vertices in $V_1$), and trivial actions on points (for vertices in $V_0$); see \cite[Lemma~4.7]{Gui04}. If the action of $A_H$ on $T$ was discrete, then the action of $F_k$ on $T$ would also be discrete, contradicting the arationality of $T$.
\end{rk}

\begin{lemma}\label{lemma:edges-sth}
Edge stabilizers of $S_{T,H}$ are non-trivial and $\calf$-peripheral.
\end{lemma}

\begin{proof}
Assume first that $e$ is an edge of $S_{T,H}$ with trivial stabilizer. 
By collapsing all edges of $S_{T,H}$ outside of the orbit of $e$, we obtain a free splitting 
of $(F_k,\calf)$ in which $A_H$ is elliptic, 
which yields a free factor of $(F_k,\calf)$ containing $A_H$ and
acting non-discretely on $T$. This is a 
contradiction to Remark~\ref{rk_non-discrete}. 

Note that edge stabilizers of $S_{T,H}$ are elliptic in $T$. Thus, if some edge $e$ of $S_{T,H}$ has non-peripheral stabilizer, then Lemma~\ref{lemma:arational-stabilizers} implies that the stabilizer of $e$ is cyclic. 
Thus, collapsing all edges of $S_{T,H}$ outside of the orbit of $e$, 
we see that the group $A_H$ is contained in a proper $\calz$-factor in the terminology of \cite[Section~11.4]{GH1}, and \cite[Proposition~11.5]{GH1} contradicts that $A_H$ acts non-discretely on $T$.
%
\end{proof}

\begin{lemma}\label{lemma:alt-desc}
  The tree $S_{T,H}$ does not depend on $T$: if $T'\in \AT$ is $H$-invariant,
  then $S_{T,H}=S_{T',H}$.
\end{lemma}

\begin{proof}
We give an alternative description of $S_{T,H}$, independent of $T$. 
Let $V_1'$ be the set of conjugates of $A_H$, and $V_0'$  be the set of free factors whose conjugacy class belongs to $\calf$ and that intersect 
at least two distinct conjugates of $A_H$ non-trivially.
We put an edge between $A_H^g\in V'_1$ and $P\in V'_0$ if and only if $A_H^g\cap P\neq \{1\}$.
Clearly, the graph $S'$ defined in this way does not depend on $T$.
We claim that $S'$ is equivariantly isomorphic to $S_{T,H}$. 

 We are going to show that there is a well defined map $f:S_{T,H}\ra S'$ sending a vertex $v$ to its stabilizer $G_v$.
The map is well defined on $V_1(S_{T,H})$ by sending $v$ to $G_v\in V'_1$.


The fact that edge stabilizers of $S_{T,H}$ are non-trivial (Lemma~\ref{lemma:edges-sth}) implies that if $v_0\in V_0(S_{T,H})$ then $G_{v_0}$ intersects at least two distinct conjugates of $A_H$
non-trivially, namely the stabilizers of the neighbors of $v_0$.
It also follows that $G_{v_0}$ is non-trivial, and we claim that the conjugacy class of  $G_{v_0}$ lies in $\calf$. If the claim does not hold, then since $G_{v_0}$ is a point stabilizer in $T$, 
Lemma~\ref{lemma:arational-stabilizers} ensures that $G_{v_0}$ is cyclic and non-$\calf$-peripheral.
In particular some edge stabilizer of $S_{T,H}$ is non-$\calf$-peripheral, 
contradicting Lemma \ref{lemma:edges-sth}.
This proves that $G_{v_0}\in V'_0$ and that $f$ is well defined and sends edge to edge.

Since $f$ is clearly injective and surjective on the set of vertices, there remains to show that $f$
preserves non-adjacency. It suffices to prove that for $u\in V_0$ , $v\in V_1$ at distance at least 3,
then  $G_u\cap G_v=\{1\}$. Consider $u'\in [u,v]\cap V_0$ at distance 1 from $v$. Since $\calf$ is a malnormal family,
$G_u\cap G_{u'}=\{1\}$ so $G_u\cap G_v=\{1\}$. This concludes the proof.
\end{proof}

\begin{lemma}\label{lemma:biflexible-trivial}
  For $\TT=\{T_1,\dots,T_p\}\in\calp_{<\infty}(\AT)$, denote by $\Gamma_\TT$
  the intersection of the isometric stabilizers
  of $T_1,\dots, T_p$ in $\ia$. Let $S^0_\TT=S_{T_i,\Gamma_\TT}$ for any $i\leq p$.
The following are equivalent:
  \begin{enumerate}
  \item $\Gamma_\TT$ is infinite; 
  \item $S^0_{\TT}$ is a non-trivial splitting of $F_k$;
  \item $S^0_{\TT}$ is a non-trivial splitting of $F_k$ which is biflexible or has an edge with $\Zmax$ stabilizer.
  \end{enumerate}
\end{lemma}

\begin{proof}
  Clearly, $3\imp 2$.
If $\Gamma_\TT$ is finite, then it is trivial because $\ia$ is torsion-free, so its fixed subgroup $A_{\Gamma_\TT}$ is $F_k$.
It follows that $S^0_\TT$ is a point.
This proves $2\imp 1$.

Now assume that $\Gamma_\TT$ is infinite. Then $A_{\Gamma_\TT}$ is a proper subgroup of $F_k$.
This implies that $F_k$ does not fix a vertex in $S^0_\TT$, so $S^0_\TT$ is non-trivial.
We know that edge stabilizers of $S^0_\TT$ are non-trivial (Lemma~\ref{lemma:edges-sth}), and they are root-closed because vertex stabilizers are.

 So we assume that all edge stabilizers are non-abelian, and prove that $S^0_\TT$ is biflexible.
 We first note that the group of twists of $S^0_\TT$ is trivial.
 The vertex set  $\bar V$ of $S^0_\TT/F_k$ is $\{v_1\}\cup \bar V_0$ where 
 $G_{v_1}$ is conjugate to $A_{\Gamma_\TT}$, and $\bar V_0$ is the image of $V_0$ in $S^0_\TT/F_k$.
 Since $S^0_\TT$ is ${\Gamma_\TT}$-invariant, let $\Gamma_\TT^0$ the finite index subgroup of $\Gamma_\TT$
 acting as the identity on the quotient graph $S^0_\TT/F_k$.
 By Proposition~\ref{prop_suite_exacte},
 the natural map $$\rho:\Gamma_\TT^0\ra \prod_{v\in \bar V} \Out(G_v,\Inc_v)$$
 is injective. 
 Since $\Gamma_\TT$ acts trivially on $A_{\Gamma_\TT}$, the image
 $\rho(\Gamma_\TT^0)$ is contained in $\prod_{v\in \bar V_0} \Out(G_v,\Inc_v)$.
 But since each edge stabilizer of $S^0_\TT$ is contained in a conjugate of $A_{\Gamma_\TT}$,
 $\Gamma_\TT$ also acts trivially on all edge groups,
 so $\rho(\Gamma_\TT^0)$ is in fact contained in $\prod_{v\in \bar V_0} \Out(G_v,\Inc_v^{(t)})$ (see Remark~\ref{rk_normalizer}).
 It follows that there exists at least
one vertex $v_0\in \bar V_0$ such that $\Out(G_{v_0};\Inc_{v_0}^{(t)})$ is infinite.

Consider the free factor system 
$\calf_{|G_{v_1}}$ of $G_{v_1}$ (which is conjugate to $A_{\Gamma_\TT}$). The tree $Y$ from Lemma~\ref{lemma:maximal-fix} is a $(G_{v_1},\calf_{|G_{v_1}})$-tree 
with dense orbits and trivial arc stabilizers. 
In particular, $G_{v_1}$ is not cyclic and $\calf_{|G_{v_1}}$ is a proper free factor system.
Since $\Inc_{v_1}$ is $\calf_{|G_{v_1}}$-peripheral (Lemma~\ref{lemma:edges-sth}), it follows that $G_{v_1}$ has a non-trivial Grushko decomposition relative to incident edge groups.
 Therefore $S^0_\TT$ is biflexible.
\end{proof}

View $\calp_{<\infty}(\AT)$ as the disjoint union over $n$ of $\AT^n/\mathfrak{S}_n$,
and endow it with the corresponding Borel structure.

\begin{lemma} \label{lem:S_T_Borel}
The map $\TT\in\calp_{<\infty}(\AT) \mapsto S^0_{\TT}$ is a Borel map.
\end{lemma}

\begin{proof}
 Using the description of $S^0_\TT$ given in the proof of Lemma \ref{lemma:alt-desc},
    it suffices to check that the map which sends $\TT$ to the conjugacy class of $A_{\Gamma_\TT}$ is Borel. To this end, it is enough to check that for every finitely generated non-$\calf$-peripheral non-abelian subgroup $A\subseteq F_k$, the set of all $\TT\in\calp_{<\infty}(\AT)$ such that $A$ is conjugate to a subgroup of $A_{\Gamma_\TT}$ is Borel. 
  
By Lemma \ref{lemma:maximal-fix}, the group $A$ is conjugate into $A_{\Gamma_\TT}$ if and only if $\Gamma_\TT\subseteq\Gamma_A$, where $\Gamma_A$ is the subgroup of all elements of $\ia$ acting trivially on $A$. This in turn is equivalent to requiring that for every $\alpha\notin\Gamma_A$, there exists $T\in\TT$ such that $\alpha\cdot T\neq T$, a Borel condition.
  \end{proof}  

\begin{proof}[Proof of Proposition \ref{prop:S_T}]
  The map $\TT\mapsto S^0_\TT$ is an equivariant Borel map on $\calp_{<\infty}(\AT)$.
  Since any two homothetic trees have the same isometric stabilizer $\Gamma_\TT$ and since
  $S^0_\TT=S_{T,\Gamma_\TT}$ for any tree $T\in\AT$ whose isometry class is fixed by $\Gamma_\TT$, this map passes to the quotient into a map on $\calp_{<\infty}(\PAT)$. 
  Also, as $S^0_\TT$ is the skeleton of a transverse covering of any tree in $\TT$, it is compatible with every tree in $\TT$ (see e.g.\ \cite[Lemma~4.7]{Gui04}).
  By Lemma~\ref{lemma:biflexible-trivial}, $S^0_\TT$ is biflexible or has a $\Zmax$ edge stabilizer
   if the common isometric
  stabilizer of the trees in $\TT$ is infinite.
\end{proof}

\subsection{A witness map for trees with non-amenable stabilizer}\label{sec:witness2}

In this section, we deduce from Proposition~\ref{prop:S_T} a witness map
on the set of (finite collections of) projective arational trees
with non-amenable stabilizer.

Given a finite collection of arational trees $\TT=\{T_1,\dots,T_n\}\subset\AT$,
we denote by $[\TT]=\{[T_1],\dots,[T_n]\}$ its image in $\PAT$.

\begin{lemma}\label{lemma:stab-amenable}
  For $\TT\in\calp_{<\infty}(\AT)$, the following are equivalent:
  \begin{enumerate}
  \item $\Stab(\TT)$ is amenable;
  \item $\Stab([\TT])$ is amenable;
  \end{enumerate}
\end{lemma}

\begin{proof}
As $\Stab(\TT)\subseteq\Stab([\TT])$, 
it suffices to prove $1\Rightarrow 2$.
Write $\TT=\{T_1,\dots,T_p\}$. 
Assume that $\Stab(\TT)$ is amenable,
and let $\Stab^0([\TT])$ be the finite-index subgroup made of all elements that 
fixes each $[T_i]$ (as opposed to permuting them).

The group $\Stab^0([\TT])$ acts by homothety on each of the trees $T_{i}$, 
and the scaling factors define a homomorphism from $\Stab^0([\TT])$ to the abelian group $(\bbR_+^*)^p$. The kernel $K$ of this homomorphism acts by isometry on each tree $T_{i}$ 
so $K\subseteq\Stab(\TT)$ is amenable. We deduce that $\Stab^0([\TT])$ and $\Stab([\TT])$ are amenable. This concludes our proof.
\end{proof}

We denote by 
$\PnaPAT$ 
the set of non-empty finite subsets $[\TT]$ of $\PAT$
whose stabilizer $\Stab([\TT])$ is non-amenable.

\begin{prop}\label{prop:PATna}
Let $\calf$ be a non-sporadic free factor system of $F_k$.

  There exists an $\Out(F_k,\calf)$-equivariant Borel map assigning to any finite set $[\TT]\in\PnaPAT$ 
a splitting $S^1_{[\TT]}$ 
which is either biflexible or a $\Zmax$-splitting, and 
which is compatible with every tree in $\TT$.
\end{prop}





\begin{proof} 
 Given $[\TT]=\{[T_1],\dots,[T_p]\}\in \PnaPAT$,
 Lemma \ref{lemma:stab-amenable} implies that the common isometric stabilizer of $T_1,\dots,T_p$
 is infinite. 
 Applying Proposition \ref{prop:S_T}, we get a non-trivial splitting $S_{[\TT]}$
 which is either biflexible or has an edge with $\Zmax$ stabilizer.
 We define $S^1_{[\TT]}=S^0_{[\TT]}$ if it is biflexible. If $S^0_{[\TT]}$ has an edge $e$ with $\Zmax$ stabilizer, we define $S^1_{[\TT]}$ by collapsing every edge not in the orbit of $e$ to a point.
\end{proof}

Combining Proposition~\ref{prop:PATna} with Corollary~\ref{cor:biflexible2nice}, we reach the following statement.

\begin{cor}\label{cor:PATna}
Let $\calf$ be a non-sporadic free factor system of $F_k$. There exists an $\Out(F_k,\calf)$-equivariant Borel map assigning to any finite set $[\TT]\in\PnaPAT$ 
a nice splitting $S_{[\TT]}$. \qed
\end{cor}

\newpage

\part{Measure equivalence rigidity of $\Out(F_N)$}\label{part_ME}

\section{Canonical invariant splittings for groupoids}\label{sec_canonical}

In this section, we fix a finitely generated free group $F_k$ with $k\ge 2$. In the sequel of the paper, we will usually apply the results of this section to the case where $F_k=F_N$ with $N\geq 3$, but we will also need to consider the case where $F_k$ is a corank one free factor of $F_N$ in Section~\ref{sec:proof-2}.

Let $\calg$ be a measured groupoid over a base space $Y$, equipped with a cocycle $\rho:\calg\to\ia$. The goal of the present section is to give a criterion ensuring (roughly) that if a subgroupoid $\calh$ has invariant splittings (with respect to the cocycle $\rho$), then so does any subgroupoid that normalizes $\calh$ (see Proposition~\ref{prop:nice-preserved} below for a precise statement). This is done by using all the results obtained in Part~\ref{part2} and transfering them to the groupoid setting.
The chain conditions proved in the previous sections are crucial for this to work.

Let $\Delta$ be a set equipped with an action of $\Out(F_k)$, for instance $\Delta$ could be a set of splittings or (almost) free factor systems.
Recall from Definition \ref{dfn_invariant} that an element $S\in \Delta$ is  \emph{$(\calg,\rho)$-invariant}  if
there  is a conull Borel subset $Y^*\subseteq Y$ such that
$\rho(\calg_{|Y^*})\subseteq\Gamma_S$, where $\Gamma_S$ denotes the stabilizer of $S$ in $\ia$. Recall that nice splittings were defined in Definition \ref{dfn_nice}.

\begin{de}[Stably nice, nice-averse and FS-averse cocycle]\label{dfn_averse_cocycle}
Let $\calg$ be a measured groupoid over a base space $Y$, and let $\rho:\calg\to\ia$ be a cocycle. We say that $(\calg,\rho)$ is
\begin{itemize}
\item \emph{stably nice} if there exists a partition  
  $Y=\Dunion_{i\in I}Y_i$ 
  into at most countably many Borel subsets such that for every $i\in I$,
  there is a $(\calg_{|Y_i},\rho)$-invariant nice splitting of $F_k$;
\item \emph{nice-averse} if for every Borel subset $U\subseteq Y$ of positive measure, there is no $(\calg_{|U},\rho)$-invariant nice splitting of $F_k$.
\item \emph{$\FS$-averse} if for every Borel subset $U\subseteq Y$ of positive measure, there is no $(\calg_{|U},\rho)$-invariant  non-trivial free splitting of $F_k$.
\end{itemize}
\end{de}

Note that a nice-averse cocycle is in particular $\FS$-averse.

We extend the definition of a pure subgroup of $\ia$ (Definition~\ref{dfn_pure}) to groupoids with a cocycle towards $\ia$.

\begin{de}[Pure cocycle]\label{dfn_pure_cocycle}
Let $\calg$ be a measured groupoid over a base space $Y$, equipped with a cocycle $\rho:\calg\to\ia$.

  We say that $(\calg,\rho)$ is \emph{pure} if $(\calg,\rho)$ is $\FS$-averse and there exists a
  $(\calg,\rho)$-invariant almost free factor system $\hat{\calf}$  such that for every Borel subset $U\subseteq Y$ of positive measure, and every $(\calg_{|U},\rho)$-invariant almost free factor system $\hat{\calf}'$, one has $\hat\calf'\sqsubseteq\hat\calf$.

  We say that $(\calg,\rho)$ is \emph{stably pure} if there is a partition
  $Y=\Dunion_{i\in I} Y_i$ 
  into at most countably many Borel subsets such that
for every $i\in I$,  $(\calg_{|Y_i},\rho)$ is pure.
\end{de}

\begin{rk}\label{rk_pure_nonsporadic}
If $(\calg,\rho)$ is pure, then the corresponding \afs\ $\HF$ is necessarily unique.
 It is also non-sporadic because the stabilizer of a sporadic free factor system preserves a free splitting (see Section \ref{sec:factor-systems}). We also insist that, as usual, the \afs\ $\HF=\es$ is allowed in this definition.
\end{rk}

\begin{rk}
  Note that if $Y^*\subset Y$ is a conull Borel subset such that $(\calg_{|Y^*},\rho)$ is stably nice or stably pure
  then so is $(\calg,\rho)$.
\end{rk}

The  goal of this section is to prove the following two propositions,
which are groupoid analogues of Theorem \ref{thm_nice} and Corollary \ref{cor_nice}. 

\begin{prop}\label{prop:stably-nice-vs-pure-nice-averse}
Let $\calg$ be a measured groupoid over a base space $Y$, and let $\rho:\calg\to\ia$ be a cocycle. 

There exists a Borel partition $Y=Y_1\dunion Y_2$ such that $(\calg_{|Y_1},\rho)$ is stably nice, and
$(\calg_{|Y_2},\rho)$ is stably pure and nice-averse.  
\end{prop}

Recall that a cocycle  $\rho$ is \emph{nowhere trivial} if there is no  Borel subset of positive measure of the base space
of $\calg$ in restriction to which $\rho$ is trivial (Definition \ref{dfn:nowhere_trivial}). 

\begin{prop}\label{prop:nice-preserved}
  Let $\calg$ be a measured groupoid over a base space $Y$, with a cocycle $\rho:\calg\to\ia$.
  Let $\calh,\calh'$ be measured subgroupoids of $\calg$. 
  Assume that  $(\calh,\rho)$ is stably nice and that   $\rho_{|\calh}$ is nowhere trivial.

If $\calh'$ stably normalizes $\calh$, then  $(\calh',\rho)$ is stably nice.

More precisely, there exists a partition $Y=\Dunion_{i\in I}Y_i$
into at most countably many Borel subsets such that for every $i\in I$, there exists a nice splitting $S_i$ which is invariant under both $(\calh_{|Y_i},\rho)$ and $(\calh'_{|Y_i},\rho)$.
\end{prop}

\subsection{Stably nice vs pure partition}

After a partition of the base space, $(\calg,\rho)$ is either stably nice or nice-averse, as shown in the following lemma.

\begin{lemma}\label{lemma:stably-nice-vs-nice-averse}
Let $\calg$ be a measured groupoid over a base space $Y$, and let $\rho:\calg\to\ia$ be a cocycle. 

There exists a Borel partition $Y=Y_1\dunion Y_2$ such that $(\calg_{|Y_1},\rho)$ is stably nice, and $(\calg_{|Y_2},\rho)$ is nice-averse.
\end{lemma}

\begin{proof}
Without loss of generality, we assume that the measure on $Y$ is a probability measure (see Remark \ref{rk:proba}).
We choose for $Y_1$ a Borel subset of maximal measure such that $(\calg_{|Y_1},\rho)$ is stably nice: this exists because if $(Y_{1,n})_{n\in\mathbb{N}}$ is a measure-maximizing sequence of such sets, then letting $Y_{1,\infty}=\bigcup Y_{1,n}$, we  have again that $(\calg_{|Y_{1,\infty}},\rho)$ is stably nice. The proof is then completed by letting $Y_2=Y\setminus Y_1$: the maximality of $Y_1$ ensures that $(\calg_{|Y_2},\rho)$ is nice-averse.
\end{proof}

To prove Proposition \ref{prop:stably-nice-vs-pure-nice-averse}, we will show that a nice-averse cocycle is stably pure (Lemma~\ref{lemma:nice-averse_is_stably_pure} below). A key step will be Lemma~\ref{lem_maxFF}, based on the chain condition, showing the existence of a collection of maximal invariant  almost free factor systems, up to partitioning the base space. We first introduce a few definitions.

\begin{de}[Everywhere maximal, dominant, stable reduction collection]
  Let $\calh$ be a measured groupoid over a base space $Y$, let $\rho:\calh\to\ia$ be a cocycle.
  
An $(\calh,\rho)$-invariant \afs\ $\HF$ is \emph{$(\calh,\rho)$-everywhere maximal} if for every Borel subset $U\subseteq Y$ of positive measure, there is no \afs\ $\HF'\sqsupsetneq \HF$ which is $(\calh_{|U},\rho)$-invariant .

A collection $\calc$ of $(\calh,\rho)$-invariant \afs{}s is \emph{$(\calh,\rho)$-dominant}
if for every Borel subset $U\subseteq Y$ of positive measure and every $(\calh_{|U},\rho)$-invariant \afs\ $\HF$, there exists $\HF'\in\calc$ such that $\HF\sqsubseteq\HF'$.

A collection $\calc$ of $(\calh,\rho)$-invariant \afs{}s is a \emph{stable reduction collection}
for $(\calh,\rho)$ if it is $(\calh,\rho)$-dominant and if every $\HF\in\calc$ is $(\calh,\rho)$-everywhere maximal.
\end{de}

\begin{rk} 
If $(\calh,\rho)$ has a stable reduction collection $\calc$, then it is unique.  Indeed, assume that $\calc'$ is another stable reduction collection for $(\calh,\rho)$, and let $\HF\in\calc$. As $\HF$ is $(\calh,\rho)$-invariant and $\calc'$ is $(\calh,\rho)$-dominant, we deduce that there exists $\HF'\in\calc'$ such that $\HF\sqsubseteq\HF'$. As $\HF$ is  $(\calh,\rho)$-everywhere maximal, it follows that $\HF=\HF'$. Hence $\calc\subseteq\calc'$, and by symmetry they are actually equal. 

Moreover, for every Borel subset $U\subseteq Y$ of positive measure, $\calc$ is also a stable reduction collection for $(\calh_{|U},\rho)$.

Note that $(\calh,\rho)$ is pure if and only if it is $\FS$-averse and it has a stable reduction collection $\calc$ consisting of a single almost free factor system (which may, as usual, be the empty \afs ).
\end{rk}

\begin{lemma}\label{lemma:maxFF_unique}
Let $\calg$ be a measured groupoid over a base space $Y$ with a cocycle $\rho:\calg\ra\ia$.
Let $\calh,\calh'$ be two measured subgroupoids.
Assume that $\calc$ is a stable reduction collection for $(\calh,\rho)$ and that $\calh'$ normalizes $\calh$.

Then $\calc$ is $(\calh',\rho)$-invariant.

If $\calc$ is finite, then every element of $\calc$ is  $(\calh',\rho)$-invariant.
\end{lemma}

We mention that finiteness of $\calc$ is in fact automatic when $(\calh,\rho)$ is $\FS$-averse, as will be established in Lemma~\ref{lem_maxFF} below.

\begin{proof} 
  By definition  of normalization of subgroupoids (see Definition~\ref{de:normal} and Remark~\ref{rk:bisection}), one can write $\calh'$
  as a countable union of bisections $B_n\subset \calh'$ such that
for all $h\in\calg$ and all $\phi,\psi\in B_n$ such that $\phi h \psi\m$ is defined, 
one has $h\in\calh$ if and only if $\phi h \psi\m\in \calh$.
Up to subdividing the bisections, we may assume additionally that the cocycle $\rho$ maps $B_n$
to a single value $\gamma_n\in\ia$.
Let $f_{B_n}:U_n\ra V_n$ be the partial isomorphism associated to $B_n$, with $U_n,V_n\subset Y$.

Assume that $U_n$ and $V_n$ have positive measure.
If $\HF$ is any $(\calh_{|U_n},\rho)$-invariant \afs, then $\gamma_n\HF$ is  $(\calh_{|V_n},\rho)$-invariant.
It follows that if $\HF$ is  $(\calh_{|U_n},\rho)$-everywhere maximal, $\gamma_n\HF$ is  $(\calh_{|V_n},\rho)$-everywhere maximal.
Similarly, since $\calc$ is a stable reduction collection for $(\calh_{|U_n},\rho)$,
$\gamma_n\calc$ is a stable reduction collection for $(\calh_{|V_n},\rho)$.
By uniqueness of the stable reduction collection, $\gamma_n\calc=\calc$ if $U_n$ has positive measure.

Discarding the bisections with zero measure, 
since $\calh'$ is the union of the bisections $B_n$, it follows that $\calc$ is $(\calh',\rho)$-invariant.
 Under the additional assumption that $\calc$ is finite, the setwise stabilizer of $\calc$ in $\ia$ coincides with its elementwise stabilizer
by Theorems~\ref{theo:ia} and~\ref{theo:ia-element}. We thus conclude that every $\HF\in\calc$ is $(\calh',\rho)$-invariant.
\end{proof}

Under the assumption that $(\calh,\rho)$ is $\FS$-averse, the following lemma shows the existence of stable reduction collections up to partitioning the base space,
and shows that such collections are finite.

\begin{lemma}\label{lem_maxFF}
 Let $\calh$ be a measured groupoid over a base space $Y$, equipped with a cocycle $\rho:\calh\to\ia$. Assume that $(\calh,\rho)$ is $\FS$-averse.

 Then there exists a countable Borel partition $Y=\Dunion_{i\in I} Y_i$
 such that for every $i\in I$, there exists a stable reduction collection $\calc_i$ for $(\calh_{|Y_i},\rho)$, and $\calc_i$ is finite and nonempty.
\end{lemma}

\begin{rk}
The non-emptiness of $\calc_i$ means that it contains at least one free factor system, 
this free factor system is allowed to be the empty free factor system. 
In other words $\calc_i=\{\emptyset\}$ is allowed.
\end{rk}
 
\begin{proof}
  We first claim that for every subset $Y'\subseteq Y$ of positive measure,
  there are at most $\max\{1,\frac{k-1}{2}\}$
   \afs{}s which are both $(\calh_{|Y'},\rho)$-invariant and $(\calh_{|Y'},\rho)$-everywhere maximal.
  Indeed, assume that $\HF_1$,\dots,$\HF_r$ are $(\calh_{|Y'},\rho)$-invariant and $(\calh_{|Y'},\rho)$-everywhere maximal.
  Let $H=\Gamma_{\{\HF_1,\dots,\HF_r\}}$ be the elementwise stabilizer of this collection. 
  Up to replacing $Y'$ by a  Borel subset of full measure, we have $\rho(\calh_{|Y'})\subseteq H$.
  Since $(\calh,\rho)$ is FS-averse, 
  the group $H$ preserves no non-trivial free splitting.
  Similarly, since $\HF_i$ is $(\calh_{|Y'},\rho)$-everywhere maximal, this implies that $\HF_i$ is a maximal $H$-invariant \afs.
  Corollary~\ref{coro:finitude-facteurs} then gives the bound  $r\leq \max\{1,\frac{k-1}{2}\}$ and proves the claim.

  We will use several times the following fact: if $\HF$ is an \afs\ which is  $(\calh_{|Y'},\rho)$-invariant,
  then one can find a countable Borel partition $Y'=\Dunion_i Y'_i$
  and an \afs\  $\HF_i\sqsupseteq \HF$ which is $(\calh_{|Y_i},\rho)$-invariant  and $(\calh_{|Y_i},\rho)$-everywhere maximal. 
  Indeed, using a maximizing sequence as in Lemma \ref{lemma:stably-nice-vs-nice-averse}, 
  consider a Borel subset of maximal measure $Z\subset Y'$ such that 
   $Z$ has a countable partition $Z=\Dunion_j Z_j$
  where there exists $\HF_j\sqsupsetneq \HF$ that is invariant on $Z_j$.
  Define $Y''=Y'\setminus Z$, and note that $\HF$ is  everywhere maximal on $Y''$.
  
  Since there is a bound
  on the size of a chain of
  almost free factor systems of $F_k$, repeating this construction finitely many times yields the desired partition.

  We first use this fact (with $\HF$ the empty free factor system) to find a countable Borel partition
  $Y=\Dunion_i Z^1_i$ such that  for every $i$, there exists an $(\calh_{|Z^1_i},\rho)$-invariant
  almost free factor system $\hat\calf_i$ which is  $(\calh_{|Z_i^1},\rho)$-everywhere maximal.
  
Using this fact again, one may then find a countable Borel  partition $Z^1_i=Y^1_i\dunion (\Dunion_j Z^{\geq 2}_{i,j})$,
  such that
  \begin{itemize}
  \item $(\calh_{|Y^1_i},\rho)$ is pure: every \afs\ $\HF'$
    which is $(\calh_{|U},\rho)$-invariant for some  Borel subset of positive measure
    $U\subseteq Y^1_i$ satisfies $\HF'\sqsubseteq\HF_i$;
  \item for each $j$, there are at least two  everywhere maximal invariant \afs{}s on $Z^{\geq 2}_{i,j}$.
\end{itemize}

Repeating this process $r$ times, one gets a partition of $Y$ into pieces where the lemma holds (with a
stable reduction collection of cardinality at most $r-1$),
and pieces on which there are at 
least $r$ distinct invariant \afs{}s  which are  everywhere maximal.
Taking $r>\max\{1,\frac{k-1}{2}\}$ concludes the proof.
\end{proof}

\begin{lemma}\label{lemma:nice-averse_is_stably_pure}
  Let $\calh$ be a measured groupoid over a base space $Y$, equipped with a cocycle $\rho:\calh\to\ia$.

  If $(\calh,\rho)$ is nice-averse, then it is stably pure.
\end{lemma}

\begin{proof}
  Being nice-averse, $(\calh,\rho)$ is $\FS$-averse
  so one may apply Lemma \ref{lem_maxFF} to get a partition $Y=\Dunion_{i\in I} Y_i$ with
  associated stable reduction collections
  $\calc_i$.

  It suffices to prove that  for every $i\in I$, the pair $(\calh_{|Y_i},\rho)$ is pure, i.e.\
  that $\calc_i$ consists
  of a single \afs.
  Assume on the contrary that $\#\calc_i\geq 2$, and let $\Gamma_{\calc_i}$ be the elementwise stabilizer of $\calc_i$.
There exists a conull Borel subset $Y_i^*\subseteq Y_i$ such that $\rho(\calh_{|Y_i^*})\subseteq \Gamma_{\calc_i}$. As $(\calh,\rho)$ is $\FS$-averse, it follows that $\Gamma_{\calc_i}$ is infinite and preserves no free splitting.
    Since $\calc_i$ is a stable reduction collection,
  each \afs\ $\HF\in \calc_i$ is maximal among all $\Gamma_{\calc_i}$-invariant \afs{}s.
  Lemma~\ref{lemma:filling} implies that $\calc_i$ is filling,
  i.e.\ there does not exist any almost free factor system $\HF$ such that for every $\HF'\in\calc_i$, one has $\HF'\sqsubseteq\HF$.
  Theorem \ref{thm_nice}(2) then provides a nice splitting $U_{\calc_i}$
  which is invariant under $\Gamma_{\calc_i}$, hence under $(\calh_{|Y_i^*},\rho)$.
  This contradicts that $(\calh,\rho)$ is nice-averse.
\end{proof}

\begin{proof}[Proof of Proposition \ref{prop:stably-nice-vs-pure-nice-averse}]
  Let $Y=Y_1\dunion Y_2$  be a Borel partition with $(\calh_{|Y_1},\rho)$ stably nice and
  $(\calh_{|Y_2},\rho)$ nice-averse, as given by Lemma \ref{lemma:stably-nice-vs-nice-averse}.
  Then by Lemma \ref{lemma:nice-averse_is_stably_pure}, $(\calh_{|Y_2},\rho)$ is stably pure, which concludes the proof.
\end{proof}

\subsection{Invariant free splittings}

The proof of Proposition \ref{prop:nice-preserved} occupies the next subsections,
where we distinguish cases according to the nature of the nice splittings preserved.
Our first lemma uses the chain condition on stabilizers of collections of free splittings.

\begin{lemma}\label{lemma:fs-part}
  Let $\calh$ be a measured groupoid over a base space $Y$, equipped with a  
  cocycle $\rho:\calh\to\ia$.

  Then there exist a countable  Borel partition $Y=\Dunion_{i\in I} Y_i$
  and for every $i\in I$, a (maybe empty) collection $\calc_i$
  of non-trivial free splittings, such that 
\begin{enumerate}
\item every splitting in $\calc_i$ is $(\calh_{|Y_i},\rho)$-invariant;
\item for every Borel subset $U\subseteq Y_i$ of positive measure, every non-trivial free splitting which is
 $(\calh_{|U},\rho)$-invariant belongs to $\calc_i$.
\end{enumerate}
\end{lemma}

Notice that for each $i\in I$ such that $Y_i$ has positive measure,
 the collection $\calc_i$ satisfying the two properties is unique;
furthermore, $\calc_i$ satisfies these two properties
with respect to $\calh_{|Y'}$ for every Borel subset $Y'\subseteq Y_i$ of positive measure.
Also note that $(\calh_{|Y_i},\rho)$ is $\FS$-averse if and only if $\calc_i=\es$.

\begin{proof}
  Let $\calc^0$ be the (maybe empty) set of all  $(\calh,\rho)$-invariant  non-trivial free splittings. 
 
  We claim that we can find a Borel partition $Y=Y_0\dunion \Dunion_{i\in I}Y'_i$ with the following properties:
  \begin{enumerate}
      \item 
      for every $i\in I$, there exists a non-trivial free splitting $S_i\notin\calc^0$ which is $(\calh_{|Y_i},\rho)$-invariant; 
      \item for every Borel subset $U\subseteq Y_0$ of positive measure, every  $(\calh_{|U},\rho)$-invariant free splitting belongs to $\calc^0$.
  \end{enumerate}
  Indeed, take for $Z_0$ a Borel subset of maximal measure having a partition $Z_0=\dunion_{i\in I}Y'_i$ satisfying 1 
  (the existence of $Z_0$ is proved using a measure-maximizing sequence, 
  as in the proof of Lemma~\ref{lemma:stably-nice-vs-nice-averse}). Then $Y_0:=Y\setminus Z_0$
  satisfies 2.
  

Now, for every $i\in I$, let $\calc^1_i$ be the set of all $(\calh_{|Y'_i},\rho)$-invariant  non-trivial free splittings. Notice that $\calc^0\subsetneq\calc^1_i$. In addition, their elementwise stabilizers in $\ia$ satisfy
$\Gamma_{\calc^1_i}\subsetneq\Gamma_{\calc^0}$: indeed, this inclusion is strict because
$\rho(\calh_{|Y^*})\subseteq \Gamma_{\calc^0}$ for some conull Borel subset $Y^*\subseteq Y$
whereas 
$\rho(\calh_{|Y^*})\not\subseteq \Gamma_{\calc^1_i}$ since $S_i$ is not  $(\calh,\rho)$-invariant.

One can now repeat this construction: for every $i\in I$, we can find a countable Borel partition $Y'_i=Y_{i,0}\dunion\Dunion_{j\in I_i} Y'_{i,j}$ such that 
\begin{enumerate}
\item for every $j\in I_i$, there exists a non-trivial free splitting $S_{i,j}\notin\calc^1_i$ which is  $(\calh_{|Y'_{i,j}},\rho)$-invariant;
\item for every Borel subset $U\subseteq Y_{i,0}$ of positive measure, every  $(\calh_{|U},\rho)$-invariant non-trivial free splitting belongs to $\calc^1_i$.
\end{enumerate}
For every $i\in I$ and every $j\in I_i$, let $\calc^2_{i,j}$ be the set of all $(\calh_{|Y'_{i,j}},\rho)$-invariant non-trivial free splittings. As above, we have $\Gamma_{\calc^2_{i,j}}\subsetneq\Gamma_{\calc^1_i}$.

As there is a bound on the size of a chain of  elementwise stabilizers of collections of free splittings (Proposition~\ref{prop:chain-FS2}), the iteration of this process stops after finitely many steps.
The output is a countable Borel partition of $Y$ as in the lemma.
\end{proof}

In the following lemma, we let $\calg,\calh,\calh'$ be as in the statement of Proposition~\ref{prop:nice-preserved}, and we apply Lemma~\ref{lemma:fs-part} to $\calh$ to get subsets $Y_i$ and collections $\calc_i$ of non-trivial free splittings.

\begin{lemma}\label{lemma:fs-canonical}
    Let $Y_i$, $\calc_i$ be as in Lemma \ref{lemma:fs-part}. 
    Let $\calh'$ be a measured subgroupoid of $\calg$ which normalizes $\calh$.
    
    Then for every $i\in I$, the collection $\calc_i$ is $(\calh'_{|Y_i},\rho)$-invariant.
\end{lemma}

Denoting by $\Gamma_{\{\calc_i\}}$ the setwise stabilizer in $\ia$ of the collection $\calc_i$, this means that there exists a conull Borel subset $Y_i^*\subseteq Y_i$ such that $\rho(\calh'_{|Y_i^*})\subseteq\Gamma_{\{\calc_i\}}$.

\begin{proof}
By definition of normalization of subgroupoids, one can write $\calh'_{|Y_i}$ as a countable union of bisections $B_n\subset \calh'_{|Y_i}$ such that
for all $h\in\calg_{|Y_i}$ and all $\phi,\psi\in B_n$ such that $\phi h \psi\m$ is defined, 
one has $h\in\calh$ if and only if
$\phi h \psi\m\in \calh$.
Up to subdividing the bisections, we may assume additionally that the cocycle $\rho$ maps $B_n$ to a single value $\gamma_n\in \ia$.
Let $f_{B_n}:U_n\ra V_n$ be the partial isomorphism associated to the bisection $B_n$, with $U_n,V_n\subseteq Y_i$.

Any free splitting $S\in\calc_i$ is $(\calh_{|U_n},\rho)$-invariant, so $\gamma_n S$ is $(\calh_{|V_n},\rho)$-invariant.
By maximality of $\calc_i$ (Lemma~\ref{lemma:fs-part}(2)), 
it follows that $\gamma _n S\in\calc_i$ if $V_n$ has positive measure.
Taking for $Y_i^*$ a subset of full measure avoiding all $U_n$ and $V_n$ of zero measure,
we deduce that $\rho(\calh'_{|Y_i^*})\subseteq\Gamma_{\{\calc_i\}}$, which concludes our proof.
\end{proof}

Since by Theorem \ref{thm_nice}, one can assign to $\calc_i$ a canonical nice splitting as soon as $\calc_i\neq \es$ and $\Gamma_{\calc_i}$ is infinite,
we get the following corollary for nowhere trivial cocycles.

\begin{cor}\label{cor:case-non-fs-averse}
  Let $\calg$ be a measured groupoid over a base space $Y$ with a  cocycle $\rho:\calg\to\ia$.
  Let $\calh$ be a measured subgroupoid of $\calg$ such that
  $\rho_{|\calh}$ is nowhere trivial.

  Then there exists a countable Borel partition $Y=Z\sqcup (\Dunion_{i\in I} Y_i)$ such that
\begin{enumerate}
\item $(\calh_{|Z},\rho)$ is $\FS$-averse,
\item for every $i\in I$, there exists a nice splitting $S_i$
  which is  $(\calh_{|Y_i},\rho)$-invariant, and in fact $(\calh'_{|Y_i},\rho)$-invariant for every measured subgroupoid $\calh'$ of  $\calg$ such that $\calh'_{|Y_i}$ normalizes $\calh_{|Y_i}$. 
\end{enumerate}
\end{cor}

\begin{proof} 
  Consider 
  a countable Borel partition $Y=\Dunion_{j\in J} Y_j$ with a collection $\calc_j$ of free splittings for each $j\in J$,
  as in Lemma \ref{lemma:fs-part}.
  Denoting by $Z$ the union of the sets $Y_j$ such that $\calc_j=\es$, we see that $(\calh_{|Z},\rho)$ is $\FS$-averse.

  Now consider $j\in J$ such that $\calc_j\neq \es$.
By Lemma~\ref{lemma:fs-canonical}, $\calc_j$ is invariant under $(\calh_{|Y_j},\rho)$ and $(\calh'_{|Y_j},\rho)$.
Since  $\rho_{|\calh}$ is nowhere trivial, the elementwise stabilizer  $\Gamma_{\calc_j}$ of $\calc_j$ is non-trivial, hence infinite.
The nice splitting $S_j=U_{\calc_j}$ provided by  Theorem~\ref{thm_nice}(1) is then
invariant under $(\calh_{|Y_j},\rho)$ and $(\calh'_{|Y_j},\rho)$.
\end{proof}

\subsection{Invariant $\Zmax$-splittings}

\begin{lemma}\label{lemma:invariant-zmax-groupoid}
  Let $\calg$ be a measured groupoid over a  base space $Y$, equipped with a cocycle $\rho:\calg\to\ia$, and let $\calh$ be a measured subgroupoid of $\calg$. Assume that $(\calh,\rho)$ is $\FS$-averse.

  Then there exists a countable Borel partition $Y=Z\sqcup(\Dunion_{i\in I} Y_i)$ such that 
\begin{enumerate}
\item for every Borel subset $U\subseteq Z$ of positive measure, there is no $(\calh_{|U},\rho)$-invariant non-trivial $\Zmax$-splitting of $F_k$,
\item for every $i\in I$, there is a non-trivial $\Zmax$-splitting $S_i$ which is $(\calh_{|Y_i},\rho)$-invariant, and in fact $(\calh'_{|Y_i},\rho)$-invariant for every measured subgroupoid $\calh'$ of $\calg$  such that $\calh'_{|Y_i}$ normalizes $\calh_{|Y_i}$. 
\end{enumerate}
\end{lemma}

\begin{proof}
Using the chain condition on stabilizers of $\Zmax$-splittings (Proposition~\ref{prop:BCC_zmax}), one constructs as in the proof of Lemma~\ref{lemma:fs-part} a countable Borel partition $Y=\Dunion_{j\in J} Y_j$, and for every $j\in J$, a (maybe empty) set $\calc_j$ of non-trivial $\Zmax$-splittings, such that 
\begin{enumerate}
\item every splitting in $\calc_j$ is $(\calh_{|Y_j},\rho)$-invariant;
\item for every Borel subset $U\subseteq Y_j$ of positive measure, every non-trivial $\Zmax$-splitting which is $(\calh_{|U},\rho)$-invariant belongs to $\calc_j$.
\end{enumerate}
We define $Z$ as the union of all $Y_j$ such that $\calc_j$ is empty. It clearly satisfies the first assertion of the lemma.

Now fix $j$ such that $\calc_j\neq\es$.
As $(\calh,\rho)$ is $\FS$-averse, 
the elementwise stabilizer $\Gamma_{\calc_j}$ of $\calc_j$ does not fix any non-trivial free splitting of $F_k$.
As in Lemma \ref{lemma:fs-canonical}, for every Borel subset $U\subseteq Y_j$ of positive measure $(\calh',\rho)$ stabilizes $\calc_j$ setwise.
The proposition thus follows by using Theorem~\ref{theo:JSJ_zmax}, which canonically associates a non-trivial $\Zmax$-splitting to the group $\Gamma_{\calc_j}$.
\end{proof}

\subsection{Invariant bi-nonsporadic splittings and conclusion}

\begin{lemma}\label{lemma:case-bi-nonsporadic}
  Let $\calg$ be a measured groupoid over a base space $Y$, with a cocycle $\rho:\calg\to\ia$.
  Let  $\calh$ be a  measured  subgroupoid of $\calg$.
  If $(\calh,\rho)$ is $\FS$-averse and stabilizes a bi-nonsporadic splitting $T$,
  then for every measured subgroupoid $\calh'$ of $\calg$ that normalizes $\calh$, $(\calh',\rho)$ is stably nice.

  More precisely, there is a countable Borel partition $Y=\Dunion_{i\in I} Y_i$
  and for each $i\in I$ a nice splitting $S_i$ of $F_k$ such that $S_i$
  is invariant under both $(\calh_{|Y_i},\rho)$ and  $(\calh'_{|Y_i},\rho)$.
\end{lemma}

\begin{proof}
  By Lemma \ref{lemma:binonsporadic2F}, one can canonically associate to $T$ a filling finite collection $\calc$ of \afs{}s.
  In particular, $\calc$ is $(\calh,\rho)$-invariant, 
  hence every almost free factor system in $\calc$ is $(\calh,\rho)$-invariant by Theorems \ref{theo:ia} and \ref{theo:ia-element}.
  Apply Lemma \ref{lem_maxFF} to get a partition $Y=\Dunion_{i\in I}Y_i$  such that for every $i\in I$,
  there is a  nonempty stable reduction collection $\calc_i$ of \afs{}s.
By Lemma \ref{lemma:maxFF_unique}, the set $\calc_i$ is also $(\calh'_{|Y_i},\rho)$-invariant.
  
Fix $i\in I$. The collection $\calc_i$ has the property that for every $\HF\in\calc$, there exists $\HF'\in \calc_i$ such that $\HF\sqsubseteq\HF'$. Since $\calc$ is filling, so is $\calc_i$.
Since $(\calh,\rho)$ is $\FS$-averse,  $\rho_{|\calh}$ is nowhere trivial, so the setwise stabilizer of $\calc_i$ in $\ia$
is non-trivial hence infinite. It follows that the elementwise stabilizer $\Gamma_{\calc_i}$ of $\calc_i$ is infinite.
Let $S_i=U_{\calc_i}$ be the nice splitting associated to $\calc_i$ by Theorem \ref{thm_nice}(2).
Then $S_i$ is invariant under both $(\calh_{|Y_i},\rho)$ and  $(\calh'_{|Y_i},\rho)$.
\end{proof}

\begin{proof}[Proof of Proposition \ref{prop:nice-preserved}]
By Corollary \ref{cor:case-non-fs-averse}, up to partitioning the base space $Y$, we may assume that $(\calh,\rho)$ is $\FS$-averse.
Since $(\calh,\rho)$ is stably nice, we may assume up to further partitioning that $(\calh,\rho)$ preserves a $\Zmax$
or a binonsporadic splitting $S$.
The proposition then follows from Lemmas \ref{lemma:invariant-zmax-groupoid} and \ref{lemma:case-bi-nonsporadic} accordingly.
\end{proof}

\section{Groupoids with cocycles to direct products of free groups}\label{sec:direct-product}

In later sections, given a measured groupoid $\calg$ equipped with a cocycle $\rho:\calg\to\IA$, we will often want to consider subgroupoids $\calh$ of $\calg$ such that $(\calh,\rho)$ has an invariant one-edge non-separating free splitting $S$, and an interesting subgroupoid of $\calh$ to consider is the preimage under $\rho$ of the group of twists of the splitting $S$, which is a finite-index subgroup of a direct product of free groups (see Section \ref{sec_twists}). This motivates the present section, which establishes a few facts about groupoids equipped with cocycles taking values in a direct product of free groups. We start by stating a first lemma which is essentially due to Adams \cite{Ada}, and actually follows from the next lemma which gives a more general statement.

\begin{lemma}\label{lemma:adams-libre}
  Let $\calg$ be a measured groupoid over a base space $Y$, let $F$ be a finitely generated free group, and let $\rho:\calg\to F$ be a cocycle. Let $\cala,\calh$ be two measured subgroupoids of $\calg$. Assume that
 $\rho_{|\cala}$ is nowhere trivial and that $\rho_{|\calh}$ 
 has trivial kernel (Definitions \ref{dfn:nowhere_trivial} and \ref{dfn:trivial_kernel}).

If $\cala$ is amenable and normalized by $\calh$, then $\calh$ is amenable.  
\end{lemma}

Lemma~\ref{lemma:adams-libre} follows from the next statement, after replacing $\rho:\calg\to F$ by $\tilde{\rho}:\calg\to F\times F$ defined by $\tilde{\rho}(g)=(\rho(g),\rho(g))$, and setting $\cala_l=\cala_r=\cala$.

\begin{lemma}\label{lemma:adams-produit}
  Let $\calg$ be a measured groupoid, let $F_l,F_r$ be two finitely generated free groups, and let $\rho:\calg\to F_l\times F_r$ be a
  cocycle. Let $\rho_l:\calg\to F_l$ and $\rho_r:\calg\to F_r$ be the two cocycles obtained by postcomposing $\rho$ with the two projections. Let $\cala_l,\cala_r,\calh$ be measured subgroupoids of $\calg$. Assume that  $(\rho_l)_{|\cala_l}$ and $(\rho_r)_{|\cala_r}$ are nowhere trivial and that $\rho_{|\calh}$ has trivial kernel. 

  If $\cala_l$ and $\cala_r$ are both amenable and normalized by $\calh$, 
  then $\calh$ is amenable.  
\end{lemma}

\begin{proof}
The proof is based on an argument due to Adams \cite{Ada}. Let $T_l$ be a Cayley tree for $F_l$ with respect to some free basis. Then $F_l$ acts by isometries on $T_l$, and this action extends continuously to an $F_l$-action by homeomorphisms on $\partial_\infty T_l$. As $\cala_l$ is amenable and $\partial_\infty T_l$ is a compact metrizable space, Proposition~\ref{prop:zimmer} ensures that there exists an $(\cala_l,\rho_l)$-equivariant Borel map $\mu:Y\to\Prob(\partial_\infty T_l)$. 

By \cite[Proposition~5.1]{Ada}, for a.e.\ $y\in Y$, the probability measure $\mu(y)$ is supported on at most two points. Let us briefly recall the argument for this fact, in order to prepare for the next section where we will need a refined version.
Assume towards a contradiction that there exists a Borel subset $U\subseteq Y$ of positive measure such that for every $y\in U$, the probability measure $\mu(y)$ is supported on at least three points. Then $\mu(y)\otimes\mu(y)\otimes\mu(y)$ gives positive measure to the subspace $(\partial_\infty T_l)^{(3)}$ made of pairwise distinct triples. 
By restriction and renormalization of the probability measures, we thus get an $((\cala_l)_{|U},\rho_l)$-equivariant Borel map $U\to\Prob((\partial_\infty T_l)^{(3)})$. 
Notice that there exists an $F_l$-equivariant Borel \emph{barycenter} map $(\partial_\infty T_l)^{(3)}\to V(T_l)$, where $V(T_l)$ denotes the (countable) vertex set of $T_l$. Pushing forward the probability measures obtained above by this barycenter map, this yields an $((\cala_l)_{|U},\rho_l)$-equivariant Borel map  $\phi:U\to \Prob(V(T_l))$. 
Denote by $\calp_{<\infty}(V(T_l))$ the (countable) collection of all nonempty finite subsets of $V(T_l)$. As $V(T_l)$ is countable, there is also an $F_l$-equivariant Borel map $\Prob(V(T_l))\to \calp_{<\infty}(V(T_l))$, sending a probability measure $\nu$ to the finite set of all elements of $V(T_l)$ with maximal $\nu$-measure. 
Combining the above two maps, we derive an $((\cala_l)_{|U},\rho_l)$-equivariant Borel map $U\to\calp_{<\infty}(V(T_l))$. 
By restricting to a Borel subset $U'\subseteq U$ of positive measure where this map is constant, we find a positive measure Borel subset of $Y$ such that $\rho_l((\cala_l)_{|U'})$ is contained in the $F_l$-stabilizer of a  finite set of vertices of $T_l$, whence trivial. 
This yields a contradiction to the fact that  $(\rho_l)_{|\cala_l}$ is nowhere trivial.

Therefore, by post-composing $\mu$ with the map sending every probability measure to its support, we get an $(\cala_l,\rho_l)$-equivariant Borel map $Y\to\calp_{\le 2}(\partial_\infty T_l)$, where $\calp_{\le 2}(\partial_\infty T_l)$ is the set of all nonempty subsets of $\partial_\infty T_l$ of cardinality at most $2$. In fact, the barycenter argument also shows that if $\phi_1,\phi_2$ are two such maps, then their union $\phi_1\cup\phi_2$ must again take its values in $\calp_{\le 2}(\partial_\infty T_l)$. Therefore, by \cite[Lemmas~3.2 and~3.3]{Ada}, there exists a unique
$(\cala_l,\rho_l)$-equivariant Borel map $\phi_{\max}:Y\to\calp_{\le 2}(\partial_\infty T_l)$ which is everywhere maximum in the following sense: for every  Borel subset $U\subseteq Y$ of positive measure and every  $((\cala_{l})_{|U},\rho_l)$-equivariant Borel map $\phi:U\to\calp_{\le 2}(\partial_\infty T_l)$, one has $\phi(y)\subseteq\phi_{\max}(y)$ for almost every $y\in U$.

Using the fact that $\calh$ normalizes $\cala_l$, we will now derive that $\phi_{\max}$ is also $(\calh,\rho_l)$-equivariant.
Indeed, one can write $\calh$ as a countable union of bisections $B_n\subset \calh$ such that
for all  $a\in\cala_l$ and all $\gamma,\delta\in B_n$ such that $\gamma a \delta\m$ is defined, 
one has  $a\in\cala_l$ if and only if  $\gamma a \delta\m\in \cala_l$.
 Up to subdividing these bisections, we may assume that the cocycle $\rho_l$ takes a single value $\gamma_n$ on $B_n$.
Let $f_{B_n}:U_n\ra V_n$ be the partial isomorphism associated to the bisection $B_n$, with $U_n,V_n\subseteq Y$. Whenever $U_n$ has positive measure, the map 
\begin{align*}
\phi':U_n&\ra \calp_{\le 2}(\partial_\infty T_l) \\
y&\mapsto \gamma_n\m.\phi_{\max}( f_{B_n}(y))  
\end{align*}
is an everywhere maximum  $((\cala_l)_{|U_n},\rho_l)$-equivariant Borel map 
so there is a subset of full measure $U'_n\subseteq U_n$ on which $\phi'$ and $\phi_{\max}$  coincide -- this is also obviously true if $U_n$ has measure zero.  Since $\calh$ is the union of the bisections $B_n$, this shows that $\phi_{\max}$ is $(\calh,\rho_l)$-equivariant.

Similarly, by working with $\cala_r$ instead of $\cala_l$, we build an $(\calh,\rho_r)$-equivariant Borel map $Y\to\calp_{\le 2}(\partial_\infty T_r)$. Combining these two maps yields an $(\calh,\rho)$-equivariant Borel map $Y\to\calp_{\le 2}(\partial_\infty T_l)\times\calp_{\le 2}(\partial_\infty T_r)$. The action of $F_l\times F_r$ on $\calp_{\le 2}(\partial_\infty T_l)\times\calp_{\le 2}(\partial_\infty T_r)$ is Borel amenable (see e.g.\ \cite[Example~1.4(3)]{AD} for the Borel amenability of the action of a finitely generated free group on the boundary of its Cayley tree, from which our claim easily follows). 
As $\rho_{|\calh}$ has trivial kernel, Proposition~\ref{prop:mackey} therefore ensures that $\calh$ is amenable. 
\end{proof}

As a consequence of Lemma~\ref{lemma:adams-libre}, we deduce the following statement.

\begin{lemma}\label{lemma:normalisateur_twist}
Let $\calg$ be a measured groupoid over a base space $Y$, let $F_l$ and $F_r$ be finitely generated non-abelian free groups, and let $\rho:\calg\to F_l\times F_r$ be an action-type cocycle. 

Then $\calg$ does not contain any normal amenable subgroupoid  of infinite type.
\end{lemma}

\begin{proof}
  Assume towards a contradiction that $\calg$ contains a normal amenable subgroupoid $\cala$  of infinite type.
  Let $\rho_l,\rho_r$ be the two cocycles obtained by postcomposing $\rho$ with the projection to $F_l$ and $F_r$, respectively.

  We claim that we can then find a Borel partition $Y=Y_l\dunion Y_r$
  such that
  $(\rho_l)_{|\cala}$ is nowhere trivial on $Y_l$
   and  $(\rho_r)_{|\cala}$ is nowhere trivial on $Y_r$.
  Indeed, let $Y_r$ be such that
\begin{enumerate}  
\item $(\rho_l)_{|\cala}$ is stably trivial on $Y_r$, i.e.\ there exists a partition $Y_r=\Dunion_{i\in I} Y_{r,i}$ into at most countably many Borel subsets such that for every $i\in I$, the cocycle $(\rho_l)_{|\cala}$ is trivial on $Y_{r,i}$, and 
\item $(\rho_l)_{|\cala}$ is nowhere trivial on $Y_l=Y\setminus Y_r$.
\end{enumerate}
We check that $(\rho_r)_{|\cala}$ is nowhere trivial on $Y_r$.
Given $U\subset Y_r$ of positive measure,  consider $i\in I$ such that $U':=U\cap Y_{r,i}$ has positive measure. In particular $\rho_l(\cala_{|U'})=\{1\}$.
Since $\rho$ has trivial kernel, $\rho_r$ has trivial kernel on $\cala_{|U'}$.
Since $\cala$ is of infinite type, this implies that $\rho_r(\cala_{|U})\neq\{1\}$ and proves our claim.

  Without loss of generality, assume that $Y_l$ has positive measure.
  Let $\calh_l=\rho^{-1}(F_l\times\{1\})$.
  Since 
  $F_l$ is a non-abelian free group,
  $\calh_l$ is everywhere non-amenable by Lemma~\ref{lemma:not-amen}.
  As $\rho$ has trivial kernel,   $\rho_l$ has trivial kernel on $(\calh_l)_{|Y_l}$.
   Since $\cala_{|Y_l}$ is normalized by $(\calh_l)_{|Y_l}$, Lemma~\ref{lemma:adams-libre} implies that $(\calh_l)_{|Y_l}$ is amenable, a contradiction.   
\end{proof}

\section{Amenable actions on arational trees}\label{sec:amenable_actions}

In a joint work with Bestvina \cite{BGH}, we proved 
that the $\Out(F_k)$-action on $\PAT$ is Borel amenable \cite[Theorem~6.4]{BGH} --  see Definition~\ref{de:borel-amenable} for the notion of Borel amenability of a group action.

More generally, for every free factor system $\calf$,
it is proved in \cite[Theorem~6.4]{BGH} that the action of $\Out(F_k,\calf^{(t)})$ on $\PAT$ is Borel amenable
(recall that $\Out(F_k,\calf^{(t)})$ denotes the subgroup of $\Out(F_k,\calf)$ made of all outer automorphisms acting trivially on each free factor in $\calf$).
 
But  when $\calf$ contains a non-cyclic free factor (and is non-sporadic),
the $\Out(F_k,\calf)$-action on $\PAT(F_k,\calf)$ is not Borel amenable
because there exist arational $(F_k,\calf)$-trees with non-amenable stabilizer. 
A typical example is the following: let $\Sigma$ be a hyperbolic surface with one boundary component, let $c$ be an essential simple closed curve on $\Sigma$ that separates $\Sigma$ into two connected components $\Sigma_0$ and $\Sigma_1$ (where $\Sigma_0$ contains the boundary, say), and let $T$ be a tree dual to an arational measured lamination on $\Sigma_0$. Then $T$ is an arational
$(F_k,\pi_1(\Sigma_1))$-tree, but its stabilizer
contains all outer automorphisms induced by homeomorphisms of $\Sigma_1$ acting as the identity on $c$.

The goal of this section is to show that, once one gets rid of points with
non-amenable stabilizers, the action becomes Borel amenable. More precisely,
consider the partition $$\calp_{<\infty}(\PAT)=\PaPAT\dunion\PnaPAT$$
where
$\PaPAT$ (resp.\ $\PnaPAT$) is the set of non-empty finite collections
 projective arational trees, whose setwise (or equivalently elementwise) stabilizer is amenable (resp.\ non-amenable). 
 
First note that this is a Borel partition. 
Indeed, given a finitely generated subgroup $H\subseteq\Out(F_k)$, the set of fixed points of $H$ in $\calp_{<\infty}(\PAT)$ is measurable.
For any $C\in \calp_{<\infty}(\PAT)$, $C\notin \PaPAT$ if and only if there is a
finitely generated non-amenable subgroup $H$ fixing $C$ (because an increasing union of amenable groups is amenable).
Since there are countably many finitely generated subgroups in $\Out(F_k)$,
it follows that $\PaPAT$ is Borel.
 
 Similary, we also consider the Borel partition
  $$\calp_{<\infty}(\PATsim)=\PaPATsim\dunion\PnaPATsim$$
  where a collection lies in $\PaPATsim$ if and only if its stabilizer is amenable.
 
Building on \cite{BGH}, the goal of this section is the following theorem.

\begin{theo}\label{theo:action-amenable-2}
  The actions of $\Out(F_k,\calf)$ on $\PaPAT$ and $\PaPATsim$ are Borel amenable.  
\end{theo}

The theorem will be proved in 
Propositions \ref{prop:action-amenable-lift}
and \ref{prop:action-amenable} in the next two subsections.


\subsection{Amenability of the action on $\PaPAT$}\label{sec:papatsanssim}

We start by proving the amenability of the action on $\PaPAT$.

\begin{prop}\label{prop:action-amenable-lift}
  The action of $\Out(F_k,\calf)$ on $\PaPAT$ is Borel amenable.  
\end{prop}

We now recall some vocabulary, and some tools from \cite{BGH}.

A map $f:S\to T$ between two $F_k$-trees is a \emph{morphism} if every segment $I\subseteq S$ can be subdivided into finitely many subsegments $I_1,\dots,I_n$, so that the restriction of $f$ to each subsegment $I_j$ is an isometry. 
It is \emph{optimal} if every point $x\in S$ is contained in an open interval $I_x$ such that $f_{|I_x}$ is injective. We mention that for every $T\in\baro$, there exists a tree $S\in\calo$ equipped with an optimal morphism towards $T$, see e.g.\ \cite[Remark~3.19]{BGH}.
We will use the Gromov topology on morphisms as defined in \cite[Section~3.2]{GL-os}
and the induced Borel structure.

A \emph{turn} in a tree $T$ is a pair of directions at a branch point of $T$. 
Following Definitions~3.1 and~3.2 in \cite{BGH}, 
a \emph{turning class} in a tree $T\in\baro$ is an $F_k$-invariant set of turns in $T$. 
The turning class of an optimal morphism $f:S\to T$, denoted $\mathrm{Turn}(f)$, is the set of turns $(d,d')$ in $T$ such that there exists a segment $I\subseteq S$ such that $f_{|I}$ is injective, and $f(I)$ contains two germs representing the directions $d$ and $d'$.

 Given an optimal morphism  $f:S\ra T$ where $S\in \calo$ and $T\in \AT$,
and given $s\leq t\in \bbR$ we define $[s,t]_f$
as the collection of all simplices $\tau$ of $\mathbb{P}\calo$ such that there exists a tree $S'\in\calo$ through which $f$ factors and representing a point in $\tau$, 
and such that $e^{-t}\leq \mathrm{vol}(S'/F_k)\leq e^{-s}$  (here $\mathrm{vol}(S'/F_k)$ denotes the volume of the quotient graph $S'/F_k$, i.e.\ the sum of its edge lengths).
By \cite[Proposition~3.6 and Remark~3.9]{BGH}, the set $[s,t]_f$ is finite.

The following lemma says that there is some confluence between $[s,t]_f$
and $[s,t]_{f'}$ when $f,f'$ have the same turning class.

\begin{lemma}[Lemma 3.3 and Remark 3.5 in \cite{BGH}]\label{lem_confluence}
Let $S,S'\in\calo$ and $f:S\ra T$, $f':S'\ra T$ be two optimal morphisms
towards the same arational tree $T\in \AT$.
Assume that $f$ and $f'$ have the same turning class.

Then there exists $s_0\geq 0$ such that for all $t\geq s\geq s_0$, one has
$[s,t]_f=[s,t]_{f'}$.\qed
\end{lemma}

We denote by $\calm(\calo,\AT)$ the Borel set of optimal morphisms
$f:S\ra T$ between a tree $S\in \calo$ and a tree $T\in \AT$.

\begin{lemma}\label{lemma:measurability}
Given a simplex $\tau$ in $\calo$,
let $X_\tau\subset \bbR\times \bbR \times \calm(\calo,\AT)$ 
be the set of triples $(s,t,f)$ 
such that $\tau\in [s,t]_f$.

Then $X_\tau$ is Borel.
\end{lemma}

\begin{proof}
This is a consequence of the measurability of the set $Z_\tau$ introduced at the beginning of the proof of \cite[Lemma~3.13]{BGH}.
\end{proof}

We now note the following consequence of the chain condition on stabilizers of collections of free splittings.

\begin{lemma}\label{lem_stabilisation_stab}
    For every optimal morphism  $f:S\ra T$ where $S\in \calo$ and $T\in \AT$,
    and every $s\geq 0$, there exists $t_1\geq s$ such that for all $t\geq t_1$,
    the elementwise stabilizer of $[s,t]_f$ in $\ia\cap \Out(F_k,\calf)$ 
    fixes $T$.
\end{lemma}

\begin{proof}
Fix $s\geq 0$.
By Proposition~\ref{prop:chain-FS2}, 
as $t$ goes to infinity,
the elementwise stabilizer $H_{s,t}$ of $[s,t]_f$ in $\ia\cap \Out(F_k,\calf)$ stabilizes to a group $H_{s,\infty}$.
Let $t_1\geq s$ be such that $H_{s,t}=H_{s,\infty}$ for all $t\geq t_1$.
Since $H_{s,t}\subset \ia$, $H_{s,t}$ acts as the identity
on each simplex of $[s,t]_f$, so $H_{s,\infty}$ acts as the identity on each simplex occuring in $\bigcup_{t\geq s} [s,t]_f$.
 Now let $\rho:[0,\infty)\ra \calo$ be a folding path from $S$ to $T$ 
 such that $f$ factors through $\rho(t)$ for all $t\geq 0$ (for the existence,
 see for instance \cite[Section~3]{GL-os}, and \cite[Lemma~2.3]{BGH} for the fact that it stays in the interior of Outer space).
 Since $H_{s,\infty}$ fixes a subray of the image of $\rho$ in $\bbP\calo$, 
 it fixes $T$. 
Since $H_{s,\infty}=H_{s,t}$  for all $t\geq t_1$, this concludes the proof.
  \end{proof}


  Let $\cald$ be the countable set of all finite collections of simplices of $\mathbb{P}\calo$. 
  For instance, given $f:S\ra T$ an optimal morphism as above,
  $[s,t]_f$ is an element of $\cald$.
  We now define a family of probability measures on $\cald$ associated to $f$
  as follows: for all $n,m\geq 1$ we define
$$\mu_{n,m}(f)=\frac{1}{n}\int_n^{2n}\delta_{[t,t+m]_f}dt \ \in \Prob(\cald),$$
where $\delta_{[t,t+m]_f}$ denotes the Dirac mass on $[t,t+m]_f\in\cald$.

To state the following lemma, note that if $[T]=[T']$, then
the turning classes of $T$ and $T'$ are naturally identified.

\begin{lemma}\label{lem_asymptotic_invariance}
Let $T,T'\in \AT$ be such that $[T]=[T']$ in $\PAT$.
Let $S,S'\in\calo$, and let $f:S\ra T$, $f':S'\ra T'$ be two optimal morphisms having the same turning class.

Then for any sequence of integers $M_n$,
$$||\mu_{n,M_n}(f)-\mu_{n,M_n}(f')||_1\to 0\text{ as $n\ra \infty$}$$
\end{lemma}

\begin{proof}
The proof is a tiny variation on Corollary 3.10, Lemma 3.11 and Corollary 3.12 in \cite{BGH}. We repeat the proof for convenience.
Assume first that $T=T'$ (not up to scaling).
Then by Lemma \ref{lem_confluence}, there exists $s_0$ such that $[s,t]_f=[s,t]_{f'}$ for all
$t\geq s\geq s_0$.
It follows that for all $n\geq s_0$, $\mu_{n,M_n}(f)=\mu_{n,M_n}(f')$ and we are done.

We now consider the case where there exists $\lambda>0$ such that 
$T'=\lambda T$, $S'=\lambda S$ (i.e. $T'$ and $S'$ are obtained from $T$ and $S$ by rescaling the metric by the same factor $\lambda$) and $f'$ is $f$ viewed as a map from $S'$ to $T'$.
In this situation, we have by definition that 
for all $s\leq t$, $[s,t]_{f'}=[s+\log\lambda,t+\log\lambda]_{f}$.
Thus, changing variables, we get
\begin{eqnarray*}
\mu_{n,M_n}(f')-\mu_{n,M_n}(f)&=&
\frac{1}{n}\int_{n+\log\lambda}^{2n+\log\lambda}\delta_{[t,t+M_n]_f}dt-
\frac{1}{n}\int_n^{2n}\delta_{[t,t+M_n]_f}dt \\
&=&\frac{1}{n}\int_{2n}^{2n+\log\lambda}\delta_{[t,t+M_n]_f}dt-\frac{1}{n}\int_n^{n+\log\lambda}\delta_{[t,t+M_n]_f}dt
\end{eqnarray*}
so $||\mu_{n,M_n}(f')-\mu_{n,M_n}(f)||_1\leq \frac{2|\log\lambda|}{n}$.

Finally, putting the two cases together proves the lemma in general.
\end{proof}

We have defined a family of probability measures $\mu_{n,m}(f)$ depending on the optimal
morphism $f$. We now want to define probability measures that depend only on $[T]$ (and not of $f$). 
To achieve this, in Lemma \ref{lem_selection_morphisms}, 
we are going to associate to $T$ a finite collection
$\Theta(T)$ of morphisms with target $T$ in a measurable way. This collection
will not be equivariant but the turning classes they represent will.
A difficulty is that there may be infinitely many turning classes of morphisms
with target $T$, and Lemma~\ref{lem_selection_morphisms} will select finitely many of them in an equivariant way.

Before stating the lemma, let us recall the action of $\Out(F_k)$ on the set of morphisms.
Given a morphism $f:S\ra T$ and $\phi\in \Aut(F_k)$, 
recall that $\phi.S$ and $\phi.T$ are just the trees $S$ and $T$
with the actions precomposed by $\phi\m$.
The equivariant map $f:S\ra T$ is also equivariant
viewed as a map $\phi.S \ra \phi.T$,
and this induces an action of $\Out(F_k)$ on the set of morphisms.

There is also an action of $\Out(F_k)$ on the set of turning classes on trees $T\in \baro$:
if $\calb$ is an $F_k$-invariant set of turns of $T$, then 
we define $\phi.\calb$ as the same set of turns on $T$ which is now
endowed with the action precomposed by $\phi$.
This descends to an action of $\Out(F_k)$ on the set of turning classes on trees in $\baro$.
The map $\Turn$, sending an optimal morphism to its turning class, is $\Out(F_k)$-equivariant for this action.

\begin{lemma}\label{lem_selection_morphisms}
There is a Borel map 
$$\Theta:\left\{
\begin{array}{rcl}
\AT &\to& \calp_{<\infty}(\calm(\calo,\AT)) \\
T&\mapsto& \Theta(T)=\{f_1^T:S_1\ra T,\dots,f_{k_T}^T:S_k\ra T\}
\end{array}\right.
$$
assigning to $T\in \AT$ a non-empty finite collection of optimal morphisms $f_i^T:S_i\ra T$ 
with $S_i\in \calo$, such that 
\begin{enumerate}
\item for $i\neq j$, $\Turn(f_i^T)\neq \Turn(f_j^T)$,
\item for every $\Phi\in\Out(F_k)$, one has  $\Turn(\Theta(\Phi.T))=\Phi.\Turn(\Theta(T))$,
\item for all $\lambda>0$, $\Turn(\Theta(\lambda T))=\Turn(\Theta(T))$.
\end{enumerate}
\end{lemma}


\begin{rk}\label{rk_nb}
Note that by Assertion 1, $\#\Theta(T)=k_T=\#\Turn(\Theta(T))$. By Assertions 2 and 3
it follows that
$\#\Theta(\lambda T)=\#\Theta(T)$ and
$\#\Theta(\Phi. T)=\#\Theta(T)$.
\end{rk}


\begin{proof}
We rely on \cite[Proposition~4.1]{BGH} saying that $\AT$
has namable turning classes in the sense of \cite[Definition~3.22]{BGH}. 
This implies that to every turning class $\calb$ on an arational tree $T\in \AT$,
we can associate a name $\Name(\calb)\in \bbN\cup\{\perp\}$ such that 
for every $T\in\AT$, there are only finitely many turning classes on $T$ whose name is a given integer $n\in\mathbb{N}$ (but this need not hold for $\perp$). 
In addition, the map $f\mapsto\Name(\Turn(f))$ sending an optimal morphism to the name of its turning class is measurable, $\Out(F_k,\calf)$-invariant, 
 and invariant under homothety.
Moreover, for every $T\in \AT$, there exists a very optimal morphism $f:S\ra T$ such that
$\Name(\Turn(f))\neq \perp$, where a morphism $f:S\ra T$ is said to be very optimal
if it is optimal and maps each branch point of $S$ to a branch point of $T$.

We will now define $\Theta(T)$ as a choice
of very optimal morphisms whose turning classes have the smallest possible name in $\bbN$, as in the proof of \cite[Proposition~3.24]{BGH}.
More precisely, by \cite[Remark~3.18 and Lemma~3.20]{BGH}, there is a measurable enumeration of 
all very optimal morphisms $F_{j,T}:S_{j,T}\ra T$ with range $T$ in the following sense:
for every $j\in \bbN$, $S_{j,T}$ and $F_{j,T}$ depend measurably on $T$,
and every very optimal morphism $S\ra T$ (with $S\in \calo$) is equal to some $F_{j,T}$. 
We let $n_0(T)$ be the smallest integer $n\in\mathbb{N}$ such that 
$n=\Name(F_{j,T})$ for some $j\in \bbN$.
Let $\calb_1,\dots,\calb_{k_T}$ be the turning classes of very optimal morphisms with range $T$ whose name is $n_0(T)$. 
For every $i\in\{1,\dots,k_T\}$, we let $f_i=F_{j,T}$
where $j$ is the smallest integer such that $\Turn(F_{j,T})=\calb_i$. We let $\Theta(T)=\{f_1,\dots,f_{k_T}\}$.

This map $\Theta$ is Borel and clearly satisfies Assertion 1.
Assertions 2 and 3 follow from the fact that 
$\Name$ is invariant under the action of $\Out(F_k,\calf)$ and under homothety.
\end{proof}

  We now define a family of probability measures on $\cald$ associated to $[T]\in\PAT$
  as follows.
  Let $\sigma:\PAT\ra \AT$ be a Borel section (maybe not equivariant).
  For all $n,m\geq 1$ and all $T\in\AT$ we define
$$\tilde\mu_{n,m}(T)=\frac{1}{\#\Theta(T)}\sum_{f\in \Theta(T)} \mu_{n,m}(f)$$
and
$$\mu_{n,m}([T])=\tilde\mu_{n,m}(\sigma([T])).$$

We now prove that this family of measures depends asymptotically equivariantly on $[T]$.
Given a probability measure $\mu$ on $\cald$, and $\Phi\in \Out(F_k,\calf)$
we denote by $\Phi_*\mu$ the push-forward of $\mu$ under the action of $\Phi$ on $\cald$.

\begin{rk}\label{rk_mu_equivariant}
Note that for all $\Phi\in \Out(F_k,\calf)$, the pushforward probability measure $\Phi_\ast\mu_{n,m}$ satisfies  
$\Phi_*\mu_{n,m}(f)=\mu_{n,m}(\Phi.f)$,
where $\Phi.f:\Phi.S\ra \Phi.T$ is the image of the morphism $f$ under $\Phi$.
\end{rk}

\begin{lemma}\label{lemma:asymptotic-equivariance}
For any $[T]\in \PAT$, any $\Phi\in \Out(F_k,\calf)$, 
and any sequence of integers $M_n$,
$$||\mu_{n,M_n}(\Phi.[T])-\Phi_*\mu_{n,M_n}([T])||_1\to 0\text{ as $n\ra \infty$}.$$
\end{lemma}

\begin{proof}
Fix $[T]\in \PAT$, $\Phi\in \Out(F_k,\calf)$ and a sequence $M_n$ of integers.
Up to changing $T$ in its projective class, assume 
 that $\sigma([T])=T$. 
Let $T_1=\Phi.T$,
and observe that $\sigma(\Phi.[T])=\lambda T_1$ for some $\lambda >0$.

Let $k=\#\Theta(T)$, and write $\Theta(T)=\{f_i:S_i\ra T \mid i\leq k\}$,
and $\Phi.\Theta(T)=\{f_{i,1}:S_{i,1}\ra T_1 \mid i\leq k\}$
where $S_{i,1}=\Phi. S_i$, and $f_{i,1}=\Phi.f_i$.
Then using Remark~\ref{rk_mu_equivariant}, we get
\begin{eqnarray*}
\Phi_*\mu_{n,M_n}([T])&=&\frac{1}{k}\sum_{i\leq k}\Phi_*\mu_{n,M_n}(f_i)\\
&=&\frac{1}{k}\sum_{i\leq k}\mu_{n,M_n}(f_{i,1})
\end{eqnarray*}

On the other hand, 
by Remark \ref{rk_nb}, we have $\#\Theta(\lambda T_1)=\#\Theta(T)=k$ 
so we can write 
$\Theta(\lambda T_1)=\{f'_{i,1}:S'_{i,1}\ra \lambda T_1\mid i\leq k\}$.
Then 
$$\mu_{n,M_n}(\Phi.[T])=\tilde \mu_{n,M_n}(\lambda T_1)=\frac{1}{k}\sum_{i\leq k}\mu_{n,M_n}(f'_{i,1}).$$ 
By Lemma~\ref{lem_selection_morphisms}, we have $\Turn(\{f_{i,1}\mid i\leq k\})=\Turn(\{f'_{i,1}\mid i\leq k\})$. The conclusion therefore follows from Lemma~\ref{lem_asymptotic_invariance}.
\end{proof}

\begin{proof}[Proof of Proposition \ref{prop:action-amenable-lift}]

    Fix $N\in \bbN$ and let $(\PAT^N)_{\amen}$ be the subset of $\PAT^N$ consisting of $N$-tuples
    whose stabilizer (under the diagonal action of $\Out(F_k,\calf)$)
    is amenable.
    We first prove that the action of $\Out(F_k,\calf)$
    on $(\PAT^N)_{\amen}$ is Borel amenable.
 
 We consider $\cald^N$ endowed with the diagonal action of $\Out(F_k,\calf)$
 and let $\Da\subset\cald^N$ be the subset consisting of tuples whose stabilizer is amenable.
 We will associate to every $\TT\in(\PAT^N)_\amen$ a sequence of probability measures $\mu_n(\TT)$ on $\Da$, as follows.
 
 Given $n,m\in\bbN$ and $\TT=([T_1],\dots,[T_N])$, we let $\mu_{n,m}(\TT)=\mu_{n,m}([T_1])\otimes\dots\otimes\mu_{n,m}([T_N])$. Noting that \[||\mu_1\otimes\dots\otimes\mu_N-\mu'_1\otimes\dots\otimes\mu'_N||_1\le \sum_{i=1}^N||\mu_i-\mu'_i||_1,\] 
 it follows from Lemma~\ref{lemma:asymptotic-equivariance} that for every sequence of integers $M_n$, the measures $\mu_{n,M_n}(\TT)$ are asymptotically equivariant.
 
 We claim that for every $n\in\bbN$ and every $\TT\in (\PAT^N)_\amen$, there exists $M'_n=M'_n(\TT)$ so that
 the support of $\mu_{n,M'_n}(\TT)$ is contained in $\Da$. Indeed, by Lemma~\ref{lem_stabilisation_stab}, for every $i\in\{1,\dots,N\}$, there exists $m_i=m_i(\TT)$ such that for every $t\in [n,2n]$ and every $m\ge m_i$, the measure $\mu_{t,m}([T_i])$ is supported on subsets in $\cald$ whose $\Out(F_k,\calf)$-stabilizer is contained in the stabilizer of $T_i$ hence of $[T_i]$. 
 For $M'_n=\max\{m_1(\TT),\dots,m_N(\TT)\}$, the measure $\mu_{n,M'_n}(\TT)$ is supported on $\Da$.
 
 Now, for every $n,m$, the measure $\mu_{n,m}(\TT)$ depends measurably on $\TT$, and therefore its support (a finite subset of the countable set $\cald^N$) 
 depends measurably on $\TT$. It follows that the smallest integer $M_n(\TT)$ such that $\mu_{n,m}(\TT)$ is supported in $\Da$ depends measurably on $\TT$. The measures $\mu_{n,M_n(\TT)}(\TT)$ depend measurably on $\TT$.
 
 At this point, we have constructed a sequence of measurable and asymptotically equivariant maps $\mu_n:(\PAT^N)_\amen\to\Prob(\Da)$. Since the stabilizer of every element of the countable set $\Da$ is amenable, it follows from \cite[Proposition~2.12]{BGH} that the $\Out(F_k,\calf)$-action on $(\PAT^N)_\amen$ is Borel amenable.
This provides measurable and asymptotically equivariant maps $\nu_n:(\PAT^N)_\amen\to\Out(F_k,\calf)$. 

The set $(\PAT^N)_\amen$ consists of ordered $N$-tuples.
Let $\PaPATN\subset\PaPAT$ consist of (unordered) subsets of cardinality $N$.
The action of $\Out(F_k,\calf)$ on $\PaPATN$ is Borel amenable: indeed,
defining 
\[\nu'_n([T_1],\dots,[T_N])=\frac{1}{N!}\sum_{\sigma\in\mathfrak{S}_N}\nu_n([T_{\sigma(1)}],\dots,[T_{\sigma(N)}]),\] 
the maps $\nu'_n$ are still measurable and asymptotically equivariant, and they descend to maps defined on $\PaPATN$. 
Since $\PaPAT=\dunion_{N\geq 1} \PaPATN$,
the Borel amenability of the action on $\PaPAT$ follows. 
 \end{proof}



\subsection{Amenability of the action on $\PaPATsim$}\label{sec:papatsim}

We now upgrade the amenability of the action established in the previous section from $\PaPAT$ to $\PaPATsim$.


\begin{prop}\label{prop:action-amenable}
  The action of $\Out(F_k,\calf)$ on $\PaPATsim$ is Borel amenable.
\end{prop}

\begin{cor}\label{cor:amenable-on-pairs}
   The action of $\Out(F_k,\calf)$ on the subset of $\PATsim$ made of points with amenable stabilizer is Borel amenable. This also holds for the action on the set of pairs in $\calp_{\leq 2}(\PATsim)$ whose stabilizer is amenable.
   \qed
\end{cor}

Given a tree $T\in\AT$, we denote by $[T]$ and $[[T]]$ its images in $\PAT$ and $\PATsim$, respectively.
Similarly, given $\TT=\{T_1,\dots,T_p\}\in\calp_{<\infty}(\AT)$ we denote by $[\TT]$ and $[[\TT]]$ its images
in $\calp_{<\infty}(\PAT)$ and in $\calp_{<\infty}(\PATsim)$, respectively.

\newcommand{\Ext}{\mathrm{Ext}}
We recall that the preimage of $[[T]]$ in $\PAT$ (under the quotient map $\PAT\to\PATsim$) is a finite-dimensional simplex, see \cite[Proposition~13.5]{GH1} relying on \cite[Section~5]{Gui00}. 
We denote by $\mathrm{Ext}([[T]])$ the finite set of its extremal points. The trees in $\Ext([[T]])$ are sometimes called \emph{ergometric} in the literature.

We view $\Ext$ as a map $\PATsim\ra\calp_{<\infty}(PAT)$,
and think of it as a multisection of the quotient map $\pi:\PAT\ra\PATsim$.

\begin{lemma}\label{lem:multisection}
 The map $\Ext:\PATsim\ra\calp_{<\infty}(\PAT)$ is an 
 $\Out(F_k,\calf)$-equivariant Borel map.
\end{lemma}

\begin{proof}
Equivariance is clear.
By \cite[Lemmas~7.4 and~7.5]{GHL},
there exists a sequence of (not equivariant)
Borel maps $f_n:\PAT\to\PAT$ such that for all $[T]\in\PAT$, one has $\mathrm{Ext}([[T]])=\{f_n([T])\mid n\in\mathbb{N}\}$.
Thus, we obtain an equivariant Borel map $\tilde\sigma:\PAT\ra \calp_{<\infty}(\PAT)$ by letting $\tilde\sigma([T]):=\{f_n([T])\mid n\in\bbN\}$.
It descends to an equivariant Borel map $\sigma$ on $\PATsim$
which coincides with $\Ext$.
\end{proof}

We immediately get the following corollary, where the second assertion follows from the
fact that the stabilizer of a projective tree $[T]$ is contained in the stabilizer of its equivalence class $[[T]]$.

\begin{cor}\label{cor:multisection}
 There exists an $\Out(F_k,\calf)$-equivariant Borel map 
 $\sigma:\calp_{<\infty}(\PATsim) \ra \calp_{<\infty}(\PAT)$;
    it maps    $\PaPATsim$ in $\PaPAT$.\qed
\end{cor}



\begin{proof}[Proof of Proposition \ref{prop:action-amenable}]
By Proposition~\ref{prop:action-amenable-lift},
there exists a sequence of Borel maps $$\nu_n:\PaPAT\to\Prob(\Out(F_k,\calf))$$ which is asymptotically $\Out(F_k,\calf)$-equivariant.
Let  $\sigma:\PaPATsim \ra \PaPAT$ be
the section defined in Corollary \ref{cor:multisection}.
Taking $\nu'_n=\nu_n\circ \sigma$,
one gets the desired asymptotically equivariant sequence of probability measures.
\end{proof}



\section{Nice-averse amenable subgroupoids: Adams' argument}\label{sec:adams}

 In this section, we fix an integer $k\ge 2$. The main result of the present section is the following.

\begin{theo}\label{theo:normalisateur_moyennable}
Let $\calg$ be a measured groupoid over a  base space $Y$, and $\cala,\calh$ two measured subgroupoids of $\calg$.
Let $\rho:\calg\ra \ia$ be a cocycle such that  $\rho_{|\calh}$ has trivial kernel
and $\rho_{|\cala}$ is nowhere trivial. 
Assume that $\cala$ is amenable, and that $\cala$ is normalized by $\calh$.

Then either $\calh$ is amenable, or there exist a Borel subset $U\subseteq Y$ of positive measure and a nice splitting which is $(\calh_{|U},\rho)$-invariant and $(\cala_{|U},\rho)$-invariant.
\end{theo}

\begin{cor}\label{cor:normalisateur_moyennable}
Let $\calg$ be a measured groupoid over a  base space $Y$, and $\cala,\calh$ two measured subgroupoids of $\calg$.
Let $\rho:\calg\ra \ia$ be a  cocycle with trivial kernel.
Assume that $\cala$ is of infinite type and amenable, and that $\cala$ is stably normalized by $\calh$.

If $(\calh,\rho)$, $(\cala,\rho)$ or $(\grp{\cala,\calh},\rho)$ is nice-averse, 
then $\calh$ is amenable.
\end{cor}

\begin{proof}
  Since $\calh$ stably normalizes $\cala$, there is a partition $Y^*=\Dunion_{i\in I}Y_i$
   of a conull Borel subset $Y^*\subset Y$ into at most countably many Borel subsets
   such that $\calh_{|Y_i}$ normalizes $\cala_{|Y_i}$  for every $i\in I$. Since amenability of every restriction
   $\calh_{|Y_i}$ implies amenability of $\calh$,
  we may restrict to $Y_i$ and assume that $\calh$ normalizes $\cala$.

  Let $\calh'=\langle \cala,\calh\rangle$, and note that $\calh'$ normalizes $\cala$.
  In all cases, $(\calh',\rho)$ is nice-averse because 
if $(\calh,\rho)$ or  $(\cala,\rho)$ is nice-averse, then so is $(\calh',\rho)$.
As $\rho$ has trivial kernel, $\rho_{|\calh'}$ has trivial kernel, and as $\cala$ is of infinite type, we also deduce that $\rho_{|\cala}$ is nowhere trivial. 
We can therefore apply Theorem~\ref{theo:normalisateur_moyennable} to the subgroupoids $\cala$ and $\calh'$
and as $(\calh',\rho)$ is nice-averse, its conclusion ensures that $\calh'$ is amenable, hence so is $\calh$.  
\end{proof}

The general strategy of our proof follows an argument due to Adams \cite{Ada}, and implemented in the mapping class group setting by Kida \cite[Chapter~4]{Kid1};  the basic idea of this argument was in fact reviewed in our proof of Lemma~\ref{lemma:adams-produit}. But we will also need to establish some new results about the amenability of certain actions of $\Out(F_k,\calf)$, and there are some real differences that arise between mapping class groups and $\Out(F_k)$. 
For instance, the analogue of Kida's statement saying that the normalizer of an irreducible amenable groupoid is amenable does not always hold here, 
an extra possibility being that the normalizer preserves a nice splitting.

Our proof of Theorem~\ref{theo:normalisateur_moyennable}  will be completed in Section \ref{sec:amen-6}.
It goes as follows. 
After a countable partition of the base space $Y$,
we may assume that $(\cala,\rho)$ is pure or that  $(\cala,\rho)$ preserves a nice splitting (Proposition~\ref{prop:stably-nice-vs-pure-nice-averse}), and in the latter case,
Proposition~\ref{prop:nice-preserved} says that up to further partitioning,
there is a canonical one which is both $(\cala,\rho)$-invariant and $(\calh,\rho)$-invariant, so we are done.
We thus assume that $(\cala,\rho)$ is pure and consider the unique maximal $(\cala,\rho)$-invariant almost free factor system $\HF$.  It is also $(\calh,\rho)$-invariant,
and we work relative to $\HF$, and in fact relative to an extracted free factor system $\calf\subset \HF$ (see Definition~\ref{de:afs}).

At first our proof follows Kida's implementation of Adams' argument closely. 
We work in the relative Outer space $\calo=\calo(G,\calf)$ -- this will be important in order to exploit the maximality of $\calf$.
We start by using the amenability of $\cala$ to construct an $(\cala,\rho)$-equivariant  Borel map $\mu$ from $Y$ to the  space of all probability measures on the compactification $\Pbaro$ of the relative  Outer space (Section~\ref{sec:amen-1}). 
Using a witness map assigning to any non-arational tree a finite collection of  conjugacy classes of free factors, maximality of $\calf$ implies that for almost every $y\in Y$, the probability measure $\mu(y)$ gives full measure to the set of arational trees (Section~\ref{sec:amen-2}). 
Using a barycenter map on triples of inequivalent arational trees, one shows that the support of $\mu(y)$ has cardinality at most two almost everywhere (Section~\ref{sec:amen-3}). 
This argument also implies that there is an essentially  unique maximal $(\cala,\rho)$-equivariant  Borel map from $Y$ to the set of pairs of points in $\AT/{\sim}$, which is therefore $(\calh,\rho)$-equivariant (Section~\ref{sec:amen-4}).

If this map takes values into the set of pairs with amenable stabilizer,
we conclude that $\calh$ is amenable using Borel amenability of the action given by Theorem~\ref{theo:action-amenable-2}. Otherwise, we exploit the witness map constructed in Section~\ref{sec:witness} to get an invariant nice splitting.
\\
\\
\textbf{Standing assumptions:} From now on and until the end of Section~\ref{sec:amen-6}, thanks to Propositions~\ref{prop:stably-nice-vs-pure-nice-averse} and~\ref{prop:nice-preserved}, we will assume that $(\cala,\rho)$ is pure in addition to all assumptions coming from the statement of Theorem~\ref{theo:normalisateur_moyennable} including amenability of $\cala$.
We denote by $\HF$ the maximal $(\cala,\rho)$-invariant almost free factor system (coming from Definition~\ref{dfn_pure_cocycle} of purity). We choose a free factor system $\calf\subseteq\HF$ extracted from $\HF$.
Then $\calf$ is $(\cala,\rho)$-invariant (as $\rho$ takes its values in $\ia$, see Theorem~\ref{theo:ia}).

In all this section, we will work relative to $\calf$, and leave $\calf$ implicit in the notations:
we denote by $\Pbaro=\Pbaro(F_k,\calf)$ the compactification of the projectivized Outer space of $(F_k,\calf)$, 
by $\PAT\subset \Pbaro$ the set of projective arational trees,
and by $\FF$ the set of conjugacy classes of proper $(F_k,\calf)$-free factors.

We record the following fact for future use.

\begin{lemma}\label{lem_calf_max}
  The free factor system $\calf$ is  nonsporadic, and everywhere maximal in the following sense:
  if $\calf'\sqsupseteq\calf$ is a free factor system which is
  $(\cala_{|U},\rho)$-invariant for some
   Borel subset $U\subset Y$ of positive measure, then $\calf'=\calf$.
\end{lemma}

\begin{proof}
 Non-sporadicity of $\calf$ follows from the fact that $(\cala,\rho)$ is $\FS$-averse (this is an assumption in Definition~\ref{dfn_pure_cocycle} of purity). 

 For the maximality statement, by purity of $(\cala,\rho)$, 
 we have $\calf\sqsubseteq\calf'\sqsubseteq\HF$. 
Since $\calf'$ is a free factor system, $\calf=\calf'$.
\end{proof}

\begin{rk}\label{rk:standing}
Amenability of $\cala$ will only be used in Section~\ref{sec:amen-1} to build an equivariant map $Y\to\Prob(\Pbaro)$. In later sections, we will only use the existence of this map, and not amenability itself.
\end{rk}

\subsection{Equivariant probability measures on the compactified Outer space}\label{sec:amen-1}

\begin{lemma}\label{lemma:mu} 
There exists an $(\cala,\rho)$-equivariant Borel map $\mu:Y\ra \Prob(\Pbaro)$.
\end{lemma}

\begin{proof}
  This follows from the fact that $\cala$ is amenable and $\Pbaro$ is compact and metrizable,
  by applying Proposition~\ref{prop:zimmer}.
\end{proof}

In the sequel, we will often denote by $\mu_y$ the probability measure $\mu(y)$.

\subsection{Arational trees and Reynolds' witness map}\label{sec:amen-2}

Recall that an $\bbR$-tree $T$ in the boundary of Outer space is \emph{not} arational
if some proper $(F_k,\calf)$-free factor acts non-discretely or fixes a point.
Theorem~\ref{theo:reducing} below, established in this form in \cite[Lemma~5.5]{Hor} by generalizing an argument of Reynolds \cite{Rey} to a relative setting, asserts the existence of a \emph{witness map}, which equivariantly assigns to any non-arational tree $T$ a finite collection of conjugacy classes of free factors, witnessing the fact that $T$ is not arational.

We denote by $\AT=\AT(F_k,\calf)$ the space of all arational $(F_k,\calf)$-trees, and by $\mathbb{P}\AT$ the space of all projective classes of arational $(F_k,\calf)$-trees.  
We will use that $\mathbb{P}\AT$ is a standard Borel space: indeed it is a Borel subset of the standard Borel space $\bbP\baro$ (see e.g.\ \cite[Lemma~5.5]{Hor}), from which it follows that it is standard Borel \cite[Corollary~13.4]{Kec}.
We recall that $\FF$ denotes the countable set of all conjugacy classes of proper $(F_k,\calf)$-free factors (i.e.\ free factors relative to $\calf$
that are non-trivial, not $\calf$-peripheral, and not equal to $F_k$ itself, see Section \ref{sec:factor-systems}).
Given a set $X$, we denote by $\calp_{<\infty}(X)$ the set of all non-empty finite subsets of $X$.

\begin{theo}[Witness map {\cite[Lemma~5.5]{Hor}, \cite{Rey}}]\label{theo:reducing}
There exists an $\Out(F_k,\calf)$-equivariant Borel map $$\Pbaro\setminus\PAT\ra \calp_{<\infty}(\FF).$$ 
\end{theo}

 Let now $\mu:Y\to\Pbaro$ be a map as in Lemma~\ref{lemma:mu}.

\begin{lemma}\label{lemma:proba_PAT}
For a.e.\ $y\in Y$, the probability measure $\mu_y$ gives full measure to $\PAT$.
\end{lemma}

\begin{proof}
  Otherwise, let $U\subseteq Y$ be a Borel subset of positive measure such that for every $y\in U$, one has $\mu_y(\Pbaro\setminus\PAT)>0$.
  By renormalizing, we deduce an  $(\cala_{|U},\rho)$-equivariant Borel map $U\to \Prob(\Pbaro\setminus\PAT)$.
  Via the witness map given by Theorem~\ref{theo:reducing}, we get for each $y\in U$ a probability measure on $\calp_{<\infty}(\FF)$. This yields a contradiction to the following lemma.
\end{proof}

\begin{lemma}\label{lemma:no-prob-on-FF} 
For every Borel subset $U\subseteq Y$ of positive measure, there is no $(\cala_{|U},\rho)$-equivariant Borel map that assigns to every $y\in U$ a probability measure on  $\calp_{<\infty}(\FF)$.
\end{lemma}

\begin{proof}
  There is an $\Out(F_k,\calf)$-equivariant Borel map $\Prob(\calp_{<\infty}(\FF))\to\calp_{<\infty}(\FF)$, assigning to every probability measure $\nu$ the union of all finite subsets of $\FF$ with maximal $\nu$-measure. It is therefore enough to show that there is no  $(\cala_{|U},\rho)$-invariant Borel map $U\to\calp_{<\infty}(\FF)$. Assume towards a contradiction that such a map $\psi$ exists.
  Since $\psi$ takes values in the countable set $\calp_{<\infty}(\FF)$, let  $V\subseteq U$ be a Borel subset of positive measure such that $\psi_{|V}$ is constant,
  with value a nonempty finite collection $\calc\subset\FF$. Since $\rho$ takes its values in $\ia$, it follows that every free factor in $\calc$ is $(\cala_{|V},\rho)$-invariant. Choosing $[A]\in\calc$, it follows that the free factor system $\calf\vee\{[A]\}$ (see Section~\ref{sec:factor-systems}) is $(\cala_{|V},\rho)$-invariant. This contradicts that $\calf$ is everywhere maximal (Lemma~\ref{lem_calf_max}).
\end{proof}

\subsection{A barycenter argument}\label{sec:amen-3}

Two arational $(F_k,\calf)$-trees $T$ and $T'$ are said to be \emph{equivalent} (which we denote as $T\sim T'$) if there exist $F_k$-equivariant alignment-preserving maps from $T$ to $T'$ and from $T'$ to $T$ (recall that a map $T\to T'$ is \emph{alignment-preserving} if it maps every segment in $T$ to a segment in $T'$).

Given a space $X$, we denote by $X^{(3)}$ the subspace of $X^3$ made of all pairwise distinct triples.

The following theorem was proved in a joint work with Lécureux \cite{GHL}. As a matter of fact $\PATsim$
is also the Gromov boundary of the free factor graph of $(F_k,\calf)$ \cite{BR,Ham2,GH}, which is hyperbolic by \cite{BF-FF,HM} (in particular $\PATsim$ is Polish by \cite[Proposition~3.4.18]{DSU}, hence a standard Borel space). The following statement can therefore be viewed as giving a barycenter map on the free factor graph. 

\begin{theo}[{\cite[Theorem~8.1]{GHL}}]\label{barycenter}
There exists an $\Out(F_k,\calf)$-equivariant Borel map  $$(\PATsim)^{(3)}\to\calp_{<\infty}(\FF).$$ 
\end{theo}

In fact \cite[Theorem~8.1]{GHL} gives a stronger statement assigning to every triple in $(\PATsim)^{(3)}$ a finite set of free splittings of $(F_k,\calf)$. Theorem~\ref{barycenter} follows as there is a canonical way to associate a finite set of conjugacy classes of proper free factors of $(F_k,\calf)$ to every free splitting $S$ of $(F_k,\calf)$, by considering all nonperipheral free factors that are equal to a vertex stabilizer in either $S$ or some collapse of $S$ (there exists one because $\calf$ is non-sporadic).

By Lemma \ref{lemma:proba_PAT}, we have for a.e.\ $y\in Y$ a probability  measure $\mu_y$ on $\PAT$. 
The following lemma shows that for a.e.\ $y\in Y$, its image under the quotient map $\PAT\onto \PATsim$  is supported on a set of cardinality at most $2$.  

\begin{lemma}\label{lemma:cardinal-2} 
For every subset of positive measure $U\subseteq Y$, every $(\cala_{|U},\rho)$-equivariant Borel map $\nu:U\to\Prob(\PATsim)$, and a.e.\ $y\in U$, the support of $\nu_y$ has cardinality at most $2$.
\end{lemma}

\begin{proof}
  Assume by contradiction that the support of  $\nu_y$ has cardinality at least $3$ on a Borel subset $U'\subseteq U$ of positive measure. Then for all $y\in U'$, the probability measure $\nu_y\otimes\nu_y\otimes\nu_y$ on
 $(\PAT/{\sim})^3$ gives positive measure to the subset $(\PAT/{\sim})^{(3)}$  of pairwise distinct triples. 
Therefore, we get an $(\cala_{|U'},\rho)$-equivariant Borel map $U'\to\Prob((\PAT/{\sim})^{(3)})$. Using Theorem~\ref{barycenter}, we deduce an  $(\cala_{|U'},\rho)$-equivariant Borel map  $U'\to\Prob(\calp_{<\infty}(\FF))$. This contradicts Lemma~\ref{lemma:no-prob-on-FF}.
\end{proof}

Given a set $X$, we denote by $\calp_{\le 2}(X)$ the collection of all nonempty subsets of $X$ of cardinality at most $2$.

\begin{cor}\label{cor:psi-0}
 There exists an $(\cala,\rho)$-equivariant Borel map
$\psi_0:Y\ra \calp_{\le 2}(\PATsim)$.
\end{cor}

\begin{proof}
  Lemmas~\ref{lemma:mu} and~\ref{lemma:proba_PAT} show that there exists an $(\cala,\rho)$-equivariant Borel map $\mu:Y\to\Prob(\PAT)$, and $\mu$ then induces an $(\cala,\rho)$-equivariant Borel map $\nu:Y\to\Prob(\PATsim)$. Lemma~\ref{lemma:cardinal-2} ensures that for a.e.\ $y\in Y$, the support of  $\nu_y$ has cardinality at most $2$. The map sending $y$ to the support of $\nu_y$ is then an $(\cala,\rho)$-equivariant Borel map $Y\to\calp_{\le 2}(\PATsim)$ (defined almost everywhere).
\end{proof}

\subsection{Canonicity}\label{sec:amen-4} 

\begin{lemma}\label{lem:canonical_pair} 
  There exists an essentially unique $(\cala,\rho)$-equivariant Borel map
$\phi:Y\ra  \calp_{\le2}(\PATsim)$
which is everywhere maximal in the following sense:
for every Borel subset $U\subseteq Y$ of positive measure, every $(\cala_{|U},\rho)$-equivariant Borel map
$\psi:U\ra \calp_{\le2}(\PATsim)$,  and almost every $y\in U$, $\psi(y)$ is contained in $\phi(y)$.
\end{lemma}

\begin{proof} The argument is exactly Adams' in \cite[Lemmas~3.2 and~3.3]{Ada}.
The uniqueness is clear.

To prove the existence, we observe that if $\psi,\psi':U\ra \calp_{\le2}(\PATsim)$ are two
$(\cala_{|U},\rho)$-equivariant Borel maps, then for almost every $y\in U$, $\psi(y)\cup \psi'(y)$
has cardinality at most 2 by Lemma~\ref{lemma:cardinal-2}.
So one would like to construct $\phi$ as the union of all maps $\psi$, but one has to 
be careful because of measurability issues.

For $U\subset Y$ we denote by $\bar U$ the $\cala$-saturation of $U$ i.e.\ the set of 
points $y\in Y$ such that there exists an element of $\cala$ with source $y$ and range in $U$.
We observe that if $\psi:U\ra \calp_{\le2}(\PATsim)$ is $(\cala_{|U},\rho)$-equivariant,
it extends naturally to an $(\cala_{|\bar{U}},\rho)$-equivariant map $\bar\psi:\bar U\ra\calp_{\le2}(\PATsim)$
by letting  $\bar\psi(r(g))=\rho(g)\psi(s(g))$ for all $g\in\calg_{|\bar U}$ with $s(g)\in U$ 
(this does not depend on choices).
It can then be extended to $Y$ by $\bar \psi(y)=\psi_0(y)$ for all $y\notin \bar U$ (where $\psi_0$ is  a map provided by Corollary~\ref{cor:psi-0}).

Let $U\subseteq Y$ be a Borel subset of maximal measure such that there exists a stably $(\cala_{|U},\rho)$-equivariant
Borel map $\phi'$ from $U$ to the set $\calp_{=2}(\PATsim)$ of all subsets of $\PATsim$ of cardinal exactly 2. Then there is actually an $(\cala_{|U},\rho)$-equivariant Borel map $\phi:U\to\calp_{=2}(\PATsim)$ (this is proved by starting from a Borel subset $V\subseteq U$ on which $\phi'$ is actually equivariant, extending it to an equivariant map on $\bar V$ as above, and repeating the process on the complement of $\bar V$, at most countably many times).  
The set $U$ is then almost equal to its saturation $\bar U$,
and extending $\phi$ on $Y\setminus \bar U$ using $\psi_0$,
one gets a map $\phi:Y\ra  \calp_{\le2}(\PATsim)$
which is $(\cala,\rho)$-equivariant and maximal as required.
\end{proof}

The advantage of the maximal map constructed in Lemma~\ref{lem:canonical_pair} is that it is also invariant under the
given subgroupoid $\calh$ which normalizes $\cala$.

\begin{cor}\label{cor:canonical_pair2}
There exists a Borel map $\phi:Y\ra \calp_{\le2}(\PATsim)$ which is both $(\cala,\rho)$-equivariant and $(\calh,\rho)$-equivariant.
\end{cor}

\begin{proof}
  Let $\phi:Y\ra \calp_{\leq 2}(\PATsim)$ be the essentially unique maximal $(\cala,\rho)$-equivariant
  map given by the previous lemma.
  Since $\calh$ normalizes $\cala$, one can write $\calh$ as a countable union of bisections $B_n\subset \calh$ such that
for all $a\in\cala$ and all $\gamma,\delta\in B_n$ such that $\gamma a \delta\m$ is defined, 
one has $a\in\cala$ if and only if
$\gamma a \delta\m\in \cala$. Up to subdividing $B_n$, we may assume without loss of generality that the cocycle $\rho$ takes a single value on $B_n$, say $\rho(B_n)=\{\beta_n\}$.
Let $f_{B_n}:U_n\ra V_n$ be the partial isomorphism associated to the bisection $B_n$, with $U_n,V_n\subset Y$.

\begin{figure}[htb]
    \centering
    \includegraphics{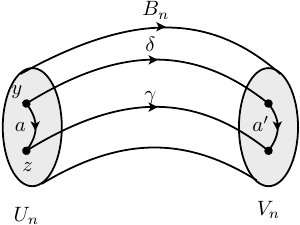}
    \caption{Equivariance of $\phi'$}
    \label{fig_eqv}
\end{figure}

Assume that $U_n$ has positive measure. We claim that the Borel map
\begin{align*}
  \phi':U_n&\ra \calp_{\le2}(\PATsim) \\
  y&\mapsto\beta_n\m .\phi( f_{B_n}(y))
\end{align*}
is  $(\cala_{|U_n},\rho)$-equivariant.
Indeed, up to replacing $U_n$ and $V_n$ by conull Borel subsets, we will assume that the equivariance relation for $\phi$ holds everywhere on $\cala_{|V_n}$. For all $a\in \cala_{|U_n}$, denoting by $y,z\in U_n$ the source and target of $h$
we need to check that $\phi'(z)=\rho(a)\phi'(y)$.
Let $\delta,\gamma\in B_n$ be the elements with source $y,z$;
the element $a'=\gamma a\delta\m$ lies in $\cala$ and its source and target are $f_{B_n}(y)$ and $f_{B_n}(z)$ respectively.
By equivariance $\phi(f_{B_n}(z))=\rho(a')\phi(f_{B_n}(y))$, so $$\phi'(z)=\beta_n\m\phi(f_{B_n}(z))=\beta_n\m\rho(a')\phi(f_{B_n}(y))=\beta_n\m\rho(a')\beta_n\phi'(y)=\rho(a)\phi'(y),$$ as desired.

A similar argument shows that $\phi'$ is maximal among $(\cala_{|U_n},\rho)$-equivariant 
maps $U_n\to\calp_{\le 2}(\PATsim)$,
so there is a subset of full measure $U'_n\subseteq U_n$ such that $\phi'$ and $\phi$ coincide on $U'_n$.
This still holds if $U_n$ has measure zero taking $U'_n=\es$.
Taking $Y'=\cup_n U'_n$, this shows that $\phi_{|Y'}$ is $(\calh_{|Y'},\rho)$-equivariant.
\end{proof}

At this point, we would like to make an observation, which will be useful later in Section~\ref{sec:extension} in order to establish a slight refinement  of the main theorem of the present section. We insist (as already mentioned in Remark~\ref{rk:standing}) that in Sections~\ref{sec:amen-2},~\ref{sec:amen-3},~\ref{sec:amen-4}, 
we never really used the amenability of $\cala$. Instead, what we used is the existence of an $(\cala,\rho)$-equivariant Borel map $\mu:Y\to\Prob(\Pbaro)$ established in Section~\ref{sec:amen-1}. We record this in the following proposition for future use.

\begin{prop}\label{prop:prepare-for-extension}
  Let $\calg$ be a measured groupoid over a base space $Y$, let $\rho:\calg\to\ia$ be a cocycle, and let $\cala,\calh$ be measured subgroupoids of $\calg$. Assume that $(\cala,\rho)$ is pure, and let $\HF$ be the maximal $(\cala,\rho)$-invariant \afs, and $\calf\subset\HF$ be an extracted free factor system.
  
  Assume that $\calh$ normalizes $\cala$ and that there exists an $(\cala,\rho)$-equivariant Borel map $Y\to \calp_{\le 2}(\PATsim)$. 

Then there exists a Borel map $Y\to \calp_{\le 2}(\PATsim)$ which is both $(\cala,\rho)$-equivariant and $(\calh,\rho)$-equivariant.
\qed
\end{prop}

As in our standing assumptions, the space $\PAT$ consists of projective arational trees relative to the free factor
system $\calf$ defined in the statement. Since $\calh$ normalizes $\cala$, the unique 
maximal $(\cala,\rho)$-invariant \afs\ $\hat\calf$ is $(\calh,\rho)$-invariant,
and as we are working in $\ia$,
 so is $\calf$.
However, $\calh$ could preserve larger \afs{}s.

\subsection{Using a witness map for trees with non-amenable stabilizer}\label{sec:witness3}


We denote by 
$\PATnasimd$ the set of pairs of (maybe not distinct)  equivalence classes of trees $[[T]],[[T']]\in \PATsim$,
such that the common stabilizer of $[[T]]$ and $[[T']]$ is not amenable.

\begin{prop}\label{prop:PATna2}
Let $\calf$ be a non-sporadic free factor system of $F_k$.

  There exists an $\Out(F_k,\calf)$-equivariant Borel map assigning to any 
$\{[[T]],[[T']]\}\in \PATnasimd$
a nice splitting $S_{\{[[T]],[[T']]\}}$.
\end{prop}

\begin{proof} 
  Given $\{[[T]],[[T']]\}\in \PATnasimd$, consider $[\TT]=\Ext([[T]])\cup\Ext([[T']])\in \calp_{<\infty}(\PAT)$, where $\Ext:\PATsim\ra \calp_{<\infty}(\PAT)$ is the Borel map from Lemma~\ref{lem:multisection}.
  The splitting $S_{[\TT]}$ from Corollary \ref{cor:PATna} is nice,
  and we just define $S_{\{[[T]],[[T']]\}}:=S_{[\TT]}$. 
\end{proof}

\subsection{Conclusion}\label{sec:amen-6}

\begin{proof}[Proof of Theorem~\ref{theo:normalisateur_moyennable}]
  As per our standing assumptions, we can assume without loss of generality that $(\cala,\rho)$ is pure.
  Let $\HF$ be the unique maximal $(\cala,\rho)$-invariant \afs\ and $\calf$ be an extracted free factor system.
  Since $\calh$ normalizes $\cala$, we have that $(\calh,\rho)$ preserves $\HF$. Since $\rho$ takes values in $\ia$,
  it follows that $(\calh,\rho)$ also preserves $\calf$ (see Theorem~\ref{theo:ia}). We now work relative to $\calf$.

  By Corollary \ref{cor:canonical_pair2},  
  there is a Borel map
  $\phi:Y\ra \calp_{\le 2}(\PATsim)\subset \calp_{<\infty}(\PATsim)$
which is both  $(\cala_{|U},\rho)$ and $(\calh_{|U},\rho)$-equivariant.
 Recall the partition
$\calp_{<\infty}(\PATsim)=\PaPATsim\dunion \PnaPATsim$.

Assume that for all $y$ in some subset $U\subset Y$ of positive measure,
$\phi(y)$ has non-amenable stabilizer, i.e.\ $\phi(y)\in\PnaPATsim$.
Then composing with the map of Proposition \ref{prop:PATna2}, one gets 
a Borel map $y\mapsto S_{\phi(y)}$  assigning a nice splitting to every $y\in U$, and which is both $(\cala_{|U},\rho)$ and $(\calh_{|U},\rho)$-equivariant.
Since there are only countably many nice splittings, there is a Borel subset  $V\subseteq U$ of positive measure
in restriction to which $\phi$ is constant and defines
 an $(\calh_{|V},\rho)$-invariant nice splitting.

Thus, we may assume that $\phi(y)\in\PaPATsim$ for almost every $y\in Y$. Since the $\Out(F_k,\calf)$-action on $\PaPATsim$ is Borel amenable (Theorem~\ref{theo:action-amenable-2})
and since  $\rho_{|\calh}$ has trivial kernel, Proposition~\ref{prop:mackey} implies that $\calh$ is amenable.
\end{proof}

\subsection{An extension}\label{sec:extension}

We now establish a refined version of Theorem~\ref{theo:normalisateur_moyennable}, that we will use in the sequel of the paper (precisely, in the proof of Lemma~\ref{lemma:check-p3}).

\begin{theo}\label{theo:a-etages}
  Let $\calg$ be a measured groupoid over a base space $Y$, and let $\caln_1,\dots,\caln_p$ be measured subgroupoids of $\calg$, such that for every $i\in\{1,\dots,p-1\}$, the subgroupoid $\caln_i$ is normalized by $\caln_{i+1}$. Let $\rho:\calg\ra \ia$ be a cocycle
 such that $\rho_{|\caln_i}$ is nowhere trivial for all $i\in\{1,\dots,p\}$.  Assume that $\caln_1$ is amenable and that $(\caln_1,\rho)$ is pure, and let $\hat\calf$ be the unique maximal $(\caln_1,\rho)$-invariant almost free factor system.

Then one of the following three assertions holds:
\begin{enumerate}
\item there exists a Borel subset $U\subseteq Y$ of positive measure such that $((\caln_p)_{|U},\rho)$ has a nice invariant splitting;
\item there exists a Borel subset $U\subseteq Y$ of positive measure such that $((\caln_p)_{|U},\rho)$ fixes an almost free factor system
  $\hat{\calf}'\sqsupsetneq\HF$;
\item for every measured subgroupoid $\calh'\subseteq\caln_p$, if $\rho_{|\calh'}$ has trivial kernel, then $\calh'$ is amenable.
\end{enumerate}
\end{theo}

\begin{proof}
  In view of Proposition~\ref{prop:nice-preserved}, and using the fact that the cocycles $\rho_{|\caln_i}$ are all nowhere trivial, the first conclusion holds as soon as there exist some $i\in\{1,\dots,p\}$ and some Borel subset $U\subseteq Y$ of positive measure for which there exists an $((\caln_i)_{|U},\rho)$-invariant nice splitting. 
  
  We now assume otherwise, i.e.\ for every $i\in\{1,\dots,p\}$, $(\caln_i,\rho)$ is nice-averse. Then Proposition~\ref{prop:stably-nice-vs-pure-nice-averse} implies that each $(\caln_i,\rho)$ is also stably pure. Thus, we can find a partition $Y=\Dunion_{j\in J}Y_j$ into at most countably many Borel subsets such that for every $i\in\{1,\dots,p\}$ and  every $j\in J$, the pair $((\caln_i)_{|Y_j},\rho)$ is pure, with maximum invariant almost free factor system $\hat{\calf}_{i,j}$. As $\caln_i$ is normalized by $\caln_{i+1}$, for every $i\in\{1,\dots,p-1\}$ and every $j\in J$, the almost free factor system $\hat{\calf}_{i,j}$ is also $((\caln_{i+1})_{|Y_j},\rho)$-invariant, so $\hat{\calf}_{i,j}\sqsubseteq\hat{\calf}_{i+1,j}$. We can therefore assume that for every $j\in J$, one has $\hat{\calf}_{p,j}=\hat{\calf}$, as otherwise the second conclusion of the theorem holds and we are done. 

    Let $\calf\subset\HF$ be an  extracted free factor system. 
    Recall that purity implies that $\calf$ is non-sporadic (Remark \ref{rk_pure_nonsporadic}).
    Moreover, as in Lemma \ref{lem_calf_max},
    for every $j\in J$ and every $i\in\{1,\dots,p\}$,  $\calf$ is an everywhere maximal $((\caln_i)_{|Y_j},\rho)$-invariant free factor system.
    In the sequel, we always work relative to $\calf$ (in particular, Outer spaces, proper free factors and arational trees are understood relative to $\calf$).

As $\caln_1$ is amenable, Corollary~\ref{cor:canonical_pair2} ensures that there exists a Borel map $\phi:Y\to\calp_{\le 2}(\PATsim)$ which is both $(\caln_1,\rho)$-equivariant and $(\caln_2,\rho)$-equivariant. Applying Proposition~\ref{prop:prepare-for-extension}
iteratively $p-2$ times ensures that there exists a Borel map $\phi:Y\to \calp_{\le 2}(\PATsim)$ which is $(\caln_p,\rho)$-equivariant.

We will finally use the partition $\calp_{<\infty}(\PATsim)=\PaPATsim\dunion \PnaPATsim$. Assume first that there is a Borel subset $U\subseteq Y$ of positive measure such that $\phi(y)\in\PnaPATsim$ for all $y\in U$. Using Proposition~\ref{prop:PATna2},
we deduce an $((\caln_p)_{|U},\rho)$-equivariant Borel map from $U$ to the countable set of all nice splittings. We can then restrict to a Borel subset $V\subseteq U$ where this map is constant to get an $((\caln_p)_{|V},\rho)$-invariant nice splitting, so the first conclusion of the theorem holds. We are thus left with the case where $\phi(y)\in\PaPATsim$ for a.e.\ $y\in Y$. In this case, as the  $\Out(F_k,\calf)$-action on $\PaPATsim$ is amenable (Theorem~\ref{theo:action-amenable-2}), Proposition~\ref{prop:mackey} ensures that a measured subgroupoid $\calh'\subseteq\caln_p$ is amenable whenever $\rho_{|\calh'}$ has trivial kernel, so the third conclusion of the theorem holds.  
\end{proof}

\section{Detecting when a subgroupoid stabilizes a free splitting}\label{sec:big}

The present section is of central importance in the proof of our main theorem; its goal is the following.
Let $\calg$ be a measured groupoid coming with an action-type cocycle $\rho:\calg\to\IA$,  with $N\ge 3$. Given a measured subgroupoid $\calh$ of $\calg$, we want to relate the fact that $(\calh,\rho)$ preserves a non-separating free splitting of $F_N$ (up to a partition of the base space of $\calh$ into at most countably many pieces), to some purely groupoid-theoretic properties of the pair $(\calg,\calh)$, without reference to the cocycle $\rho$.
This will be crucial in later sections to get a map at the level of (a variation of) the free splitting graph intertwining a pair of action-type cocycles $\calg\to\IA$ as above.


To begin with, let us start by describing the stabilizer of a
non-separating free splitting $S$, i.e.\ of an HNN extension $F_N=A*$ over the trivial group.
By Proposition~\ref{prop_suite_exacte}, an index 2 subgroup $H^0$ of the stabilizer $H$ of $S$ fits in a short exact sequence
$$ 1 \ra A\times A \ra H^0 \ra \Out(A) \ra 1$$
where the map to $\Out(A)$ is the natural restriction map
and the kernel is the group of twists; if $t$ is a stable letter of the HNN extension,
the twists are the outer automorphisms of the form $\tau_{u,v}$ (with $u,v\in A$) acting as the identity on $A$ 
and sending $t$ to $utv$.
When $N\geq 3$, 
$H^0$ normalizes two non-amenable subgroups that form a direct product -- this will be translated to the groupoid-theoretic context as Property~\PI.
To introduce this property, we first need an analogue of direct products for groupoids, which is the role of our first subsection.


In all this section, we fix $N\geq 3$.
\subsection{Schottky tuples and pseudo-products}

\begin{de}[Schottky tuples of subgroupoids]
Let $\calg$ be a measured groupoid over a base space $Y$. A \emph{Schottky tuple of subgroupoids} of $\calg$ is a tuple $(\cala^1,\dots,\cala^k)$ of amenable subgroupoids of $\calg$ of infinite type such that for every Borel subset $U\subseteq Y$ of positive measure, the subgroupoid $\langle \cala^1_{|U},\dots,\cala^k_{|U}\rangle$ is non-amenable.
\end{de}

\begin{rk}\label{rk:schottky-stabilities}
  This definition implies that $k\geq 2$ and that $\grp{\cala^1,\dots,\cala^k}$ is everywhere non-amenable. But this is not equivalent
  because in general,
  $\grp{\cala^1_{|U},\dots,\cala^k_{|U}}$ does not coincide with the restriction of $\grp{\cala^1,\dots,\cala^k}$ to $U$.
  The point of this notion is that it is stable under restriction and stabilization. More precisely, 
\begin{itemize}
\item If $V\subseteq Y$ is a Borel subset of positive measure, and if $(\cala^1,\dots,\cala^k)$ is a Schottky tuple of subgroupoids of $\calg$, then $(\cala^1_{|V},\dots,\cala^k_{|V})$ is a Schottky tuple of subgroupoids of $\calg_{|V}$.
\item If there exists a partition $Y=\Dunion_{i\in I}Y_i$ into at most countably many Borel subsets, and if $\cala^1,\dots,\cala^k$ are subgroupoids of $\calg$ such that for every $i\in I$, the tuple $(\cala^1_{|Y_i},\dots,\cala^k_{|Y_i})$ is a Schottky tuple of subgroupoids of $\calg_{|Y_i}$, then $(\cala^1,\dots,\cala^k)$ is a Schottky tuple of subgroupoids of $\calg$.
\end{itemize}
\end{rk}

An important source of examples, which justifies the name \emph{Schottky}, is given by the following lemma.

\begin{lemma}[{Kida \cite[Lemma~3.20]{Kid2}}]\label{lemma:schottky}
  Let $\calg$ be a measured groupoid, let $\Gamma$ be a countable group, and let $\rho:\calg\to \Gamma$ be an
  action-type cocycle. Let $A_1$ and $A_2$ be two infinite cyclic subgroups of $\Gamma$ such that $\langle A_1,A_2\rangle$ is a non-abelian free group. 

  Then the groupoids $\cala_1=\rho^{-1}(A_1)$ and $\cala_2=\rho^{-1}(A_2)$ form
  a Schottky pair of subgroupoids of $\calg$.
\end{lemma}

We now introduce pseudo-products of subgroupoids. These give a counterpart for groupoids to the notion of a direct product $G_1\times\dots\times G_k$ of groups that all contain non-abelian free subgroups. Our definition in the groupoid setting translates the fact that each of the factors $G_i$ of the direct product contains two infinite cyclic subgroups $\grp{a_i},\grp{b_i}$ such that
$\grp{a_i,b_i}\simeq F_2$ with $\grp{a_i}$ and $\grp{b_i}$ both normalized by all other factors $G_j$.

\begin{de}
  Given a measured subgroupoid $\calh$ of $\calg$ and a Schottky tuple $(\cala^1,\dots,\cala^k)$ of subgroupoids of $\calg$,
  we say that $(\cala^1,\dots,\cala^k)$ is \emph{stably normalized} by $\calh$ if each $\cala^i$ is.
\end{de}

\begin{de}[Pseudo-products]\label{de:pseudo-product}
    Let $\calg$ be a measured groupoid,
    let $k\in\mathbb{N}$, and let $\calg_1,\dots,\calg_k$ be measured subgroupoids of $\calg$. 

    The subgroupoids $\calg_1,\dots,\calg_k$ \emph{form a pseudo-product} if each $\calg_i$ contains a Schottky tuple of subgroupoids which is stably normalized by all $\calg_j$ with $j\neq i$.
\end{de}

\begin{rk} 
  In view of Remark~\ref{rk:schottky-stabilities}, this notion is stable under restriction and stabilization. More precisely,
  denoting by $Y$ the base space of $\calg$,
\begin{itemize}
\item If $U\subseteq Y$ is a Borel subset of positive measure, and if the subgroupoids $\calg_i$ form a pseudo-product, then so do the subgroupoids $(\calg_i)_{|U}$ of $\calg_{|U}$. 
\item If the  subgroupoids $\calg_i$ stably form a pseudo-product, then they form a pseudo-product. 
\end{itemize}
\end{rk}

A typical example, which justifies the terminology and will be of particular importance in the sequel, is the following. 

\begin{lemma}\label{lemma:pseudo-product}
  Let $\calg$ be a measured groupoid, let $\Gamma$ be a countable group, and let $\rho:\calg\to \Gamma$ be an action-type cocycle. Assume that $\Gamma$ contains subgroups $G_1,\dots,G_k$ that generate their direct product in $\Gamma$, and that all contain a non-abelian free subgroup.

Then the groupoids $\calg_i=\rho^{-1}(G_i)$ for $i\leq k$ form a pseudo-product.
\end{lemma}

\begin{proof}
For every $i\in\{1,\dots,k\}$, let $A_i^1$ and $A_i^2$ be two infinite cyclic subgroups of $G_i$ that generate a non-abelian free subgroup. By Lemma~\ref{lemma:schottky}, for every $i\in\{1,\dots,k\}$, the pair $(\rho^{-1}(A_i^1),\rho^{-1}(A_i^2))$ is a Schottky pair of subgroupoids of $\rho^{-1}(G_i)$. In addition, since $G_1,\dots,G_k$ form a direct product, the subgroupoids $\rho^{-1}(A_i^1)$ and $\rho^{-1}(A_i^2)$ are normalized by $\rho^{-1}(G_j)$ for $j\neq i$.
\end{proof}

\begin{rk}\label{rk:nowhere-nonamenable}
  If $\calg_1,\dots,\calg_k$ form a pseudo-product, then each $\calg_i$ is everywhere non-amenable since it contains a Schottky tuple of subgroupoids.
\end{rk}

\subsection{Main statement}

Let $\calg$ be a measured groupoid over a base space $Y$, and let $\calh$ be a measured subgroupoid of $\calg$. We consider the following properties of the pair $(\calg,\calh)$. 

\begin{enumerate}
\item[$(P_1)$] \hypertarget{Pi}{}
  The groupoid $\calh$ is everywhere non-amenable and stably normalizes two subgroupoids of $\calg$
  that form a pseudo-product.
\item[$(P_1^{\max})$] \hypertarget{Pimax}{}
  The pair $(\calg,\calh)$ satisfies $(P_1)$, and 
  for every Borel subset $U\subseteq Y$ of positive measure, if $\calh'$ is
  a subgroupoid of $\calg_{|U}$ such that $(\calg_{|U},\calh')$ satisfies $(P_1)$, with $\calh_{|U}$ stably contained in $\calh'$, then $\calh_{|U}$ is stably equal to $\calh'$.
\item[$(P_2)$] \hypertarget{Pii}{}
  For every Borel subset $U\subseteq Y$ of positive measure, the groupoid $\calh_{|U}$ does not
  normalize any amenable subgroupoid of $\calg_{|U}$ of infinite type.
\item[$(P_3)$]\hypertarget{Piii}{}
  There do not exist a Borel subset $U\subseteq Y$ of positive measure and two normal subgroupoids $\caln^-,\caln^+$ of $\calh_{|U}$ which contain respectively three subgroupoids
  $\calb_1^-,\calb_2^-,\calb_3^-$ and $\calb_1^+,\calb_2^+,\calb_3^+$ such that for all $\eps\in\{\pm\}$, the groupoids $\calb_1^\eps,\calb_2^\eps,\calb_3^\eps,\caln^{-\eps}$  form a pseudo-product.
\end{enumerate}

Sometimes, when the ambient groupoid $\calg$ is understood from context, we will slightly abuse notation and simply say that $\calh$ satisfies the property instead of $(\calg,\calh)$.

Property \PI\ is our groupoid-theoretic version of the property of normalizing a direct product
of non-amenable groups, which is a key property of the stabilizer of a non-separating free splitting. Property \PImax\ is a maximality property among subgroupoids satisfying \PI. 
As a first step, we will prove that Property~\PI\ forces the subgroupoid to (stably) stabilize a nice splitting (with respect to any action-like cocycle towards $\IA$). The next two properties will be used to distinguish free splittings from $\Zmax$ and bi-nonsporadic splittings. 

Property~\PII\ is the groupoid-theoretic version of the property of not normalizing any infinite amenable subgroup. It will be used to 
exclude stabilizers of $\Zmax$ splittings.

Property~\PIII\ translates the idea that whenever $H$ normalizes two subgroups that form a direct product, then those two subgroups cannot themselves contain direct products (recall that the group of twists $A\times A$ above is a direct product of two free groups).
It will be used to exclude stabilizers of bi-nonsporadic splittings.

\begin{rk}
  Notice that the assumptions $\PI$, $\PImax$, $\PII$, and $\PIII$ 
  are preserved under restriction and stabilization. More precisely, 
\begin{enumerate}
\item If $(\calg,\calh)$ satisfies any of these properties, and $U\subseteq Y$ is a
   Borel subset of positive measure, then $(\calg_{|U},\calh_{|U})$ satisfies the property.
\item If there exists a partition $Y=\Dunion_{i\in I}Y_i$ into at most countably many Borel subsets, such that for every $i\in I$, the pair $(\calg_{|Y_i},\calh_{|Y_i})$ satisfies one of these properties, then $(\calg,\calh)$ satisfies the corresponding property.
\end{enumerate}
\end{rk}

Our goal in the sequel of the present section is to consider an action-type cocycle $\rho:\calg\ra \IA$
and to relate the above purely groupoid-theoretic properties of the pair $(\calg,\calh)$ to
the fact that $(\calh,\rho)$ stabilizes a free splitting of $F_N$.

A \emph{one-edge non-separating free splitting} of $F_N$ is a free splitting of $F_N$ equal to the Bass--Serre tree of an HNN extension of the form $F_N=P\ast$, with $P$ a subgroup of $F_N$ of rank $N-1$. We denote by $\NS$ the countable collection of all one-edge non-separating free splittings. We say that a free splitting of $F_N$ is a \emph{$(2,2^+)$-free splitting} if it is a free splitting of $F_N$ equal to the Bass--Serre tree of a decomposition of the form $F_N=P*P'$ with
$\rank(P)=2$ and $\rank(P')\geq 2$.
We denote by $\FSii$ the countable set of all $(2, 2^+)$-free splittings of $F_N$.
Note that when $N=3$, $\FSii=\es$. 

The goal of the present section is to prove the following theorem. 

\begin{theo}\label{theo:non-characterization}
  Let $\calg$ be a measured groupoid over a base space $Y$, and let $\rho:\calg\to\IA$ be an  action-type cocycle with $N\geq 3$.
  Let $\calh$ be a measured subgroupoid of $\calg$.  
\begin{enumerate}
\item If $(\calg,\calh)$ satisfies Properties~$\PImax$,~$\PII$ and~$\PIII$, then there exists a conull Borel subset $Y^*\subseteq Y$ and a partition $Y^*=\Dunion_{i\in I} Y_i$ into at most countably many Borel subsets, such that for every $i\in I$, the groupoid $\calh_{|Y_i}$ is equal to the $(\calg_{|Y_i},\rho)$-stabilizer of a splitting $S_i\in\NS\cup\FSii$, i.e.\ $\calh_{|Y_i}=\rho^{-1}(\Stab_{\IA}(S_i))_{|Y_i}$.
\item If there exist a conull Borel subset $Y^*\subseteq Y$ and a partition $Y^*=\Dunion_{i\in I}Y_i$ into at most countably many Borel subsets, such that for every $i\in I$, the groupoid $\calh_{|Y_i}$ is equal to the $(\calg_{|Y_i},\rho)$-stabilizer of some splitting $S_i\in\NS$, then $(\calg,\calh)$ satisfies Properties~$\PImax$,~$\PII$ and~$\PIII$.
\end{enumerate}
\end{theo}

\begin{proof}
The first part is proved in Section~\ref{sec:proof-1} below, and the second is proved in Section~\ref{sec:proof-2}  (Lemmas~\ref{lemma:check-p1max},~\ref{lemma:check-p2} and~\ref{lemma:check-p3}).
\end{proof}

\begin{rk}
We do not know however whether stabilizers of splittings $\FSii$ satisfy Property~$\PIII$ or not, whence the slight asymmetry in the statement of Theorem~\ref{theo:non-characterization}.
\end{rk}

\subsection{A sufficient condition to stabilize a free splitting}\label{sec:proof-1}

Our goal in this section is to prove the first part of Theorem~\ref{theo:non-characterization}.

\subsubsection{Step 1: Getting an invariant nice splitting from Property~$(P_1)$}

\begin{lemma}\label{lemma:1-implies-splitting}
Let $\calg$ be a measured groupoid over a base space $Y$, and let $\rho:\calg\to\IA$ be a cocycle with trivial kernel. Let $\calh$ be a measured subgroupoid of $\calg$. Assume that $(\calg,\calh)$ satisfies Property~$\PI$.

Then $(\calh,\rho)$ is stably nice.
\end{lemma}

\begin{proof}
  By assumption $\calh$
  stably normalizes two subgroupoids $\caln^-,\caln^+$ of $\calg$ which form a pseudo-product.
  In particular, $\calh$ stably normalizes the subgroupoid $\caln^-$, which is
  everywhere non-amenable (Remark \ref{rk:nowhere-nonamenable}), 
  and $\caln^-$ stably normalizes an amenable subgroupoid $\cala^+$ of $\caln^+$ of infinite type.
  Applying Lemma~\ref{lemma:stably-nice-vs-nice-averse} to $\calh$, consider a  Borel partition $Y=Y_1\dunion Y_2$ with $(\calh_{|Y_1},\rho)$ stably nice
  and $(\calh_{|Y_2},\rho)$ nice-averse. We assume that $Y_2$ has positive measure and argue towards a contradiction.
  Because the property of  stably preserving a nice splitting goes to the normalizer (Proposition \ref{prop:nice-preserved}),
  $(\cala^+_{|Y_2},\rho)$ and $(\caln^-_{|Y_2},\rho)$ are nice-averse.
  As $\cala^+_{|Y_2}$ is of infinite type and amenable, and $\rho$ has trivial kernel,
  Corollary \ref{cor:normalisateur_moyennable}
   shows that $\caln^-_{|Y_2}$ is amenable,
  a contradiction.
\end{proof}

\subsubsection{Step 2: Excluding $\Zmax$-splittings thanks to Property~$(P_2)$}

\begin{lemma}\label{lemma:exclude-cyclic}
Let $\calg$ be a measured groupoid over a base space $Y$, and let $\rho:\calg\to\IA$ be an action-type cocycle. Let $\calh$ be a measured subgroupoid of $\calg$.

If there exists an $(\calh,\rho)$-invariant $\Zmax$-splitting, then $\calh$ normalizes an amenable subgroupoid of $\calg$ of infinite type.
\end{lemma}

\begin{proof}
 Let $S$ be a $\Zmax$-splitting which is $(\calh,\rho)$-invariant. Then there exists a conull Borel subset $Y^*\subseteq Y$ such that $\calh_{|Y^*}\subseteq\rho^{-1}(\Stab_{\IA}(S))$. 
 Let $A\subseteq\Stab_{\IA}(S)$ be the group of twists of the splitting $S$: this is an infinite abelian normal subgroup of $\Stab_{\IA}(S)$ (see Section~\ref{sec_twists} and Proposition~\ref{prop_suite_exacte}). 
 Let $\cala=\rho^{-1}(A)$: this is an amenable subgroupoid of $\calg$ (Corollary~\ref{cor:amenable-subgroup-subgroupoid}), of infinite type because $\rho$ is action-type. 
 In addition, as $A$ is normal in $\Stab_{\IA}(S)$, it follows that $\cala_{|Y^*}$ is normal in $\rho^{-1}_{|Y^*}(\Stab_{\IA}(S))$  (Lemma~\ref{lemma:normal-subgroup-subgroupoid}),
and therefore $\cala$ is normalized by $\calh$.
\end{proof}

Combining Lemmas~\ref{lemma:1-implies-splitting} and~\ref{lemma:exclude-cyclic}, we reach the following statement. 

\begin{cor}\label{cor:exclude-cyclic-splitting}
  Let $\calg$ be a measured groupoid over a base space $Y$, and let $\rho:\calg\to\IA$ be an action-type cocycle.
  Let $\calh$ be a measured subgroupoid of $\calg$. Assume that $(\calg,\calh)$ satisfies Properties~$\PI$ and~$\PII$.

Then there exists a partition $Y^*=\Dunion_{i\in I}Y_i$ of a conull Borel subset $Y^*\subseteq Y$ into at most countably many Borel subsets such that for every $i\in I$, there exists a non-trivial free splitting or bi-nonsporadic splitting which is $(\calh,\rho)$-invariant.
\qed
\end{cor}

\subsubsection{Step 3: Excluding bi-nonsporadic splittings via $(P_1^{\max})$ and $(P_3)$}

\begin{lemma}\label{lemma:f2-f2-f2}
Let $S$ be a bi-nonsporadic splitting of $F_N$, and let $H:=\Stab_{\IA}(S)$.

Then $H$ contains a direct product $N^-\times N^+$, with $N^-,N^+$ normal in $H$,
and such that for every $\epsilon \in \{\pm\}$, the group $N^\eps$ contains a subgroup isomorphic to $F_2\times F_2\times F_2$.  
\end{lemma}

\begin{proof}
As $S$ is bi-nonsporadic, there exist two vertices $v^-,v^+\in S$ in distinct $F_N$-orbits such that for every $\eps\in\{\pm\}$, the Grushko decomposition of $G_{v^\eps}$ relative to the set $\Inc_{v^\eps}$ of incident edge groups is non-sporadic.
  
  We claim that the group $\Out(G_{v^\eps},\Inc_{v^\eps}^{(t)})$ made of all outer automorphisms of $G_{v^\eps}$ acting by conjugation on every incident edge group contains a copy of $F_2\times F_2\times F_2$.  Indeed, let $S_{v^\epsilon}$ be a Grushko splitting of $G_{v^\epsilon}$ relative to $\Inc_{v^\epsilon}$
  with all vertex stabilizers non-trivial (hence non-abelian because incident edge groups are non-abelian).
  By non-sporadicity, the splitting $S_{v^\epsilon}$ contains at least four half-edges in distinct orbits based at vertices with non-trivial stabilizer. 
  The group of twists of $S_{v^\eps}$ is a subgroup of $\Out(G_{v^\eps})$
  containing a copy of $F_2\times F_2\times F_2$, associated to three of these half-edges. And since twists acts trivially on the vertex groups of $S_{v^\eps}$, it is contained in $\Out(G_{v^\eps},\Inc_{v^\eps}^{(t)})$. This proves our claim.
  

   Now by Lemma~\ref{lem_produit}, the product
  $\Out(G_{v^+},\Inc_{v^+}^{(t)})\times \Out(G_{v^-},\Inc_{v^-}^{(t)})$ embeds in $\Stab_{\Out(F_N)}(S)$.
   We then denote by $H$ the stabilizer of $S$ in $\IA$, and by $N^\eps$ the intersection of $H$ with $\Out(G_{v^\eps},\Inc_{v^\eps}^{(t)})$.  Corollary~\ref{cor:ia-nice} ensures that $N^+$ and $N^-$ are normal in $H$, which concludes the proof.
\end{proof}  

\begin{cor}\label{cor:step3}
 Let $\calg$ be a measured groupoid over a base space $Y$, and let $\rho:\calg\to\IA$ be an action-type cocycle. Let $\calh$ be a measured subgroupoid of $\calg$. Assume that $(\calg,\calh)$ satisfies Properties~$\PImax$, $\PII$, and $\PIII$.

Then $(\calh,\rho)$ stably preserves a free splitting.
\end{cor}

\begin{proof}
 We argue by contradiction, so assume otherwise. Then by Corollary \ref{cor:exclude-cyclic-splitting}, up to restricting to a subset of positive measure we may assume
  that $(\calh,\rho)$ preserves a bi-nonsporadic splitting $S$, and let $\Gamma_S\subset\IA$ be its stabilizer.
  Let $N^+,N^-$ be the normal subgroups of $ \Gamma_S$ be given by Lemma \ref{lemma:f2-f2-f2}.

 Let $\hat \calh=\rho\m(\Gamma_S)$. We claim that $(\calg,\hat\calh)$ 
   satisfies $\PI$. Indeed,  the subgroupoids $\caln^+=\rho\m(N^+)$
   and $\caln^-=\rho\m(N^-)$ 
   form a pseudo-product by Lemma \ref{lemma:pseudo-product} and are normalized
   by $\hat \calh$ (Lemma~\ref{lemma:normal-subgroup-subgroupoid}). And $\hat\calh$ is everywhere non-amenable (Lemma~\ref{lemma:not-amen}), which proves the claim.

   Since $(\calg,\calh)$ satisfies Property~$\PImax$ it follows that $\calh$ is stably equal to $\hat\calh$,
   so  up to further restricting to a Borel subset of positive measure, we may assume that $\calh=\rho\m(\Gamma_S)$.
   For each $\eps\in\{\pm\}$, denote by $B_1^\eps\times B_2^\eps\times B_3^\eps\simeq F_2\times F_2\times F_2$
   the subgroup of $N^\eps$ given by Lemma \ref{lemma:f2-f2-f2}. 
   For each $\eps\in\{\pm\}$, and each $j\leq 3$, consider $\calb_j^\eps=\rho\m(B_j^\eps)$.
   Then by Lemma \ref{lemma:pseudo-product}, $(\caln^+,\calb_1^-,\calb_2^-,\calb_3^-)$ form a pseudo-product, and so do
   $(\caln^-,\calb_1^+,\calb_2^+,\calb_3^+)$. This contradicts Property~$\PIII$.
\end{proof}

\subsubsection{Step 4: Excluding some types of free splittings}\label{sec_step4}

We say that a free splitting $S$ of $F_N$ is \emph{acyclic} if $S/F_N$ does not contain any terminal
vertex with vertex group isomorphic to $\mathbb{Z}$. When $S$ is a one-edge splitting, this is equivalent to saying that $S$ is not the Bass--Serre tree of a decomposition of $F_N$ of the form $F_N=\mathbb{Z}\ast F_{N-1}$. Note that the stabilizer in $\Out(F_N)$ of the splitting $F_N=\mathbb{Z}\ast F_{N-1}$ is contained in the stabilizer of the acyclic splitting $F_N=F_{N-1}\ast$. This observation leads to the following result. 

\begin{lemma}\label{lemma:nonspecial-free-splitting}
Let $S$ be a free splitting of $F_N$, and let $H$ be its stabilizer in $\IA$.

Then $H$ stabilizes an acyclic one-edge free splitting of $F_N$.
\end{lemma}

\begin{proof}
As $H\subseteq\IA$, it stabilizes every one-edge collapse of $S$ (Lemma~\ref{lemma:ia-free-splitting}). In addition, if $H$ stabilizes a splitting of the form $F_N=A\ast B$ with $B$ cyclic, then it also stabilizes the acyclic splitting $F_N=A\ast$.
\end{proof}

In order to be able to use Property $\PImax$, we now prove
that the stabilizer of a free  splitting satisfies Property $\PI$. This will also be useful later for proving the second part of Theorem~\ref{theo:non-characterization}.

\begin{lemma}\label{lemma:nonsep-satisfies-1}
Let $\calg$ be a measured groupoid over a base space $Y$, and let $\rho:\calg\to\IA$ be an action-type cocycle. Let $S$ be a one-edge acyclic free splitting of $F_N$.

Then the $(\calg,\rho)$-stabilizer of $S$ satisfies Property~$\PI$.
\end{lemma}

\begin{proof}
We denote by $H$ the stabilizer of $S$ in $\IA$, and we let $\calh=\rho^{-1}(H)$, so that $\calh$ is the $(\calg,\rho)$-stabilizer of $S$. Our proof will show that $H$ contains a non-abelian free subgroup, in particular $\calh$ is everywhere non-amenable by Lemma~\ref{lemma:not-amen}.

We first assume that $S$ is a separating free splitting, equal to a Bass--Serre tree of a free product decomposition $F_N=A\ast B$, with both $A$ and $B$ noncyclic.
The group of twists of $S$ in $\Out(F_N)$ is isomorphic to the direct product $A\times B$ (see Section \ref{sec_twists}). 
We then let $N^-:=(A\times\{1\})\cap\IA$ and $N^+:=(\{1\}\times B)\cap\IA$. Since $H\subseteq\IA$, both $N^-$ and $N^+$ are normal subgroups of $H$. For every $\epsilon\in\{\pm\}$, we then let $\caln^\eps:=\rho^{-1}(N^\epsilon)$, a  normal subgroupoid of $\calh$ by  Lemma~\ref{lemma:normal-subgroup-subgroupoid}.
In addition, as $N^+$ and $N^{-}$ generate their direct product in $H$ and both contain a non-abelian free subgroup, Lemma~\ref{lemma:pseudo-product} ensures that $\caln^-$ and $\caln^+$ form a pseudo-product.
This proves that $(\calg,\calh)$ satisfies $\PI$ when $S$ is separating.

We now assume that $S$ is a non-separating free splitting, equal to a Bass--Serre tree of a decomposition $F_N=A\ast$ (where $A$ is a corank one free factor of $F_N$).  The group of twists of $S$ in $\Out(F_N)$ is isomorphic to a direct product $A^{-}\times A^+$ of two copies of $A$ (see Section \ref{sec_twists}).
We let $N^-:=(A^{-}\times\{1\})\cap\IA$ and $N^+:=(\{1\}\times A^+)\cap\IA$. Since $H\subseteq\IA$, both $N^-$ and $N^+$ are normal subgroups of $H$.
For every $\epsilon\in\{\pm\}$, we then let $\caln^\eps:=\rho^{-1}(N^\epsilon)$. 
As above,  $\caln^-$ and $\caln^+$ are normalized by $\calh$ and form a pseudo-product.
\end{proof}

Property $\PIII$ will allow us to rule out  one-edge acyclic free splittings not in $\NS\cup\FSii$ 
thanks to the following lemma.
 
\begin{lemma}\label{lemma:f2-f2-f2-aut}
  Let $S$ be a separating free splitting of $F_N$ corresponding to a decomposition of the form $F_N=F_k\ast F_{N-k}$ with $k,N-k\ge 3$, and let $H:=\Stab_{\IA}(S)$.

Then $H$ contains a direct product $N^-\times N^+$, with $N^-,N^+$ normal in $H$,
and such that for every $\epsilon \in \{\pm\}$, the group $N^\eps$ contains a subgroup isomorphic to $F_2\times F_2\times F_2$. 
\end{lemma}

\begin{proof}
The subgroup of $\Stab_{\Out(F_N)}(S)$ of index at most $2$ that preserves the conjugacy classes of both $F_k$ and $F_{N-k}$ (as opposed to permuting them) is isomorphic to $\Aut(F_k)\times\Aut(F_{N-k})$ (see for instance \S5 and the short exact sequence (4) in \cite{GL-os}). Since $H\subseteq\IA$, it does not swap these two free factors. 
The intersections $N^{\pm}$ of $H$ with each of the factors $\Aut(F_k)$ and $\Aut(F_{N-k})$ are therefore normal in $H$, and they have finite index in $\Aut(F_k)$ and $\Aut(F_{N-k})$. The conclusion thus follows from the fact that for every $k\ge 3$, the group $\Aut(F_k)$ contains a subgroup isomorphic to $F_2\times F_2\times F_2$, see e.g.\ \cite[Theorem~6.1]{HW}.
\end{proof}

We are now ready to complete the proof of the first part of Theorem~\ref{theo:non-characterization}.

\begin{proof}[Proof of Theorem~\ref{theo:non-characterization}, Part 1]
  Assume that $(\calg,\calh)$ satisfies Properties \PImax, \PII\ and \PIII.
  By Corollary \ref{cor:step3}, $\calh$ stably preserves a free splitting: there is a countable Borel partition  $Y^*=\Dunion_{i\in I} Y_i$ of a conull subset
  $Y^*\subseteq Y$ into subsets of positive measure
  such that  $\rho(\calh_{|Y_i})\subseteq \Gamma_{S_i}$ where $S_i$ is a free splitting and $\Gamma_{S_i}$ is its stabilizer in $\IA$.
  By Lemma~\ref{lemma:nonspecial-free-splitting}, up to increasing $\Gamma_{S_i}$, we may assume that $S_i$ is a one-edge acyclic free splitting.
  Since the subgroupoid  $\calh'=\rho\m(\Gamma_{S_i})_{|Y_i}$ contains $\calh_{|Y_i}$ and satisfies Property~$\PI$ by Lemma~\ref{lemma:nonsep-satisfies-1},  
  Assumption $\PImax$ on $\calh$ ensures that, up to refining our partition,  $\calh_{|Y_i}=\rho\m(\Gamma_{S_i})_{|Y_i}$.

  We assume that $S_i\notin \NS\cup\FSii$ and argue towards a contradiction.
  Then Lemma~\ref{lemma:f2-f2-f2-aut} applies to $S_i$
  and yields two normal subgroups  $N^+,N^-\unlhd \Gamma_{S_i}$  which generate their direct product, and such that
  for every $\epsilon \in \{\pm\}$, the group $N^\epsilon$ contains a direct product  $B_1^\epsilon\times B_2^\epsilon\times B_3^\epsilon$
  of three non-abelian free groups.
  The subgroupoids $\caln^+=\rho\m(N^+)_{|Y_i}$ and $\caln^-=\rho\m(N^-)_{|Y_i}$ are normal subgroupoids of $\calh_{|Y_i}$.
  For all $\eps\in\{\pm\}$ and all $j\in\{1,2,3\}$, we let  $\calb_j^\epsilon=\rho^{-1}(B_j^\epsilon)_{|Y_i}$.
   Then by Lemma \ref{lemma:pseudo-product}, $(\caln^+,\calb_1^-,\calb_2^-,\calb_3^-)$ form a pseudo-product, and so do
   $(\caln^-,\calb_1^+,\calb_2^+,\calb_3^+)$. This contradicts that $\calh$ satisfies Property $\PIII$. 
\end{proof}

\subsection{A necessary condition to stabilize a non-separating free splitting}\label{sec:proof-2}

The goal of the present section is to prove the second part of Theorem~\ref{theo:non-characterization}.

Since Properties~$\PI$, $\PImax$,~$\PII$ and~$\PIII$ are preserved under stabilization, it is enough to check that the $(\calg,\rho)$-stabilizer $\calh$ of any non-separating free splitting satisfies Properties~$\PI$, $\PImax$,~$\PII$ and~$\PIII$. 
We already checked that  $\calh$ satisfies Property~$\PI$ in Lemma \ref{lemma:nonsep-satisfies-1}.

In the rest of this subsection, we fix a measured groupoid $\calg$ over a base space $Y$, and let $\rho:\calg\to\IA$ be an action-type cocycle (with $N\ge 3$).
Let $S_0$ be a one-edge non-separating free splitting of $F_N$, let $H_0\subseteq\IA$ be the stabilizer of $S_0$ in $\IA$, and let $\calh_0:=\rho^{-1}(H_0)$, which is a measured subgroupoid of infinite type because $H_0$ is infinite and $\rho$ is action-type. 
We let $A_0$ be a free factor of corank 1 stabilizing a vertex of $S_0$
so that $S_0$ is dual to the HNN extension $F_N=A_0*$.
We denote by $v_0$ the vertex of $S_0$ fixed by  $A_0$,
choose an edge $e_0$ incident on $v_0$, and consider $t\in F_N$ such that $e_0=[v_0,tv_0]$,
so that $t$ is a stable letter of the HNN extension.
For each $w,w'\in A_0$, consider the automorphism $\tilde \tau_{w,w'}\in \Aut(F_N)$
 that acts as the identity on $A_0$
and sends $t$ to $wtw'$: this is the lift to $\Aut(F_N)$ of an element of the group of twists of $S_0$, in particular it stabilizes $S_0$. The associated $\tilde \tau_{w,w'}$-equivariant isometry $I_{\tilde \tau_{w,w'}}$ of $S_0$
maps $e_0$ to $we_0$ and $t\m e_0$ to $w'^{-1}t\m e_0$.

Finally, we let $\rho_{A_0}:\calh_0\to\Out(A_0)$ be the cocycle obtained by postcomposing $\rho$ with the natural map $\Stab(S_0)\to\Out(A_0)$. 

\begin{rk}\label{rk:rho_A0}
We mention that $\rho_{A_0}$ is not action-type (it has a non-trivial kernel, coming from the $\rho$-preimage of the group of twists of $S_0$). Nevertheless, as a consequence of the fact that $\rho$ is action-type, we see that $(\calh_0,\rho_{A_0})$ is nowhere trivial, e.g.\ by Remark~\ref{rk:action-type}.
\end{rk}

\subsubsection{Proof of Property~$(P_1^{\max})$}

\begin{lemma} \label{lem_conj_invariant}
Consider an element $a\in F_N$ whose conjugacy class is $H_0$-invariant. Then $a$  is conjugate into $A_0$.
\end{lemma}

\begin{rk}\label{rk_N3}
  This uses the fact that $N\geq 3$. When $N>3$ (so that $A_0$ has rank at least 3), $a$ has in fact to be trivial.
  When $N=3$ (so that $A_0$ has rank 2),  one may have $a=[x,y]^k$ for some free basis $\{x,y\}$ of $A_0$ and some $k\in\mathbb{Z}$.
\end{rk}

\begin{proof}
Fix a free basis $\{s_1,\dots, s_{N-1},t\}$ of $F_N$ such that $\{s_1,\dots,s_{N-1}\}$ is a basis of $A_0$. We denote by $|\cdot|$ the word length in this basis.
Since $a$ is not conjugate into $A_0$, up to conjugacy, 
write $$a=a_1 t^{\eps_1} a_2 t^{\eps_2}\dots a_r t^{\eps_r}$$
as a cyclically reduced word in the graph of groups underlying the HNN extension,
with $r\geq 1$, $a_i\in A_0$ for all $i\leq r$, $\eps_i=\pm 1$ 
and  $a_i\neq 1$ when $\eps_i=-\eps_{i-1}$ (indices being understood modulo $r$).

Given $w\in A_0$, consider the automorphism $\tau_w$ fixing $A_0$ and sending $t$ to $tw$.
Then $$\tau_{w^k}(a)=a'_1(k) t^{\eps_1} a'_2(k) t^{\eps_2}\dots a'_r(k) t^{\eps_r}$$
where $a'_i(k)$ is either $a_i$, $w^ka_i$, $a_iw^{-k}$, $w^ka_i w^{-k}$
according to the signs of $\eps_i$ and $\eps_{i-1}$.
Note that the cases $a'_i(k)=a_i$ and $a'_i(k)=w^k a_i w^{-k}$ occur when 
$\eps_i=-\eps_{i-1}$, so $a_i\neq 1$ in these cases.
Also note that the case where $\eps_{i-1}=-1$ and $\eps_{i}=1$ (corresponding to $a'_i(k)=a_i$) cannot hold for all $i\leq r$.

We choose $w\in A_0\setminus\{ 1\}$ so that it does not commute with any non-trivial $a_i$ 
(which is possible because $A_0$ has rank at least 2).
This ensures that, writing $a'_i(k)$ as a reduced word on $\{s_1,\dots,s_{N-1}\}$, 
one obtains a cyclically reduced expression of $\tau_{w^k}(a)$. 
Moreover, as $k\ra\infty$, $|w^ka_i|$, $|a_iw^{-k}|$, and $|w^ka_i w^{-k}|$ all go to infinity.
It follows that $|\tau_{w^k}(a)|$ goes to infinity.
This prevents $\tau_w$ from preserving the conjugacy class of $a$.
\end{proof}

\begin{proof}
Let $a\in F_N$ be an element which is not conjugate into $A_0$; we are going to prove that the conjugacy class of $a$ is not $H_0$-invariant.

Given a non-trivial element $w\in A_0$ which is not a proper power, consider
the $\Zmax$-splitting $S_w$ of $F_N$ dual to the HNN extension
with vertex group $\grp{A_0,twt\m}$ and with edge group $\grp{w}$
with the two obvious inclusions.
The splitting $S_w$ may be obtained from $S_0$ as follows:  
recall that $v_0\in S_0$ is the vertex fixed by $A_0$, and $e_0=[v_0,tv_0]$; then $S_w$ is obtained by folding $e_0$ with $we_0$.

As $a$ is not conjugate into $A_0$, it is hyperbolic in $S_0$. For all but finitely many choices of $w$, the axis of $a$ in $S_0$ isometrically embeds in $S_w$ under the folding map. It follows that we may choose $w$ so that $a$ is hyperbolic in $S_w$.
Let $\tau$ be the twist associated to the splitting $S_w$: this is the automorphism restricting to the identity on $A_0$ and sending $t$ to $tw$, so in particular $\tau$ has a power in $H_0$.

Let $R$ be a Cayley graph of $F_N$ with respect to some basis, equipped with the simplicial metric where every edge is assigned length $1$.
By \cite{CL}, the rescaled trees $\frac{1}{n}(\tau^n.R)$ converge to a tree homothetic to $S_w$ as $n\ra \infty$.
It follows that the translation length $||\tau^n(a)||_R$ tends to infinity, so the conjugacy class of $a$ is not $H_0$-invariant.
\end{proof}

\begin{lemma}\label{lemma:unique-invariant-factor}
Let $B$ be any proper free factor of $F_N$ whose conjugacy class is $H_0$-invariant.

Then $B$ is conjugate to $A_0$. 
\end{lemma}

\begin{proof}
Observe that since $H_0$ surjects onto $\mathrm{IA}(A_0,\mathbb{Z}/3\mathbb{Z})$, there is no $H_0$-invariant proper free factor of $F_N$ strictly contained in $A_0$. We can therefore assume that $B$ is not conjugate into $A_0$, and aim for a contradiction.

 We first claim that $B$ intersects every conjugate of $A_0$ trivially, equivalently that $A_0$ intersects every conjugate of $B$ trivially. Indeed, for every $g\in F_N$,  the subgroup $A_0\cap B^g$ is a free factor contained in $A_0$ whose conjugacy class is $H_0$-invariant (Lemma~\ref{lemma:intersection-factors}). 
 Thus, $A_0\cap B^g$ is either trivial or equal to $A_0$, but the latter case implies $B^g=A_0$ since $A_0$ has corank $1$.

We next claim that $H_0$ acts trivially on $B$. 
Let $S_B\subseteq S_0$ be the minimal $B$-invariant subtree.
As $B$ intersects every conjugate of $A_0$ trivially, it acts freely on $S_B$, so $S_B$ defines a point in the Outer space of $B$, and its stabilizer in $\mathrm{IA}(B,\bbZ/3\bbZ)$
is trivial (because point stabilizers in Outer space are finite, and $\mathrm{IA}(B,\bbZ/3\bbZ)$ is torsion free).
For all $\alpha\in H_0$ and all $b\in B$,
one has $||\alpha(b)||_{S_B}=||\alpha(b)||_{S_0}=||b||_{S_0}=||b||_{S_B}$,
so $S_B$ is invariant by the image of $H_0$ in $\mathrm{IA}(B,\mathbb{Z}/3\mathbb{Z})$
(indeed every tree in the Outer space of $B$ is determined by its length function by \cite[Theorem~3.7]{CM}). 
The image of $H_0$ in $\Out(B)$ is therefore trivial, which proves our claim.

Together with Lemma~\ref{lem_conj_invariant}, the claim implies that every element of $B$ is conjugate into $A_0$. 
It follows that $B$ is conjugate into $A_0$. Indeed,
we know that every non-trivial element of $B$ fixes a point in $S_0$ (necessarily unique). 
If two distinct non-trivial elements of $B$ had distinct fixed points, their product would be hyperbolic in $S_0$, a contradiction. Therefore $B$ fixes a point in $S_0$, so $B$ is conjugate into $A_0$. This contradicts our initial assumption, so the lemma follows.
\end{proof}

 Recall that an outer automorphism $\alpha\in\Out(A_0)$ is \emph{fully irreducible} if no non-trivial power of $\alpha$ fixes the conjugacy class of a proper free factor of $A_0$.

\begin{lemma}\label{lem:invariant_A0}
    For every fully irreducible element $\alpha\in \Out(A_0)$, the following two assertions hold:
\begin{enumerate}
\item no non-trivial power of $\alpha$ preserves the conjugacy class a non-cyclic finitely generated subgroup
of $A_0$ of infinite index;
\item no non-trivial power of $\alpha$ preserves a non-trivial splitting of $A_0$
  with finitely generated edge or vertex stabilizers.
  \end{enumerate}
\end{lemma}

\begin{proof}
  For the first assertion, 
  consider the attracting $\bbR$-tree $T$ of $\alpha$
  and denote by  $L_T(g)$ the translation length in $T$ of an element $g\in F_N$.
 Recall that there exists $\lambda >1$ such that $L_T(\alpha(g))=\lambda L_T(g)$.
  By \cite[Theorem~5.4]{BFH0},
any finitely generated subgroup $B\subseteq A_0$ of infinite index acts discretely on $T$. 
In particular, there exists $\eps>0$ such that $L_T(B)\cap (0,\eps)=\es$.
If $B$ is invariant under some power of $\alpha$, this forces $B$ to be elliptic in $T$.
It is well known that this implies that $B$ is cyclic
(this can be proved using train-track theory as in \cite{BH,BFH0}, 
or one can apply \cite[Theorem~1.1]{Rey} since $T$ is arational by \cite[Corollary 1.4]{Rey}).

We now prove that $\alpha$ satisfies the second conclusion of the lemma. Recall that a splitting with finitely generated edge
stabilizers also has finitely generated vertex stabilizers.
So let $S$ be a splitting of $A_0$ with finitely generated vertex stabilizers which is invariant
under a power of $\alpha$. Then the vertex and edge stabilizers have to be cyclic.
We claim that some edge stabilizer has to be trivial. Otherwise,
any two edges incident on the same vertex have commuting stabilizers, so by
commutative transitivity, all edge stabilizers of $S$ commute, a contradiction.
We may thus assume that $S$ is a free splitting. Since there is no free factor
invariant by a power of $\alpha$, this concludes the proof.
\end{proof}

\begin{lemma}\label{lemma:unique-invariant-nice-splitting}
Let $S'$ be a nice splitting such that every element of $H_0$ has a power that preserves $S'$.

Then $S'=S_0$.
\end{lemma}

\begin{proof}
  Assume first that $S'$ is a free splitting.
  By assumption, $S'$ is periodic under the action of every element of $H_0$.
  As $H_0\subseteq\IA$, Lemma~\ref{lemma:ia-free-splitting} implies that $S'$ is in fact $H_0$-invariant,
  and so is every collapse of $S'$.
  By Lemma~\ref{lemma:unique-invariant-factor}, point stabilizers of $S'$ are either trivial or conjugate to $A_0$,
  and the same is true for collapses of $S'$.
  If the action of $F_k$ on $S'$ was free, then one could construct some collapse of $S'$ with a cyclic vertex stabilizer, a contradiction. It follows that $A_0$ fixes a point in $S'$, so either $S'=S_0$ or $S'$ is dual
  to a graph of groups of the form
  $$\includegraphics[width=2cm]{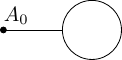}$$
in which case $S'$ has a collapse with a cyclic vertex stabilizer, a contradiction.

We now assume that $S'$ is a $\Zmax$-splitting of $F_N$. Theorem~\ref{theo:ia-element} ensures that the the conjugacy class of any edge group of $S'$ 
is $H_0$-invariant so if $N>3$, we may use Lemma~\ref{lem_conj_invariant} and Remark~\ref{rk_N3}
to exclude the existence of $S'$. The following argument works for all $N\geq 3$.

Let $H'$ be the intersection of $H_0$ with the stabilizer of $S'$ in $\IA$.
Let $H'_{|A_0}\subset\Out(A_0)$ be its image under the restriction map.
The action of $A_0$ on its minimal subtree $S'_{A_0}\subset S'$ is a splitting of $A_0$  with finitely generated edge stabilizers. This splitting is invariant under $H'_{|A_0}$.
As  every element of $H_0$ has a power contained in $H'$, it follows that every element of $\Out(A_0)$ has a power contained in $H'_{|A_0}$,
in particular $H'_{|A_0}$ contains a fully irreducible outer automorphism of $A_0$. By Lemma~\ref{lem:invariant_A0}, it follows that $A_0$ fixes a (unique) point in $S'$.

View $S_0$ and $S'$ as $\tilde H'$-trees where $\tilde H'$ is the preimage of $H'$ 
in $\Aut(F_N)$ (see Section~\ref{sec:tildeH}).
Let $f:S_0\to S'$ be the unique $F_N$-equivariant map that is linear on edges.
Being unique, $f$ is $\tilde H$-equivariant. 
Since $f$ is not an isomorphism, there exists a pair of adjacent edges $e,e'$ in $S_0$ such that $f$ identifies an initial segment of $e$ with an initial segment of $e'$.
Without loss of generality, the common endpoint of $e$ and $e'$
is the vertex $v_0\in S_0$ fixed by $A_0$.

Assume first that $e'=ae$ for some  $a\in A_0\setminus\{1\}$. Let $w\in A_0$  be an element that does not commute with $a$.
We use the fact that 
there exists a  twist $\alpha\in\Out(F_N)$ of $S_0$ which has a lift $\tilde{\alpha}\in\Aut(F_N)$ acting as the identity on $A_0$ and such that
the $\tilde\alpha$-equivariant isometry $I_{\tilde \alpha}:S_0\ra S_0$
(see Section \ref{sec:tildeH})
satisfies $I_{\tilde\alpha}(e)=we$
and $I_{\tilde\alpha}(ae)=awe$. 
Take  $n\ge 1$ such that $\alpha^n\in H'$.
Since $\tilde\alpha$ acts as the identity on $A_0$, $I_{\tilde\alpha^n}(e)=w^ne$ and $I_{\tilde\alpha^n}(e')=aw^ne$.
The $\tilde H$-equivariance of $f$ translates into the fact that $f$ identifies initial segments of $I_{\tilde\alpha^n}(e)$ and $I_{\tilde\alpha^n}(e')$.
It follows that $e$ and $w^{-n} aw^ne$ have initial segments that are identified by $f$. Thus, the non-abelian group $\grp{a,w^{-n}aw^n}$ fixes a segment in $S'$, a contradiction.

Now assume that $e,e'$ are not in the same $A_0$-orbit. Choose $a\in A_0\setminus\{1\}$. 
There exists a twist $\alpha\in \Stab(S_0)$ and a preimage $\tilde \alpha\in\Aut(F_N)$ acting as the identity on $A_0$ and such that
the isometry $I_{\tilde \alpha}$ of $S_0$ fixes $e'$ and sends $e$ to $ae$.
In particular, if  $n\ge 1$ is such that $\alpha^n\in H'$,
we get that $f$ identifies some initial segments of $e$ and  $a^ne$. The previous case then yields a contradiction.

We finally assume that $S'$ is a bi-nonsporadic splitting. 
By Lemma~\ref{lemma:binonsporadic2F} there is a finite collection $\{\calf_1,\dots,\calf_n\}$ of $n\geq 2$ non-empty free factor systems of $F_N$ such that every element of $H_0$ has a power that preserves each $\calf_i$. 
Since $H_0\subset \IA$, each $\calf_i$ is in fact $H_0$-invariant.
This contradicts Lemma~\ref{lemma:unique-invariant-factor}. 
\end{proof}

\begin{lemma}\label{lemma:unique-nice-invariant}
For every Borel subset $U\subseteq Y$ of positive measure,
\begin{enumerate}
\item $S_0$ is the only nice splitting of $F_N$ which is $({\calh_0}_{|U},\rho)$-invariant; 
\item there is no non-trivial splitting of $A_0$  with finitely generated edge stabilizers
which is $({\calh_0}_{|U},\rho_{A_0})$-invariant. 
\end{enumerate}
\end{lemma}

\begin{proof}
  Let $S$ be an $({\calh_0}_{|U},\rho)$-invariant nice splitting. Then  there exists a conull Borel subset $U^*\subseteq U$ such that $S$ is invariant by $\rho({\calh_0}_{|U^*})$. Since ${\calh_0}=\rho^{-1}(H_0)$ and $\rho$ is action-type, Remark~\ref{rk:action-type} implies that every element $\alpha\in H_0$ has a power that preserves $S$. Lemmas~\ref{lemma:unique-invariant-nice-splitting} and~\ref{lem:invariant_A0} thus conclude the proof
  (to apply Lemma~\ref{lem:invariant_A0}, notice indeed that $H_0$ contains an element that restricts to a fully irreducible outer automorphism of $A_0$).
\end{proof}

\begin{lemma}\label{lemma:check-p1max}
The pair $(\calg,{\calh_0})$ satisfies Property~$\PImax$.
\end{lemma}

\begin{proof}
  By Lemma~\ref{lemma:nonsep-satisfies-1}, the pair $(\calg,{\calh_0})$ satisfies Property~$\PI$, so we only need to check maximality. Let $U\subseteq Y$ be a Borel subset of positive measure, and let $\calh'$ be a subgroupoid of $\calg_{|U}$ such that $(\calg_{|U},\calh')$ satisfies Property~$\PI$, and such that ${\calh_0}_{|U}$ is stably contained in $\calh'$. By Lemma~\ref{lemma:1-implies-splitting}, there exists a partition $U=\Dunion_{i\in I} U_i$ 
  into at most countably many Borel subsets of positive measure, such that for every $i\in I$, there exists a nice splitting $S_i$ of $F_N$ which is  $(\calh'_{|U_i},\rho)$-invariant,  
 and up to refining the partition it is also $({\calh_0}_{|U_i},\rho)$-invariant.
  
  Lemma~\ref{lemma:unique-nice-invariant}(1) implies that $S_i=S_0$ for every $i\in I$. Thus,
  there exists a conull  Borel subset
 $U_i^*\subset U_i$ such that $\calh'_{|U_i^*}$ is contained in $\rho\m(H_0)_{|U^*_i}={\calh_0}_{|U^*_i}$, and hence ${\calh_0}_{|U}$ is stably equal to $\calh'$. 
\end{proof}

\subsubsection{Proof of Property~$(P_2)$}

\begin{lemma}\label{lemma:check-p2}
The pair $(\calg,{\calh_0})$ satisfies Property~$\PII$.
\end{lemma}

\begin{proof}
  Assume towards a contradiction that there exists a Borel subset $U\subseteq Y$ of positive measure such that ${\calh_0}_{|U}$
 normalizes an amenable subgroupoid $\cala$ of $\calg_{|U}$ of infinite type.

We first claim that there exists a Borel subset $V\subseteq U$ of positive measure such that $S_0$ is $(\cala_{|V},\rho)$-invariant. Indeed, by Lemma~\ref{lemma:stably-nice-vs-nice-averse}, there exists a Borel partition $U=U_1\dunion U_2$ such that $(\cala_{|U_1},\rho)$ is nice-averse and $(\cala_{|U_2},\rho)$ is stably nice. As $\cala_{|U_1}$ is amenable and
normalized by ${\calh_0}_{|U_1}$, and as $(\cala_{|U_1},\rho)$ is nice-averse and $\rho$ has trivial kernel, Corollary~\ref{cor:normalisateur_moyennable} implies that ${\calh_0}_{|U_1}$ is amenable. But we already know that $\calh_0$ is everywhere non-amenable (Lemma~\ref{lemma:nonsep-satisfies-1}), so $U_1$ is a null set. Using Proposition~\ref{prop:nice-preserved}, we deduce that there exists a partition $U=\Dunion_{i\in I} U'_{i}$  into at most countably many Borel subsets of positive measure such that for every $i\in I$, there exists a nice splitting $S_i$ which is both $(\cala_{|U'_{i}},\rho)$-invariant and $({\calh_0}_{|U'_{i}},\rho)$-invariant. Lemma~\ref{lemma:unique-nice-invariant}(1) then implies that $S_i=S_0$ for every $i\in I$. This completes the proof of our claim by letting $V$ be any $U'_i$ with positive measure.

Up to replacing $V$ by a conull Borel subset, we can assume that $\cala_{|V}$ is contained in ${\calh_0}_{|V}$. In particular, the cocycle $\rho_{A_0}:{\calh_0}\to\Out(A_0)$ is defined on $\cala_{|V}$. 

We aim to prove that $\rho_{A_0}$ is stably trivial in restriction to $\cala_{|V}$ (in the sense of Definition~\ref{dfn:nowhere_trivial}).
Using Lemma \ref{lem_trivial_partition} and Lemma \ref{lemma:stably-nice-vs-nice-averse}, 
consider a partition $V=V_0\dunion V_1\dunion V_2$ where 
$\rho_{A_0}$ is stably trivial in restriction to $\cala_{|V_0}$,
$(\cala_{|V_1},\rho_{A_0})$ is nice-averse (so in particular nowhere trivial) and
$(\cala_{|V_2},\rho_{A_0})$ is stably nice and nowhere trivial.

If $V_2$ has positive measure, then
by Proposition~\ref{prop:nice-preserved}, $({\calh_0}_{|V_2},\rho_{A_0})$ is stably nice, which contradicts
Lemma \ref{lemma:unique-nice-invariant}(2). So $V_2$ is a null set.

Assume towards a contradiction that $V_1$ has positive measure. Let $F\subseteq H_0$ be a non-abelian free subgroup that injects in $\Out(A_0)$ under the natural restriction map $H_0\to\Out(A_0)$. Let $\calh'=\rho^{-1}(F)$. Then $(\rho_{A_0})_{|\calh'}$ has trivial kernel. Since $\cala_{|V_1}$ is amenable and 
normalized by ${\calh_0}_{|V_1}$, and $(\cala_{|V_1},\rho_{A_0})$ is nowhere trivial, Theorem~\ref{theo:normalisateur_moyennable} implies that either $\calh'_{|V_1}$ is amenable or there exist a Borel subset $W_1\subseteq V_1$ of positive measure and a nice splitting of $A_0$ which is both $(\cala_{|W_1},\rho_{A_0})$-invariant 
and $(\calh'_{|W_1},\rho_{A_0})$-invariant. The former case is excluded by Lemma~\ref{lemma:not-amen} because $F$ is a non-abelian free group, and the latter case is excluded because $(\cala_{|W_1},\rho_{A_0})$ is nice-averse.

We thus conclude that $\rho_{A_0}$ is stably trivial in restriction to $\cala_{|V}$.
So consider a Borel subset $W\subseteq V$ of positive measure such that for all $g\in\cala_{|W}$, one has $\rho_{A_0}(g)=1$. Then $\rho_{|\cala_{|W}}$ takes its values in the group of twists $\mathrm{Tw}\simeq A_0\times A_0$ of the splitting $S_0$.
Let now $\calh''=\rho\m(\mathrm{Tw})_{|W}$.  Then $\cala_{|W}$ is an amenable subgroupoid of infinite type contained in $\calh''$ and
normalized by $\calh''$.
Notice that $\rho$ restricts to an action-type cocycle $\calh''\to\mathrm{Tw}$. We thus get a contradiction to Lemma~\ref{lemma:normalisateur_twist}.
\end{proof}

\subsubsection{Proof of Property~$(P_3)$}

\begin{lemma}\label{lemma:empty_afs}
  If $\rank(A_0)\geq 3$, then
  for every Borel subset $U\subseteq Y$ of positive measure, $\es$ is the only \afs\ which is 
  $({\calh_0}_{|U},\rho_{A_0})$-invariant.

  If $A_0=\grp{a,b}\simeq F_2$, then the \afs\ $\HF_{A_0}=\{[\grp{aba\m b\m}]\}$ is $\Out(A_0)$-invariant and
  for every Borel subset $U\subseteq Y$ of positive measure, $\HF_{A_0}$ is the only non-empty \afs\ which is 
  $({\calh_0}_{|U},\rho_{A_0})$-invariant.
\end{lemma}

\begin{proof}
  If $\rank(A_0)\geq 3$, there exists $\alpha\in\mathrm{IA}(A_0)$ which is atoroidal and fully irreducible: indeed, by \cite{Sta} and \cite{GS}, there exists an outer automorphism which is both fully irreducible and non-geometric (i.e.\ not induced by a surface homeomorphism), and by \cite[Theorem~4.1]{BH} this implies that it is atoroidal.
  Then for all $k\geq 1$, $\es$ is the only \afs\ which is $\alpha^k$-invariant.
  By Remark~\ref{rk:action-type}, for every Borel
  subset $U\subset Y$ of positive measure, $\rho({\calh_0}_{|U})$ contains a non-trivial power of $\alpha$, which proves the first assertion.

  If $\rank(A_0)=2$, then $\Out(A_0)$ is isomorphic to the mapping class group of a punctured torus,
  with $\HF_{A_0}$ corresponding to the peripheral element.
   To prove the second assertion,
   it suffices to consider $\alpha\in \Out(F_2)$ induced by a pseudo-Anosov homeomorphism.
   Then $[\grp{aba\m b\m}]$
   is the only $\alpha$-periodic conjugacy class of maximal cyclic subgroup, and one concludes using Remark \ref{rk:action-type} as above.
\end{proof}

\begin{lemma}\label{lemma:subgroupoid-p3}
  There exists a measured subgroupoid $\calh'\subseteq{\calh_0}$ which is everywhere non-amenable and such that $(\rho_{A_0})_{|\calh'}$ has trivial kernel.
\end{lemma}

\begin{proof}
  Let $F_0=\grp{\ol\alpha,\ol\beta}\subset \mathrm{IA}(A_0)$ be a non-abelian free group,
  and
  let $F=\grp{\alpha,\beta}\subset H_0$ be a subgroup mapping isomorphically to $F_0$ under the restriction map.
  The subgroupoid $\calh'=\rho^{-1}(F)$ is everywhere non-amenable by Lemma~\ref{lemma:not-amen}.
  As $F$ injects into $\Out(A_0)$ and $\rho$ has trivial kernel, it follows that
   $(\rho_{A_0})_{|\calh'}$ has trivial kernel.
\end{proof}

\begin{lemma}\label{lemma:check-p3}
The pair $(\calg,{\calh_0})$ satisfies Property~$\PIII$.
\end{lemma}

\begin{proof}
Assume by contradiction that $\PIII$ does not hold, and let $U\subseteq Y$, and subgroupoids 
$\caln^+,\caln^-$, $\calb_i^+,\calb_i^-$ ($i\leq 3$) of ${\calh_0}_{|U}$ be such that
$\caln^-,\caln^+$ are normal in ${\calh_0}_{|U}$, and for all $\eps\in\{\pm\}$, the four subgroupoids $\calb_1^\eps,\calb_2^\eps,\calb_3^\eps,\caln^{-\eps}$ form a pseudo-product (Definition~\ref{de:pseudo-product}),
 with all $\calb_i^+$ contained in $\caln^+$ and all $\calb_i^-$ contained in $\caln^-$.

Thus consider for all $\eps\in\{\pm\}$ and all $i\leq 3$
a pair of amenable subgroupoids $\cala_{i,1}^\eps,\cala_{i,2}^\eps\subseteq \calb_i^\eps$
stably normalized by $\calb_{i'}^{\eps}$ for all $i'\neq i$, and by $\caln^{-\eps}$
and such that for every Borel subset $V\subseteq U$ of positive measure, the groupoid $\langle (\cala_{i,1}^{\eps})_{|V},(\cala_{i,2}^{\eps})_{|V}\rangle$ is non-amenable.
Notice in particular that
$\caln^+$ and $\caln^-$ are everywhere non-amenable.
Up to restricting to a smaller subset of positive measure, we may assume that
$\cala_{i,j}^\eps$ is actually normalized (and not only stably normalized)
by $\calb_{i'}^\eps$ and $\caln^{-\eps}$.

Since $\caln^+,\caln^-\subseteq {\calh_0}_{|U}$, 
the cocycle $\rho_{A_0}:{\calh_0}\ra\Out(A_0)$ is well-defined on all the subgroupoids under consideration.

\begin{claim*} 
There exists $\epsilon\in\{\pm\}$ and a Borel subset $V\subseteq U$ of positive measure such that the cocycle $\rho_{A_0}$ is trivial on the six subgroupoids $(\cala_{i,j}^{\eps})_{|V}$ for $i\le 3$, $j\le 2$.
\end{claim*}

\begin{proof}[Proof of the claim]
  We can assume that $(\rho_{A_0})_{|\caln^-}$ is nowhere trivial since otherwise, the claim would hold
as $\rho_{A_0}$ would be trivial on the six subgroupoids $(\cala_{i,j}^{-})_{|V}$ for some $V\subseteq U$ of positive measure.
  Let $k\leq 6$ be maximal such that there exists $V\subseteq U$ of positive measure
  such that $\rho_{A_0}$ is trivial on $k$ of the six subgroupoids $(\cala_{i,j}^{+})_{|V}$.
  We assume that $k<6$ and argue towards a contradiction.
  
  Up to renumbering, 
  we may assume that 
the cocycle $\rho_{A_0}$ is nowhere trivial on 
  $(\cala_{1,1}^+)_{|V}$.  Using Proposition \ref{prop:stably-nice-vs-pure-nice-averse}, we may additionally assume up to replacing $V$ by 
  a Borel subset of positive measure that 
  $((\cala_{1,1}^+)_{|V},\rho_{A_0})$ is either stably nice, or pure and nice-averse.
  If stably nice, then so is $(\caln^-_{|V},\rho_{A_0})$
  by Proposition~\ref{prop:nice-preserved}. Since  $\rho_{A_0}$ is nowhere trivial on $\caln^-_{|V}$,
  one can apply Proposition~\ref{prop:nice-preserved} again and get that $(({\calh_0})_{|V},\rho_{A_0})$ is stably nice.
  This contradicts Lemma~\ref{lemma:unique-nice-invariant}(2), so $((\cala_{1,1}^+)_{|V},\rho_{A_0})$ is pure and nice-averse.
  We denote by $\HF_1$ the unique maximal $((\cala_{1,1}^+)_{|V},\rho_{A_0})$-invariant \afs\ of $A_0$.
  
  We will now set up a contradiction to Theorem~\ref{theo:a-etages}.
  First, by Lemma~\ref{lemma:subgroupoid-p3}, we can find a measured subgroupoid $\calh'$ of $({\calh_0})_{|V}$ which is everywhere non-amenable, such that $(\rho_{A_0})_{|\calh'}$ has trivial kernel.
  As $(\cala_{1,1}^+)_{|V}$ is amenable and 
  normalized by $\caln^-_{|V}$, and $\caln^-_{|V}$ is normalized by $\calh'$, and $((\calh_0)_{|V},\rho_{A_0})$ is nowhere trivial (Remark~\ref{rk:rho_A0}), we can therefore apply
  Theorem~\ref{theo:a-etages} to  $(\caln_1,\caln_2,\caln_3)=((\cala_{1,1}^+)_{|V},\caln^-_{|V},({\calh_0})_{|V})$ with respect to $\rho_{A_0}$.
  Since $(({\calh_0})_{|V},\rho_{A_0})$ is nice-averse
  (Lemma~\ref{lemma:unique-nice-invariant}(2)) and $\calh'_{|V}$ is everywhere non-amenable, the only possible conclusion
  from this theorem is Assertion~2 saying that $({\calh_0})_{|V}$ preserves an \afs\ $\HF$ of $A_0$ with $\HF\sqsupsetneq\HF_1$.
  In particular, $\HF\neq\es$ which contradicts Lemma~\ref{lemma:empty_afs} if
  $\rank(A_0)\geq 3$ (i.e.\ $N\geq 4$). If $\rank(A_0)=2$, then with $\HF_{A_0}$ as in Lemma~\ref{lemma:empty_afs}, one has
  $\HF_{A_0}\sqsubseteq\HF_1\sqsubsetneq \HF$, again a contradiction.
\end{proof}

Thanks to the claim, up to swapping signs, we may find a Borel subset
$V\subseteq U$ of positive measure such that for every $i\in\{1,2,3\}$ and every $j\in\{1,2\}$, one has  $\rho_{A_0}((\cala_{i,j}^+)_{|V})=\{\id\}$.
Since $\rho$ takes values in $\IA$, this means that $\rho((\cala_{i,j}^+)_{|V})\subset \Tw$
where $\Tw$ is the group of twists of the splitting $S_0$ in $\Out(F_N)$.
Recall that $\Tw$ is isomorphic to the direct product
$A_l\times A_r$ of two copies of $A_0$. We denote by $\rho_{l}$ and $\rho_{r}$ the corresponding cocycles (given by postcomposing $\rho$ with the projection to one of the two factors).

We now consider the four groupoids $\cala_{i,j}^+$ for $i,j\leq 2$.
Clearly, one of the following assertions holds:
\begin{enumerate}
\item there exists a Borel subset $W\subseteq V$ of positive measure such that $\rho_l((\cala^+_{i,j})_{|W})=\{\id\}$ for all $i,j\leq 2$;
\item there exists a Borel subset $W\subseteq V$ of positive measure such that $\rho_r((\cala^+_{i,j})_{|W})=\{\id\}$ for all $i,j\leq 2$;
\item or there exist a Borel subset $W\subseteq V$ of positive measure, and $(i_l,j_l),(i_r,j_r)\in \{1,2\}\times\{1,2\}$, 
such that $\rho_l$ is nowhere trivial on $(\cala^+_{i_l,j_l})_{|W}$, and $\rho_{r}$ is nowhere trivial on $(\cala^+_{i_r,j_r})_{|W}$.
\end{enumerate}


First assume that Assertion 1 holds. Since $\rho$ has trivial kernel,
then $\rho_r$ has trivial kernel on the groupoid $\grp{(\cala^+_{1,1})_{|W},(\cala^+_{1,2})_{|W}}$ which is everywhere non-amenable.
This groupoid normalizes the amenable groupoid $(\cala^+_{2,2})_{|W}$ on which $\rho_r$ is nowhere trivial.
This contradicts Lemma~\ref{lemma:adams-libre}.

Similarly, Assertion 2 leads to a contradiction.

If Assertion 3 holds, we apply Lemma~\ref{lemma:adams-produit} to
the two amenable groupoids  $(\cala^+_{i_l,j_l})_{|W}$,  $(\cala^+_{i_r,j_r})_{|W}$
which are normalized by the groupoid $\grp{(\cala^+_{3,1})_{|W},(\cala^+_{3,2})_{|W}}$ on which $\rho$ has trivial kernel.
We get that $\grp{(\cala^+_{3,1})_{|W},(\cala^+_{3,2})_{|W}}$ is amenable, a contradiction.
\end{proof}

\section{Characterizing compatibility}\label{sec:compatibility}

Let $\calg$ be a measured groupoid over a base space $Y$, and let $\rho:\calg\to\IA$ be an action-type cocycle, with $N\ge 3$. Given two splittings $S_1,S_2\in\NS\cup\FSii$, the goal of the present section is to characterize the compatibility of $S_1$ and $S_2$ in terms of a purely groupoid-theoretic statement about their $(\calg,\rho)$-stabilizers. 

We refer to Section~\ref{sec:big} for Properties~$\PImax$,~$\PII$ and~$\PIII$. We introduce the following property of a triple
$(\calg;\calg_1,\calg_2)$, where $\calg$ is a measured groupoid over a base space $Y$, and $\calg_1$ and $\calg_2$ are measured subgroupoids of $\calg$.

\newcommand{\Pcompi}{\hyperlink{Pcompi}{\ensuremath{(P_{\mathrm{comp}}^1)}}}
\newcommand{\Pcompii}{\hyperlink{Pcompii}{\ensuremath{(P_{\mathrm{comp}}^2)}}}
\begin{enumerate}
\item[\Pcompi] \hypertarget{Pcompi}{}
Every measured subgroupoid of $\calg$ that stably normalizes $\calg_1\cap\calg_2$ is stably contained in $\calg_1\cap\calg_2$.
\item[\Pcompii] \hypertarget{Pcompii}{}
  For every Borel subset $Z\subseteq Y$ of positive measure, and every measured subgroupoid $\calh$ of $\calg_{|Z}$ that stably contains $(\calg_1\cap\calg_2)_{|Z}$ and such that $(\calg_{|Z},\calh)$ satisfies Properties~$\PImax$,~$\PII$ and~$\PIII$, there exists a Borel partition $Z=Z_1\dunion Z_2$ such that the groupoids $\calh_{|Z_1}$ and $\calh_{|Z_2}$ are stably contained in $(\calg_1)_{|Z_1}$ and $(\calg_2)_{|Z_2}$ respectively.
\end{enumerate}

\begin{rk}
In the group setting, \Pcompi\ translates into the fact that the stabilizer of
a pair of distinct compatible splittings $S_1,S_2\in \NS\cup\FSii$
is its own normalizer.
Property \Pcompii\ is meant to express the fact that if $H$ is any group containing the stabilizer of $\{S_1,S_2\}$
and which is the stabilizer of some splitting $S'\in \NS\cup\FSii$,
then $H$ stabilizes $S_1$ or $S_2$ (this analogy is not completely faithful because \PImax, \PII, \PIII\ don't quite characterize stabilizers of splittings in $\NS\cup\FSii$).
\end{rk}

Notice that Properties~$\Pcompi$ and $\Pcompii$ are preserved by restriction to a Borel subset of positive measure, and by stabilization.

\begin{theo}\label{theo:compatibility-ns}
  Let $\calg$ be a measured groupoid over a base space $Y$, and let $\rho:\calg\to\IA$ be an action-type cocycle with $N\geq 3$.
  Let $S_1,S_2\in\NS\cup\FSii$ be distinct splittings, and for every $j\in\{1,2\}$, let $\calg_j$ be the $(\calg,\rho)$-stabilizer of $S_j$.
  The following assertions are equivalent.
\begin{itemize}
\item[(i)] The splittings $S_1$ and $S_2$ are compatible.
\item[(ii)] The triple $(\calg;\calg_1,\calg_2)$ satisfies Properties~$\Pcompi$ and $\Pcompii$.
\end{itemize}
\end{theo}

We will use the following lemma from \cite{HW}.

\begin{lemma}[{\cite[Lemma~2.7]{HW}}]\label{lemma:zmax-compatible}
  Let $U$ be a one-edge $\Zmax$-splitting of $F_N$, let $\tau_U$ be a twist about $U$, and let $S$ be a free splitting of $F_N$.

  If $\tau_U^k(S)=S$ for some $k\neq 0$, then $S$ is compatible with $U$.
\end{lemma}

We recall from Section~\ref{sec_step4} that a free splitting $S$ of $F_N$ is \emph{acyclic} if the quotient graph $S/F_N$ does not contain any terminal vertex with vertex group isomorphic to $\mathbb{Z}$.

\begin{lemma}\label{lemma:stabilizer-free-splitting}
Let $S$ be an acyclic free splitting of $F_N$, such that all vertex stabilizers of $S$ are non-trivial. Let $H\subseteq\IA$ be the stabilizer of $S$ in $\IA$.

If $S'$ is a free splitting of $F_N$ such that every element in $H$ has a non-trivial power that fixes $S'$, then $S'$ is a collapse of $S$.
\end{lemma}

\begin{rk}
The conclusion does not hold if $S$ is not supposed to be acyclic. Indeed, in the case where $S/F_N$ has a terminal vertex with edge group isomorphic to $\mathbb{Z}$, there is a canonical way to blow up that vertex into a loop-edge, and this blowup is $H$-invariant (but is not a collapse of $S$). Likewise, if $S$ contains a vertex $v$ of valence at least $4$ with trivial stabilizer, the conclusion of the lemma fails. Indeed, in this case, there are finitely possible blowups of $S$ at $v$, and these are permuted by $H$, whence $H$-invariant because $H\subseteq\IA$.   
\end{rk}

\begin{proof}
    Let $E$ be a set of representatives of the $F_N$-orbits of half-edges of $S$. For every $e\in E$ with initial vertex $v$,
choose 
an element $g_e\in G_v$
  which is not a proper power and is not contained in any proper free factor of $G_v$.
  This is possible because all vertex stabilizers are non-trivial.

Let $U$ be the $\Zmax$-splitting of $F_N$
 obtained from $S$ by folding every half-edge  $e\in E$ with its $g_e$-translate, and extending this operation equivariantly to the whole tree $S$.
 The fact that $S$ is acyclic implies that $U$ is minimal.
 Notice that the splitting $U$ is naturally bipartite over the vertex set $V=V_0\dunion V_1$, where vertices in $V_0$ are the images in $U$ of vertices of $S$, and vertices in $V_1$ are the images in $U$ of midpoints of edges in $S$. Notice that the stabilizer of every vertex in $V_0$ is the same in $S$ and in $U$, and that the stabilizer of every vertex in $V_1$ is isomorphic to $F_2$:
 more precisely, if $v$ is a vertex in $V_1$, and if $e$ is an edge of $S$ whose midpoint is mapped to $v$ under the folding map from $S$ to $U$, denoting by $e_1$ and $e_2$ the two half-edges of $e$ (based at two distinct vertices of $S$,  and folded with respective translates $g_1e_1$ and $g_2e_2$), then $G_v$ is the subgroup of $F_N$ generated by
  $\langle g_1,g_2\rangle$. 
  
  Denote by $S_e$ the one-edge splitting obtained from $S$ by collapsing all the edges outside the orbit of $e$.
  Then one can recover $S_e$ from $U$ and the vertex $v\in V_1$ by
  blowing up the vertex $v$ into the free splitting $G_v=\grp{g_1}*\grp{g_2}$ and collapsing all the original edges of $U$.
  Note that there is no choice in this construction as $G_v$ has a unique free splitting in which $g_1$ and $g_2$ are elliptic.
 On the other hand, for $v\in V_0$, the choice of the elements $g_e$ ensures
 that there does not exist any free splitting of $G_v$ relative to the incident edge stabilizers. This shows the only one-edge free splittings of $F_N$ that are compatible with $U$ are exactly the splittings $S_e$.
 
Given a half-edge $\eps\subseteq S$, we let $U_\eps$ be the collapse of $U$ obtained by collapsing all edges of $U$ except those in the orbit of the image of $\eps$. 
Let $\tau_\eps$ be the twist of $S$ by $g_\eps$ around $\eps$ near $v$ (see Section~\ref{sec_twists}),
  or equivalently, the twist about the cyclic splitting $U_\eps$.
  Since $\tau_\eps$ preserves $S$, some power of $\tau_\eps$ lies in $H$
  so there exists $k\geq 1$ such that  $\tau_\eps^k(S')=S'$.
 
By Lemma~\ref{lemma:zmax-compatible}, this implies that $S'$ is compatible with $U_\eps$.
As this is true for every one-edge collapse $U_\eps$ of $U$, we deduce that $S'$ is compatible with $U$ by \cite[Theorem~A.26(2)]{GL-jsj}, 
because, in the terminology of \cite{GL-jsj},
$U$ is the lcm of the trees $U_\eps$ (the length function of $U$  is the sum of the length functions of the trees $U_\eps$).

Since the only one-edge free splittings compatible with $U$ are the splittings $S_e$,
$S'$ is a collapse of $S$ (possibly $S'=S$).

\end{proof}

We are now in position to prove Theorem~\ref{theo:compatibility-ns}.

\begin{proof}[Proof of Theorem~\ref{theo:compatibility-ns}]
 We first prove that $\Pcompi$ holds assuming that $S_1$ and $S_2$ are compatible.
Note that $S_1$ and $S_2$ are acyclic.
Let $\calh$ be a measured subgroupoid of $\calg$ that stably normalizes $\calg_1\cap\calg_2$. Let $S$ be 
a common refinement of $S_1$ and $S_2$ with the minimal number of orbits of edges.
Since $S_1\neq S_2$, $S$ has at least 2 orbits of edges, and it has exactly 2 
because otherwise, some orbit of edges of $S$ is collapsed both in $S_1$ and $S_2$,
contradicting the choice of $S$.
 Then $S$ is acyclic as otherwise one of its one-edge collapses (i.e.\ $S_1$ or $S_2$) would fail to be acyclic. 
We denote by $K$ the stabilizer of $S$ in $\IA$, which coincides with the intersection of the stabilizers of $S_1$ and $S_2$.
Notice that $\calg_1\cap\calg_2=\rho^{-1}(K)$.

Lemma~\ref{lemma:stabilizer-free-splitting} ensures that $S,S_1,S_2$ are the only non-trivial free splittings of $F_N$ that are invariant by a power of every element of $K$. Using the fact that $\rho$ is action-type (in the form of Remark~\ref{rk:action-type}), we deduce that for every Borel subset $U\subseteq Y$ of positive measure, the splittings $S,S_1,S_2$ are the only  non-trivial free splittings of $F_N$ which are $((\calg_1\cap\calg_2)_{|U},\rho)$-invariant. 

Together with the fact that $\calh$ stably normalizes $\calg_1\cap\calg_2$,
this uniqueness property ensures that the collection $\calc_0=\{S,S_1,S_2\}$ is stably
$(\calh,\rho)$-invariant (see Lemma~\ref{lemma:fs-canonical}), and therefore $S$ (which is the unique 2-edge splitting in $\calc_0$) is stably $(\calh,\rho)$-invariant.
This shows that $\calh$ is stably contained in $\calg_1\cap\calg_2$, so $\Pcompi$ holds.

To prove $\Pcompii$, let $Z\subseteq Y$ be a Borel subset of positive measure, and $\calh$ be a subgroupoid of $\calg_{|Z}$ that stably contains $(\calg_1\cap\calg_2)_{|Z}$, and such that $(\calg_{|Z},\calh)$ satisfies Properties~$\PImax$,~$\PII$ and~$\PIII$. By the first part of Theorem~\ref{theo:non-characterization}, we can find a partition $Z=\Dunion_{i\in I}Z_i$ into at most countably many Borel subsets such that for every $i\in I$, there exists a non-trivial free splitting 
$S'_i\in \NS\cup\FSii$ 
which is $(\calh_{|Z_i},\rho)$-invariant. Up to further partitioning the base space and replacing 
$Z$ by a conull Borel subset, we can also assume that for every $i\in I$, one has
$(\calg_1\cap\calg_2)_{|Z_i}\subseteq\calh_{|Z_i}$.
Since $\rho$ is action-type and
$\calg_1\cap\calg_2=\rho\m(K)$, every element of $K$ has a power that fixes  $S'_i$ (see Remark~\ref{rk:action-type}).
By Lemma~\ref{lemma:stabilizer-free-splitting}, this implies that $S'_i$ is a collapse of $S$,
so either $S_1$ or $S_2$ is  $(\calh_{|Z_i},\rho)$-invariant
and $\Pcompii$ follows.

We now prove that $(ii)\Rightarrow (i)$. Assume that $(\calg;\calg_1,\calg_2)$ satisfies Properties $\Pcompi$ and $\Pcompii$.  Let $\calc=\{S_1,S_2\}$, and as above let $K=\Gamma_\calc$ be the elementwise stabilizer of $\calc$ in $\IA$. Observe that $K$ is infinite: otherwise $K$ is trivial (since we are working in the torsion-free subgroup $\IA$), and as $\rho$ has trivial kernel, the groupoid $\calg_1\cap\calg_2=\rho\m(K)$ would be trivial, contradicting Property~$\Pcompi$.

Let $U=\Uun_K$ be the non-trivial splitting constructed from the group $K$ and the collection of $K$-invariant free splittings in
Theorem \ref{theo:tree-of-cyl}. 
 Recall that $U$ comes with a bipartition of its vertex set $V=V^0\dunion V^1$
and that $S_1$ and $S_2$ may be obtained 
from $U$ by blowing up some vertices $v_1,v_2\in V^1$
(using a possibly non-minimal free splitting relative to the incident edge groups) and collapsing all edges coming from $U$. 
If $v_1$ and $v_2$ are in different orbits, then one can perform the two blowups simultaneously and conclude that $S_1$ and $S_2$ are compatible, so up to changing $v_2$ to another vertex in the same orbit, we may assume that $v_1=v_2$, and we denote $v=v_1=v_2$.

Let $H$ be the stabilizer of $U$ in $\IA$. Then $H$ normalizes $K$  (see Proposition~\ref{prop:stab-u1})
so the subgroupoid  $\calh=\rho\m(H)$ normalizes $\calk=\rho\m(K)=\calg_1\cap\calg_2$.
By $\Pcompi$, $\calh$ 
is stably contained in $\calk$. 
As $\rho$ is action-type, this implies in particular that every element of $H$ has a power contained in $K$.
We can therefore apply Corollary~\ref{cor_V1_compatible}, and deduce that if $v$ has non-trivial stabilizer,
then  $S_1$ and $S_2$ are compatible.
So from now on, we assume that $G_v$ is trivial. In particular, since $U$ is non-trivial, it has edges incident on $v$, and they have trivial stabilizer.

We denote the quotient graph $\bar U=U/F_N$.
For each edge $e$ in $\bar U$ with trivial stabilizer, we denote by $U_e$ the free splitting obtained from $U$ by collapsing all the edges whose image in $\bar U$ is different from $e$.
Since $U$ is compatible with both splittings in $\calc$, so is $U_e$, hence
we are done if $S_1$ or $S_2$ coincides with one of the trees $U_e$.

We now prove the result assuming that $\bar U$ has some non-separating edge $e$ with trivial stabilizer. 
Let $H'_e$ be the stabilizer in $\IA$ of $U_e$, and let $\calh'_e:=\rho^{-1}(H'_e)$.
Since  $K\subseteq H\subseteq H'_e$, 
we have $\calg_1\cap\calg_2\subseteq\calh'_e$.
As $e$ is non-separating, 
Theorem~\ref{theo:non-characterization} shows that $(\calg,\calh'_e)$ satisfies Properties~$\PImax$,~$\PII$ and~$\PIII$.
We can therefore apply Property~$\Pcompii$, and get a Borel partition $Y=Y_1\dunion Y_2$ such that for every $i\in\{1,2\}$, the groupoid $(\calh'_e)_{|Y_i}$ is stably contained in $(\calg_i)_{|Y_i}$. 
Let $j\in\{1,2\}$ be such that $Y_j$ has positive measure.
Then every element of $H'_e$ has a power contained in the stabilizer of $S_j$. As $H'_e\subseteq\IA$, this implies that $H'_e$ is contained in the stabilizer of $S_j$. By Lemma~\ref{lemma:stabilizer-free-splitting}, $U_e=S_j$, so as noticed above, we are done in this case. 

In the remaining case, every edge of $U$ with trivial stabilizer in $\bar U$ is separating.
We denote by $E_v=\{e_1,\dots,e_p\}$ the set of edges incident on $v$. They all have trivial stabilizer since
 $G_v=\{1\}$.
 This gives a decomposition of $F_N$ as a free product $F_N\simeq A_1*\dots *A_p$
where for every $k\in\{1,\dots,p\}$, $A_k$ is the fundamental group (as a graph of groups) of the connected component of $\bar U\setminus (\rond e_1\cup\dots \cup \rond e_p\cup\{v\})$
containing the endpoint of $e_k$. Notice that $A_k$ is non-trivial. 
  The tree $S_1$ is obtained from the graph of groups $\bar U$ by choosing a partition $E_v=E_a\dunion E_b$ of the edges incident on $v$, and blowing-up $v$ into a new edge that separates the two subsets $E_a$ and $E_b$. It follows that $S_1$ is a separating free splitting.
Since $S_1$ belongs to $\NS\cup\FSii$, we deduce that up to exchanging the roles of $E_a$ and $E_b$, the splitting $S_1$ is the Bass--Serre tree of a decomposition of the form $F_N=A*B$, where $A\simeq F_2$ is the free product of the groups $A_k$ with $e_k$ ranging in $E_a$, and $B\simeq F_{N-2}$. We may assume that  $E_a$ and $E_b$ both contain at least 2 edges since otherwise $S_1=U_{e_k}$ for some $k\leq p$, in which case we are done by the above.
It follows that $E_a$ contains exactly two edges, and that the two corresponding groups $A_k$ are infinite cyclic.

Since all edges with trivial stabilizer are separating,  $\bar U$ contains a terminal vertex whose stabilizer is isomorphic to $\mathbb{Z}$. Blowing up this vertex into a loop-edge and collapsing all other edges coming from $U$ yields a one-edge non-separating free splitting $U'$ of $F_N$ which is $H$-invariant.
As above, $\Pcompii$ implies that $U'=S_1$ or $U'=S_2$. But as observed above $S_1$ is a separating free splitting (and so is $S_2$ by the same argument), a contradiction.
\end{proof}

\section{Maps from $\NSG$ to $\FSG$}\label{sec:maps}

The \emph{free splitting graph} $\FSG$ is the simplicial graph whose vertex set
is the set $\FS$ of one-edge free splittings of $F_N$, and where two splittings are joined by an edge if they are compatible.
The \emph{non-separating free splitting graph}  $\NSG\subset \FSG$ is the induced subgraph
whose vertex set $\NS$ is the set of non-separating one-edge free splittings of $F_N$. 
The goal of the present section is to prove the following proposition.

\begin{prop}\label{prop:ns-fs}
Let $N\ge 3$, and let $\theta:\NS\to \FS$ be an injective map that preserves adjacency and non-adjacency.

Then $\theta(\NS)\subseteq\NS$.
\end{prop}

\begin{rk}
In fact, the proposition remains true for $N=2$, as in this case one can check that a vertex of $\mathbb{FS}$ is in $\NS$ if and only if it has infinite valence (there are exactly two one-edge free splittings that are compatible with a given one-edge separating free splitting $S$, other than $S$ itself; they are given by blowing up each of the two vertex groups isomorphic to $\mathbb{Z}$). 
\end{rk}

Recall that the \emph{link} of a vertex $x$ in a simplicial graph $\mathbb{G}$ is the induced subgraph $\lk_{\mathbb{G}}(x)\subseteq \mathbb{G}$
whose vertex set is the set of vertices $v\neq x$ in $\mathbb{G}$ that are adjacent to $x$.
Recall that a graph with vertex set $V$ splits non-trivially as a join $X_1\star X_2$ if $V=X_1\dunion X_2$
with $X_1,X_2\neq\es$ 
and for all $x_1\in X_1$ and all $x_2\in X_2$ there is an edge joining $x_1$ to $x_2$.

\begin{lemma}\label{lemma:link-ns}
  For any $S\in \NS$, the link of $S$ in $\NSG$ does not split non-trivially as a join.
\end{lemma}

\begin{proof}
  It is enough to show that given any two compatible non-separating free splittings $S_1,S_2$ of $F_N$ in the link of $S$, there exists a non-separating free splitting $S'$ of $F_N$ which is neither compatible with $S_1$ nor with $S_2$. Indeed, if $\lk_{\NSG}(S)$ decomposes non-trivially as a join $\lk_{\NSG}(S)=X_1\star X_2$, then choosing any splittings $S_1\in X_1$ and $S_2\in X_2$, one has that $S_1$ and $S_2$ are compatible, and every splitting $S'\in\NS$ is then compatible with $S_1$ if $S'\in X_2$  or with $S_2$ otherwise.

Let $A$ be a corank one free factor of $F_N$ such that $S$ is the Bass--Serre tree of the HNN extension $F_N=A\ast$. Then the $A$-minimal invariant subtrees of $S_1$ and $S_2$ yield two (possibly isomorphic) non-trivial free splittings of $A$, which we denote by $(S_1)_{|A}$ and $(S_2)_{|A}$. As $\mathrm{rk}(A)\ge 2$, we can find a one-edge non-separating free splitting $S'_A$ of $A$ that is neither compatible with $(S_1)_{|A}$ nor with $(S_2)_{|A}$ -- this is, for instance, a consequence of the fact that the free splitting graph as infinite diameter \cite{HM3,HM4}. Let $\hat{S}$ be a refinement of $S$ such that $\hat{S}/F_N$ is a two-petalled rose, and whose $A$-minimal subtree $(\hat S)_{|A}$ is isomorphic to $S'_A$: such a splitting can be obtained by blowing up $S$ at the vertex stabilized by $A$, using the splitting $S'_A$. Let $S'\in\NS$ be the splitting obtained from $\hat{S}$ by collapsing every edge coming from $S$ to a point. Then $S'$ is a one-edge non-separating free splitting of $F_N$ in the link of $S$. It is neither compatible with $S_1$ nor with $S_2$, as otherwise $S'_A$ would be compatible with either $(S_1)_{|A}$ or $(S_2)_{|A}$.
\end{proof}

\begin{lemma}\label{lemma:maximal-clique-in-link}
  For any $S\in \NS$, the link of $S$  in $\NSG$
  contains a complete subgraph on $3N-4$ vertices.
\end{lemma}

\begin{proof}
  The splitting $S$ has a refinement $\hat{S}$ such that $\hat{S}/F_N$ has $3N-3$ edges and no separating edge (recall that $3N-3$ is the maximal number of edges of a graph in Culler--Vogtmann's reduced Outer space \cite{CV}). The one-edge splittings associated to the edges of $\hat S/F_N$
  yield in addition to $S$, $3N-4$ distinct non-separating one-edge free splittings in the link of $S$ that are pairwise compatible.
\end{proof}

\begin{lemma}\label{lemma:free-splitting-link}
  For any $S\in \FS\setminus\NS$,
  the link of $S$ in $\FSG$ decomposes as a non-trivial join $\lk_{\FSG}(S)=X_1\star X_2$ 
  where neither $X_1$ nor $X_2$ contains a complete subgraph on $3N-4$ vertices. 
\end{lemma}

\begin{rk}
  We allow  $X_1$ or $X_2$ to be reduced to a point (this occurs if one of the factors of the splitting is cyclic).
\end{rk}

\begin{proof}
  Let $A_1,A_2$ be proper free factors of $F_N$ such that $S$ is the Bass--Serre tree of the free product $F_N=A_1\ast A_2$. Every one-edge free splitting of $F_N$ that is compatible with $S$ but distinct from $S$
  is obtained in the following way: choose $i\in\{1,2\}$, construct a refinement $\hat{S}$ of $S$ by blowing up $S$ at the vertex fixed by $A_i$, 
  using a one-edge free splitting of $A_i$ and choosing an attaching point for the edge of $S$, 
  and then collapse every edge of $\hat{S}$ that comes from $S$ to a point. We say that such a splitting is a \emph{compatible splitting of type $i$}. Then the link $\lk_{\FSG}(S)$ decomposes as a non-trivial join $\lk_{\FSG}(S)=X_1\star X_2$, where $X_i$ consists of all compatible splittings of type $i$.

Let $i\in\{1,2\}$. Assume towards a contradiction that $X_i$ contains a complete subgraph on $3N-4$ vertices. 
Then by \cite[Lemma A.17]{GL-jsj} there exists a refinement of $S$ with $3N-3$ orbits of edges and in which $A_j$ is elliptic (with $j\neq i$). 
This is a contradiction because every free splitting of $F_N$ with $3N-3$ orbits of edges corresponds to a free action of $F_N$.
\end{proof}

\begin{proof}[Proof of Proposition~\ref{prop:ns-fs}]
  Assume towards a contradiction that there exists a vertex $S\in\NS$ such that $\theta(S)$ is a separating splitting. By Lemma~\ref{lemma:free-splitting-link}, the link $\lk_{\FSG}(\theta(S))$ decomposes as a non-trivial join $\lk_{\FSG}(\theta(S))=X_1\star X_2$, where no $X_i$ contains a complete subgraph on $3N-4$ vertices. By Lemma~\ref{lemma:link-ns}, the link $\lk_{\NSG}(S)$ does not decompose non-trivially as a join, and as $\theta$ preserves both adjacency and non-adjacency, neither does $\theta(\lk_{\NSG}(S))$. Since $$\theta(\lk_{\NSG}(S))=(\theta(\lk_{\NSG}(S))\cap X_1)\star (\theta(\lk_{\NSG}(S))\cap X_2),$$ we deduce that $\theta(\lk_{\NSG}(S))$ is contained in either $X_1$ or $X_2$. On the other hand, Lemma~\ref{lemma:maximal-clique-in-link} ensures that $\lk_{\NSG}(S)$ contains a complete subgraph on $3N-4$ vertices, a contradiction.
\end{proof}

\section{Measure equivalence rigidity of $\Out(F_N)$}\label{sec:conclusion}

The goal of this section is to complete the proof of the main theorem of the paper.

\begin{theo}\label{thm_superrigid}
For every $N\ge 3$, the group $\Out(F_N)$ is ME-superrigid.
\end{theo}

The following lemma is folklore. It follows for instance from \cite[Lemma 3.2, Corollary 1.4]{FH}.

\begin{lemma}\label{lem_ICC}
For every $N\geq 3$, the group $\Out(F_N)$ is ICC: the conjugacy class of every non-trivial element of $\Out(F_N)$ is infinite.\qed
\end{lemma}

\begin{proof}[Proof of Theorem \ref{thm_superrigid}]
Since $\Out(F_N)$ is ICC,
the theorem follows from Theorem~\ref{theo:me}
together with the following statement, which shows that $\Out(F_N)$ is rigid with respect to action-type cocycles (Definition~\ref{de:rigid}).
\end{proof}

\begin{theo}\label{theo:two-cocycles}
For every $N\ge 3$, the group $\Out(F_N)$ is rigid with respect to action-type cocycles. 

More precisely, let $\calg$ be a measured groupoid over a base space $Y$, and let $\rho,\rho':\calg\to\IA$ be two action-type cocycles.

Then $\rho$ and $\rho'$ are $\Out(F_N)$-cohomologous, i.e.\ there exist a Borel map  
$\phi:Y\to\Out(F_N)$ and a conull Borel subset $Y^*\subseteq Y$ such that for all $g\in\calg_{|Y^*}$, one has $\rho'(g)=\phi(r(g))\rho(g)\phi(s(g))\m$. 
\end{theo}

We are now left proving Theorem~\ref{theo:two-cocycles}.

\begin{lemma}\label{lemma:two-stab-different}
  Let $\calg$ be a measured groupoid over a base space $Y$, and let $\rho:\calg\to\IA$ be an action-type cocycle. Let $S_1,S_2\in\NS\cup\FSii$ be two distinct splittings and $H_1,H_2\subset\IA$ their stabilizers.

Then for every Borel subset $U\subseteq Y$ of positive measure, one has $\rho^{-1}(H_1)_{|U}\neq \rho^{-1}(H_2)_{|U}$. 
\end{lemma}

\begin{proof}
The splitting $S_1$ is acyclic (i.e.\ its quotient graph does not contain any terminal vertex with vertex group isomorphic to $\mathbb{Z}$), so Lemma~\ref{lemma:stabilizer-free-splitting} shows that there exists an element in $H_1$ none of whose powers belongs to $H_2$. As $\rho$ is action-type, this implies that $\rho\m(H_1)_{|U}$ is not contained in $\rho\m(H_2)_{|U}$ (see Remark~\ref{rk:action-type}), which concludes the proof. 
\end{proof}

Recall that the $(\calg,\rho)$-stabilizer of a Borel map $\sigma:Y\ra \FS$ is
the subgroupoid of $\calg$ consisting of all $g\in\calg$ such that $\sigma(r(g))=\rho(g)\sigma(s(g))$. Note that if $Y=\dunion Y_i$ is a countable Borel partition such that $\sigma$ is constant on $Y_i$ with value $S_i$,
and if $H_i\subset \IA$ denotes the stabilizer of $S_i$,
then the $(\calg,\rho)$-stabilizer of $\sigma$ stably coincides with the groupoid $\cup_i (\rho\m(H_i))_{|Y_i}$.

\begin{rk}\label{rk:two-stab-different}
  As a consequence of Lemma~\ref{lemma:two-stab-different}, if $\sigma,\sigma':Y\to\NS\cup\FSii$ are two Borel maps such that the $(\calg,\rho)$-stabilizers of $\sigma$ and $\sigma'$ are stably equal, then $\sigma$ and $\sigma'$ coincide almost everywhere.
\end{rk}

The following key lemma is based on the groupoidal (almost) characterization of stabilizers of non-separating free splittings of
Section \ref{sec:big}.

\begin{lemma}\label{lemma:rho-rho'}
Let $\calg$ be a measured groupoid over a base space $Y$, and let $\rho,\rho':\calg\to\IA$ be two action-type cocycles. 

For every Borel map $\sigma:Y\to\NS$, there exists an essentially unique Borel map $\sigma':Y\to\NS\cup\FSii$ such that the $(\calg,\rho)$-stabilizer of $\sigma$ is stably equal  to the $(\calg,\rho')$-stabilizer of $\sigma'$.  
\end{lemma}

We say that $\sigma'$ is the \emph{$(\calg,\rho,\rho')$-image} of $\sigma$.
Equivalently, a Borel map $\sigma':Y\to\NS\cup\FSii$ is (up to measure $0$) the $(\calg,\rho,\rho')$-image of $\sigma$ if and only if
there exist a conull Borel subset $Y^*\subseteq Y$ and a partition $Y^*=\Dunion_{i\in I} Y_i$ into at most countably many Borel subsets such that for every $i\in I$, the maps $\sigma$ and $\sigma'$ are constant on $Y_i$ with respective values $S_i,S'_i$,
and denoting by $H_i,H'_i$ the respective stabilizers of $S_i,S'_i$ in $\IA$, one has $(\rho\m(H_i))_{|Y_i}=({\rho'}\m(H'_i))_{|Y_i}$.

Note that \emph{a priori}, there is a slight asymmetry between $\sigma$ and $\sigma'$ (coming from the asymmetry in Theorem~\ref{theo:non-characterization}): indeed $\sigma$ is assumed to take its values in $\NS$ while $\sigma'$ is assumed to take its values in $\NS\cup\FSii$. However, using the results from  Section~\ref{sec:maps}, we will see in the course of the proof of Theorem~\ref{theo:two-cocycles} that $\sigma'$ actually takes its values in $\NS$. 

\begin{proof}[Proof of Lemma \ref{lemma:rho-rho'}]
  The essential uniqueness of $\sigma'$ follows from Remark \ref{rk:two-stab-different}, so we now prove the existence of $\sigma'$.
  
  Let $\calh$ be the $(\calg,\rho)$-stabilizer of $\sigma$. By the second part of Theorem~\ref{theo:non-characterization} (applied to the cocycle $\rho$, with the partition given by $\sigma$-preimages of points in $\NS$), the pair $(\calg,\calh)$ satisfies Properties~$\PImax,~\PII$ and~$\PIII$. By the first part of Theorem~\ref{theo:non-characterization} (applied to the cocycle $\rho'$),
  there exist a conull subset $Y^*\subset Y$, a countable Borel partition $Y^*=\dunion Y_i$, and splittings $S'_i\in\NS\cup\FSii$,
  such that $\calh_{|Y_i}$ is equal to the $(\calg_{|Y_i},\rho')$-stabilizer of $S'_i$. 
  Consider the map $\sigma'$ assigning $S'_i$ to $y\in Y_i$ (and defined in an arbitrary way on $Y\setminus Y^*$).
  Then $\calh$ is stably  equal  to the $(\calg,\rho')$-stabilizer of $\sigma'$.
\end{proof}

\begin{rk}\label{rk_symmetry}
If $\sigma'$ is the $(\calg,\rho,\rho')$-image of $\sigma$, and if additionally $\sigma'$ takes its values in $\NS$, then by symmetry, $\sigma$ is the $(\calg,\rho',\rho)$-image of $\sigma'$. 
\end{rk}

\begin{rk} \label{rk_restriction} If $\sigma'$ is the $(\calg,\rho,\rho')$-image of $\sigma$, then for every Borel subset $U\subseteq Y$ of positive measure, 
$\sigma'_{|U}$ is the $(\calg_{|U},\rho,\rho')$-image of $\sigma_{|U}$.
\end{rk}

The following lemma is based on the groupoidal characterization of compatibility in Section \ref{sec:compatibility}.

\begin{lemma}\label{lemma:groupoid-isomorphism-ns-fs}
Let $\calg$ be a measured groupoid over a base space $Y$, and let $\rho,\rho':\calg\to\IA$ be two action-type cocycles. Consider $S_1,S_2,S'_1,S'_2\in\NS\cup\FSii$ with $S_1\neq S_2$ and $S'_1\neq S'_2$, and assume that for every $i\in\{1,2\}$, the $(\calg,\rho)$-stabilizer of $S_i$ is stably equal to the $(\calg,\rho')$-stabilizer of $S'_i$. 

Then $S_1$ is compatible with $S_2$ if and only if $S'_1$ is compatible with $S'_2$.
\end{lemma}

\begin{proof}
It suffices to assume that $S_1$ and $S_2$ are compatible and prove that so are $S'_1$ and $S'_2$. For every $i\in\{1,2\}$, 
let $\calh_i$ be the $(\calg,\rho)$-stabilizer of $S_i$, and let $\calh'_i$ be the $(\calg,\rho')$-stabilizer of $S'_i$. 

Theorem~\ref{theo:compatibility-ns} (applied to the cocycle $\rho$) ensures that the triple $(\calg;\calh_1,\calh_2)$ satisfies Properties~$\Pcompi$ and $\Pcompii$. As $\calh_i$ is stably equal to $\calh'_i$, we deduce that $(\calg;\calh'_1,\calh'_2)$ also satisfies Properties~$\Pcompi$ and $\Pcompii$. Therefore, the converse implication from Theorem~\ref{theo:compatibility-ns} (applied to the cocycle $\rho'$) ensures that $S'_1$ and $S'_2$ are compatible.    
\end{proof}

\begin{proof}[Proof of Theorem~\ref{theo:two-cocycles}]
  In this proof, given two 
  simplicial graphs $\mathbb{X}$ and $\mathbb{Y}$
  with vertex sets $V_{\mathbb{X}}$ and $V_{\mathbb{Y}}$, we denote by $\Inj(\mathbb{X}\to\mathbb{Y})$ the set of all injective maps from $V_{\mathbb{X}}$ to $V_{\mathbb{Y}}$ that preserve adjacency and non-adjacency.
  We denote by $\Aut(\mathbb{X})$ the group of all graph automorphisms of $\mathbb{X}$, i.e.\ maps in $\Inj(\mathbb{X}\to\mathbb{X})$ which are bijective.

  Before actually starting the proof, let us briefly explain the general strategy.
  We will use the graph $\FSG$ of free splittings and its subgraph $\NSG$ of non-separating free splittings.
  We will start by constructing a map $\Psi:Y\ra \Inj(\NSG\ra \FSG)$
  such that for every Borel map  $\sigma:Y\ra \NS$, the $(\calg,\rho,\rho')$-image $\sigma'$ of $\sigma$ is obtained by applying $\Psi$ to $\sigma$, i.e.\ $\sigma'(y)=\Psi(y)(\sigma(y))$.  
  There is also a symmetric map $\Psi'$ obtained by reversing the roles of $\rho$ and $\rho'$. Using the fact that $\Inj(\NSG\ra \FSG)=\Inj(\NSG\ra \NSG)$ (Proposition~\ref{prop:ns-fs}), we will actually deduce that $\Psi'(y)$ is (almost everywhere) an inverse of $\Psi(y)$. Therefore $\Psi$ takes its values in $\Aut(\NSG)$. By a theorem of Pandit (\cite{Pan}, which builds on work of Bridson and Vogtmann regarding the symmetries of Outer space \cite{BV2}), we know that $\Aut(\NSG)$ is isomorphic to $\Out(F_N)$. This will allow us to construct the map $\phi:Y\ra \Out(F_N)$ that conjugates $\rho$ to $\rho'$. 
  
 Let us now construct the map $\Psi$.
  Given $S\in\NS$, we let $\sigma_S^0:Y\to\NS$ be the constant map with value $S$, and $\sigma'_S:Y\to\NS\cup\FSii$ be the $(\calg,\rho,\rho')$-image of $\sigma^0_S$ (the map $\sigma'_S$ is not necessarily constant).
 This allows us to define a Borel map 
  \begin{align*}
    \tilde\Psi:Y\times \NS \ra &\  \NS\cup\FSii \\
    (y,S)\mapsto &\ \sigma'_S(y).
  \end{align*}
  Then for each $y\in Y$, we define $\Psi_y:\NS\ra\NS\cup\FSii$ by $\Psi_y=\tilde\Psi(y,\cdot)$. 
  
  We note that for every Borel map $\sigma:Y\ra \NS$, the
  pointwise evaluation
  \begin{align*}
    \Psi \bullet \sigma :Y\ra &\ \NS\cup\FSii\\
    y \mapsto& \ \Psi_y(\sigma(y))\ =\ \Tilde \Psi(y,\sigma(y)) 
\end{align*}
coincides almost everywhere with the $(\calg,\rho,\rho')$-image of $\sigma$.
Indeed, using countability of $\FS$, this  follows from the definition after partitioning $Y$ into countably many
Borel subsets on which $\sigma$ is constant.

 Let us prove that for a.e.\ $y\in Y$, the map $\Psi_y:\NS\ra\NS\cup\FSii$ is injective.
 Otherwise,
 there exists a Borel subset $U\subseteq Y$ of positive measure, two distinct splittings $S_1,S_2\in\NS$ and  $S'\in\NS\cup\FSii$ such that for all $y\in U$, one has $\Psi_y(S_1)=\Psi_y(S_2)=S'$.
 For every $i\in\{1,2\}$, let $H_i\subset\ia$ be the stabilizer of $S_i$, and let $H'\subset \ia$ be the stabilizer of $S'$. The fact that $\Psi_y(S_1)=\Psi_y(S_2)=S'$ for $y\in U$ ensures that $\rho^{-1}(H_1)_{|U}$ and $\rho^{-1}(H_2)_{|U}$ are both stably equal to $(\rho')^{-1}(H')_{|U}$ (see Remark \ref{rk_restriction}).
 This is a contradiction to Lemma~\ref{lemma:two-stab-different} which implies that $\rho^{-1}(H_1)_{|U}$ and $\rho^{-1}(H_2)_{|U}$ are not stably equal.

 Second, for a.e.\ $y\in Y$, the map $\Psi_y$ preserves adjacency and non-adjacency.
Indeed, given a pair of distinct splittings $S_1,S_2\in\NS$, up to partitioning $Y$,
 one may restrict to a Borel subset of positive measure
 $U\subset Y$ such that $S'_1:=\Psi_y(S_1)$ and
 $S'_2:=\Psi_y(S_2)$ do not depend on $y\in U$. 
 Applying Remark \ref{rk_restriction} and Lemma~\ref{lemma:groupoid-isomorphism-ns-fs} (applied to $\calg_{|U}$), we get that $S_1$ and $S_2$ are compatible
 if and only if $S'_1$ and $S'_2$ are.

  Until now, we have proved that for a.e.\ $y\in Y$, the map $\Psi_y$ belongs to $\Maps$.
  As recalled above (Proposition~\ref{prop:ns-fs}), $\Maps=\mathrm{Inj}(\NSG\to\NSG)$,
  so we get a Borel map
  \begin{align*}
    \Psi:Y \ra &\ \mathrm{Inj}(\NSG\to\NSG) \\
    y\mapsto &\ \Psi_y.
  \end{align*}
Note that in particular, 
 for any Borel map $\sigma:Y\to\NS$, the $(\calg,\rho,\rho')$-image  of $\sigma$ also takes its values  (essentially) in $\NS$.

 We finally claim that for a.e.\ $y\in Y$, the map $\Psi_y$ is actually an automorphism of $\NSG$.
 Indeed, let $\Psi':Y\ra \Inj(\NSG\to\NSG)$ be the map obtained in place of $\Psi$ by reversing the roles of $\rho$ and $\rho'$.
 The map $\Psi'$ has the property that for any Borel map $\sigma':Y\ra \NS$,
 the pointwise evaluation $\Psi'\bullet\sigma'$ is the $(\calg,\rho',\rho)$-image
 of $\sigma'$. Then for any  $S\in \NS$,
  the map $\sigma''_S:=\Psi'\bullet(\Psi\bullet\sigma^0_S)$ is the essentially unique map
  $\sigma''_S:Y\ra \NS\cup\FSii$ whose $(\calg,\rho)$-stabilizer
 stably coincides with the $(\calg,\rho)$-stabilizer of $\sigma^0_S$, so 
 $\sigma''_S=\sigma^0_S$ almost everywhere (see Remark \ref{rk_symmetry}).
 Since this holds for every $S\in\NS$, this means that for almost every $y\in Y$, we have $\Psi'_y\circ\Psi_y=\id_{\NSG}$,
 and for symmetrical reasons, $\Psi_y\circ\Psi'_y=\id_{\NSG}$ for almost every $y\in Y$, which proves our claim.

 We have thus constructed a Borel map $\Psi:Y\to\Aut(\NSG)$. Composing with the isomorphism $\Aut(\NSG)\ra\Out(F_N)$
 yields a map $\phi:Y\ra \Out(F_N)$ with the property that
 for every Borel map $\sigma:Y\ra \NS$, its $(\calg,\rho,\rho')$-image is $\sigma'=\Psi\bullet\sigma$ i.e.\
 $\sigma'(y)=\phi(y)\cdot\sigma(y)$ for almost every $y\in Y$.
 
 We now check that  there exists a conull Borel subset $Y^*\subseteq Y$ such that the desired cohomology relation
 $\rho'(g)=\phi(r(g))\rho(g)\phi(s(g))\m$ holds for all $g\in \calg_{|Y^*}$.
 
 Since $\calg$ is a countable union of bisections, it suffices to fix a bisection $B$ of $\calg$ and show that the above cohomology relation holds for almost every $g\in B$. Without loss of generality, 
 we can also assume that  there exist $\gamma,\gamma'\in\IA$ such that $\rho(B)=\{\gamma\}$ and $\rho'(B)=\{\gamma'\}$.
 Let $f:U\ra V$ be the partial isomorphism between Borel subsets of $Y$ associated to $B$. 
 It suffices to check that for almost every $x\in U$, one has 
 $$\gamma'=\phi(f(x))\gamma\phi(x)\m.$$
 Since the action of $\Out(F_N)$ on $\NS$ is faithful,
 it suffices to show that 
 for every Borel map $\sigma:U\to\NS$ and almost every $x\in U$, one has $$\gamma' \phi(x) \cdot\sigma(x) =\phi(f(x))\gamma\cdot\sigma(x)$$
 (it is even enough to prove it when $\sigma$ is constant).

 Let $\calh_U$ be the $(\calg_{|U},\rho)$-stabilizer of $\sigma$.
The $(\calg_{|U},\rho,\rho')$-image of $\sigma$ is $\sigma'=\Psi\bullet\sigma=\phi\cdot\sigma$
(i.e.\ $\sigma'(x)=\phi(x)\cdot\sigma(x)$ for almost every $x\in U$),
so the $(\calg_{|U},\rho')$-stabilizer of $\sigma'$ stably coincides with $\calh_{U}$.
 Conjugating by $B$, we see that 
 the $(\calg_{|V},\rho)$-stabilizer of the map $\sigma_V:=\gamma\cdot\sigma\circ f\m$ is  $B \calh_U B\m$.
 Similarly, the $(\calg_{|V},\rho')$-stabilizer of the map $\sigma'_V:=\gamma'\cdot\sigma'\circ f\m$
 stably coincides with $B \calh_U B\m$.
 It follows that $\sigma'_V$ is the $(\calg_{|V},\rho,\rho')$-image of $\sigma_V$, so
 for almost every $y\in V$,
 $$\sigma'_V(y)=\phi(y)\cdot\sigma_V(y).$$
 Replacing $\sigma'_V$ and $\sigma_V$ by their definition,
 we get that for almost every $y\in V$,
 $$\gamma'\cdot\sigma'(f\m(y))=\phi(y)\gamma\cdot\sigma(f\m(y))$$
 and applying this equality to $y=f(x)$ for almost every $x\in U$, we get
 $$\gamma' \cdot(\phi(x)\cdot\sigma(x))=\phi(f(x))\gamma\cdot\sigma(x)$$
as desired.
\end{proof}

We conclude this section by proving the remaining theorems from the introduction.

The following result slightly generalizes Theorem \ref{theo:intro-oe}.
With the terminology introduced in Section \ref{sec:oe-rigidity}, it says that
any ergodic free measure-preserving action of $\Gamma$ on a probability space
is OE-superrigid. 
\begin{theo}\label{theo:oe-out(fn)}
Let $N\ge 3$. Let $\Gamma\subset\Out(F_N)$ be a finite index subgroup, and $\Lambda$ a countable group. Consider two standard probability spaces $X_1,X_2$ and two ergodic free measure-preserving actions $\Gamma\actson X_1$, $\Lambda\actson X_2$.

If the actions $\Gamma\actson X_1$ and $\Lambda\actson X_2$ are stably orbit equivalent, then they are virtually conjugate.
\end{theo}

\begin{proof}[Proof of Theorems \ref{theo:oe-out(fn)} and \ref{theo:intro-oe}.]
The group $\Out(F_N)$ is ICC by Lemma~\ref{lem_ICC}.
Theorem~\ref{theo:two-cocycles} establishes that $\Out(F_N)$ is rigid with respect action-type cocycles.
Theorem~\ref{theo:oe-out(fn)} then follows from 
Theorem~\ref{theo:oe}. And Theorem~\ref{theo:intro-oe} follows from Remark~\ref{rk_aperiodic}. 
\end{proof}

Likewise, we complete the proof of Theorem~\ref{theo:lattice-embeddings} and Corollary~\ref{corintro:cayley} from the introduction. 

\begin{proof}[Proof of Theorem \ref{theo:lattice-embeddings} and Corollary \ref{corintro:cayley}]\label{proof:lattice}
The group $\Out(F_N)$ is ICC by Lemma~\ref{lem_ICC}.
Theorem~\ref{theo:two-cocycles} establishes that $\Out(F_N)$ is rigid with respect action-type cocycles.
The results then follow from Theorem~\ref{theo:lattice} and 
Corollary~\ref{cor:cayley}.
\end{proof}

\footnotesize

\bibliographystyle{alpha}
\bibliography{ME-bib}

\begin{thebibliography}{BTW07}

\bibitem[AD07]{AD}
C.~Anantharaman-Delaroche.
\newblock Amenability and exactness for groups, group actions and operator
  algebras.
\newblock {\em École thématique. Erwin Schrödinger Institute, Mars 2007.
  cel-00360390}, 2007.

\bibitem[Ada94]{Ada}
S.~Adams.
\newblock Indecomposability of equivalence relations generated by word
  hyperbolic groups.
\newblock {\em Topol.}, 33(4):785--798, 1994.

\bibitem[ADR00]{ADR}
C.~Anantharaman-Delaroche and J.~Renault.
\newblock {\em Amenable groupoids}, volume~36 of {\em Monographies de
  l'Enseignement Mathématique}.
\newblock L'Enseignement Mathématique, Geneva, 2000.

\bibitem[AEG94]{AEG}
S.~Adams, G.A. Elliott, and T.~Giordano.
\newblock Amenable actions of groups.
\newblock {\em Trans. Amer. Math. Soc.}, 344(2):803--822, 1994.

\bibitem[AS11]{AS}
J.~Aramayona and J.~Souto.
\newblock Automorphisms of the graph of free splittings.
\newblock {\em Michigan Math. J.}, 60(3):483--493, 2011.

\bibitem[BF91]{BF91}
M.~Bestvina and M.~Feighn.
\newblock Bounding the complexity of simplicial group actions on trees.
\newblock {\em Invent. Math.}, 103(3):449--469, 1991.

\bibitem[BF94]{BF}
M.~Bestvina and M.~Feighn.
\newblock Outer limits.
\newblock 1994.
\newblock preprint, available at
  \url{https://www.math.utah.edu/~bestvina/eprints/bestvina.feighn..outer_limits.pdf}.

\bibitem[BF95]{BF95}
M.~Bestvina and M.~Feighn.
\newblock Stable actions of groups on real trees.
\newblock {\em Invent. Math.}, 121(2):287--321, 1995.

\bibitem[BF14]{BF-FF}
M.~Bestvina and M.~Feighn.
\newblock Hyperbolicity of the complex of free factors.
\newblock {\em Adv. Math.}, 256:104--155, 2014.

\bibitem[BFH97]{BFH0}
M.~Bestvina, M.~Feighn, and M.~Handel.
\newblock Laminations, trees, and irreducible automorphisms of free groups.
\newblock {\em Geom. Funct. Anal.}, 7(2):215--244, 1997.

\bibitem[BFH00]{BFH}
M.~Bestvina, M.~Feighn, and M.~Handel.
\newblock The {T}its alternative for $\text{{O}ut}({F}_n)$ {I}: {D}ynamics of
  exponentially-growing automorphisms.
\newblock {\em Ann. of Math. (2)}, 151(2):517--623, 2000.

\bibitem[BFS13]{BFS}
U.~Bader, A.~Furman, and R.~Sauer.
\newblock Integrable measure equivalence and rigidity of hyperbolic lattices.
\newblock {\em Invent. Math.}, 194(2):313--379, 2013.

\bibitem[BGH22]{BGH}
M.~Bestvina, V.~Guirardel, and C.~Horbez.
\newblock Boundary amenability of {${\rm Out}(F_N)$}.
\newblock {\em Ann. Sci. \'Ec. Norm. Sup\'er. (4)}, 55(5):1379--1431, 2022.

\bibitem[BH92]{BH}
M.~Bestvina and M.~Handel.
\newblock Train tracks and automorphisms of free groups.
\newblock {\em Ann. of Math. (2)}, 135(1):1--51, 1992.

\bibitem[BIP21]{BIP}
R.~Boutonnet, A.~Ioana, and J.~Peterson.
\newblock Properly proximal groups and their von {N}eumann algebras.
\newblock {\em Ann. Sci. \'Ec. Norm. Sup\'er. (4)}, 54(2):445--482, 2021.

\bibitem[BR15]{BR}
M.~Bestvina and P.~Reynolds.
\newblock The boundary of the complex of free factors.
\newblock {\em Duke Math. J.}, 164(11):2213--2251, 2015.

\bibitem[BTW07]{BTW}
M.R. Bridson, M.~Tweedale, and H.~Wilton.
\newblock Limit groups, positive-genus towers and measure-equivalence.
\newblock {\em Ergodic Theory Dynam. Systems}, 27(3):703--712, 2007.

\bibitem[BV00]{BV}
M.~Bridson and K.~Vogtmann.
\newblock Automorphisms of automorphism groups of free groups.
\newblock {\em J. Algebra}, 229(2):785--792, 2000.

\bibitem[BV01]{BV2}
M.~Bridson and K.~Vogtmann.
\newblock The symmetries of outer space.
\newblock {\em Duke Math. J.}, 106(2):391--409, 2001.

\bibitem[CFW81]{CFW}
A.~Connes, J.~Feldman, and B.~Weiss.
\newblock An amenable equivalence relation is generated by a single
  transformation.
\newblock {\em Ergodic Theory Dynam. Systems}, 1(4):431--450, 1981.

\bibitem[CK15]{CK}
I.~Chifan and Y.~Kida.
\newblock ${OE}$ and ${W}^*$ superrigidity results for actions by surface braid
  groups.
\newblock {\em Proc. Lond. Math. Soc. (3)}, 111(6):1431--1470, 2015.

\bibitem[CL95]{CL}
M.M. Cohen and M.~Lustig.
\newblock Very small group actions on $\mathbb{R}$-trees and {D}ehn twist
  automorphisms.
\newblock {\em Topol.}, 34(3):575--617, 1995.

\bibitem[CM87]{CM}
M.~Culler and J.W. Morgan.
\newblock Group actions on $\mathbb{R}$-trees.
\newblock {\em Proc. London Math. Soc.}, 55(3):571--604, 1987.

\bibitem[CV86]{CV}
M.~Culler and K.~Vogtmann.
\newblock Moduli of graphs and automorphisms of free groups.
\newblock {\em Invent. Math.}, 84(1):91--119, 1986.

\bibitem[DG11]{DG_isomorphism}
F.~Dahmani and V.~Guirardel.
\newblock The isomorphism problem for all hyperbolic groups.
\newblock {\em Geom. Funct. Anal.}, 21(2):223--300, 2011.

\bibitem[dlST19]{dlST}
M.~de~la Salle and R.~Tessera.
\newblock Characterizing a vertex-transitive graph by a large ball.
\newblock {\em J. Topol.}, 12(3):705--743, 2019.

\bibitem[DSU17]{DSU}
T.~Das, D.~Simmons, and M.~Urba\'nski.
\newblock {\em Geometry and dynamics in {G}romov hyperbolic metric spaces},
  volume 218 of {\em Mathematical Surveys and Monographs}.
\newblock American Mathematical Society, Providence, RI, 2017.
\newblock With an emphasis on non-proper settings.

\bibitem[DV96]{DV}
W.~Dicks and E.~Ventura.
\newblock {\em The group fixed by a family of injective endomorphisms of a free
  group}, volume 195 of {\em Contemp. Math.}
\newblock Amer. Math. Soc., Providence, RI, 1996.

\bibitem[Dye59]{Dye1}
H.A. Dye.
\newblock On groups of measure preserving transformation. {I}.
\newblock {\em Amer. J. Math.}, 81:119--159, 1959.

\bibitem[Dye63]{Dye2}
H.A. Dye.
\newblock On groups of measure preserving transformations. {II}.
\newblock {\em Amer. J. Math.}, 85:551--576, 1963.

\bibitem[Eps07]{Epstein}
I.~Epstein.
\newblock Orbit inequivalent actions of non-amenable groups.
\newblock {\em arXiv:0707.4215}, 2007.

\bibitem[FH07]{FH}
B.~Farb and M.~Handel.
\newblock Commensurations of $\text{{O}ut}({F}_n)$.
\newblock {\em Publ. Math. Inst. Hautes \'Etudes Sci.}, 105(1):1--48, 2007.

\bibitem[FH11]{FeH}
M.~Feighn and M.~Handel.
\newblock The recognition theorem for {${\rm Out}(F_n)$}.
\newblock {\em Groups Geom. Dyn.}, 5(1):39--106, 2011.

\bibitem[FM15]{FM}
S.~Francaviglia and A.~Martino.
\newblock Stretching factors, metrics and train tracks for free products.
\newblock {\em Illinois J. Math.}, 59(4):859--899, 2015.

\bibitem[For02]{For_deformation}
M.~Forester.
\newblock Deformation and rigidity of simplicial group actions on trees.
\newblock {\em Geom. Topol.}, 6:219--267 (electronic), 2002.

\bibitem[FP06]{FP}
K.~Fujiwara and P.~Papasoglu.
\newblock {JSJ}-decompositions of finitely presented groups and complexes of
  groups.
\newblock {\em Geom. Funct. Anal.}, 16(1):70--125, 2006.

\bibitem[FSZ89]{FSZ}
J.~Feldman, C.E. Sutherland, and R.J. Zimmer.
\newblock Subrelations of ergodic equivalence relations.
\newblock {\em Ergodic Theory Dynam. Systems}, 9(2):239--269, 1989.

\bibitem[Fur99a]{Fur1}
A.~Furman.
\newblock Gromov's measure equivalence rigidity of higher-rank lattices.
\newblock {\em Ann. of Math. (2)}, 150:1059--1081, 1999.

\bibitem[Fur99b]{Fur2}
A.~Furman.
\newblock Orbit equivalence rigidity.
\newblock {\em Ann. of Math. (2)}, 130(3):1083--1108, 1999.

\bibitem[Fur11a]{Fur3}
A.~Furman.
\newblock Mostow--{M}argulis rigidity with locally compact targets.
\newblock {\em Geom. Funct. Anal.}, 11(1):30--59, 2011.

\bibitem[Fur11b]{Fur-survey}
A.~Furman.
\newblock A survey of measured group theory.
\newblock In {\em Geometry, rigidity, and group actions}, Chicago Lectures in
  Math., pages 296--374, Chicago, IL, 2011. Univ. Chicago Press.

\bibitem[Gab02a]{Gab_L2}
D.~Gaboriau.
\newblock Invariants {$l^2$} de relations d'\'{e}quivalence et de groupes.
\newblock {\em Publ. Math. Inst. Hautes \'{E}tudes Sci.}, (95):93--150, 2002.

\bibitem[Gab02b]{Gab2}
D.~Gaboriau.
\newblock On orbit equivalence of measure preserving actions.
\newblock In {\em Rigidity in dynamics and geometry (Cambridge, 2000)}, pages
  167--186, Berlin, 2002. Springer.

\bibitem[Gab05]{Gab}
D.~Gaboriau.
\newblock Examples of groups that are measure equivalent to the free group.
\newblock {\em Ergodic Theory Dynam. Systems}, 25(6):1809--1827, 2005.

\bibitem[Gab10]{Gab-survey}
D.~Gaboriau.
\newblock Orbit equivalence and measured group theory.
\newblock In {\em Proceedings of the International Congress of Mathematicians},
  volume III, pages 1501--1527, New Dehli, 2010. Hindustan Book Agency.

\bibitem[GH19]{GH1}
V.~Guirardel and C.~Horbez.
\newblock Algebraic laminations for free products and arational trees.
\newblock {\em Algebr. Geom. Topol.}, 19(5):2283--2400, 2019.

\bibitem[GH22]{GH}
V.~Guirardel and C.~Horbez.
\newblock Boundaries of relative factor graphs and subgroup classification for
  automorphisms of free products.
\newblock {\em Geom. Topol.}, 26(1):71--126, 2022.

\bibitem[GHL22]{GHL}
V.~Guirardel, C.~Horbez, and J.~L\'ecureux.
\newblock Cocycle superrigidity from higher rank lattices to {${\rm
  Out}(F_N)$}.
\newblock {\em J. Mod. Dyn.}, 18:291--344, 2022.

\bibitem[GL]{GL-on}
V.~Guirardel and G.~Levitt.
\newblock {\em in preparation}.

\bibitem[GL95]{GL}
D.~Gaboriau and G.~Levitt.
\newblock The rank of actions on $\mathbb{R}$-trees.
\newblock {\em Ann. Sci. \'Ecole Norm. Sup. (4)}, 28(5):549--570, 1995.

\bibitem[GL07a]{GL_deformation}
V.~{Guirardel} and G.~{Levitt}.
\newblock {Deformation spaces of trees}.
\newblock {\em {Groups Geom. Dyn.}}, 1(2):135--181, 2007.

\bibitem[GL07b]{GL-os}
V.~Guirardel and G.~Levitt.
\newblock The outer space of a free product.
\newblock {\em Proc. London Math. Soc.}, 94(3):695--714, 2007.

\bibitem[GL09]{GabLyons}
D.~Gaboriau and R.~Lyons.
\newblock A measurable-group-theoretic solution to von {N}eumann's problem.
\newblock {\em Invent. Math.}, 177(3):533--540, 2009.

\bibitem[GL10]{GL5}
V.~Guirardel and G.~Levitt.
\newblock {S}cott and {S}warup's regular neighbourhood as a tree of cylinders.
\newblock {\em Pacific J. Math.}, 245(1):79--98, 2010.

\bibitem[GL11]{GL-cyl}
V.~Guirardel and G.~Levitt.
\newblock Trees of cylinders and canonical splittings.
\newblock {\em Geom. Topol.}, 15(2):977--1012, 2011.

\bibitem[GL15a]{GL4}
V.~Guirardel and G.~Levitt.
\newblock Mc{C}ool groups of toral relatively hyperbolic groups.
\newblock {\em Algebr. Geom. Topol.}, 15(6):3485--3534, 2015.

\bibitem[GL15b]{GL_automorphisms}
V.~Guirardel and G.~Levitt.
\newblock Splittings and automorphisms of relatively hyperbolic groups.
\newblock {\em Groups Geom. Dyn.}, 9(2):599--663, 2015.

\bibitem[GL17]{GL-jsj}
V.~Guirardel and G.~Levitt.
\newblock J{SJ} decompositions of groups.
\newblock {\em Ast\'{e}risque}, (395):vii+165, 2017.

\bibitem[GN19]{GN}
D.~Gaboriau and C.~Noûs.
\newblock On the top-dimensional $\ell^2$-{B}etti numbers.
\newblock {\em arXiv:1909.01633}, 2019.

\bibitem[GP05]{GabPopa}
D.~Gaboriau and S.~Popa.
\newblock An uncountable family of nonorbit equivalent actions of {$\Bbb F_n$}.
\newblock {\em J. Amer. Math. Soc.}, 18(3):547--559, 2005.

\bibitem[Gro93]{Gro}
M.~Gromov.
\newblock Asymptotic invariants of infinite groups.
\newblock In {\em Geometric group theory}, volume~2 of {\em London Math. Soc.
  Lecture Note Ser.}, pages 1--295, Cambridge Univ. Press, Cambridge, 1993.

\bibitem[Gru40]{Grushko}
I.~Grushko.
\newblock On the bases of a free product of groups.
\newblock {\em Mat. Sbornik}, 8:168--182, 1940.

\bibitem[GS91]{GS}
S.M. Gersten and J.R. Stallings.
\newblock Irreducible outer automorphisms of a free group.
\newblock {\em Proc. Amer. Math. Soc.}, 111(2):309--314, 1991.

\bibitem[Gui00]{Gui00}
V.~Guirardel.
\newblock Dynamics of $\text{{O}ut}({F}_n)$ on the boundary of outer space.
\newblock {\em Ann. Sci. \'Ecole Norm. Sup. (4)}, 33(4):433--465, 2000.

\bibitem[Gui04]{Gui04}
V.~Guirardel.
\newblock Limit groups and groups acting freely on $\mathbb{R}^n$-trees.
\newblock {\em Geom. Topol.}, 8:1427--1470, 2004.

\bibitem[Gui05]{Gui_core}
V.~Guirardel.
\newblock C\oe{}ur et nombre d'intersection pour les actions de groupes sur les
  arbres.
\newblock {\em Ann. Sci. \'Ecole Norm. Sup. (4)}, 38(6):847--888, 2005.

\bibitem[Gui08]{Gui_actions}
V.~Guirardel.
\newblock Actions of finitely generated groups on {$\Bbb R$}-trees.
\newblock {\em Ann. Inst. Fourier (Grenoble)}, 58(1):159--211, 2008.

\bibitem[Ham12]{Ham2}
U.~Hamenstädt.
\newblock The boundary of the free splitting graph and the free factor graph.
\newblock {\em arXiv:1211.1630}, 2012.

\bibitem[HH20]{HH}
C.~Horbez and J.~Huang.
\newblock Boundary amenability and measure equivalence rigidity among
  two-dimensional {A}rtin groups of hyperbolic type.
\newblock {\em arXiv:2004.09325}, 2020.

\bibitem[HH22]{HH2}
C.~Horbez and J.~Huang.
\newblock Measure equivalence classification of transvection-free right-angled
  {A}rtin groups.
\newblock {\em J. \'Ec. polytech. Math.}, 9:1021--1067, 2022.

\bibitem[HHL23]{HHL}
C.~Horbez, J.~Huang, and J.~L\'ecureux.
\newblock Proper proximality in non-positive curvature.
\newblock {\em Amer. J. Math.}, 145(5):1327--1364, 2023.

\bibitem[Hjo05]{Hjorth}
G.~Hjorth.
\newblock A converse to {D}ye's theorem.
\newblock {\em Trans. Amer. Math. Soc.}, 357(8):3083--3103, 2005.

\bibitem[HM13]{HM3}
M.~Handel and L.~Mosher.
\newblock The free splitting complex of a free group, {I}: hyperbolicity.
\newblock {\em Geom. Topol.}, 17(3):1581--1672, 2013.

\bibitem[HM14]{HM}
M.~Handel and L.~Mosher.
\newblock Relative free splitting and free factor complexes {I}:
  {H}yperbolicity.
\newblock {\em arXiv:1407.3508}, 2014.

\bibitem[HM19]{HM4}
M.~Handel and L.~Mosher.
\newblock The free splitting complex of a free group, {II}: {L}oxodromic outer
  automorphisms.
\newblock {\em Trans. Amer. Math. Soc.}, 372(6):4053--4105, 2019.

\bibitem[HM20]{HM2}
M.~Handel and L.~Mosher.
\newblock Subgroup decomposition in $\mathrm{Out}({F}_n)$.
\newblock {\em Mem. Amer. Math. Soc.}, 2020.

\bibitem[Hor14]{Hor}
C.~Horbez.
\newblock The {T}its alternative for the automorphism group of a free product.
\newblock {\em arXiv:1408.0546v2}, 2014.

\bibitem[Hor17]{Hor1}
C.~Horbez.
\newblock The boundary of the outer space of a free product.
\newblock {\em Israel J. Math.}, 221(1):179--234, 2017.

\bibitem[HR15]{HR}
C.~Houdayer and S.~Raum.
\newblock Baumslag-{S}olitar groups, relative profinite completions and measure
  equivalence rigidity.
\newblock {\em J. Topol.}, 8(1):295--313, 2015.

\bibitem[HW20]{HW}
C.~Horbez and R.D. Wade.
\newblock Commensurations of subgroups of $\mathrm{{O}ut}({F}_{N})$.
\newblock {\em Trans. Amer. Math. Soc.}, 373(4):2699--2742, 2020.

\bibitem[Ioa11]{Ioana}
A.~Ioana.
\newblock Orbit inequivalent actions for groups containing a copy of {$\Bbb
  F_2$}.
\newblock {\em Invent. Math.}, 185(1):55--73, 2011.

\bibitem[Iva97]{Iva}
N.V. Ivanov.
\newblock Automorphism of complexes of curves and of {T}eichmüller spaces.
\newblock {\em Int. Math. Res. Not.}, 14:651--666, 1997.

\bibitem[JKL02]{JKL}
S.~Jackson, A.S. Kechris, and A.~Louveau.
\newblock Countable {B}orel equivalence relations.
\newblock {\em J. Math. Logic}, 2(1):1--80, 2002.

\bibitem[Kec95]{Kec}
A.S. Kechris.
\newblock {\em Classical descriptive set theory}, volume 156 of {\em Graduate
  Texts in Mathematics}.
\newblock Springer-Verlag, New York, 1995.

\bibitem[Khr90]{Khr}
D.G. Khramtsov.
\newblock Completeness of groups of outer automorphisms of free groups.
\newblock In {\em Group-theoretic investigations (Russian)}, pages 128--143,
  Sverdlovsk, 1990. Akad. Nauk SSSR Ural. Otdel.

\bibitem[Kid08a]{Kid1}
Y.~Kida.
\newblock The mapping class group from the viewpoint of measure equivalence
  theory.
\newblock {\em Mem. Amer. Math. Soc.}, 196(916), 2008.

\bibitem[Kid08b]{Kid3}
Y.~Kida.
\newblock Orbit equivalence rigidity for ergodic actions of the mapping class
  group.
\newblock {\em Geom. Dedic.}, 131(1):99--109, 2008.

\bibitem[Kid09]{Kid-survey}
Y.~Kida.
\newblock Introduction to measurable rigidity of mapping class groups.
\newblock In {\em Handbook of {T}eichmüller theory, volume II}, volume~13 of
  {\em IRMA Lect. Math. Theor. Phys.}, pages 297--367, Zürich, 2009. Eur.
  Math. Soc.

\bibitem[Kid10]{Kid2}
Y.~Kida.
\newblock Measure equivalence rigidity of the mapping class group.
\newblock {\em Ann. of Math. (2)}, 171(3):1851--1901, 2010.

\bibitem[Kid11]{Kid4}
Y.~Kida.
\newblock Rigidity of amalgamated free products in measure equivalence.
\newblock {\em J. Topol.}, 4(3):687--735, 2011.

\bibitem[KSS06]{KSS}
I.~Kapovich, P.~Schupp, and V.~Shpilrain.
\newblock Generic properties of {W}hitehead's algorithm and isomorphism
  rigidity of random one-relator groups.
\newblock {\em Pacific J. Math.}, 223(1):113--140, 2006.

\bibitem[LdlS22]{LdlS}
P.-H. Leemann and M.~de~la Salle.
\newblock Cayley graphs with few automorphisms: the case of infinite groups.
\newblock {\em Ann. H. Lebesgue}, 5:73--92, 2022.

\bibitem[Lev05]{Lev}
G.~Levitt.
\newblock Automorphisms of {H}yperbolic {G}roups and {G}raphs of {G}roups.
\newblock {\em Geom. Dedic.}, 114(1):49--70, 2005.

\bibitem[MM00]{MM}
H.A. Masur and Y.N. Minsky.
\newblock Geometry of the complex of curves {II}. {H}ierarchical structure.
\newblock {\em Geom. Funct. Anal.}, 10(4):902--974, 2000.

\bibitem[Mor88]{Mor}
J.W. Morgan.
\newblock Ergodic theory and free actions of groups on {${\bf R}$}-trees.
\newblock {\em Invent. Math.}, 94(3):605--622, 1988.

\bibitem[MS06]{MS}
N.~Monod and Y.~Shalom.
\newblock Orbit equivalence rigidity and bounded cohomology.
\newblock {\em Ann. of Math. (2)}, 164:825--878, 2006.

\bibitem[MV04]{MV}
A.~Martino and E.~Ventura.
\newblock Fixed subgroups are compressed in free groups.
\newblock {\em Comm. Algebra}, 32(10):3921--3935, 2004.

\bibitem[OW80]{OW}
D.S. Ornstein and B.~Weiss.
\newblock Ergodic theory of amenable group actions. {I}. {T}he {R}ohlin lemma.
\newblock {\em Bull. Amer. Math. Soc. (N.S.)}, 2(1):161--164, 1980.

\bibitem[Pan14]{Pan}
S.~Pandit.
\newblock The complex of non-separating embedded spheres.
\newblock {\em Rocky Mountain J. Math.}, 44(6):2029--2054, 2014.

\bibitem[Pau88]{Pau}
F.~Paulin.
\newblock Topologie de {G}romov équivariante, structures hyperboliques et
  arbres réels.
\newblock {\em Invent. Math.}, 94(1):53--80, 1988.

\bibitem[Pau89]{Pau3}
F.~Paulin.
\newblock The {G}romov topology on $\mathbb{{R}}$-trees.
\newblock {\em Topology Appl.}, 32(3):197--221, 1989.

\bibitem[Pau91]{Pau2}
F.~Paulin.
\newblock Outer automorphisms of hyperbolic groups and small actions on
  $\mathbb{R}$-trees.
\newblock In {\em Arboreal group theory (Berkeley, CA, 1988)}, pages 331--343.
  Springer, 1991.

\bibitem[Per11]{Perin}
C.~Perin.
\newblock Elementary embeddings in torsion-free hyperbolic groups.
\newblock {\em Ann. Sci. \'{E}c. Norm. Supér. (4)}, 44(4):631--681, 2011.

\bibitem[Pop07]{Pop}
S.~Popa.
\newblock Cocycle and orbit equivalence superrigidity for malleable actions of
  $w$-rigid groups.
\newblock {\em Invent. Math.}, 170(2):243--295, 2007.

\bibitem[Ren15]{Ren}
J.~Renault.
\newblock Topological {A}menability {I}s a {B}orel {P}roperty.
\newblock {\em Math. Scand.}, 117(1):5--30, 2015.

\bibitem[Rey12]{Rey}
P.~Reynolds.
\newblock Reducing systems for very small trees.
\newblock {\em arXiv:1211.3378}, 2012.

\bibitem[Sel97]{Sel}
Z.~Sela.
\newblock Structure and rigidity in ({G}romov) hyperbolic groups and discrete
  groups in rank $1$ {L}ie groups. {II}.
\newblock {\em Geom. Funct. Anal.}, 7(3):561--593, 1997.

\bibitem[Ser77]{Ser}
J.-P. Serre.
\newblock {\em Arbres, amalgames, {${\rm SL}_{2}$}}.
\newblock Soci\'{e}t\'{e} Math\'{e}matique de France, Paris, 1977.
\newblock Avec un sommaire anglais, R\'{e}dig\'{e} avec la collaboration de
  Hyman Bass, Ast\'{e}risque, No. 46.

\bibitem[Sha05]{Sha-survey}
Y.~Shalom.
\newblock Measurable group theory.
\newblock In {\em European Congress of Mathematics}, pages 391--423, Eur. Math.
  Soc., Zürich, 2005.

\bibitem[Sin55]{Sin}
I.M. Singer.
\newblock Automorphisms of finite factors.
\newblock {\em Amer. J. Math.}, 77:117--133, 1955.

\bibitem[Sta82]{Sta}
J.R. Stallings.
\newblock Topologically unrealizable automorphisms of free groups.
\newblock {\em Proc. Amer. Math. Soc.}, 84(1):21--24, 1982.

\bibitem[Zim78]{Zim2}
R.J. Zimmer.
\newblock Amenable ergodic group actions and an application to {P}oisson
  boundaries of random walks.
\newblock {\em J. Funct. Anal.}, 27(3):350--372, 1978.

\bibitem[Zim80]{Zim3}
R.J. Zimmer.
\newblock Strong rigidity for ergodic actions of semisimple groups.
\newblock {\em Ann. of Math. (2)}, 112(3):511--529, 1980.

\bibitem[Zim84]{Zim}
R.J. Zimmer.
\newblock {\em Ergodic {T}heory and {S}emisimple {G}roups}, volume~81 of {\em
  Monographs in Mathematics}.
\newblock Birkhäuser, Basel; Boston; Stuttgart, 1984.

\bibitem[Zim91]{Zim4}
R.J. Zimmer.
\newblock Groups generating transversals to semisimple {L}ie group actions.
\newblock {\em Israel J. Math.}, 73(2):151--159, 1991.

\end{thebibliography}

 \begin{flushleft}
 Vincent Guirardel\\
 Univ Rennes,  CNRS, IRMAR - UMR 6625, F-35000 Rennes, France\\
 \emph{e-mail: }\texttt{vincent.guirardel@univ-rennes1.fr}\\[8mm]
 \end{flushleft}

\begin{flushleft}
Camille Horbez\\ 
Universit\'e Paris-Saclay, CNRS,  Laboratoire de math\'ematiques d'Orsay, 91405, Orsay, France \\
\emph{e-mail:~}\texttt{camille.horbez@universite-paris-saclay.fr}\\[4mm]
\end{flushleft}

\end{document}